\documentclass[aos,preprint]{imsart}

\usepackage[authoryear]{natbib}
\usepackage{color}
\usepackage{ifthen}

\usepackage[lined,algonl,boxed]{algorithm2e}
\usepackage{graphicx}
\usepackage{latexsym}
\usepackage{amssymb,amsbsy}
\usepackage{amsmath}
\usepackage{amsfonts,amsthm,enumerate,CJK,array,mathrsfs}
\usepackage{bm,dsfont,multicol}
\usepackage{caption}
\usepackage[hypertexnames=false]{hyperref}
\usepackage{chngcntr} 
\usepackage{pdflscape}
\usepackage{verbatim}
\usepackage{subfig}
\usepackage{float}
\usepackage{pifont}
\usepackage[mathscr]{euscript}

\arxiv{arXiv:1701.06088}

\startlocaldefs
\newboolean{tab}
\setboolean{tab}{false}

\makeatletter\@addtoreset{equation}{section}\makeatother
\renewcommand{\theequation}{\thesection.\arabic{equation}}

\newcommand{\Var}{\mbox{Var}}
\def\eps{\varepsilon}
\def \R{\mathbb{R}}
\def \E{\mathbb{E}}
\def \N{\mathbb{N}}

\def \Cov{\mbox{Cov}}

\newcommand{\Ac}{\mathcal{A}}

\newcommand{\Xc}{\mathcal{X}}
\newcommand{\Gc}{\mathcal{G}}

\newcommand{\Ec}{\mathcal{E}}

\newcommand{\Fc}{\mathcal{F}}
\newcommand{\Zc}{\mathcal{Z}}

\newcommand{\Nc}{\mathcal{N}}
\newcommand{\Ic}{\mathcal{I}}
\newcommand{\Jc}{\mathcal{J}}
\newcommand{\Tc}{\mathcal{T}}

\newcommand{\Sc}{\mathcal{S}}
\newcommand{\Pc}{\mathcal{P}}
\newcommand{\Uc}{\mathcal{U}}

\newcommand{\Ps}{\mathscr{P}}

\newcommand{\ba}{\mathbf{a}}
\newcommand{\bb}{\mathbf{b}}

\newcommand{\eee}{\mathbf{e}}

\newcommand{\bj}{\mbox{\boldmath $j$}}

\newcommand{\bl}{\mathbf{l}}

\newcommand{\bv}{\mathbf{v}}
\newcommand{\bu}{\mathbf{u}}

\newcommand{\zz}{\mathbf{z}}

\newcommand{\ZZ}{\mathbf{Z}}
\newcommand{\BB}{\mathbf{B}}

\newcommand{\WW}{\mathbf{W}}

\newcommand{\za}{\boldsymbol\alpha}
\newcommand{\zb}{\boldsymbol\beta}
\newcommand{\zg}{\boldsymbol\gamma}

\newcommand{\bU}{\bm{U}}

\newcommand{\weak}{\rightsquigarrow}

\newcommand{\Pb}{\mathbb{P}}
\newcommand{\Gb}{\mathbb{G}}

\newcommand{\bean}{\begin{eqnarray*}}
\newcommand{\eean}{\end{eqnarray*}}
\newcommand{\bea}{\begin{eqnarray}}
\newcommand{\eea}{\end{eqnarray}}
\newcommand{\be}{\begin{eqnarray}}
\newcommand{\ee}{\end{eqnarray}}
\newcommand{\beq}{\begin{equation}}
\newcommand{\eeq}{\end{equation}}

\renewcommand{\hat}{\widehat}
\renewcommand{\tilde}{\widetilde}

\newcommand{\IF}{\boldsymbol{1}}  

\newtheorem{theo}{Theorem}[section]
\newtheorem{lemma}[theo]{Lemma}
\newtheorem{cor}[theo]{Corollary}
\newtheorem{rem}[theo]{Remark}

\newtheorem{example}[theo]{Example}
\newtheorem{remark}[theo]{Remark}

\counterwithin{table}{section}
\counterwithin{figure}{section} 
\renewenvironment{proof}[1][\proofname]{{\noindent\bfseries #1.}}{\qed} 

\DeclareMathOperator*\argmin{\mbox{argmin}}

\endlocaldefs

\counterwithin{table}{section}

\begin{document}

\begin{frontmatter}
\title{Distributed Inference for Quantile Regression Processes} 
\runtitle{Distributed Inference for Quantile Regression Processes}
\thankstext{T1}{The authors would like to thank the Editors and two anonymous Referees for helpful comments that helped to considerably improve an earlier version of this manuscript.}
\begin{aug}
\author{\fnms{Stanislav} \snm{Volgushev}\thanksref{m1,t3}},
\author{\fnms{Shih-Kang} \snm{Chao}\thanksref{m2,t1}}
\and
\author{\fnms{Guang} \snm{Cheng}\thanksref{m2,t2,t1}}

%
\thankstext{t1}{Partially supported by Office of Naval Research (ONR N00014-15-1-2331).}
\thankstext{t2}{Partially supported by NSF CAREER Award DMS-1151692, DMS-1418042.}
\thankstext{t3}{Partially supported by discovery grant from Natural Sciences and Engineering Research Council of Canada.}
%
%
\affiliation{University of Toronto\thanksmark{m1} and Purdue University\thanksmark{m2}}
%
%
%
\end{aug}


\begin{abstract}
	The increased availability of massive data sets provides a unique opportunity to discover subtle patterns in their distributions, but also imposes overwhelming computational challenges. To fully utilize the information contained in big data, we propose a two-step procedure: (i) estimate conditional quantile functions at different levels in a parallel computing environment; (ii) construct a conditional quantile regression process through projection based on these estimated quantile curves. Our general quantile regression framework covers both linear models with fixed or growing dimension and series approximation models. We prove that the proposed procedure does not sacrifice any statistical inferential accuracy provided that the number of distributed computing units and quantile levels are chosen properly. In particular, a sharp upper bound for the former and a sharp lower bound for the latter are derived to capture the minimal computational cost from a statistical perspective. As an important application, the statistical inference on conditional distribution functions is considered. Moreover, we propose computationally efficient approaches to conducting inference in the distributed estimation setting described above. Those approaches directly utilize the availability of estimators from sub-samples and can be carried out at almost no additional computational cost. Simulations confirm our statistical inferential theory.
\end{abstract}
	
		

\end{frontmatter}

\section{Introduction}\label{sec:intro}

The advance of technology in applications such as meteorological and environmental surveillance or e-commerce has led to extremely large data sets which cannot be processed with standalone computers due to processor, memory, or disk bottlenecks. Dividing data and computing among many machines is a common way to alleviate such bottlenecks, and can be implemented by parallel computing systems like Hadoop \citep{W12} with the aid of communication-efficient algorithms. 

In statistics, this approach to estimation has recently been adopted under the name divide-and-conquer, see \cite{LLL12,J13}. Pioneering contributions on theoretical analysis of divide and conquer algorithms focus on mean squared error rates \cite{ZDW13,ZDW15}. The analysis therein does not take into account a core question in statistics -- inferential procedures. In the last two years, such procedures have been developed for various non- and semi-parametric estimation approaches that focus on the mean or other notions of central tendency, this line of work includes \cite{CS15,ZCL15,SLS16,BDS16}, among others. Focusing on the mean tendency illuminates one important aspect of the dependence between predictors and response, but ignores all other aspects of the conditional distribution of the response which can be equally important. Additionally, most of the work cited above assumes homoskedastic errors or sub-Gaussian tails. Both assumptions are easily violated in modern messy and large-scale data sets, and this limits the applicability of approaches that are available to date.  

We propose to use quantile regression in order to extract features of the conditional distribution of the response while avoiding tail conditions and taking heteroskedasticity into account. This approach focuses on estimating the conditional quantile function $x \mapsto Q(x;\tau) := F_{Y|X}^{-1}(\tau|x)$ where $F_{Y|X}$ denotes the conditional distribution of response given predictor. A flexible class of models for conditional quantile functions can be obtained by basis expansions of the form $Q(x;\tau) \approx \ZZ(x)^\top \zb(\tau)$. This includes linear models of fixed or increasing dimension where the approximation is replaced by an equality, as well as a rich set of non- and semi-parametric procedures such as tensor product B-splines or additive models.  

Given data $\{(X_i,Y_i)\}_{i=1}^N$, quantile regression for such models is classically formulated through the following minimization problem:   
\begin{align}
\widehat \zb_{or}(\tau) := \arg\min_{\bb\in\R^m} \sum_{i=1}^N \rho_\tau\{Y_{i} - \bb^\top \ZZ(X_i)\}, \label{eq:or}
\end{align}
where $\rho_\tau(u) := \{\tau-\IF(u \leq 0)\}u$ and $\IF(\cdot)$ is the indicator function. However, solving this minimization problem by classical algorithms requires that the complete data set can be processsed on a single computer. For large samples, this might not be feasible. This motivates us to utilize the divide and conquer approach where the full sample is randomly split across several computers into $S$ smaller sub-samples of size $n$, the minimization in \eqref{eq:or} is solved on each sub-sample, and results are averaged in the end to obtain the final estimator $\overline{\zb}(\tau)$ (see \eqref{eq:zbbar} in Section~\ref{sec:DCgen} for a formal definition). While this approach is easy to implement, the resulting estimator $\overline{\zb}(\tau)$ is typically not a solution to the original minimization problem~\eqref{eq:or}. As illustrated in Figure~\ref{fig:intro}, this can be problematic for inference procedures. More precisely, the left panel of Figure~\ref{fig:intro} depicts the coverage probabilities (on the y-axis) against number of sub-samples (on the x-axis) for asymptotic confidence intervals that are based on asymptotic distributions of $\hat\zb_{or}(\tau)$ but computed using $\overline \zb(\tau)$ for three data generating processes (linear models with different dimension) and a fixed quantile $\tau = 0.1$. This indicates that choosing $S$ too large results in invalid inference, as reflected by a rapid drop in coverage after a certain threshold is reached. For different models this threshold is different and it intrinsically depends on various properties of the underlying model such as dimension of predictors. These observations indicate that developing {\em valid} statistical inference procedures based on $\overline\zb(\tau)$ requires a careful theoretical analysis of divide and conquer procedures. The first major contribution of the present paper is to provide statistical inferential theory for this framework.  

\begin{figure}[!ht]
\centering
\includegraphics[width=4.8cm, height = 3.8cm]{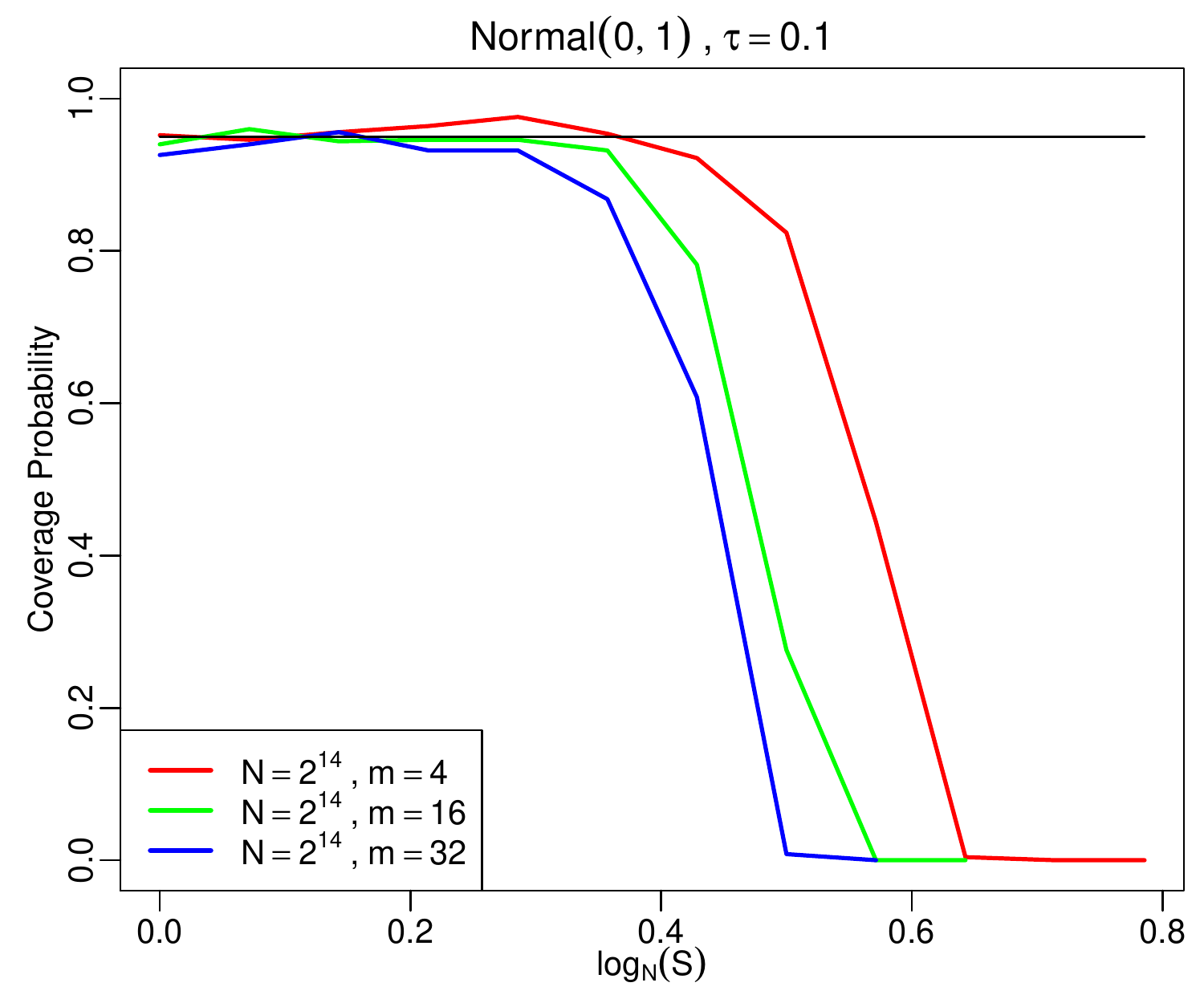}\hspace{1cm}
\includegraphics[width=4.8cm, height = 3.8cm]{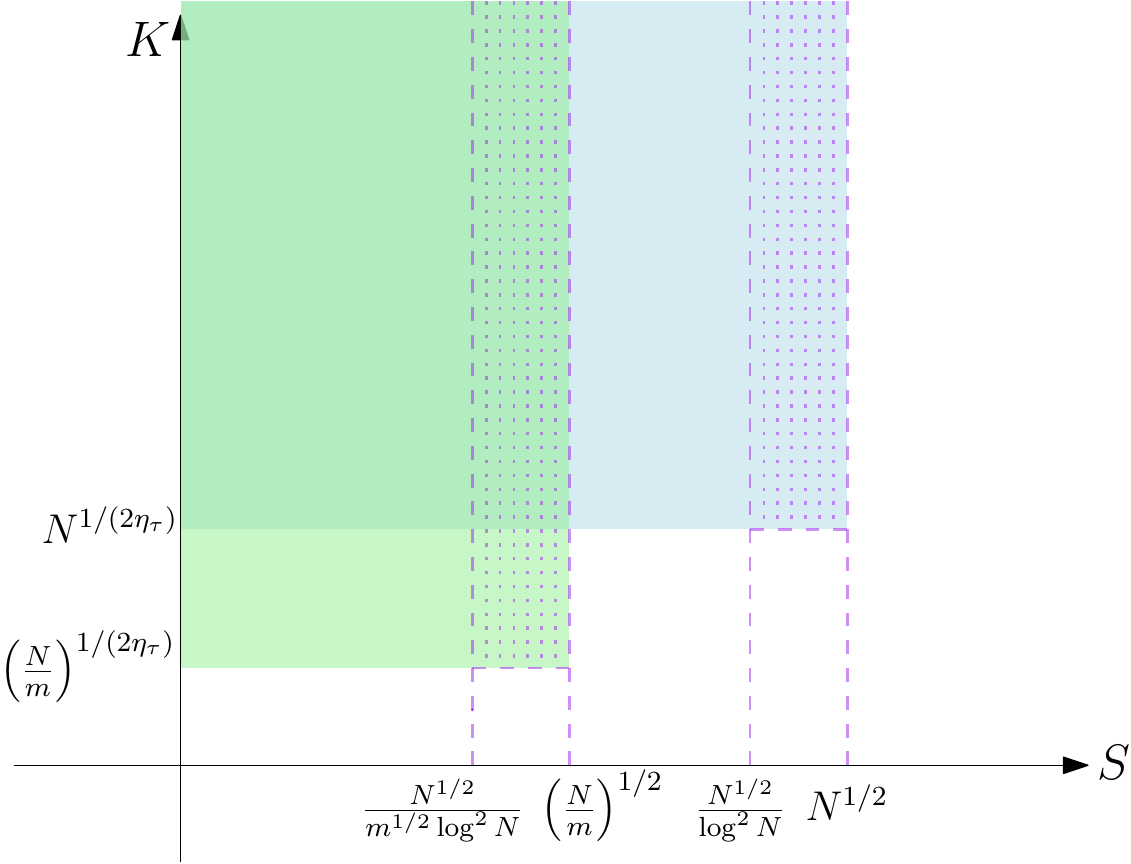}
\caption{Left penal: coverage probabilities (y-axis) of confidence intervals for estimators computed from divide and conquer; x-axis: number of sub-samples $S$ in log-scale. Different colors correspond to linear models with different $m=\dim(\ZZ(X))$
. Right panel: necessary and sufficient conditions on $(S,K)$ for the oracle rule in linear models with fixed dimension (Blue) and B-spline nonparametric models (Green). The dotted region is the discrepancy between the sufficient and necessary conditions. } \label{fig:intro}
\end{figure}

The approach described so far provides a way to estimate the conditional quantile function for a single quantile. To obtain a picture of the entire conditional distribution of response given predictors, estimation of the conditional quantile function at several quantile levels is required. Performing such an estimation for a dense grid of quantiles can be computationally costly as estimation at each new quantile requires running the above algorithm anew. To relax this computational burden, we introduce the two-step \textit{quantile projection algorithm}. In the first step, the divide and conquer algorithm is used to compute $\overline{\zb}(\tau_1),...,\overline{\zb}(\tau_K)$ for a grid of quantile values $\tau_1,...,\tau_K$. In the second step, a matrix $\hat \Xi$ is computed from $\overline{\zb}(\tau_1),...,\overline{\zb}(\tau_K)$ (see~\eqref{eq:hatXi} for a formal definition). Given this matrix $\hat\Xi$, the \textit{quantile projection estimator} $\hat \zb(\tau)$ can be computed for any $\tau$ by performing a multiplication of this matrix with a known vector (depending on $\tau$) without access to the original dataset. Based on $\hat \zb(\tau)$, we can estimate the conditional distribution function; see (\ref{eq:Fhat}) and (\ref{eq:cdf}).

The computational cost of our procedure is determined by the number of sub-samples $S$ and the number of quantile levels $K$. To minimize the computational burden, $K$ should be chosen as small as possible. Choosing $S$ large will allow us to maintain a small sample size $n$ on each local machine, thus enabling us to process larger data sets. At the same time, classical inference procedures for $\hat\zb_{or}$ (e.g. procedures based on asymptotic normality of $\hat\zb_{or}$) should remain valid for $\overline\zb(\tau), \hat\zb(\tau)$. 

A detailed analysis of conditions which ensure this "oracle rule" is conducted in Section~\ref{SEC:THEO}. There, we derive \textit{sufficient} conditions for both $S$ and $K$, which are also proved to be necessary in some cases up to $\log N$ terms. Sufficient conditions of $S$ take the form $S = o\big(N^{1/2} a_{N,m}\big)$ where $a_{N,m} \to 0$. The specific form of $a_{N,m}$ depends on the precise form of the model, with more complex models leading to $a_{N,m}$ converging to zero faster, thus placing more restrictions on $S$. Necessary and sufficient conditions on $K$ take the form $K \gg N^{1/(2\eta)} \|\ZZ(x)\|^{-1/\eta}$, where $\eta>0$ is the degree of H\"older continuity (see \eqref{eq:hoelder1}) of $\tau\mapsto Q(x;\tau)$. In particular, this lower bound is completely independent of $S$. An interesting insight which will be explained in more detail in Section \ref{sec:loc} is that for large values of $\|\ZZ\|$, which typically corresponds to more complex models, the restriction on $K$ becomes weaker; see the paragraph right after Corollary~\ref{th:orlocpr} for a more detailed explanation. 

A graphical summary of the necessary and sufficient conditions on $K$ and $S$ derived in Sections~\ref{sec:lin} and~\ref{sec:loc} is provided in the right panel of Figure~\ref{fig:intro}.  

Deriving such necessary conditions on $S,K$ is a crucial step in understanding the limitations of divide and conquer procedures, as it shows that the interplay between $S,K$ and model complexity is indeed a feature of the underlying problem and not an artifact of the proof; see also~\cite{CS15}. 

While the oracle rules described above provide justification for using the asymptotic distribution of $\hat \zb_{or}$ for inference, this distribution is typically not pivotal. We discuss several ways to overcome this problem. First, we propose simple and computationally efficient approaches which directly make use of the fact that $\bar{\zb}$ is formed by averaging results from independent sub-samples. Those approaches are based on normal and t-approximations as well as a novel bootstrap procedure. Second, we comment on extensions of traditional approaches to inference which rely on estimating the asymptotic variance of $\hat \zb_{or}$ by using kernel smoothing techniques. Simulations demonstrate finite sample  properties of the proposed inference strategies.

The rest of this paper is organized as follows. In Section~\ref{sec:DCgen}, we provide a formal description of all algorithms and estimation procedures discussed in this paper. Section~\ref{SEC:THEO} contains the main theoretical results, with additional results presented in Appendix~\ref{SEC:GEN}. In Section~\ref{SEC:INF}, we propose several practical approaches to inference. In Section~\ref{sec:sim}, we validate the finite sample relevance of our theoretical results through extensive Monte Carlo experiments. Proofs of all theoretical results are deferred to the online supplement. The following notation will be used throughout the paper.

\bigskip
\noindent\textbf{Notation.} 
Let $\Xc := \mbox{supp}(X)$. Let $\ZZ = \ZZ(X)$ and $\ZZ_i = \ZZ(X_i)$ and assume $\Tc=[\tau_L,\tau_U]$ for some $0<\tau_L<\tau_U<1$. Define the true underlying measure of $(Y_i,X_i)$ by $P$ and denote the corresponding expectation as $\E$. Denote by $\|\bb\|$ the $L^2$-norm 
of a vector $\bb$. $\lambda_{\min}(A)$ and $\lambda_{\max}(A)$ are the smallest and largest eigenvalue of a matrix 
$A$. Let $\lfloor \eta \rfloor$ denote the integer part of a real number $\eta$, and $|\bj|=j_1+...+j_d$ for $d$-tuple $\bj=(j_1,...,j_d)$. $\Sc^{m-1}\subset\R^m$ is the unit sphere. $a_n \asymp  b_n$ means that $(|a_n/b_n|)_{n\in \N}$ and $(|b_n/a_n|)_{n\in \N}$ are bounded, and $a_n \gg b_n$ means $a_n/b_n \to \infty$. Define the class of functions
\begin{align}
\Lambda_c^{\eta}(\Tc) :=
\Big\{&f \in \mathcal{C}^{\lfloor \eta \rfloor}(\Tc): \sup_{j \leq \lfloor \eta \rfloor} \sup_{\tau \in \Tc} |D^{j}f(\tau)|\leq c,\notag\\ 
&\sup_{j= \lfloor \eta \rfloor}\sup_{\tau\neq \tau'}\frac{|D^j f(\tau) - D^j f(\tau')|}{\|\tau-\tau'\|^{\eta - \lfloor \eta \rfloor}} \leq c \Big\},\label{eq:hoelder1}
\end{align}
where $\eta$ is called the "degree of H\"older continuity", and $\mathcal{C}^\alpha(\Xc)$ denotes the class of $\alpha$-continuously differentiable functions on a set $\Xc$.


\section{A two-step procedure for computing quantile process}\label{sec:DCgen}
In this section, we formally introduce two algorithms studied throughout the paper - \emph{divide and conquer} and \emph{quantile projection}. The former is used to estimate quantile functions at fixed quantile levels, based on which the latter constructs an estimator for the quantile process. We also note that algorithms presented below can be applied for processing data that are locally stored for any reason and not necessarily large. As an important application, the estimation of conditional distribution functions will be presented.

We consider a general approximately linear model:
\beq\label{eq:series0}
Q(x;\tau) \approx \ZZ(x)^\top \zb_N(\tau),
\eeq
where $Q(x;\tau)$ is the $\tau$-th quantile of the distribution of $Y$ conditional on $X = x\in \R^d$, and $\ZZ(x)\in \R^m$ is a vector of transformation of $x$. This framework \eqref{eq:series0} incorporates various estimation procedures, for instance, series approximation models; see \cite{bechfe2011} for more discussion. In this paper, we will focus on three classes of transformation $\ZZ(x)\in\R^m$ which include many models as special cases: fixed and finite $m$, diverging $m$ with local support structure (e.g. B-splines) and diverging $m$ \emph{without} local support structure. 

The \emph{divide and conquer} algorithm for estimating $\zb_N(\tau)$ at a fixed $\tau\in(0,1)$ is described below:
\begin{enumerate}
\item Divide the data $\{(X_i,Y_i)\}_{i=1}^N$ into $S$ sub-samples of size $n$\footnote{The equal sub-sample size is assumed for ease of presentation, and can be certainly relaxed.}. Denote the $s$-th sub-sample as $\{(X_{is},Y_{is})\}_{i=1}^n$ where $s=1,...,S$.

\item For each $s$ and fixed $\tau$, estimate the sub-sample based quantile regression coefficient as follows
\begin{equation}\label{eq:betas}
\widehat \zb^{s}(\tau) := \arg\,\min_{\zb \in \R^m} \sum_{i=1}^n \rho_\tau\big\{Y_{is} - \zb^\top \ZZ(X_{is})\big\}.
\end{equation}

\item Each local machine sends $\widehat \zb^{s}(\tau)\in\R^m$ to the master that outputs a \emph{pooled estimator}
\begin{align}
\overline\zb(\tau):=S^{-1}\sum_{s=1}^{S}\widehat\zb^s(\tau).\label{eq:zbbar}
\end{align}
\end{enumerate}

\begin{remark} \label{rem:unique} \rm The minimization problem in~\eqref{eq:betas} is in general not strictly convex. Thus, several minimizers could exist. In the rest of this paper, we will only focus on the minimizer with the smallest $\ell_2$ norm. This is assumed for presentational convenience, and a close look at the proofs reveals that all statements remain unchanged if any other minimizer is chosen instead.
\end{remark}

While $\overline\zb(\tau)$ defined in \eqref{eq:zbbar} gives an estimator at a fixed $\tau\in\Tc$, a complete picture of the conditional distribution is often desirable. To achieve this, we propose a two-step procedure. First compute $\overline\zb(\tau_k)\in\mathbb R^m$ for each $\tau_k\in\Tc_K$, where $\Tc_K \subset \Tc=[\tau_L,\tau_U]$ is grid of quantile values in $\Tc$ with $|\Tc_K|=K\in\N$. Second project each component of the vectors $\{\overline\zb(\tau_1),...,\overline\zb(\tau_K)\}$ on a space of spline functions in $\tau$. More precisely, let
\begin{align}
	\hat\za_j := \arg\,\min_{\za\in\R^{q}} \sum_{k=1}^K \big(\overline{\beta}_j(\tau_k)-\za^\top\BB(\tau_k)\big)^2, \quad j=1,...,m. \label{eq:defza}
\end{align}
where $\BB := (B_1,...,B_{q})^\top$ is a B-spline basis defined on $G$ equidistant knots $\tau_L = t_1<...<t_G = \tau_U$ with degree $r_\tau\in\N$ (we use the normalized version of B-splines as given in Definition 4.19 on page 124 in~\cite{schumaker:81}). Here, uniformity of knots is assumed for simplicity, all theoretical results can be extended to knots with quasi uniform partitions. Using $\hat\za_j$, we define the quantile projection estimator $\hat\zb(\cdot):\Tc\to\R^m$: 
\begin{equation} \label{eq:hatbeta}
\hat\zb(\tau) := \hat\Xi^\top \BB(\tau),
\end{equation} 
where $\hat\Xi$ is defined below. 
The $j$th element $\hat\beta_j(\tau)=\hat\za_j^\top\BB(\tau)$ can be viewed as projection, with respect to $\|f\|_K := (\sum_{k=1}^K f^2(\tau_k))^{1/2}$, of $\overline{\beta}_j$ onto the polynomial spline space with basis $B_1,...,B_q$. In what follows, this projection is denoted by $\Pi_K$. Although we focus on B-splines in this paper, other basis functions are certainly possible.

The algorithm for computing the quantile projection matrix $\hat\Xi$ is summarized below, here the divide and conquer algorithm is applied as a subroutine: 
\begin{enumerate}
	\item Define a grid of quantile levels $\tau_k = \tau_L + (k/K)(\tau_U - \tau_L)$ for $k = 1,...,K$. For each $\tau_k$, compute $\overline{\zb}(\tau_k)$ as \eqref{eq:zbbar}.
	
	\item For each $j=1,...,m$, compute 
	\begin{align}
	\hat \za_j = \bigg(\sum_{k=1}^K \BB(\tau_k)\BB(\tau_k)^\top\bigg)^{-1} \bigg(\sum_{k=1}^K \BB(\tau_k) \overline{\beta}_j(\tau_k)\bigg),\label{eq:za}
	\end{align}
	which is a closed form solution of \eqref{eq:defza}.
	\item Set the matrix 
	\begin{align}\label{eq:hatXi}
	\hat\Xi := [\hat\za_1\,\hat\za_2\,...\,\hat\za_m].
	\end{align}
\end{enumerate}
 
A direct application of the above algorithm is to estimate the quantile function for {\em any} $\tau\in\Tc$. 
A more important application of $\hat\zb(\tau)$ is the estimation of the conditional distribution function $F_{Y|X}(y|x)$ for any fixed $x$. More precisely, we consider
\begin{align}
\hat F_{Y|X}(y|x) := \tau_L + \int_{\tau_L}^{\tau_U} \IF\{\ZZ(x)^\top\hat \zb(\tau) < y\} d\tau,\label{eq:Fhat}
\end{align}
where $\tau_L$ and $\tau_U$ are chosen close to 0 and 1. The intuition behind this approach is the observation that
\begin{align}\label{eq:cdf}
\tau_L + \int_{\tau_L}^{\tau_U} \IF\{Q(x;\tau) < y\} d\tau = \left\{
\begin{array}{lcl}
\tau_L & \mbox{if} & F_{Y|X}(y|x) < \tau_L;
\\
F_{Y|X}(y|x) & \mbox{if} & \tau_L \leq F_{Y|X}(y|x) \leq \tau_U;
\\
\tau_U & \mbox{if} & F_{Y|X}(y|x) > \tau_U, 
\end{array}
\right.
\end{align}
see \cite{chfega2010,V13}. The estimator $\hat F_{Y|X}$ is a smooth functional of the map $\tau \mapsto \ZZ(x)^\top\hat \zb(\tau)$ \citep{chfega2010}. Hence, properties of $\hat F_{Y|X}$ depend on the behavior of $\ZZ(x)^\top\hat \zb(\tau)$ as a \emph{process} in $\tau$, which we will study in detail in Section \ref{SEC:THEO}.


\section{Theoretical analysis}\label{SEC:THEO}

The following regularity conditions are needed throughout this paper.
\begin{itemize}
\item[(A1)]\label{A1} Assume that $\|\ZZ_i\| \leq \xi_m < \infty$, where $\xi_m$ is allowed to diverge, and that $1/M\leq\lambda_{\min}(\E [\ZZ \ZZ^\top])\leq\lambda_{\max}(\E [\ZZ \ZZ^\top])\leq M$ holds uniformly in $N$ for some fixed constant $M$.
\item[(A2)]\label{A2} The conditional distribution $F_{Y|X}(y|x)$ is twice differentiable w.r.t. $y$, with the corresponding derivatives $f_{Y|X}(y|x)$ and $f'_{Y|X}(y|x)$. Assume $\bar{f} :=\sup_{y \in \R,x\in \Xc} |f_{Y|X}(y|x)| < \infty$, $\overline{f'} :=\sup_{y \in \R,x\in \Xc} |f'_{Y|X}(y|x)| < \infty$ uniformly in $N$.
\item[(A3)]\label{A3} Assume that uniformly in $N$, there exists a constant $f_{\min} < \overline{f}$ such that
\[
0 < f_{\min} \leq \inf_{\tau \in \Tc} \inf_{x \in \Xc} f_{Y|X}(Q(x;\tau)|x).
\]
\end{itemize}

In these assumptions, we explicitly work with triangular array asymptotics for $\{(X_i,Y_i)\}_{i=1}^N$, where $d=\dim(X_i)$ is allowed to grow as well. 

Assumptions \hyperref[A1]{(A1)} -- \hyperref[A3]{(A3)} are fairly standard in the quantile regression literature, see \cite{bechfe2011}. The eigenvalue condition in \hyperref[A1]{(A1)} is imposed for the purpose of normalization, which in turn determines the value of $\xi_m$. For basis expansions with local support such as B-splines, the right scaling to ensure the eigenvalue condition results is $\xi_m \asymp \sqrt{m}$.


\subsection{Fixed dimensional linear models} \label{sec:lin}

In this section, we assume for all $\tau \in \Tc$ and $x \in \Xc$
\begin{equation} \label{eq:qlin}
Q(x;\tau) = \ZZ(x)^\top \zb(\tau),
\end{equation}
where $\ZZ(x)$ has fixed dimension $m$. This simple model setup allows us to derive a simple and clean bound on the difference between $\overline{\zb}, \widehat{\zb}$ and the oracle estimator $\hat\zb_{or}$ and, leading to a better understanding of resulting oracle rules. Our first main result is as follows.

\begin{theo}\label{th:linub}
Suppose that~\eqref{eq:qlin} and assumptions~\hyperref[A1]{(A1)} -- \hyperref[A3]{(A3)} hold and that $K \ll N^2$, $S = o(N(\log N)^{-1})$. Then
\begin{equation}\label{eq:linubpo}
\sup_{\tau \in \Tc_K} \|\overline{\zb}(\tau) - \hat\zb_{or}(\tau)\| = O_P\Big(\frac{S \log N}{N} + \frac{S^{1/4}(\log N)^{7/4}}{N^{3/4}}\Big) + o_P(N^{-1/2}).
\end{equation}
If additionally $K \gg G \gg 1$ 
we also have
\begin{align*} 
\sup_{\tau \in \Tc} |\ZZ(x_0)^\top(\widehat{\zb}(\tau) - \hat\zb_{or}(\tau))| \leq~& O_P\Big(\frac{S \log N}{N} + \frac{S^{1/2}(\log N)^2}{N}\Big) + o_P(N^{-1/2})
\\
& + \sup_{\tau \in \Tc} |(\Pi_K Q(x_0;\cdot))(\tau) - Q(x_0;\tau)|
\end{align*}
where the projection operator $\Pi_K$ was defined right after~\eqref{eq:hatbeta}.
\end{theo}

The proof of Theorem~\ref{th:linub} is given in Section \ref{sec:proof_par}. To obtain this result, we develop a new Bahadur representation for each local estimator $\hat\zb^s(\tau)$, see Section~\ref{sec:ABR}. The main novelty compared to existing results is a sharp bound on the \textit{expectation} of the remainder term. This is crucial to obtain the bound in~\eqref{eq:linubpo}. In contrast, relying on previously available versions of the Bahadur representation would result in a bound of the form $(S/N)^{3/4}(\log N)^{3/4}$, which is somewhat more restrictive. See Remark~\ref{rem:fixeddiminter} for additional details. Theorem~\ref{th:linub} can be generalized to potentially miss-specified linear models with dimension that depends on the sample size. Technical details are provided in the Appendix, Section~\ref{SEC:GEN}.

The bound in~\eqref{eq:linubpo} quantifies the difference between $\overline{\zb}$ and $\hat\zb_{or}$ in terms of the number of sub-samples $S$ and can be equivalently formulated in terms of the sub-sample size $n = N/S$ provided that $\log N \sim \log n$. When considering the projection estimator, an additional bias term $\Pi_K Q(x_0;\cdot) - Q(x_0;\cdot)$ is introduced. Provided that for a given $x_0 \in \Xc$ one has $\tau \mapsto Q(x_0;\tau) \in \Lambda_c^{\eta}(\Tc)$, this bias term can be further bounded by $O(G^{-\eta})$. Note that in the setting of Theorem~\ref{th:linub} the oracle estimator converges to $\zb(\tau)$ at rate $O_P(N^{-1/2})$, uniformly in $\tau \in \Tc$. By combining the results in Theorem~\ref{th:linub} with this bound, upper bounds on convergence rates for $\overline{\zb}(\tau) - \zb(\tau)$ and $\sup_{\tau \in \Tc} |\ZZ(x_0)^\top(\widehat{\zb}(\tau) - Q(x_0;\tau)|$ follow as direct Corollaries.

Theorem~\ref{th:linub} only provides upper bounds on the differences between $\overline{\zb}, \widehat{\zb}$ and $\hat\zb_{or}$. While a more detailed expression for this difference is derived in the proof of Theorem~\ref{th:aggbah}, this expression is very unwieldy and does not lead to useful explicit formulas. However, it is possible to prove that the bounds given in Theorem~\ref{th:linub} are sharp up to $\log N$ factors, which is the main result of the next Theorem. Before stating this result, we need to introduce some additional notation. Denote by $\Ps_1(\xi_m,M,\bar{f},\overline{f'},f_{min})$ all pairs $(P,\ZZ)$ of distributions $P$ and transformations $\ZZ$ satisfying~\eqref{eq:qlin} and \hyperref[A1]{(A1)}-\hyperref[A3]{(A3)} with constants $0 < \xi_m, M,\bar{f},\overline{f'} < \infty$, $f_{min} > 0$. Since $m,\xi_m$ are constant in this section, we use the shortened notation $\Ps_1(\xi,M,\bar{f},\overline{f'},f_{min})$.

\begin{theo}\label{th:lilob}
For any $\tau$ in $\Tc$ there exists $(P,\ZZ) \in \Ps_1(\xi,M,\bar{f},\overline{f'},f_{min})$ and a $C>0$ such that
\begin{equation}\label{eq:lilobpo}
\limsup_{N \to \infty} P\Big(\|\overline{\zb}(\tau) - \hat\zb_{or}(\tau)\| \geq \frac{CS}{N} \Big) > 0. 
\end{equation} 
Moreover for any $c, \eta > 0$ there exist $(P,\ZZ) \in \Ps_1(\xi,M,\bar{f},\overline{f'},f_{min})$ such that $\tau \mapsto \beta_j(\tau) \in \Lambda_c^{\eta}(\Tc), j=1,...,d$ and
\begin{equation}\label{eq:lilobpr}
\limsup_{N \to \infty} P\Big(\sup_{\tau \in \Tc}\|\widehat{\zb}(\tau) - \hat\zb_{or}(\tau)\| \geq \frac{CS}{N} + CG^{-\eta}\Big) > 0. 
\end{equation} 
\end{theo}

%
%

The proof of Theorem~\ref{th:lilob} is given in Section \ref{sec:proof_par}. The result provided in~\eqref{eq:lilobpo} has an interesting implication: the best possible precision of estimating $Q$ in a divide-and-conquer framework is restricted by $n = N/S$, the sample size that can be processed on a single machine, regardless of the total sample size $N$. A related observation was made in Example~1 of~\cite{ZDW13} who construct a data-generating process where the MSE rate of a divide and conquer estimator is limited by the sample size that can be processed on a single computer. 




As corollaries to the above results, we derive sufficient and necessary conditions on $S$ under which $\overline{\zb}$ and $\ZZ(x)^\top\widehat{\zb}$ satisfy the oracle rule. Note that the asymptotic distribution of  the oracle estimator $\hat\zb_{or}(\tau)$ under various conditions has been known for a long time, see for instance Theorem 4.1 of \cite{K2005}. Under \hyperref[A1]{(A1)}-\hyperref[A3]{(A3)} it was developed in \cite{bechfe2011} and \cite{ChaVolChe2016} who show that
\begin{align}\label{eq:fixedpro}
\sqrt{N}\big(\widehat{\zb}_{or}(\cdot) - \zb(\cdot)\big) \weak \Gb(\cdot) \mbox{ in } (\ell^\infty(\Tc))^d,
\end{align}
where $\Gb$ is a centered Gaussian process with covariance structure 
\begin{multline} \label{eq:covlin}
H(\tau,\tau') := \E\big[\Gb(\tau)\Gb(\tau')^\top\big] 
\\
= J_m(\tau)^{-1} \E\big[\ZZ(X)\ZZ(X)^\top\big] J_m(\tau')^{-1}(\tau\wedge\tau' - \tau\tau').
\end{multline}
where $J_m(\tau) := \E[\ZZ\ZZ^\top f_{Y|X}(Q(X;\tau)|X)]$.

\begin{cor}[Oracle rules for $\overline{\zb}$] \label{th:orlipo} 
A sufficient condition for $\sqrt{N}(\bar \zb(\tau) - \zb(\tau)) \weak \Nc(0,H(\tau,\tau))$ for any $(P,\ZZ)\in \Ps_1(\xi,M,\bar{f},\overline{f'},f_{min})$ is that $S = o(N^{1/2}/\log N)$. A necessary condition for the same result is that $S = o(N^{1/2})$.
\end{cor}

The necessary and sufficient conditions above match up to a factor of $\log N$. In order to guarantee the oracle rule for the process, we need to impose additional conditions on the smoothness of $\tau \mapsto \zb(\tau)$ and on the number of grid points $G$.    

\begin{cor}[Oracle rules for $\widehat{\zb}$]\label{th:orlipr}  Assume that $\tau \mapsto \beta_j(\tau) \in \Lambda_c^{\eta}(\Tc)$ for $j=1,...,d$ and given $c, \eta > 0$, that $N^2 \gg K \gg G$ and $r_\tau \geq \eta$.
A sufficient condition for $\sqrt{N}(\widehat \zb(\cdot) - \zb(\cdot)) \weak \Gb(\cdot)$ for any $(P,\ZZ)\in \Ps_1(\xi,M,\bar{f},\overline{f'},f_{min})$ satisfying above conditions is $S = o(N^{1/2} (\log N)^{-1})$ and $G \gg N^{1/(2\eta)}$. A necessary condition for the same result is $S = o(N^{1/2})$ and $G \gg N^{1/(2\eta)}$.

\end{cor}

Corollary~\ref{th:orlipr} characterizes the choice for parameters $(S,K)$ which determine computational cost. The conditions on $S$ remain the same as in Corollary~\ref{th:orlipo}. The process oracle rule requires restrictions on $K$ based on the smoothness of $\tau\mapsto Q(x_0;\tau)$, denoted as $\eta$. Note that, compared to the results in~\cite{bechfe2011}, smoothness of $\tau\mapsto Q(x_0;\tau)$ is the only additional condition of the data that is needed to ensure process convergence of $\sqrt{N}(\widehat \zb(\cdot) - \zb(\cdot))$. Specifically, the lower bound on $K$ in terms of $N$ becomes smaller as $\eta$ increases, which implies that smoother $\tau\mapsto Q(x_0;\tau)$ allow for larger computational savings. Corollary~\ref{th:orlipr} implies that the condition on $K$ is necessary for the oracle rule, no matter how $S$ is chosen.

Next we apply Corollary~\ref{th:orlipr} to the estimation of conditional distribution functions. Define 
\begin{align}\label{eq:hatFor}
\hat F_{Y|X}^{or}(\cdot|x_0):= \tau_L + \int_{\tau_L}^{\tau_U} \IF\{\ZZ(x)^\top\hat \zb_{or}(\tau) < y\} d\tau.
\end{align}
The asymptotic distribution of $\hat F_{Y|X}^{or}(\cdot|x_0)$ was derived in \cite{ChaVolChe2016}.

\begin{cor} \label{cor:fhatli}
Under the same conditions as Corollary~\ref{th:orlipr}, we have for any $x_0 \in \Xc$,
\begin{align*}
\sqrt{N}\big(\hat F_{Y|X}(\cdot|x_0) - F_{Y|X}(\cdot|x_0)\big) \weak -f_{Y|X}&(\cdot|x_0) \ZZ(x_0)^\top\Gb\big(F_{Y|X}(\cdot|x_0)\big)\\
&\mbox{in } \ell^\infty\big((Q(x_0;\tau_L),Q(x_0;\tau_U))\big),
\end{align*}
where $\hat F_{Y|X}(\cdot|x_0)$ is defined in \eqref{eq:Fhat} and $\Gb$ is a centered Gaussian process with covariance given in \eqref{eq:covlin}. The same process convergence result holds with $\hat F_{Y|X}^{or}(\cdot|x_0)$ replacing $\hat F_{Y|X}(\cdot|x_0)$.
\end{cor}

The proof of Corollary \ref{cor:fhatli} uses the fact that $y \mapsto \hat F_{Y|X}(y|x)$ is a Hadamard differentiable functional of $\tau \mapsto \ZZ(x)^\top\hat\zb(\tau)$ for any fixed $x$. The result of this corollary can be extended to other functionals with this property.


\subsection{Approximate linear models with local basis structure}\label{sec:loc}

In this section, we consider models with $Q(x;\tau) \approx \ZZ(x)^\top \zb(\tau)$ with $m=\dim(\ZZ)\to\infty$ as $N\to\infty$ where the transformation $\ZZ$ corresponds to a basis expansion. The analysis in this section focuses on transformations $\ZZ$ with a specific local support structure, which will be defined more formally in Condition \hyperref[L]{(L)}. Examples of such transformations include polynomial splines or wavelets. 
	
Since the model $Q(x;\tau) \approx \ZZ(x)^\top \zb(\tau)$ holds only approximately, there is no unique "true" value for $\zb(\tau)$. Theoretical results for such models are often stated in terms of the following vector
\beq \label{eq:gamman}
\zg_N(\tau) := \arg\,\min_{\zg\in\R^m} \E\big[(\ZZ^\top\zg - Q(X;\tau))^2f(Q(X;\tau)|X)\big],
\eeq
see \cite{ChaVolChe2016} and Remark~\ref{rem:center}. Note that $\ZZ^\top\zg_N(\tau)$ can be viewed as the (weighted $L_2$) projection of $Q(X;\tau)$ onto the approximation space. The resulting $L_\infty$ approximation error is defined as
\begin{align}\label{def:cn}
c_m(\zg_N) := \sup_{x \in \Xc, \tau \in \Tc} \big|Q(x;\tau) - \zg_N(\tau)^\top\ZZ(x)\big|.
\end{align}

For any $\bv\in\R^m$ define the matrix $\tilde J_m(\bv) := \E[\ZZ\ZZ^\top f(\ZZ^\top \bv|X)]$ . For any $\ba\in \R^m, \bb(\cdot): \Tc \to \R^m$, define $\tilde{\mathcal{E}}(\ba,\bb) := \sup_{\tau \in \Tc} \E[|\ba^\top \tilde J_m^{-1}(\bb(\tau)) \ZZ|]$. Throughout the rest of this subsection, we assume the following condition:

\begin{enumerate}
\item[(L)]\label{L} For each $x \in \Xc$, the vector $\ZZ(x)$ has zeroes in all but at most $r$ consecutive entries, where $r$ is fixed. Furthermore, $\sup_{x\in \Xc}\tilde{\mathcal{E}}(\ZZ(x),\zg_N) = O(1)$.
\end{enumerate} 

Condition \hyperref[L]{(L)} ensures that the matrix {$\tilde J_m(\bv)$ has a band structure for any $\bv \in\R^m$} such that the off-diagonal entries of $\tilde J_m^{-1}(\bv)$ decay exponentially fast (Lemma 6.3 in \cite{zsw98}). Next, we discuss an example of $\ZZ$ which satisfies \hyperref[L]{(L)}.

\begin{example}[Univariate polynomial spline] \rm \label{ex:spline}
Suppose that \hyperref[A2]{(A2)}-\hyperref[A3]{(A3)} hold and that $X$ has a density on $\Xc = [0,1]$ uniformly bounded away from zero and infinity. Let $\tilde \BB(x) = (\tilde B_1(x),...,\tilde B_{J-p-1}(x))^\top$ be a polynomial spline basis of degree $p$ defined on $J$ uniformly spaced knots $ 0 = t_1 <... < t_J = 1$, such that the support of each $\tilde B_j$ is contained in the interval $[t_{j},t_{j+p+1})$ and normalization is as given in Definition 4.19 on page 124 in~\cite{schumaker:81}. Let $\ZZ(x) := m^{1/2}(\tilde B_1(x),...,\tilde B_{J-p-1}(x))^\top$, then there exists a constant $M>1$ such that $M^{-1}<\E[\ZZ \ZZ^\top]<M$ (by Lemma 6.2 of \cite{zsw98}). With this scaling we have $\xi_m \asymp \sqrt{m}$. Moreover, the first part of assumption \hyperref[L]{(L)} holds with $r = p+1$, while the second part, i.e.,  $\sup_{x\in \Xc}\tilde{\mathcal{E}}(\ZZ(x),\zg_N) = O(1)$, is verified in Lemma~\ref{lem:locexp}. 
\end{example}

\begin{theo}\label{th:locub}
Suppose that assumptions~\hyperref[A1]{(A1)} -- \hyperref[A3]{(A3)} and \hyperref[L]{(L)} hold, that $K \ll N^2$ and $S \xi_m^4 \log N = o(N), c_m(\zg_N) = o(\xi_m^{-1}\wedge (\log N)^{-2})$. Then
\begin{multline} 
\sup_{\tau \in \Tc_K} \Big|\ZZ(x_0)^\top(\overline{\zb}(\tau) - \hat\zb_{or}(\tau))\Big| = o_P(\|\ZZ(x_0)\|N^{-1/2})
\\ 
+ O_P\Big(\Big(1+\frac{\log N}{S^{1/2}}\Big)\Big(c_m^2(\zg_N) + \frac{S\xi_m^2 (\log N)^2}{N}\Big)\Big)
\\
+ O_P\Big(\frac{\|\ZZ(x_0)\| \xi_m S \log N}{N}  + \frac{\|\ZZ(x_0)\|}{N^{1/2}}\Big(\frac{S \xi_m^2 (\log N)^{10}}{N}\Big)^{1/4}\Big).  \label{eq:locubpo}
\end{multline}
If additionally $K \gg G \gg 1$ and $c_m^2(\zg_N) = o(N^{-1/2})$ we also have
\begin{multline*} 
\sup_{\tau \in \Tc} |\ZZ(x_0)^\top(\widehat{\zb}(\tau) - \hat\zb_{or}(\tau))| \leq  \|\ZZ(x_0)\|\sup_{\tau \in \Tc_K} \|\overline{\zb}(\tau) - \hat\zb_{or}(\tau)\|  + o_P(\|\ZZ(x_0)\|N^{-1/2})
\\
  + \sup_{\tau \in \Tc} \big\{|(\Pi_K Q(x_0;\cdot))(\tau) - Q(x_0;\tau)| + |\ZZ(x_0)^\top\zg_N(\tau) - Q(x_0;\tau)|\big\} 
\end{multline*}
where the projection operator $\Pi_K$ was defined right after~\eqref{eq:hatbeta}.
\end{theo}

The strategy for proving Theorem \ref{th:locub} is similar to that for Theorem \ref{th:linub}, the difference is that we now apply an aggregated Bahadur representation which makes explicit use of the local basis structure of $\ZZ$ (Section \ref{sec:ABRloc}). 


Similarly to Theorem~\ref{th:linub}, Theorem~\ref{th:locub} only provides upper bounds on the differences between $\overline{\zb}, \widehat{\zb}$ and $\hat\zb_{or}$. As in the setting of fixed-dimensional linear models, this result can be complemented by a corresponding lower bound which we state next.

\begin{theo}\label{th:loclob}
For any $\tau$ in $\Tc$ there exists a sequence of distributions of $(Y,X)$ and sequence of transformations $\ZZ$ such that assumptions~\hyperref[A1]{(A1)} -- \hyperref[A3]{(A3)} and \hyperref[L]{(L)} hold, that $S \xi_m^4 (\log N)^{10} = o(N), c_m^2(\zg_N) = o(N^{-1/2})$ and there exists a $C>0$ with
\begin{equation}\label{eq:loclobpo}
\limsup_{N \to \infty} P\Big(|\ZZ(x_0)^\top\overline{\zb}(\tau) - \ZZ(x_0)^\top\hat\zb_{or}(\tau)| \geq \frac{CS\xi_m }{N} \Big) > 0. 
\end{equation} 
Moreover for any $c, \eta > 0$ there exists a sequence of distributions of $(Y,X)$ and sequence of transformations $\ZZ$ satisfying the above conditions and an $x_0 \in \Xc$ with $\tau \mapsto Q(x_0;\tau) \in \Lambda_c^{\eta}(\Tc)$ such that
\begin{equation}\label{eq:loclobpr}
\limsup_{N \to \infty} P\Big(\sup_{\tau \in \Tc}|\ZZ(x_0)^\top\widehat{\zb}(\tau) - \ZZ(x_0)^\top\hat\zb_{or}(\tau)| \geq \frac{CS\xi_m}{N} + CG^{-\eta}\Big) > 0. 
\end{equation} 
\end{theo}

Compared to Section~\ref{sec:lin}, we obtain an interesting insight. The sufficient and necessary conditions on $S$ explicitly depend on $m$ (note that under \hyperref[A1]{(A1)} -- \hyperref[A3]{(A3)} we have $\xi_m \gtrsim m$) and become more restrictive as $m$ increases. This shows that there is a fundamental limitation on possible computational savings that depends on the model complexity. In other words more complex models (as measured by $m$) allow for less splitting and require larger sub-samples, resulting in a larger computational burden. 

We conclude this section by providing sufficient and necessary conditions for oracle rules in the local basis expansion setting. To state those results, denote by $\Ps_L(M,\bar f,\overline{f'},f_{min},R)$ the collection of all sequences $P_N$ of distributions of $(X,Y)$ on $\R^{d+1}$ and fixed $\ZZ$ with the following properties: \hyperref[A1]{(A1)}-\hyperref[A3]{(A3)} hold with constants $M,\bar{f},\overline{f'}<\infty,f_{min}>0$, \hyperref[L]{(L)} holds for some $r < R$, $\xi_m^4 (\log N)^6 = o(N), c_m^2(\zg_N) = o(N^{-1/2})$. The conditions in $\Ps_L(M,\bar f,\overline{f'},f_{min},R)$ ensure the weak convergence of the oracle estimator $\hat\zb_{or}(\tau)$, see \cite{ChaVolChe2016}. 

The following condition characterizes the upper bound on $S$ which is sufficient to ensure the oracle rule for $\overline{\zb}(\tau)$. 
\begin{enumerate}
	\item[(L1)]\label{L1} Assume that 
	\[
	S = o\bigg(\frac{N}{m\xi_m^2\log N} \wedge \frac{N}{\xi_m^2(\log N)^{10}} \wedge \frac{N^{1/2}}{\xi_m\log N} \wedge \frac{N^{1/2}\|\ZZ(x_0)\|}{\xi_m^2 (\log N)^2} \bigg).
	\] 
\end{enumerate}

For specific examples, Condition \hyperref[L1]{(L1)} can be simplified. For instance, in the setting of Example \ref{ex:spline}, we can reduce \hyperref[L1]{(L1)} to the form 
$$S = o(N^{1/2} (\log N)^{-2} m^{-1/2} \wedge N (\log N)^{-10} m^{-2}).$$

We now present the sufficient and necessary conditions for the oracle rule of $\overline\zb(\tau)$ under the Condition \hyperref[L]{(L)} for $\ZZ$. 

\begin{cor}[Oracle Rule for $\overline{\zb}(\tau)$] \label{th:orloc} 
Assume that \hyperref[L1]{(L1)} holds and data are generated from $P_N$ with $(P_N,\ZZ) \in \Ps_L(M,\bar f,\overline{f'},f_{min},R)$. Then the pooled estimator $\overline{\zb}(\tau)$ defined in \eqref{eq:zbbar} satisfies for any fixed $\tau \in \Tc, x_0 \in\Xc$, 
\begin{align} \label{eq:weakloc}
\frac{\sqrt{N}\ZZ(x_0)^\top(\overline \zb(\tau) - {\zg_N(\tau)})}{(\ZZ(x_0)^\top J_m(\tau)^{-1}\E[\ZZ\ZZ^\top]J_m(\tau)^{-1}\ZZ(x_0) )^{1/2}} \weak \Nc\big(0,\tau(1-\tau)\big),
\end{align}
where $J_m(\tau)$ is defined in Corollary~\ref{th:orlipo}. This matches the limiting behaviour of $\hat \zb_{or}$. If $S\gtrsim N^{1/2}\xi_m^{-1}$ the weak convergence result (\ref{eq:weakloc}) fails for some $(P_N,\ZZ) \in \Ps_L(1,\bar f,\overline{f'},f_{min},R), x_0 \in \Xc$.  
\end{cor}

To compare the necessary and sufficient conditions in Corollary~\ref{th:orloc}, assume for simplicity that $m \gg N^{\alpha}$ for some $\alpha>0$ and $\|\ZZ(x_0)\| \asymp \xi_m$ (this is, for instance, the case for univariate splines as discussed in Example~\ref{ex:spline}). Since under \hyperref[A1]{(A1)}-\hyperref[A3]{(A3)} we have $m^{1/2} \lesssim \xi_m$, it follows that \hyperref[L1]{(L1)} holds provided that $S = o(N^{1/2}\xi_m^{-1}(\log N)^{-2})$ and the necessary and sufficient conditions match up to a factor of $(\log N)^2$.


%
%


%


Next we discuss sufficient conditions for the process oracle rule.

\begin{cor}
\label{th:orlocpr}
Consider an arbitrary vector $x_0 \in \Xc$. Assume that data are generated from $P_N$ with $(P_N,\ZZ) \in \Ps_L(M,\bar f,\overline{f'},f_{min},R)$, that \hyperref[L1]{(L1)} holds, that $\tau \mapsto Q(x_0;\tau) \in \Lambda_c^{\eta}(\Tc)$, $r_\tau\geq \eta$, $\sup_{\tau \in \Tc} |\ZZ(x_0)^\top\zg_N(\tau) - Q(x_0;\tau)| = o(\|\ZZ(x_0)\|N^{-1/2})$, that $N^2 \gg K \gg G \gg N^{1/(2\eta)}\|\ZZ(x_0)\|^{-1/\eta}$, $c_m^2(\zg_N) = o(N^{-1/2})$ and that the limit 
\begin{align}\label{eq:covloc}
H_{x_0}(\tau_1,\tau_2) := \lim_{N \to \infty} \frac{\ZZ(x_0)^\top J_m^{-1}(\tau_1)\E[\ZZ\ZZ^\top] J_m^{-1}(\tau_2)\ZZ(x_0) (\tau_1 \wedge \tau_2-\tau_1 \tau_2)}{\|\ZZ(x_0)\|^2} 
\end{align}
exists and is non-zero for any $\tau_1,\tau_2 \in \Tc$, where $J_m(\tau)$ is defined in the statement of Corollary~\ref{th:orlipo}. 

1. The projection estimator $\hat\zb(\tau)$ defined in \eqref{eq:hatbeta} satisfies
\begin{align}
\frac{\sqrt{N}}{\|\ZZ(x_0)\|}\big(\ZZ(x_0)^\top\widehat{\zb}(\cdot) - Q(x_0;\cdot)\big) \weak \Gb_{x_0}(\cdot) \mbox{ in } \ell^\infty(\Tc),\label{eq:prloc} 
\end{align}
where $\Gb_{x_0}$ is a centered Gaussian process with $\E[\Gb_{x_0}(\tau)\Gb_{x_0}(\tau')] = H_{x_0}(\tau,\tau')$. The same holds for the oracle estimator $\hat\zb_{or}(\tau)$.

 If $G \lesssim N^{1/(2\eta)}\|\ZZ(x_0)\|^{-1/\eta}$ or $S \gtrsim N^{1/2}\xi_m^{-1}$ the weak convergence in \eqref{eq:prloc} fails for some $(P_N,\ZZ)$ which satisfies the above conditions.

2. 
For $\hat F_{Y|X}(\cdot|x_0)$ defined in \eqref{eq:Fhat} and $\Gb$ defined above we have
\begin{align*}
	\frac{\sqrt{N}}{\|\ZZ(x_0)\|}\big(\hat F_{Y|X}(\cdot|x_0) - F_{Y|X}&(\cdot|x_0)\big) \weak -f_{Y|X}(\cdot|x_0)\Gb_{x_0}\big(F_{Y|X}(\cdot|x_0)\big)\\
	 &\mbox{ in } \ell^\infty\big((Q(x_0;\tau_L),Q(x_0;\tau_U))\big).
\end{align*}
This matches the process convergence for $\hat F_{Y|X}^{or}(\cdot|x_0)$.
\end{cor} 

The proof of the sufficient conditions in Corollary~\ref{th:orlocpr} is presented in Section \ref{sec:pforlocpr}, and the collapse of the weak convergence \eqref{eq:prloc} is shown in Section \ref{sec:pfnec}. Similar to the discussion after Corollary~\ref{th:orlipr}, the process oracle rule does not place additional restrictions on the number of sub-samples $S$ besides \hyperref[L1]{(L1)}. However, the process oracle rule requires additional assumptions on the quantile grid $\Tc_K$. An interesting observation is that $G \gg (N/\|\ZZ(x_0)\|^2)^{1/(2\eta)}$ in Theorem \ref{eq:covloc} can be weaker than $G \gg N^{1/(2\eta)}$ from Corollary~\ref{th:orlipr}. For instance, this is true in the setting of Example~\ref{ex:spline} where $\|\ZZ(x_0)\| \asymp m^{1/2}$. Consequently, the computational burden is reduced since $K$ can be chosen smaller. The intuition behind this surprising phenomenon is that the convergence rate for the estimator $\ZZ(x_0)^\top\hat\zb_{or}(\tau)$ in nonparametric models is typically slower. Thus, less stringent assumptions are needed to ensure that the bias induced by quantile projection is negligible compared to the convergence rate of $\hat\zb_{or}(\tau)$. 

The sufficient conditions in this section can be extended to cover approximately linear models with increasing dimension that do not satisfy condition \hyperref[L]{(L)}. Technical details are provided in the Appendix, Section~\ref{SEC:GEN}.

\begin{remark}
\rm 
To the best of our knowledge, the only paper that also studies the sharpness of upper bounds for $S$ that guarantee valid inference in a divide and conquer setting with nonparametric regression is \cite{CS15}. However, \cite{CS15} only consider nonparametric \textit{mean} regression in the smoothing spline setup. 
\end{remark}

\section{Practical aspects of inference}\label{SEC:INF}

In the previous section, we derived conditions which guarantee that the divide and conquer estimator $\overline \zb$ and the quantile projection estimator $\widehat \zb$ share the same asymptotic distribution as the 'oracle' estimator $\widehat \zb_{or}$, so that inference based on the asymptotic distribution of $\widehat \zb_{or}$ remains valid. In practice, this result can not be directly utilized since the asymptotic variance of the oracle estimator $\widehat \zb_{or}$ is in general not pivotal. Classical approaches to inference for this estimator are typically based on estimating its asymptotic variance from the data directly, or conducting bootstrap to approximate the asymptotic distribution. 

Estimating the limiting variance requires the choice of a bandwidth parameter, and existing research indicates that classical rules for bandwidth selection need to be adjusted in a divide and conquer setting (see e.g. \cite{BDS16, xu2016}). We discuss related issues for variance estimation in Section~\ref{sec:asyvar}. 

Conducting bootstrap substantially increases the computational burden of any procedure, which is problematic in a massive data setting we are considering here. While recent proposals by \cite{kleiner2014, sengupta2016} provide a way to reduce the computational cost of bootstrap in a big data setting, the approaches described in those papers are not easily combined with the divide and conquer setting which we consider here. 

As an alternative to variance estimation or classical bootstrap approaches, we propose several simple inference procedures which directly utilize the fact that in a divide and conquer setting estimators from sub-samples are available. Those procedures are very easy to implement, and require only a very small amount of computation on the central computer without additional communication costs. Details are provided in section~\ref{sec:subs}. 

\subsection{Inference utilizing results from subsamples}\label{sec:subs}

We begin by discussing inference at a fixed quantile level $\tau$. The key idea in this section is to make direct use of the fact that $S$ estimators $\hat \zb_s(\tau)$ from subsamples are available. Observe that the estimator $\overline\zb(\tau)$ is simply an average of $\hat \zb_1(\tau),...,\hat \zb_S(\tau)$ which can be seen as iid realizations (provided groups are homogeneous) of a random variable with approximately normal distribution. This suggests two simple options

1. Use the sample covariance matrix, say $\hat \Sigma^D$, of $\hat \zb_1(\tau),...,\hat \zb_S(\tau)$ in order to conduct inference on $\overline\zb(\tau)$ or linear functionals thereof such as $\overline Q(x;\tau) := \ZZ(x)^\top \overline\zb(\tau)$. For example, a simple asymptotic level $\alpha$ confidence interval for $Q(x;\tau)$ is given by
\begin{equation}\label{ci:simple}
\big[\ZZ(x_0)^\top \overline{\zb}(\tau) \pm S^{-1/2} (\ZZ(x_0)^\top \hat \Sigma^D \ZZ(x_0))^{1/2} \Phi^{-1}(1-\alpha/2)\big],
\end{equation} 

2. A refined version of the previous approach is to additionally exploit the fact that a suitably scaled version of $\bu_N^\top\hat \zb_s(\tau)$ should be approximately normal since each $\hat \zb_s(\tau)$ is itself an estimator based on sample of iid data. Hence for small $S$ (say $S \leq 30$) more precise confidence intervals can be obtained by using quantiles of the student $t$ distribution (if $\bu_N$ is a vector) or $F$ distribution (if $\bu_N$ is a fixed-dimensional matrix). For example, a modification of the confidence interval in~\eqref{ci:simple} would take the form
\begin{equation}\label{ci:simple_t}
\big[\ZZ(x_0)^\top \overline{\zb}(\tau) \pm S^{-1/2} (\ZZ(x_0)^\top \hat \Sigma^D \ZZ(x_0))^{1/2} t_{S-1,1-\alpha/2}\big]
\end{equation} 
where $t_{S-1,1-\alpha/2}$ denotes the $1-\alpha/2$-quantile of the t-distribution with $S-1$ degrees of freedom. The asymptotic validity of both intervals discussed above is provided in the following theorem. 

\begin{theo} \label{th:simpleinf} Assume that the conditions of either Corollary~\ref{th:orlipo}, Corollary~\ref{th:orloc} or Corollary~\ref{th:orgenpo} hold, that $c_m(\zg_N) = o(\|\ZZ(x_0)\|N^{-1/2})$ and that $S \geq 2$ (S can be fixed). Then the confidence interval~\eqref{ci:simple_t} has asymptotic (for $N \to \infty$) coverage probability $1-\alpha$. 

If additionally $S \to \infty$, the confidence interval given in~\eqref{ci:simple} also has asymptotic coverage probability $1 - \alpha$. 
\end{theo}

See Section~\ref{sec:prsimple} for a proof of Theorem~\ref{th:simpleinf}. The main advantage of the two approaches discussed above lies in their simplicity as they do not require any costly computation or communication between machines. There are two main limitations. 

First, for small values of $S$ (say $S\leq 30$) the confidence intervals in~\eqref{ci:simple} will not have the correct coverage while the interval in~\eqref{ci:simple_t} can be substantially wider than the one based on the exact asymptotic distribution since quantiles of the $t$-distribution with few degrees of freedom can be substantially larger than corresponding normal quantiles. Moreover, the approach is not applicable if $S$ is smaller than the dimension of the parameter of interest. Second the approaches are not straightforward to generalize to inference on non-linear functionals of $\zb(\tau)$ such as $\widehat F_{Y|X}(y|x)$ or inference which is uniform over a continuum of quantiles $\tau$. 

The first limitation is not likely to become relevant in a big data setting since here we typically expect that $S$ is large due to computational bottlenecks and memory constraints. For the sake of completeness we discuss this case in the next section. To deal with the second limitation, we propose to use a bootstrap procedure that can be conducted by solely using the subsample estimators $\{\hat \zb_s(\tau_k)\}_{s=1,..,S, k=1,...,K}$ which are stored on the central machine. Details are provided below.

\begin{enumerate}
\item Sample iid random weights $\{\omega_{s,b}\}_{s=1,...,S, b=1,...,B}$ from taking value $1-1/\sqrt{2}$ with probability 2/3 and $1+\sqrt{2}$ with probability 1/3 (i.e. weights are chosen such that $\omega_{s,b}\geq 0$, $\E[\omega_{s,b}] = Var(\omega_{s,b}) = 1$)
and let $\bar \omega_{\cdot,b} := S^{-1}\sum_{s=1}^S \omega_{s,b}$.
\item For $b=1,...,B, k = 1,...,K$ compute the bootstrap estimators 
\beq\label{eq:betaboot}
\overline \zb^{(b)}(\tau_k) := \frac{1}{S} \sum_{s = 1}^S \frac{\omega_{s,b}}{\bar \omega_{\cdot,b}} \hat \zb^s(\tau_k)
\eeq
and define the matrix $\hat\Xi^{(b)}$ from~\eqref{eq:hatXi} with $\overline \zb^{(b)}(\tau_k)$ replacing $\overline \zb(\tau_k)$.
\item Similarly as in~\eqref{eq:hatbeta} compute $\hat\zb^{(b)}(\cdot)$ from the matrix $\hat\Xi^{(b)}$ defined above. For a functional of interest $\Phi$ approximate quantiles of the distribution of $\Phi(\hat\zb(\cdot)) - \Phi(\zb(\cdot))$ by the empirical quantiles of $\{\Phi(\hat\zb^{(b)}(\cdot)) - \Phi(\hat\zb(\cdot))\}_{b=1,...,B}$.
\end{enumerate}   

A formal justification for this bootstrap approach is provided by the following result.
\begin{theo} \label{th:boot}
Let the assumptions of either Corollary~\ref{th:orlipr}, Corollary~\ref{th:orlocpr} or Corollary~\ref{th:orgenpr} hold and assume that additionally $S \to \infty$. Then we have conditionally on the data $(X_{i},Y_{i})_{i=1,...,N}$
\[
\frac{\sqrt{N}}{\|\ZZ(x_0)\|}(\ZZ(x_0)^\top \hat\zb^{(1)}(\cdot) - \ZZ(x_0)^\top \hat\zb(\cdot)) \weak \Gb_{x_0}(\cdot) \quad \mbox{in } \ell^\infty(\mathcal{T}). 
\]
where the limit $\Gb_{x_0}$ denotes the centered Gaussian process from Corollary~\ref{th:orlocpr} or Corollary~\ref{th:orgenpr} under their respective assumptions and $\Gb_{x_0} = \ZZ(x_0)^\top \Gb$ under the assumptions of Corollary~\ref{th:orlipr}. 
\end{theo}

The proof of Theorem~\ref{th:boot} is given in Section~\ref{sec:prboot}. We conclude this section by remarking that the bootstrap proposed differ from the cluster-robust bootstrap of \cite{hagema2017}. The main difference is that we propose to directly utilize the sub-sample estimators $\hat \zb^s$ while the approach of \cite{hagema2017} requires repeated estimation on the complete sample. 
 

\subsection{Inference based on estimating the asymptotic covariance matrix}
\label{sec:asyvar}
As suggested by a referee, an alternative way to conduct inference is to compute, for each sub-sample, not only the estimator $\hat\zb_s$ but also a variance estimator and pool those estimators. Here, we provide some details on this approach. For the sake of simplicity, we only discuss the case where $m$ is fixed and the model is specified correctly, i.e. the setting of Section~\ref{sec:lin}.
 
It is well known that the asymptotic variance-covariance matrix of the difference $\sqrt{n}(\hat \zb_{or}(\tau) - \zb(\tau))$ (for fixed $m$) takes the 'sandwich' form
\[
\Sigma(\tau) = \tau(1-\tau)J_m(\tau)^{-1} \E[\ZZ\ZZ^\top] J_m(\tau)^{-1} , \quad J_m(\tau) = \E[\ZZ\ZZ^\top f_{Y|X}(Q(X;\tau)|X)].
\] 
The middle part $\E[\ZZ\ZZ^\top]$ is easily estimated by $\frac{1}{nS}\sum_i\sum_s \ZZ_{is}\ZZ_{is}^\top$. Since this is a simple mean of the subsample-based quantities $\frac{1}{n}\sum_i \ZZ_{is}\ZZ_{is}^\top$, implementing this in a distributed computing setting is straightforward. 

The matrix $J_m(\tau)$ involves the conditional density $f_{Y|X}(Q(X;\tau)|X)$ and is more difficult to estimate. A popular approach is based on Powell's estimator (\cite{powell1986})
\begin{equation} \label{eq:powell}
\hat J_{ms}^P (\tau) := \frac{1}{2nh_n} \sum_{i=1}^n \ZZ_{is}\ZZ_{is}^\top \IF \Big\{|Y_{is} - \ZZ_{is}^\top \hat \zb^s(\tau)| \leq h_n \Big\}. 
\end{equation}
Here, $h_n$ denotes a bandwidth parameter that needs to be chosen carefully in order to balance the resulting bias and variance.

There are several possible approaches to estimate $J_{m}^P (\tau)$ in a parallel computing environment. If an additional round of communication is acceptable, it is possible to construct estimators with the same convergence rate and asymptotic distribution as the estimator based on the full sample. Details are provided in section~\ref{sec:powor} in the appendix. If only a single round of communication is allowed, the following algorithm can be used instead. 

\begin{enumerate}
\item For $s=1,...,S$, in the same round as computing $\widehat \zb^s(\tau)$, compute $\hat J_{ms}^P (\tau)$ from~\eqref{eq:powell} and $\hat \Sigma_{1s} := \frac{1}{n}\sum_i \ZZ_{is}\ZZ_{is}^\top$.
\item Along with $(\widehat \zb^s(\tau))_{s=1,...,S}$ send $(\hat J_{ms}^P (\tau),\hat \Sigma_{1s})_{s=1,...,S}$ to the master machine and compute $\overline J_{m}^P (\tau) := \frac{1}{S}\sum_{s=1}^S \hat J_{ms}^P (\tau)$, $\overline\Sigma_{1} := \frac{1}{S}\sum_{s=1}^S \hat \Sigma_{1s}$.
\item The final variance estimator is given by $\bar \Sigma(\tau) = \tau(1-\tau) \overline J_{m}^P (\tau)^{-1} \overline\Sigma_{1} \overline J_{m}^P (\tau)^{-1}$.
\end{enumerate}

\begin{remark}\rm
Note that in the above algorithm we first take averages over the subsampled estimators $\hat J_{ms}^P (\tau)$ and only invert the aggregated matrix $\overline J_{m}^P (\tau)$. An alternative approach would have been to compute the estimator $\hat J_{ms}^P (\tau)^{-1} \hat\Sigma_{s} \hat J_{ms}^P (\tau)^{-1}$ for each subsample and average in the end. However, given that $A \mapsto A^{-1}$ is non-linear, this might result in additional bias since in general for random matrices $(\E[A])^{-1} \neq \E[A^{-1}]$. 
\end{remark}

An important question for implementing the above algorithm is the choice of the bandwidth parameter $h_n$. To gain some intuition about the optimal choice of $h_n$, we will formally discuss the case of a linear model fixed dimension. First, observe that by a Taylor expansion we have almost surely
\begin{multline*}
\overline J_{m}^P (\tau)^{-1} \overline\Sigma_{1} \overline J_{m}^P (\tau)^{-1} - J_{m}(\tau)^{-1} \Sigma_{1} J_{m} (\tau)^{-1}
\\
= - J_{m}(\tau)^{-1}(\overline J_m^P(\tau) - J_{m} (\tau)) J_{m} (\tau)^{-1} \Sigma_{1}J_{m} (\tau)^{-1}
\\ 
-  J_{m}(\tau)^{-1} \Sigma_{1} J_{m}(\tau)^{-1}(\overline J_m^P(\tau) - J_{m} (\tau)) J_{m} (\tau)^{-1}
\\
+ J_{m}(\tau)^{-1} (\Sigma_{1} - \overline\Sigma_{1})J_{m} (\tau)^{-1}
\\
 + O(\|\overline J_{m}^P (\tau) - J_{m}(\tau)\|^2 + \|\overline\Sigma_{1} - \Sigma_1\|\|\overline J_{m}^P (\tau) - J_{m}(\tau)\|).
\end{multline*}
Ignoring higher-order terms and terms that do not depend on $h_n$ it suffices to analyze the properties of $\overline J_{m}^P (\tau) - J_{m}(\tau)$. 

\begin{theo} \label{th:powell} Under assumptions (A1)-(A3), assume that additionally $y\mapsto f_{Y|X}'(y|x)$ is continuously differentiable with first derivative being jointly continuous and uniformly bounded as a function of $x,y$ and that $nh_n (\log n)^{-1}\to \infty.$ Then
\[
\overline J_{m}^P (\tau) - J_m(\tau) = A_n(\tau) + O_p\Big(\frac{\log n}{n h_n}\Big)
\]
where the exact form of $A_n(\tau)$ is given in the proof. Moreover
\[
\E[A_n(\tau)] = \frac{h_n^2}{6} \E\Big[ \ZZ\ZZ^\top f''_{Y|X}(\ZZ^\top \zb(\tau)|X)\Big] + O\Big(\frac{\log n}{n}\Big) + o(h_n^2)
\]
and for $A_{N,(j,k)}$ denoting the entry of $A_N$ in row $j$ and column $k$
\[
\Cov(A_{N,(j,k)}(\tau),A_{N,(u,v)}(\tau)) = \frac{1}{Nh_n} \E[f_{Y|X}(\ZZ^\top \zb(\tau)|X) Z_jZ_kZ_uZ_v] + o\Big( \frac{1}{Nh_n}\Big).
\]
\end{theo}

The proof of Theorem~\ref{th:powell} is given in Section~\ref{sec:prpowell}. Theorem~\ref{th:powell} has several interesting implications. First, note that the asymptotic MSE of $A_N(\tau)$ is of the order $h_n^4 + (Nh_n)^{-1}$, which is minimized for $h_N \sim N^{-1/5}$ (note that the term $\log n/n$ is negligible). Under the additional condition $S = o(N^{2/5}(\log N)^{-1})$ we have $\log n/(nh_n) = o(N^{-2/5})$, and in this setting the MSE of the first order expansion of $\overline J_{m}^P (\tau)$ matches that of the Powell 'oracle' estimator as derived in \cite{kato2012}. This shows that, despite using estimators from subsamples, the same rate for estimating $J_{m}^P (\tau)$ as from the full sample can be achieved. That requires a stronger condition on $S$ than the oracle rate for the estimator $\hat \zb$. It is not clear if the latter condition is sharp, and we leave an investigation of this issue to future research.
 
Second, the choice $h_n \sim n^{-1/5}$ which would have been optimal for estimating $J_m(\tau)$ based on a subsample of size $n$ does not lead to the optimal error rate for the averaged estimator $\overline J_{m}^P (\tau)$. In fact the optimal bandwidth for $\overline J_{m}^P (\tau)$ is always smaller, which corresponds to undersmoothing. Similar effects were observed in various settings by~\cite{ZDW15}, \cite{BDS16} (see their Appendix A.13) and~\cite{CS15}.

%
%
%
%
%
%
%


\section{Monte Carlo experiments}\label{sec:sim}

In this section, we demonstrate our theory with simulation experiments. Due to space limitations, we restrict our attention to correctly specified linear models with different dimensions of predictors. More precisely, we consider data generated from
\begin{align}
Y_i = 0.21 +  \zb_{m-1}^\top X_i+\varepsilon_i, 
\quad i=1,...,N.\label{eq:simlin}
\end{align}
where $\varepsilon_i \sim \Nc(0,0.01)$ iid and $m \in \{4,16,32\}$. For each $m$, the covariate $X_i$ follows a multivariate uniform distribution $\Uc([0,1]^{m-1})$ with $\textrm{Cov}(X_{ij}, X_{ik})=0.1^2 0.7^{|j-k|}$ for $j,k=1,...,m-1$, and the vector $\zb_{m-1}$ takes the form
\begin{align}
	\begin{split}\label{eq:betaspeci}
		\zb_{3}&=(0.21,-0.89,0.38)^\top;\\
		\zb_{15}&=(\zb_{3}^\top,0.63,0.11,1.01,-1.79,-1.39,0.52,-1.62,
		\\
		&\quad\quad\quad\quad1.26,-0.72,0.43,-0.41,-0.02)^\top;\\
		\zb_{31}&=(\zb_{15}^\top,0.21,\zb_{15}^\top)^\top.
	\end{split}
\end{align}  


Throughout this section, we fix $\Tc = [0.05,0.95]$. Section~\ref{sec:simdc} contains results for the estimator $\overline\zb(\tau)$, while results for $\hat F_{Y|X}(y|x)$ are collected in Section~\ref{sec:simF}. Additional simulations (including models with heteroskedastic errors) are presented in Section~\ref{sec:addsim} of the online supplement.

\subsection{Results for the divide and conquer estimator $\overline\zb(\tau)$} \label{sec:simdc}

We fix the sub-sample size $n$ and consider the impact of the number of sub-samples $S$ on the coverage probabilities of different $95\%$ confidence intervals. To benchmark our results, we use the infeasible asymptotic confidence interval
\begin{align}
\big[x_0^\top \overline{\zb}(\tau) \pm N^{-1/2} \sigma(\tau) \Phi^{-1}(1-\alpha/2)\big],\label{eq:lincovp}
\end{align}
where $\sigma^2(\tau)$ denotes the theoretical asymptotic variance of $\zb^{or}(\tau)$; this CI is valid by the oracle rule but contains unknown quantities. The coverage properties of this interval also indicate whether we are in a regime where the oracle rule holds.

In a first step, we consider the properties of confidence intervals discussed in Section~\ref{sec:subs} which directly utilize the availability of results from sub-samples. We consider the following three types of confidence intervals. 
\begin{enumerate}
\item The normal confidence interval~\eqref{ci:simple}.
\item The confidence interval~\eqref{ci:simple_t} based on quantiles of the t distribution. 
\item The bootstrap confidence interval based on sample quantiles of $(\bar\zb^{(1)}(\tau)-\bar \zb(\tau)),...,(\bar\zb^{(B)}(\tau)-\bar \zb(\tau))$, where $\bar\zb^{(1)}(\tau),...,\bar\zb^{(B)}(\tau)$ are computed as in~\eqref{eq:betaboot}; in this section we set $B=500$. 
\end{enumerate} 

The coverage probabilities of corresponding confidence intervals are summarized in Figure~\ref{fig:simlecoverageSmS} where we fix two sub-sample sizes $n = 512, 2048$ and present two types of plots: coverage probabilities for 'small' values of $S$ ranging from $S=2$ to $S=50$ and coverage probabilities for 'large' values of $S = 2^k, k=1,...,10$. Only $m = 4, 32$ is considered here, $m=16$ can be found in the online supplement.

Plots for small $S$ help to access how different procedures behave if the number of sub-samples is small. As expected from the theory, oracle confidence intervals and simple confidence  intervals based on the t distribution maintain nominal coverage for small values of $S$. Confidence intervals based on normal approximation and the bootstrap under-cover for smaller values of $S$. Starting from about $S=20$, coverage probabilities of all four types of intervals are very close. We also see that for $m=32$, the oracle rule does not apply for a sub-sample size of $n=512$ with any number of sub-samples $S > 10$; the situation improves for $n=2048$. This is in line with our asymptotic oracle theory. For 'larger' values of $S$, there is no difference in the coverage probabilities of different intervals. As predicted by the oracle theory, coverage starts to drop earlier for models with larger $m$.

Next, we analyze the properties of asymptotic confidence intervals which are based on estimating the asymptotic variance of $x_0^\top \zb(\tau)$ from data. We compare three different ways of estimating the asymptotic variance 
\begin{enumerate}
\item A simple pooled estimator which uses the default settings in the package \texttt{quantreg} to obtain estimated variances in each sub-sample and takes the average over all sub-samples (additional details are provided in the online Supplement, Section~\ref{sec:simdet}). 
\item The estimator $\bar \Sigma$ based on the bandwidth $c^*(\tau) n^{-1/5}$ which minimizes the asymptotic MSE of the estimator $J_m$ within sub-samples of size $n$ (additional details on the choice of optimal constant $c^*(\tau)$ are provided in the online Supplement, Section~\ref{sec:simdet}). 
\item The estimator $\bar \Sigma$ based on the bandwidth $c^*(\tau) N^{-1/5}$ which is motivated by the theory developed in Theorem~\ref{th:powell}.
\end{enumerate}   
 
The results are reported in Table~\ref{tab:bwcomp}. Since there is no notable difference between all approaches when $S$ is large, only results for $S \leq 50$ are displayed for the sake of brevity. Interestingly, we do not observe a big difference between the naive bandwidth choice $h_n \sim n^{-1/5}$ and the optimal undersmoothing choice $h_n \sim N^{-1/5}$. This finding is quite intuitive since, once the asymptotic variance is estimated with at most $5-10\%$ relative error, a further increase in estimation accuracy does not lead to substantial improvements in coverage probabilities. The completely automatic choice implemented in the \texttt{quantreg} package also performs reasonably well.

\begin{table}[!h] \scriptsize
\begin{tabular}{c||c|c|c|c||c|c|c|c||c|c|c|c}
\hline
S & 1 & 10 & 30 & 50 & 1 & 10 & 30 & 50 & 1 & 10 & 30 & 50\\
\hline
&\multicolumn{4}{c||}{n = 512, m = 4, $\tau = 0.1$}&\multicolumn{4}{c||}{n = 512, m = 16, $\tau = 0.1$}&\multicolumn{4}{c}{n = 512, m = 32, $\tau = 0.1$}
\\
\hline
or & 94.9 & 94.8 & 94.9 & 94.3 & 94.7 & 94.3 & 92.2 & 90.7 & 94.4 & 92.5 & 88.6 & 85.5\\
def & 92.6 & 93.7 & 93.6 & 92.8 & 92 & 92.3 & 90.1 & 88.4 & 92.5 & 90.6 & 86.2 & 82.9\\
nai & 94.2 & 93.9 & 93.6 & 93.2 & 96.6 & 93.3 & 91 & 89.4 & 98.9 & 92.7 & 88.4 & 85\\
adj & 94.2 & 94.5 & 94.3 & 94.2 & 96.6 & 94.3 & 92.1 & 90.8 & 98.9 & 94.4 & 89.5 & 87.2\\
\hline
&\multicolumn{4}{c||}{n = 512, m = 4, $\tau = 0.5$}&\multicolumn{4}{c||}{n = 512, m = 16, $\tau = 0.5$}&\multicolumn{4}{c}{n = 512, m = 32, $\tau = 0.5$}
\\
\hline
or & 94.7 & 96 & 95.9 & 95 & 94.2 & 95.7 & 94.8 & 95.4 & 95.4 & 95.5 & 94.8 & 95\\
def & 97.8 & 98.2 & 98.2 & 98.2 & 96.7 & 98.2 & 98 & 98 & 97.6 & 97.5 & 97.3 & 97.7\\
nai & 95.9 & 96.8 & 96.4 & 96.3 & 96.7 & 97 & 96.4 & 97 & 99 & 97 & 96.2 & 96.7\\
adj & 95.9 & 96.4 & 96 & 95.2 & 96.7 & 96.4 & 95.7 & 96.2 & 99 & 96.6 & 95.7 & 95.9\\
\hline
&\multicolumn{4}{c||}{n = 512, m = 4, $\tau = 0.9$}&\multicolumn{4}{c||}{n = 512, m = 16, $\tau = 0.9$}&\multicolumn{4}{c}{n = 512, m = 32, $\tau = 0.9$}
\\
\hline
or & 95.4 & 94.6 & 94.2 & 93.6 & 94.6 & 94 & 91.7 & 90.1 & 95.2 & 92.1 & 90.6 & 87.1\\
def & 94 & 93.6 & 92.8 & 92.4 & 92.6 & 92.2 & 90 & 88.1 & 92.2 & 90.2 & 88.2 & 84.3\\
nai & 94.8 & 93.8 & 93.1 & 92.6 & 96.6 & 93.2 & 90.8 & 88.6 & 99 & 92.5 & 90.1 & 86.6\\
adj & 94.8 & 94.3 & 93.8 & 93.7 & 96.6 & 94 & 91.9 & 90.4 & 99 & 93.8 & 91.4 & 88.4\\
\hline
\hline
&\multicolumn{4}{c||}{n = 2048, m = 4, $\tau = 0.1$}&\multicolumn{4}{c||}{n = 2048, m = 16, $\tau = 0.1$}&\multicolumn{4}{c}{n = 2048, m = 32, $\tau = 0.1$}
\\
\hline
or & 95.6 & 94.7 & 94.4 & 94.3 & 95 & 94.2 & 94.4 & 94.4 & 95 & 95 & 94 & 94\\
def & 94.7 & 94.1 & 93.8 & 94 & 93.8 & 93.2 & 93.5 & 93.7 & 93.6 & 94.2 & 92.7 & 93.1\\
nai & 95.1 & 94.1 & 93.7 & 94 & 95.3 & 93.5 & 93.5 & 93.7 & 95.7 & 94.6 & 93.1 & 93.2\\
adj & 95.1 & 94.6 & 94.4 & 94.5 & 95.3 & 94 & 94.1 & 94.4 & 95.7 & 94.9 & 93.8 & 93.8\\
\hline
&\multicolumn{4}{c||}{n = 2048, m = 4, $\tau = 0.5$}&\multicolumn{4}{c||}{n = 2048, m = 16, $\tau = 0.5$}&\multicolumn{4}{c}{n = 2048, m = 32, $\tau = 0.5$}
\\
\hline
or & 94.6 & 94.9 & 95.2 & 95.3 & 95.3 & 94.2 & 94.2 & 95.5 & 95.3 & 95.5 & 95.1 & 95.5\\
def & 96 & 96.2 & 96.7 & 96.2 & 96.7 & 96.2 & 95.4 & 96.7 & 96 & 96.4 & 96.3 & 96.6\\
nai & 95.1 & 95.4 & 95.5 & 95.6 & 96.8 & 95.2 & 95 & 96 & 96.5 & 95.8 & 95.7 & 95.9\\
adj & 95.1 & 95.1 & 95.1 & 95.3 & 96.8 & 95 & 94.7 & 95.6 & 96.5 & 95.6 & 95.2 & 95.5\\
\hline
&\multicolumn{4}{c||}{n = 2048, m = 4, $\tau = 0.9$}&\multicolumn{4}{c||}{n = 2048, m = 16, $\tau = 0.9$}&\multicolumn{4}{c}{n = 2048, m = 32, $\tau = 0.9$}
\\
\hline
or & 94.6 & 95.6 & 95.8 & 94.6 & 95.1 & 95 & 94.5 & 94.7 & 95.2 & 94.4 & 92.9 & 92.9\\
def & 94.2 & 94.9 & 94.9 & 94.1 & 94.3 & 94.5 & 94 & 94.3 & 94.4 & 93.8 & 92.6 & 92.2\\
nai & 94.4 & 94.9 & 94.8 & 94.1 & 95.5 & 94.7 & 94.2 & 94.4 & 96.7 & 94.4 & 92.7 & 92.5\\
adj & 94.4 & 95.3 & 95.6 & 94.6 & 95.5 & 95.1 & 94.6 & 95 & 96.7 & 94.9 & 93.2 & 93.3\\
\hline
\end{tabular}
\caption{Coverage probabilities based on estimating the asymptotic variance. Different rows correspond to different methods for obtaining covariance matrix. or: using true asymptotic variance matrix, def: default choice implemented in quantreg package, nai: asymptotically optimal constant with scaling $h_n \sim n^{-1/5}$, adj: asymptotically optimal constant with scaling $h_n \sim N^{-1/5}$ as suggested by Theorem~\ref{th:powell}. \label{tab:bwcomp} }
\end{table}

\begin{figure}[!h]
\centering
\centering {\scriptsize (a) $m=4$}\\	
\includegraphics[width=13cm]{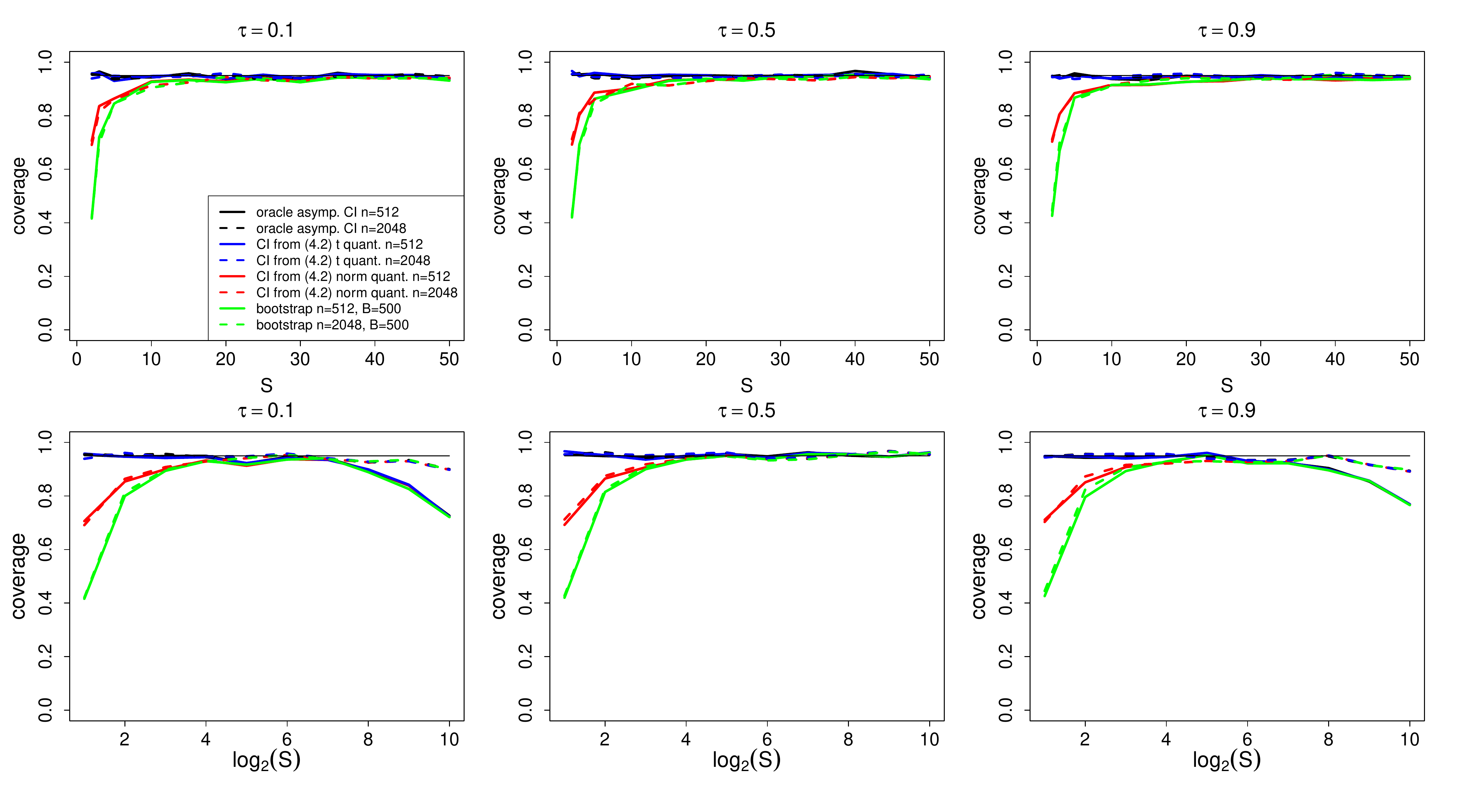}	
\\\vspace{-0.2cm}	
\centering {\scriptsize (b) $m=32$}\\	
\includegraphics[width=13cm]{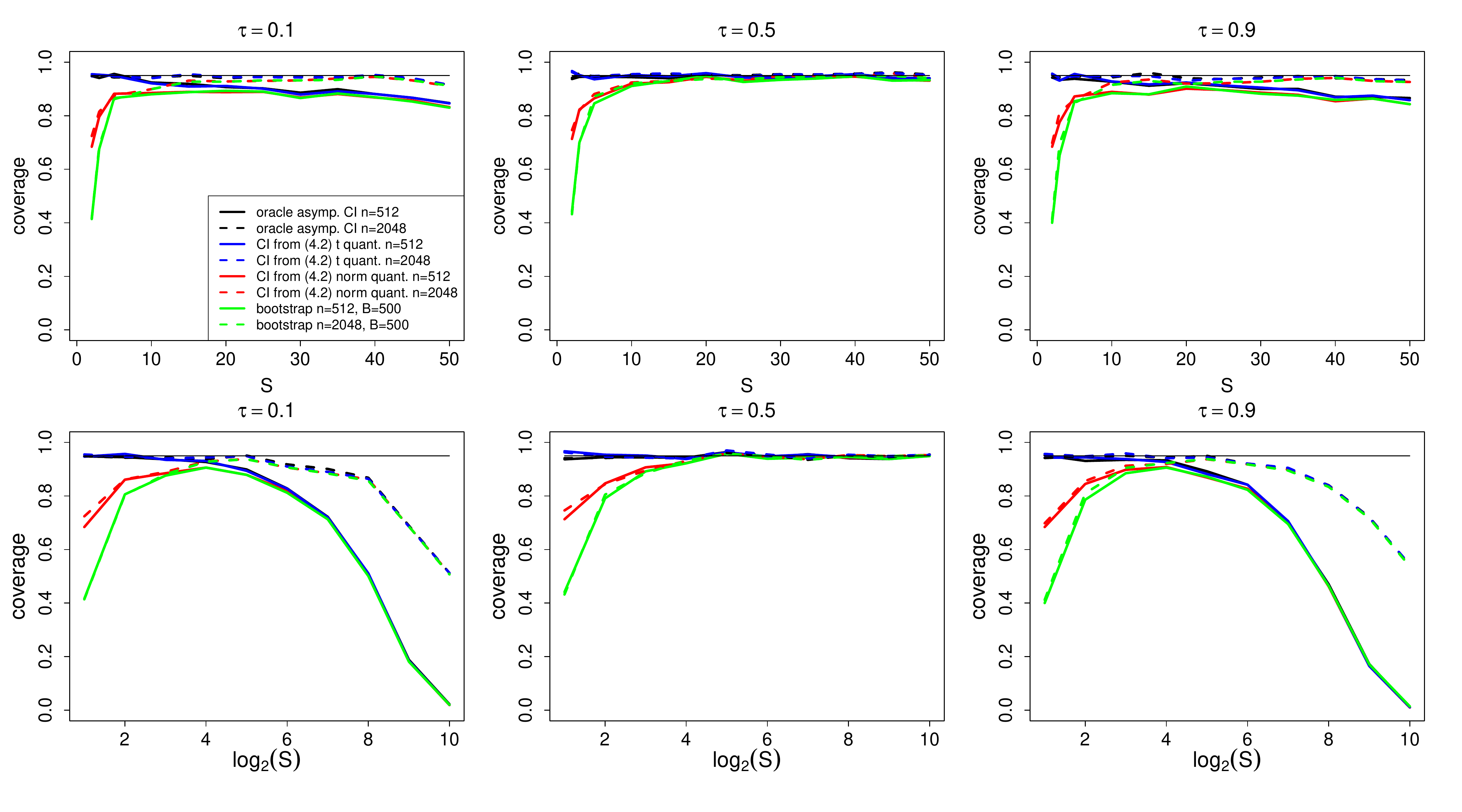}		
\caption{Coverage probabilities for $x_0^\top\zb(\tau)$ for different values of $S$ and $\tau = 0.1, 0.5, 0.9$ (left, middle, right row). Solid lines: $n = 512$, dashed lines: $n = 2048$. Black: asymptotic oracle CI, blue: CI from~\eqref{ci:simple_t} based on t distribution, red: CI from~\eqref{ci:simple} based on normal distribution, green: bootstrap CI.
Throughout $x_0 = (1,...,1)/m^{1/2}$, nominal coverage $0.95$. }\label{fig:simlecoverageSmS}
\end{figure}

Finally, note that the pattern of coverage probabilities varies at different $\tau$. For example, in the linear models with normal errors, the coverage probabilities at tail quantiles ($\tau=0.1,0.9$) drop to zero much faster than those at $\tau=0.5$. These empirical observations are not inconsistent with our theory where only the orders of the upper bound for $S$ are shown to be the same irrespective of the value of $\tau$. Rather, this phenomenon might be explained by different derivatives of the error density that appear in the estimation bias, and is left for future study.

\subsection{Results for the estimator $\hat F_{Y|X}(y|x)$}\label{sec:simF}

In this section, we consider inference on $F_{Y|X}(y|x)$. We compare the coverage probability of the oracle asymptotic confidence interval 
\begin{align}
\big[\hat F_{Y|X}(Q(x_0;\tau)|x_0) \pm N^{-1/2} \sigma^2_F(\tau) \Phi^{-1}(1-\alpha/2)\big].\label{eq:funcovp2}
\end{align}
(here $\sigma^2_F(\tau)$ is the asymptotic variance of the oracle estimator) and the bootstrap confidence interval described above Theorem~\ref{th:boot}. Note that the other approaches described in Section~\ref{SEC:INF} are not directly applicable here since $\hat F_{Y|X}(y|x)$ is a functional of the whole process $\hat\zb(\cdot)$. Since we focus on bootstrap reliability, the number of quantile levels $K = 65$ and knots for spline interpolation $G = 32$ are chosen sufficiently large to ensure nominal coverage of oracle intervals. A detailed study of the impact of $K,G$ on coverage of oracle intervals is provided in Section~\ref{sec:lincdf} of the online supplement. Due to space limitations, we only show the results for small values of $S$, results for large values of $S$ do not give crucial additional insights and are deferred to the online supplement. Coverage probabilities for $m=4,32$ are reported in Figure~\ref{fig:Fhatcov}. For $m=4$, the bootstrap and oracle confidence interval show a very similar performance as soon as $S \geq 20$; this is in line with the coverage properties for the bootstrap for $\bar\zb$. For $m=32$, coverage probabilities of the oracle confidence interval indicate that the sub-sample size $n$ is too small and the oracle rule does not apply, even for $S=2$. Interestingly, coverage probabilities of bootstrap and asymptotic confidence intervals differ in this setting. This does not contradict our theory for the bootstrap since that was only developed under the assumption that we are in a regime where the oracle rule holds.  

Summarizing all results obtained so far, we can conclude that inference based on results from sub-samples is reliable for $S \geq 20$. Since this does not require additional computation, we recommend using the normal approximation for $S>20$ for point-wise inference and the bootstrap if process level results are required. For $S < 20$, estimating the asymptotic variance within sub-samples and aggregating is recommendable. The simplest approach which is based on averaging variance estimators from the \texttt{quanteg} package works well and does not require additional implementation, so we recommend to use this for $S<20$.

\begin{figure}[!h]
\includegraphics[width=13cm]{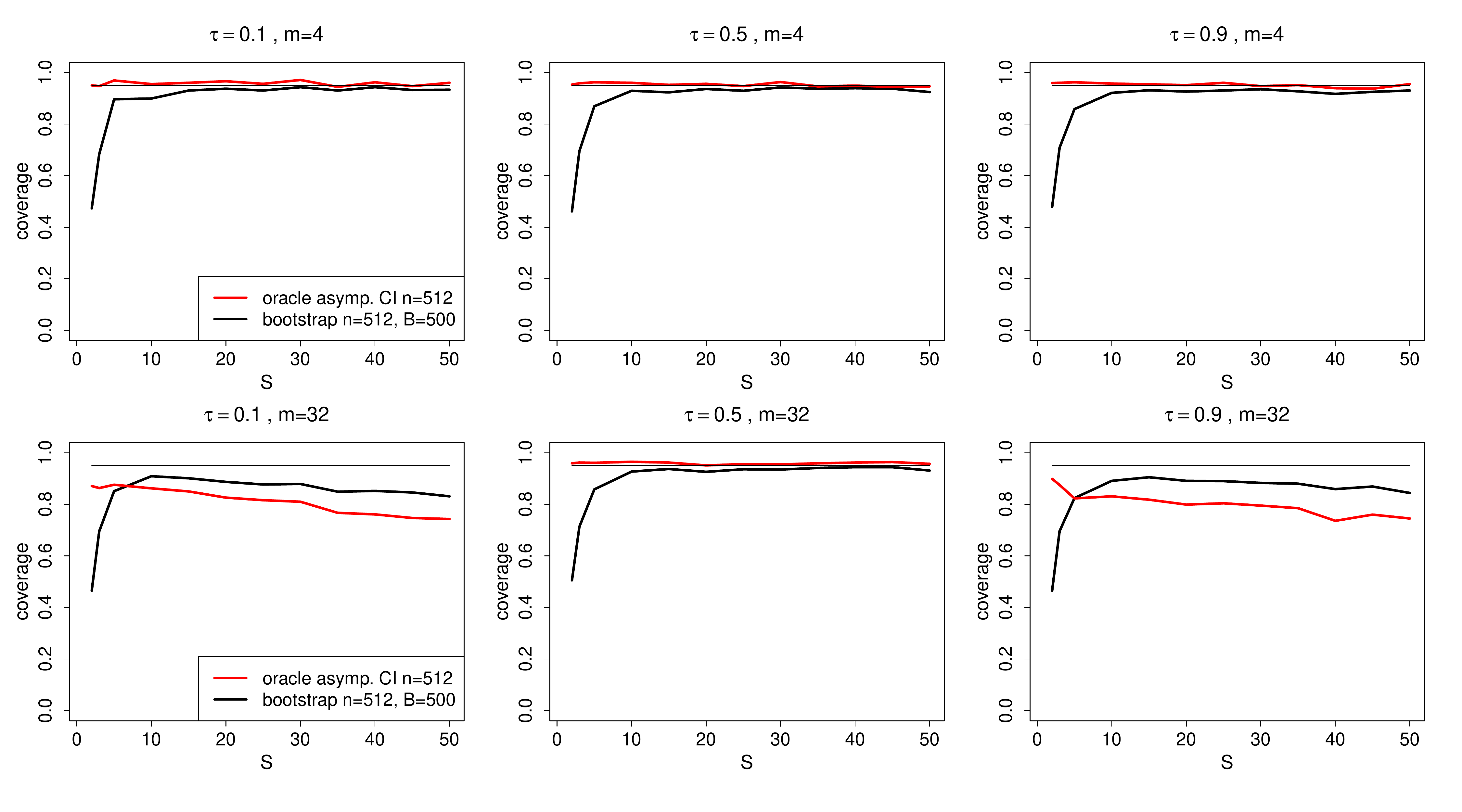}	
\caption{Coverage probabilities for oracle confidence intervals (red) and bootstrap confidence intervals (black) for $F_{Y|X}(y|x_0)$ for $x_0 = (1,...,1)/m^{1/2}$ and $y = Q(x_0;\tau)$, $\tau = 0.1, 0.5, 0.9$. $n = 512$ and nominal coverage $0.95$.}\label{fig:Fhatcov}
\end{figure}


\bibliography{_2017-11-21_BDQR}

\newpage

\appendix

\section{Approximate linear models without local basis structure} \label{SEC:GEN}


In this section, we consider models with transformation $\ZZ$ of increasing dimension that do not have the special local structure considered in Section~\ref{sec:loc}. 
The price for this generality is that we need to assume a more stringent upper bound for $S$ and put additional growth restrictions on $m$ in order to prove the oracle rule. The conditions on $K$ remain the same. 

\begin{theo}\label{th:gendiffup}
Assume that conditions \hyperref[A1]{(A1)}-\hyperref[A3]{(A3)} hold and that additionally $m \xi_m^2 \log N = o(N/S)$, $c_m(\zg_N) = o(\xi_m^{-1})$, $K\ll N^2$. Then
\begin{multline} \label{eq:genubpo}
\sup_{\tau \in \Tc_K} \|\overline{\zb}(\tau) - \hat\zb_{or}(\tau)\| 
\\
=~O_P\Big( \Big(\frac{m c_m^2\log N}{N} \Big)^{1/2} + c_m^2 \xi_m  + \Big(\frac{Sm\xi_m^2 \log N}{N} + c_m^4\xi_m^4\Big)\Big(1 + \frac{\log N}{S^{1/2}} \Big)\Big) 
\\
+ \frac{1}{N^{1/2}}O_P\Big( \Big(m\xi_m^2 c_m^2 (\log N)^3\Big)^{1/2}+\Big(\frac{Sm^3\xi_m^2(\log N)^7}{N}\Big)^{1/4}\Big) 
 + o_P(N^{-1/2}). 
\end{multline}
If additionally $K \gg G \gg 1$, $m^3\xi_m^2(\log N)^3 = o(N), c_m^2(\zg_N)\xi_m = o( N^{-1/2})$ we also have for any $x_0 \in \Xc$
\begin{multline} \label{eq:genubpr}
\sup_{\tau \in \Tc} |\ZZ(x_0)^\top\hat{\zb}(\tau) - \ZZ(x_0)^\top\hat\zb_{or}(\tau)| 
\\
\leq \|\ZZ(x_0)\|\sup_{\tau \in \Tc_K} \|\overline{\zb}(\tau) - \hat\zb_{or}(\tau)\| + \sup_{\tau \in \Tc} |(\Pi_K Q(x_0;\cdot))(\tau) - Q(x_0;\tau)|
\\
+ \sup_{\tau \in \Tc} |\ZZ(x_0)^\top\zb_N(\tau) - Q(x_0;\tau)| + o_P(\|\ZZ(x_0)\|N^{-1/2})  
\end{multline}
where the projection operator $\Pi_K$ was defined right after~\eqref{eq:hatbeta}.
\end{theo}

The proof of Theorem \ref{th:gendiffup} is given in Section \ref{sec:proof_dcolin}. The general upper bound provided in~\eqref{eq:genubpo} takes a rather complicated form. Under additional assumptions on $c_m$, major simplifications are possible. Due to space considerations we do not provide additional details here, rather we implicitly carry out the simplifications when proving the following result.

\begin{cor}[Oracle Rule for $\overline{\zb}(\tau)$] \label{th:orgenpo} Assume that conditions \hyperref[A1]{(A1)}-\hyperref[A3]{(A3)} hold and that additionally $m^4(\log N)^{10} = o(N)$, $c_m^2(\zg_N)\xi_m = o(N^{-1/2})$. Provided that additionally $S = o(N^{1/2}/(m\xi_m^2(\log N)^2))$ the estimator $\overline{\zb}(\tau)$ defined in \eqref{eq:zbbar} satisfies
\begin{align}
 \frac{\sqrt{N} \bu_N^\top(\overline \zb(\tau) - \zg_N(\tau))}{(\bu_N^\top J_m^{-1}(\tau)\E[\ZZ\ZZ^\top] J_m^{-1}(\tau)\bu_N)^{1/2}} \weak \Nc\big(0,\tau(1-\tau)\big), \label{eq:dcolin_asym}
\end{align}
for any $\tau \in \Tc, \bu_N \in \R^m$, where $J_m(\tau)$ is defined in the statement of Corollary~\ref{th:orlipo}. The same holds for the oracle estimator $\hat \zb_{or}(\tau)$.
\end{cor}

We note that $m^4(\log N)^{10} = o(N)$ imposes restriction on model complexity. In particular, this requires $m = o(N^{1/4}(\log N)^{-5/2})$. An immediate consequence of Corollary \ref{th:orgenpo} is the oracle rule for correctly specified models, including linear quantile regression with increasing dimension as a special case. In this case $c_m(\zg_N) = 0$.

\begin{cor}[Oracle rule for correctly specified models]\label{cor:orinc} 
	Assume that conditions \hyperref[A1]{(A1)}-\hyperref[A3]{(A3)} hold and that the quantile function satisfies $Q(x;\tau) = \ZZ(x)^\top \zg_N(\tau)$ with a transformation $\ZZ(x)$ of possibly increasing dimension $m$ with each entry bounded almost surely. Then
	$\bu_N^\top\overline\zb(\tau)$ satisfies the oracle rule provided that $m^4 (\log N)^{10} = o(N)$ and $S = o(N^{1/2}m^{-2}(\log N)^{-2})$. 
\end{cor}

This corollary reduces to Corollary~\ref{th:orlipo} in Section~\ref{sec:lin} when $m$ is fixed. It describes the effect of allowing $m$ to increase on the sufficient upper bound for $S$. We note that $c_m(\zg_N)=0$ and $\xi_m \asymp m^{1/2}$ under the settings of Corollary \ref{cor:orinc}, whose proof follows directly from Corollary \ref{th:orgenpo}. 

Both Corollary~\ref{th:orloc} and Corollary \ref{th:orgenpo} can be applied to local polynomial spline models, but Corollary~\ref{th:orloc} puts less assumptions on $S$ and $m$ than Corollary~\ref{th:orgenpo}, because Corollary~\ref{th:orloc} exploits the local support property of splines. This is illustrated in the following Remark \ref{rem:compspl} for the specific setting of Example~\ref{ex:spline}. 

\begin{remark}[Comparing Corollary~\ref{th:orloc} and Corollary \ref{th:orgenpo} with univariate splines] \label{rem:compspl} \rm 
Let $\ZZ$ denote the univariate splines from Example~\ref{ex:spline} and let $\bu_N := \ZZ(x_0)$ for a fixed $x_0 \in \Xc$. We assume that \hyperref[A2]{(A2)} and \hyperref[A3]{(A3)} hold and that $X$ has a density on $\Xc = [0,1]$ that is uniformly bounded away from zero and infinity.  
We will verify in Section \ref{sec:compsplpro} that \hyperref[A1]{(A1)} holds with $\xi_m \asymp m^{1/2}$. For simplicity, assume that the bias $c_m(\zg_N)$ satisfies $\xi_m c_m(\zg_N)^2 = o(N^{-1/2})$. Corollary~\ref{th:orloc} shows that sufficient conditions for the oracle rule are $m^2 (\log N)^6=o(N)$ and $S = o((Nm^{-1}(\log N)^{-4})^{1/2} \wedge (Nm^{-2}(\log N)^{-10}))$. On the other hand, Corollary~\ref{th:orgenpo} requires the more restrictive conditions $m^4 (\log N)^{10}=o(N)$ and $S = o(N^{1/2}m^{-2}(\log N)^{-2})$.  

\end{remark}  

Remark~\ref{rem:compspl} indicates that Corollary~\ref{th:orloc} gives sharper bounds than Corollary~\ref{th:orgenpo} when both results are applicable. Note however that Corollary~\ref{th:orgenpo} applies to more general settings since, in contrast to Corollary~\ref{th:orloc}, it does not require Condition \hyperref[L]{(L)}. For instance, linear models as discussed in Corollary~\ref{cor:orinc} can in general not be handled by Corollary~\ref{th:orloc}. Finally, we discuss sufficient conditions for the process oracle rule.

\newpage

\begin{cor}[Oracle Rule for $\widehat{\zb}(\tau)$]\label{th:orgenpr}
Let $x_0 \in \Xc$, let the conditions of Corollary~\ref{th:orgenpo} hold. Suppose that $\tau \mapsto Q(x_0;\tau) \in \Lambda_c^{\eta}(\Tc)$, $r_\tau\geq \eta$, $N^2 \gg K \gg G \gg N^{1/(2\eta)}\|\ZZ(x_0)\|^{-1/\eta}$, and $\sup_{\tau \in \Tc} |\ZZ(x_0)^\top\zg_N(\tau) - Q(x_0;\tau)| = o(\|\ZZ(x_0)\|N^{-1/2})$. Let the limit $H_{x_0}(\tau_1,\tau_2)$ defined in \eqref{eq:covloc} exist and be non-zero for any $\tau_1,\tau_2 \in \Tc$ 
\begin{enumerate}
\item The projection estimator $\widehat{\zb}(\tau)$ defined in \eqref{eq:hatbeta} satisfies
\begin{align} \label{eq:weakorgen}
\frac{\sqrt{N}}{\|\ZZ(x_0)\|}\big(\ZZ(x_0)^\top \hat\zb(\cdot) - Q(x_0;\cdot)\big) \weak \Gb_{x_0}(\cdot) \mbox{ in } \ell^\infty(\Tc),
\end{align}
where $\Gb_{x_0}$ is a centered Gaussian process with $\E[\Gb_{x_0}(\tau)\Gb_{x_0}(\tau')] = H_{x_0}(\tau,\tau')$. This matches the process convergence of $\ZZ(x_0)^\top \hat\zb_{or}(\tau)$. 
\item The estimator $\hat F_{Y|X}(\cdot|x_0)$ defined in \eqref{eq:Fhat} satisfies, 
\begin{align*}
\frac{\sqrt{N}}{\|\ZZ(x_0)\|}\big(\hat F_{Y|X}(\cdot|x_0) - F_{Y|X}&(\cdot|x_0)\big) \weak -f_{Y|X}(\cdot|x_0)\Gb_{x_0}\big(F_{Y|X}(\cdot|x_0)\big)\\
 &\mbox{ in } \ell^\infty\big((Q(x_0;\tau_L),Q(x_0;\tau_U))\big),
\end{align*}
where and $\Gb_{x_0}$ is the centered Gaussian process from~\eqref{eq:weakorgen}. The same is true for $\hat F_{Y|X}^{or}$.
\end{enumerate}
\end{cor} 

The proof of Corollary \ref{th:orgenpr} is given in Section \ref{sec:proof_orgenpr}. Note that the condition on $K$ is the same as in Corollary~\ref{th:orlocpr}. Results along the lines of Corollary~\ref{th:orgenpr} can be obtained for any estimator of the form $\bu_N^\top \hat\zb(\cdot)$, as long as $\bu_N$ satisfies certain technical conditions. For example, the partial derivative, the average derivative and the conditional average derivative of $Q(x;\tau)$ in $x$ fall into this framework. For brevity, we refer the interested reader to Section 2.3 of \cite{bechfe2011} for examples of vectors $\bu_N$ and omit the technical details. 

\begin{rem}\label{rem:center} \rm
The weak convergence in~\eqref{eq:dcolin_asym} was derived with the centering $\zg_N$ for notational convenience. As pointed out in~\cite{ChaVolChe2016}, other choices of centering sequences can be advantageous (for instance, this is the case in the setting of partial linear models, see also~\cite{bechfe2011} for an alternative centering sequence). The conclusions of Theorem~\ref{th:orgenpo} can be generalized by replacing $\zg_N$ in there by other vectors $\zb_N(\tau)$ as long as certain technical conditions are satisfied. However, this requires additional notation and technicalities. To keep the presentation simple we provide those details in Section~\ref{sec:zbN} in the appendix. Similarly, the results in Section~\ref{sec:loc} could be stated in terms of more general sequences $\zb_N$ instead of the one considered in~\eqref{eq:gamman} at the cost of additional notation and conditions. 
\end{rem}

\newpage


\newpage

\newpage

\begin{supplement}[id=suppA]
 \sname{ONLINE SUPPLEMENTARY MATERIAL}
 \stitle{Distributed inference for quantile regression processes}
 \slink[doi]{COMPLETED BY THE TYPESETTER}
 \sdatatype{.pdf}
 \sdescription{The supplementary materials contain all the proofs, additional technical details, and additional simulation results.}
\end{supplement}
\setcounter{page}{1}

\setcounter{section}{0}
\renewcommand{\thesection}{S.\arabic{section}}
\renewcommand{\thetheo}{S.\arabic{section}.\arabic{theo}}
\renewcommand{\theequation}{S.\arabic{section}.\arabic{equation}}
\renewcommand{\thesubsection}{S.\arabic{section}.\arabic{subsection}}
%

\noindent\textbf{Additional Notation.} This notation will be used throughout the supplement. Let $\Sc^{m-1} := \{\bu \in \R^m: \|\bu\| = 1\}$ denote the unit sphere in $\R^m$. For a set $\Ic \subset \{1,...,m\}$, define 
\begin{align}
	\R_\Ic^m &:= \{\bu=(u_1,...,u_m)^\top\in \R^m: u_j \neq 0 \mbox{ if and only if } j \in \Ic \}
	\label{eq:defRI}\\
	\Sc_\Ic^{m-1} &:= \{\bu=(u_1,...,u_m)^\top\in \Sc^{m-1}: u_j \neq 0 \mbox{ if and only if } j \in \Ic \}\label{eq:defSI}
\end{align}
For simplicity, we sometimes write $\sup_{\tau}$($\inf_{\tau}$) and $\sup_{x}$($\inf_{x}$) instead of $\sup_{\tau\in\Tc}$($\inf_{\tau\in\Tc}$) and $\sup_{x\in\Xc}$($\inf_{x\in\Xc}$) throughout the Appendix. $I_k$ is the $k$-dimensional identity matrix for $k\in\N$.

\section{An alternative estimator for $J_{m}^P (\tau)$}\label{sec:powor}

Here, we provide a simple way to estimate the matrix $J_m(\tau)$ in a divide and conquer setup provided that an additional round of communication is acceptable. 

\begin{enumerate}
\item Compute $\overline \zb(\tau)$ on the master computer and send to each distributed unit.
\item On each machine, compute
\begin{equation*} 
\widetilde J_{ms}^P (\tau) := \frac{1}{2nh_n} \sum_{i=1}^n \ZZ_{is}\ZZ_{is}^\top \IF \Big\{|Y_{is} - \ZZ_{is}^\top \overline \zb(\tau)| \leq h_n \Big\} 
\end{equation*}
and send results back to master machine.
\item On the master, compute $\widetilde J_{m}^P (\tau) := \frac{1}{S}\sum_{s=1}^S \widetilde J_{ms}^P (\tau)$. The final variance estimator is given by $\tau(1-\tau) \widetilde J_{m}^P (\tau)^{-1} \overline\Sigma_{1} \widetilde J_{m}^P (\tau)^{-1}$.
\end{enumerate} 
Provided that $\overline\zb(\tau)$ satisfies the oracle rule, the theory developed in \cite{kato2012} for Powell's estimator applies (note that the results developed in \cite{kato2012} only require $\sqrt{N}$ consistency of the estimator $\overline \zb$, which is guaranteed by the oracle rule). In particular, following \cite{kato2012}, the optimal rate for the bandwidth $h_n$ is given by $h_N \sim N^{-1/5}$, and an explicit formula for the optimal bandwidth can be found on page 263 of \cite{kato2012}. In practice, the 'rule of thumb' provided on page 264 in \cite{kato2012} can be used to determine the bandwidth $h_n$. 

The algorithm described above provides a simple way to estimate the asymptotic covariance matrix of $\overline\zb(\tau)$ at oracle rate $N^{-2/5}$ under the same assumptions that guarantee $\overline \zb$ to satisfy the oracle rule. The additional communication and computational costs are also fairly low since only the vectors $\overline \zb$ and the matrices $\widetilde J_{ms}^P (\tau)$ need to be communicated and computation of $\widetilde J_{ms}^P (\tau)$ on each sub-sample does not involve any optimization. 

\vspace{2em}

\section{Aggregated Bahadur representations}\label{SEC:APP_ABR}


\subsection{Aggregated Bahadur representation for general series approximation model}\label{sec:ABR}

Note that Assumptions \hyperref[A1]{(A1)}-\hyperref[A3]{(A3)} imply that for any sequence of $\R^m$-valued functions $\zb_N(\tau)$ satisfying $\sup_{\tau\in\Tc}\sup_x |\zb_N(\tau)^\top\ZZ(x)-Q(x;\tau)| = o(1)$, the smallest eigenvalues of the matrices
\begin{align}
\tilde J_m(\zb_N(\tau)) := \E[\ZZ\ZZ^\top f_{Y|X}(\zb_N(\tau)^\top\ZZ|X)], \quad J_m(\tau) := \E[\ZZ\ZZ^\top f_{Y|X}(Q(X;\tau)|X)]\label{eq:defJm}
\end{align}
are bounded away from zero uniformly in $\tau$ for sufficiently large $n$. Define for any vector of functions $\bb_N:\Tc\to\R^m$,
\begin{align}
g_N(\bb_N) := \sup_{\tau \in \Tc} \big\|\E[\ZZ_i(\IF\{Y_i \leq \ZZ_i^\top  \bb_N(\tau)\}- \tau)]\big\|.
\label{eq:defgn}
\end{align}

\begin{theo}[ABR for General Series Model] \label{th:aggbah} 
	Assume that the underlying distribution of $(X_i,Y_i)$ satisfies conditions \hyperref[A1]{(A1)}-\hyperref[A3]{(A3)} and that $m \xi_m^2 \log n = o(n)$. Consider any sequence $\zb_N(\tau)$ such that $g_N := g_N(\zb_N) = o(\xi_m^{-1}), c_m := c_m(\zb_N) = o(1)$. Then we have
\[
\overline\zb(\tau) - \zb_N(\tau) = -\frac{1}{N}J_m(\tau)^{-1} \sum_{i=1}^N \ZZ_i(\IF\{Y_i \leq \ZZ_i^\top \zb_N(\tau)\}- \tau)  + r_N^{(1)}(\tau) + r_N^{(2)}(\tau)
\]
where, 
for any $\kappa_n \ll n/\xi_m^2$, any sufficiently large $n$, and a constant $C$ independent of $n$
\begin{equation} \label{eq:th1rn1}
P\Big(\sup_{\tau\in\Tc} \|r_N^{(1)}(\tau)\| \geq C R_{1,n}(\kappa_n) \Big) \leq S e^{-\kappa_n}
\end{equation}
with
\beq\label{eq:th1RN1}
R_{1,n}(\kappa_n) := \xi_m \Big( \Big(\frac{m}{n} \log n\Big)^{1/2}  + \Big(\frac{\kappa_n}{n}\Big)^{1/2}\Big)^2 + c_m \Big( \Big(\frac{m}{n} \log n\Big)^{1/2}  + \Big(\frac{\kappa_n}{n}\Big)^{1/2}\Big) + g_N.
\eeq
Moreover there exists a constant $C$ independent of $n,S$ such that for $n$ sufficiently large and $A \geq 1$
\begin{equation} \label{eq:th1rn2}
\sup_{\tau\in\Tc} P\Big(\big\| r_N^{(2)}(\tau)\big\| > C R_{2,n}(A,\kappa_n) \Big) \leq 2Sn^{-A} + 2e^{-\kappa_n^2}
\end{equation}
where 
\begin{multline}\label{eq:th1RN2}
R_{2,n}(A,\kappa_n) := \frac{A \kappa_n}{S^{1/2}} \Big(\frac{m}{n}\Big)^{1/2}\Big( (\xi_m g_N \log n)^{1/2} + \Big(\frac{m\xi_m^2 (\log n)^3}{n} \Big)^{1/4}\Big)
\\
 + \Big(1+ \frac{\kappa_n}{S^{1/2}}\Big)\Big(g_N^2\xi_m^2 + \frac{m \xi_m^2 \log n}{n} \Big).
\end{multline}
\end{theo}

See Section \ref{sec:proof_aggbah} of the online Appendix for a proof for Theorem \ref{th:aggbah}. We give a simpler expression for the remainder terms in the following remark.

\begin{rem}
	For subsequent applications of the Bahadur representation Theorem \ref{th:aggbah}, we will often set $\kappa_n = A \log n$ for a sufficiently large constant $A>0$, and assume the remainder terms $r_N^{(1)}(\tau)+r_N^{(2)}(\tau)$, defined in Theorem \ref{th:aggbah}, are bounded with certain order. 
	
	Under the conditions $\kappa_n = A \log n$, $(\log n)^{2}S^{-1}=o(1)$ and $m \geq 1$, $R_{1,n}(\kappa_n)+R_{2,n}(\kappa_n)$ can be bounded by (neglecting a positive constant)
\begin{multline}
c_m \Big(\frac{m}{n} \log n\Big)^{1/2} + g_N+\frac{\log n}{S^{1/2}} \Big(\frac{m}{n}\Big)^{1/2}\Big( (\xi_m g_N \log n)^{1/2} + \Big(\frac{m\xi_m^2 (\log n)^3}{n} \Big)^{1/4}\Big)
\\
+ g_N^2\xi_m^2 + \frac{m \xi_m^2 \log n}{n}.\label{eq:bahsimres}
\end{multline}
	We note that without the condition $(\log n)^{2}S^{-1}=o(1)$, we have an additional $\log n$ factor to the last two terms in \eqref{eq:bahsimres}.
\end{rem}

\begin{rem} \label{rem:fixeddiminter} 
A result along the lines of Theorem~\ref{th:aggbah} can also be obtained by a direct application of the classical Bahadur representation (e.g. Theorem~5.1 in \cite{ChaVolChe2016}), but that leads to bounds on remainder terms that are less sharp compared to the ones provided in \eqref{eq:th1rn1}-\eqref{eq:th1RN2}. To illustrate this point, assume that $Q(x;\tau) = x^\top\zb(\tau)$ with $x$ denoting a covariate of fixed dimension. In this case $g_N(\zb_N) = c_m(\zb_N) = 0$ and $\xi_m, m$ are constant. The classical Bahadur representation in Theorem~5.1 of \cite{ChaVolChe2016} yields an expansion of the form 
	\[
	\hat \zb^s(\tau) - \zb_N(\tau) = -\frac{1}{n}J_m(\tau)^{-1} \sum_{i=1}^n X_{is}(\IF\{Y_{is} \leq Q(X_{is};\tau)\}- \tau) + r_n^s 
	\]
	where $J_m(\tau) := \E[XX^\top f_{Y|X}(Q(X;\tau)|X)]$, $r_n^s = O_P(n^{-3/4}\log n)$ and the exponent of $n$ in the bound on $r_n^s$ can in general not be improved. Applying this expansion to the estimator in each group we obtain
	\begin{align}
	\overline{\zb}(\tau) - \zb_N(\tau) = -\frac{1}{N}J_m(\tau)^{-1} \sum_{i=1}^N X_i(\IF\{Y_i \leq Q(X_i;\tau)\}- \tau) + \frac{1}{S}\sum_{s=1}^S r_n^s.\label{eq:agglinfix}
	\end{align}
	Without additional information on the moments of $r_n^s$, the bound $r_n^s = O_P(n^{-3/4}\log n)$ would at best provide a bound of the form $\frac{1}{S}\sum_{s=1}^S r_n^s = O_P(n^{-3/4}\log n)$. 
	Thus by the representation \eqref{eq:agglinfix} the oracle rule holds provided that $n^{-3/4}\log n = o(N^{-1/2})$, which is equivalent to $S = o(N^{1/3}(\log N)^{-4/3})$ for an arbitrarily small $\eps>0$. In contrast, for $c_m(\zb_N) = g_N(\zb_N) = 0$ the bound in \eqref{eq:bahsimres} reduces to $n^{-1} \log n +n^{-3/4}S^{-1/2}(\log n)^{7/4}$. Since $N = Sn$ this is of order $o(N^{-1/2})$ if $S = o(N^{1/2}(\log N)^{-2})$. This is a considerably less restrictive condition than $S = o(N^{1/3}(\log N)^{-4/3})$ which was derived by the 'direct' approach above. Similar arguments apply in a more general setting of a linear model with increasing dimension or series expansions.

\end{rem} 

\subsection{Aggregated Bahadur representation for local basis functions}\label{sec:ABRloc}

\begin{theo}[ABR for local basis functions] \label{th:aggbah_loc}  Assume that the underlying distribution satisfies conditions \hyperref[A1]{(A1)}-\hyperref[A3]{(A3)} and let Condition \hyperref[L]{(L)} hold. Assume that $m\xi_m^2\log n = o(n)$, $c_m(\zg_N) = o(1)$ and that $c_m(\zg_N)^2 = o(\xi_m^{-1})$.
Then we have for $\bu_N \in \Sc_\Ic^{m-1}$ (defined in \eqref{eq:defSI}) with $\Ic$ consisting of $L = O(1)$ consecutive integers ($\Ic$ is allowed to depend on $n$)
\begin{multline}\label{eq:ag_loc1}
\bu_N^\top \overline\zb(\tau) - \bu_N^\top \zg_N(\tau) = -\frac{1}{N} \bu_N^\top \tilde J_m(\zg_N(\tau))^{-1} \sum_{i=1}^N \ZZ_i(\IF\{Y_i \leq \ZZ_i^\top \zg_N(\tau)\}- \tau)
\\
+ r_N^{(1)}(\tau,\bu_N) + r_N^{(2)}(\tau,\bu_N)
\end{multline}
where, for any $\kappa_n \ll n/\xi_m^2$, any sufficiently large $n$, and a constant $C$ independent of $n$
\begin{equation} \label{eq:th11rn1}
P\Big(\sup_{\tau\in\Tc} \sup_{\bu_N\in\Sc_\Ic^{m-1}}\|r_N^{(1)}(\tau,\bu_N)\| \geq C R_{1,n}^{(L)}(\kappa_n)\Big) \leq S e^{-\kappa_n}.
\end{equation}
where
\beq\label{eq:R1NL}
R_{1,n}^{(L)}(\kappa_n) := \frac{\xi_m \log n}{n}+ \sup_{\bu_N \in \Sc_\Ic^{m-1}}\tilde{\mathcal{E}}(\bu_N,\zg_N)\Big(c_m(\zg_N)^2 + \frac{\xi_m^2(\log n)^2 + \xi_m^2\kappa_n}{n} \Big).
\eeq
Moreover there exists a constant $C$ independent of $n,S$ such that for $n$ sufficiently large and $A \geq 1$
\begin{equation} \label{eq:th11rn2}
\sup_\tau P\Big(\sup_{\bu_N \in \Sc_\Ic^{m-1}}\Big| r_N^{(2)}(\tau,\bu_N)\Big| > C R_{n,2}^{(L)}(A,\kappa_n) \Big) \leq 2Sn^{-A} + 2e^{-\kappa_n^2}
\end{equation}
where 
\begin{align} \nonumber
R_{n,2}^{(L)}(A,\kappa_n) :=\  & A \sup_{\bu_N \in \Sc_\Ic^{m-1}}\tilde{\mathcal{E}}(\bu_N,\zg_N)\Big(1+\frac{\kappa_n}{S^{1/2}}\Big) \Big(c_m(\zg_N)^4 + \frac{\xi_m^2 (\log n)^2}{n}\Big) 
\\
&+ A \frac{\kappa_n}{n^{1/2}S^{1/2}}\Big(c_m(\zg_N)(\kappa_n^{1/2}\vee \log n) + \frac{\xi_m^{1/2}(\kappa_n^{1/2}\vee \log n)^{3/2}}{n^{1/4}}\Big). \label{eq:DnLK}
\end{align}
\end{theo} 

See Section \ref{sec:proof_aggbah_loc} of the online Appendix for a proof of Theorem \ref{th:aggbah_loc}. 

The following corollary gives an ABR representation which centered at the true function.

\begin{cor}\label{cor:aggbah_loctrue}
	If the conditions in Theorem \ref{th:aggbah_loc} hold, and additionally for a fixed $x_0\in\Xc$, $c_m(\zg_N) = o(\|\ZZ(x_0)\|N^{-1/2})$  
	then for any $\tau \in \Tc$, \eqref{eq:ag_loc1} in Theorem \ref{th:aggbah_loc} holds with $\bu_N$ and $\bu_N^\top \zg_N(\tau)$ being replaced by $\ZZ(x_0)$ and $Q(x_0;\tau)$.
\end{cor}

\begin{lemma}\label{lem:locexp}
In setting of Example~\ref{ex:spline} we have
\begin{align*}
\sup_{\tau \in \Tc, x\in [0,1]}\E|\ZZ(x)\tilde J_m^{-1}(\zg_N(\tau))\ZZ| &= O(1)
\\
\sup_{\tau \in \Tc, x\in [0,1]}\E|\ZZ(x) J_m^{-1}(\tau)\ZZ| &= O(1)
\end{align*}
\end{lemma}

The proof of Lemma~\ref{lem:locexp} follows by similar arguments as the proof of Example 2.3 on page 3279 in \cite{ChaVolChe2016} and is omitted for brevity. 



\section{Proofs for Section \ref{SEC:THEO} and Section \ref{SEC:GEN}}

\subsection{Proofs for Section \ref{sec:lin}} \label{sec:proof_par}

Here we provide proofs for Theorem~\ref{th:linub} and the sufficient conditions for the oracle rule in Corollary~\ref{th:orlipo} and Corollary~\ref{th:orlipr}. Theorem~\ref{th:lilob} and the necessary conditions of Corollary~\ref{th:orlipo} and Corollary~\ref{th:orlipr} will be established in Section~\ref{sec:pfnec}.  

Theorem~\ref{th:linub} follows from Theorem~\ref{th:gendiffup} by noting that for any $P \in \Ps_1$, the quantities $m,\xi_m$ are constants that do not depend on $N$. Moreover, since $J_m(\tau)$ is invertible by assumptions \hyperref[A1]{(A1)}-\hyperref[A3]{(A3)}, $\zg_N \equiv \zb$, and additionally we have $c_m(\zg_N) = g_N(\zg_N) = 0$. 

The sufficiency of $S = o(N^{1/2}(\log N)^{-1})$ for weak convergence of $\bar\zb$ in Corollary~\ref{th:orlipo} follows from the corresponding weak convergence of $\hat\zb_{or}$ and~\eqref{eq:linubpo} which implies that under $S = o(N^{1/2}(\log N)^{-1})$ we have $\bar \zb - \hat\zb_{or} = o_P(N^{-1/2})$. 

It remains to prove the 'sufficient' part of Corollary~\ref{th:orlipr}. As a direct consequence of the second part of Theorem~\ref{th:linub} we obtain 
\[
\sup_{\tau\in\Tc} |\hat \beta_j(\tau) - \hat\beta_{or,j}(\tau)) | \leq \|\Pi_K \beta_j(\cdot) - \beta_{j}(\cdot)\|_\infty + o_P(N^{-1/2}).
\]

Next, note that by assumption, $\beta_j \in \Lambda_c^{\eta}(\Tc)$. Applying Theorem~2.62 and Theorem~6.25 of \cite{schumaker:81}, we find that
\begin{align}
\inf_{g \in \Theta_G} \|\beta_j(\cdot) - g\|_{\infty} \lesssim G^{-\eta},\label{eq:ha1}
\end{align}
where $\Theta_G$ denotes the space of splines of degree $r_\tau$ over the grid $t_1,...,t_G$ on $\Tc=[\tau_L,\tau_U]\subset(0,1)$. This together with Lemma~5.1 of \cite{huang2003} 
shows that
\begin{align}
\|\Pi_K \beta_j(\cdot) - \beta_j(\cdot)\|_\infty = O(G^{-\eta}).\label{eq:bdh_N}
\end{align}
Thus the 'sufficient' part of the statement of Corollary~\ref{th:orlipr} follows. \hfill $\Box$



\subsection{Proofs for Section \ref{sec:loc}}\label{sec:proof_loc}

This section contains the proofs of Theorem~\ref{th:locub} and the sufficient conditions for weak convergence in Corollary~\ref{th:orloc} and Corollary~\ref{th:orlocpr}. The proof of Theorem~\ref{th:loclob} and the necessary conditions in Corollary~\ref{th:orloc} and Corollary~\ref{th:orlocpr} are given in section~\ref{sec:pfnec}.

\subsubsection{Proof of Theorem~\ref{th:locub}}\label{sec:proof_locub} 

We begin with a proof of~\eqref{eq:locubpo}. Define
\begin{multline}\label{eq:RNS}
R_n(S) := \Big(1+\frac{\log N}{S^{1/2}}\Big)\Big(c_m^2(\zg_N) + \frac{S\xi_m^2 (\log N)^2}{N}\Big)
\\
+ \frac{\|\ZZ(x_0)\| \xi_m S \log N}{N}  + \frac{\|\ZZ(x_0)\|}{N^{1/2}}\Big(\frac{S \xi_m^2 (\log N)^{10}}{N}\Big)^{1/4}.
\end{multline}
From Theorem~\ref{th:bahsimple_bspl} in the supplementary material we obtain the following representation for the oracle estimator
\begin{align}\label{eq:bahgen1_loc}
\sup_{\tau\in\Tc} \Big| \frac{\ZZ(x_0)^\top}{\|\ZZ(x_0)\|}(\hat\zb_{or}(\tau) - \zg_N(\tau)) - \tilde\bU_{1,N}(\tau)\Big|
= O_P(R_N(1))
\end{align} 
where 
\begin{align}
	\tilde \bU_{1,N}(\tau):=- N^{-1} \frac{\ZZ(x_0)^\top}{\|\ZZ(x_0)\|} \tilde J_m(\zg_N(\tau))^{-1}\sum_{i=1}^N \ZZ_i(\IF\{Y_i \leq \ZZ_i^\top \zg_N(\tau)\}- \tau).\label{eq:tU1N}
\end{align}
Let $\bu_N := \frac{\ZZ(x_0)}{\|\ZZ(x_0)\|}$ and apply Theorem~\ref{th:aggbah_loc} to obtain 
\begin{align}
\bu_N^\top(\overline\zb(\tau) - \zg_N(\tau)) = \tilde\bU_{1,N}(\tau) + r_N^{(1)}(\tau,\bu_N) + r_N^{(2)}(\tau,\bu_N)
\end{align}
with 
\begin{align*}
&P\Big(\sup_{\tau\in\Tc_K} |r_N^{(1)}(\tau,\bu_N) + r_N^{(2)}(\tau,\bu_N)|\geq \frac{C}{2}R_{1,n}^{(L)}(A\log n) + \frac{C}{2}R_{2,n}^{(L)}(A,A\log n) \Big)
\\
 \leq&  2N^2(3S+2)N^{-A}
\end{align*}
and $R_{1,n}^{(L)}(A\log n), R_{2,n}^{(L)}(A,A\log n)$ defined in Theorem~\ref{th:aggbah_loc}. Note that under assumption~\hyperref[L]{(L)} we have $\tilde \Ec(\bu_N,\zg_N) = O(\xi_m^{-1})$. Now a straightforward but tedious calculation shows that under the assumptions of Theorem~\ref{th:locub} we have, for any fixed constant $A>0$,
\begin{align*}
R_{1,n}^{(L)}(A\log n) + R_{2,n}^{(L)}(A,A\log n) = O(R_n(S)).
\end{align*} 
Choosing $A=3$, a combination of~\eqref{eq:bahgen1_loc} and the bound above yields~\eqref{eq:locubpo}. 

To prove the second claim we first shall prove that 
\beq\label{eq:unloc}
\sup_{\tau \in \Tc} \big\|\tilde\bU_{1,N}(\tau)-\bU_N(\tau) \big\| = o_P\big(N^{-1/2}\big),
\eeq
where
\begin{align}
	\bU_N(\tau):=- N^{-1} \frac{\ZZ(x_0)^\top}{\|\ZZ(x_0)\|} J_m(\tau)^{-1}\sum_{i=1}^N \ZZ_i(\IF\{Y_i \leq Q(X_i;\tau)\}- \tau).\label{eq:UN}
\end{align}
{This follows by similar arguments as Step 1 in the proof of Theorem 2.4 in Section A.2 on pages 3294-3296 of  \cite{ChaVolChe2016}, by noting that the conditions required there follow exactly by our \hyperref[A1]{(A1)}-\hyperref[A3]{(A3)}, \hyperref[L]{(L)}, \hyperref[L1]{(L1)} where we use that $\xi_m^4 (\log N)^6 = o(N)$, $c_m(\zg_N)^2=o(N^{-1/2})$ by \hyperref[L1]{(L1)}, and $\Ic$ and $L$ in the statement of Theorem 2.4 of \cite{ChaVolChe2016} being the coordinates of the non-zero components of $\ZZ(x_0)$ and $r$, respectively.} 

The proof of the second claim is very similar to the proof of~\eqref{eq:genubpr}. Use~\eqref{eq:unloc} instead of~\eqref{eq:barh} and~\eqref{eq:bahgen1_loc} instead of~\eqref{eq:bahgen1} and similar arguments as in the proof of~\eqref{eq:genubpr} to obtain    
\begin{multline*}
\sup_{\tau \in \Tc} \Big|(\Pi_K \ZZ(x_0)^\top \hat\zb_{or}(\cdot))(\tau) - \ZZ(x_0)^\top \hat \zb_{or}(\tau)\Big| \leq o_P(\|\ZZ(x_0)\|N^{-1/2})
\\ 
+ \sup_{\tau \in \Tc} \Big|[\Pi_K Q(x_0;\cdot)](\tau) - Q(x_0;\tau)\Big| + \sup_{\tau \in \Tc} |\ZZ(x_0)^\top\zg_N(\tau) - Q(x_0;\tau)|
\\
+ \sup_{\tau \in \Tc} \Big|[\Pi_K \ZZ(x_0)^\top \bU_N](\tau) - \ZZ(x_0)^\top \bU_N(\tau)\Big|, 
\end{multline*}
and bound the last term exactly as in the proof of~\eqref{eq:genubpr}. 

\subsubsection{Proof of \eqref{eq:weakloc} and the weak convergence of $\hat\zb_{or}(\tau)$ in Corollary~\ref{th:orloc}}\label{sec:proof_prop_orloc1} 

Combining~\eqref{eq:bahgen1_loc} and~\eqref{eq:unloc} we obtain the representation
\begin{multline*}
\ZZ(x_0)^\top(\hat\zb_{or}(\tau) - \zg_N(\tau))
\\
 = - N^{-1} \ZZ(x_0)^\top J_m(\tau)^{-1}\sum_{i=1}^N \ZZ_i(\IF\{Y_i \leq Q(X_i;\tau)\}- \tau)
 + o_P\Big(\frac{\|\ZZ(x_0)\|}{N^{1/2}}\Big). 
\end{multline*}
Moreover, by straightforward calculations bounding the terms in~\eqref{eq:locubpo} we obtain under assumption~\hyperref[L1]{(L1)} the representation $\ZZ(x_0)^\top(\hat\zb_{or}(\tau)-\hat\zb(\tau)) = o_P(\|\ZZ(x_0)\|N^{-1/2})$.

To prove the weak convergence result for both $\hat\zb_{or}$ and $\overline{\zb}$, it thus suffices to prove that 
\[
\frac{1}{\sqrt{N}} \sum_{i=1}^N \frac{\ZZ(x_0)^\top J_m(\tau)^{-1}\ZZ_i(\IF\{Y_i \leq Q(X_i;\tau)\} -\tau)}{(\ZZ(x_0)^\top J_m(\tau)^{-1}\E[\ZZ\ZZ^\top]J_m(\tau)^{-1}\ZZ(x_0))^{1/2}} \weak \Nc(0,\tau(1-\tau)).
\] 
This follows by an application of the Lindeberg CLT. The verification of the Lindeberg condition here is a simple modification from finite dimensional joint convergence in the proof of (2.4) in Section A.1.2 in \cite{ChaVolChe2016} to pointwise convergence. \hfill$\qed$

\subsubsection{Proof of~\eqref{eq:prloc} and part 2 of Corollary \ref{th:orlocpr}}\label{sec:pforlocpr}

The proof of the process convergence of $\hat\zb_{or}(\tau)$ follows directly by Theorem 2.4 of \cite{ChaVolChe2016}. Moreover, the bound 
\[
\sup_{\tau\in \Tc}|\ZZ(x_0)(\hat \zb(\tau) - \hat \zb_{or}(\tau))| = o_P(N^{-1/2}\|\ZZ(x_0)\|)
\]
follows from Theorem~\ref{th:locub} after some simple computations to bound the term $\sup_{\tau\in \Tc_K}|\ZZ(x_0)(\overline{\zb}(\tau) - \hat \zb_{or}(\tau))|$ and using similar arguments as in the proof of~\eqref{eq:bdh_N} to obtain the bound 
\[
\sup_{\tau \in \Tc} \Big|[\Pi_K Q(x_0;\cdot)](\tau) - Q(x_0;\tau)\Big|  = O(G^{-\eta}).
\]
This implies~\eqref{eq:prloc}. The proof of the second part follows by similar arguments as the proof of Corollary~4.1 in \cite{ChaVolChe2016} and is omitted for the sake of brevity.  \hfill $\qed$


\subsection{Proof of Theorem~\ref{th:lilob}, Theorem~\ref{th:loclob} and necessary conditions for $S,G$ in Section~\ref{sec:lin} and Section~\ref{sec:loc}}
\label{sec:pfnec}

~

First, we will prove in Section~\ref{sec:sharpG} that the oracle rule in Corollary~\ref{th:orlipr} fails whenever $G \lesssim N^{1/(2\eta)}$ and that the oracle rule in Corollary~\ref{th:orlocpr} fail whenever $G \lesssim N^{1/(2\eta)}\|\ZZ(x_0)\|^{-1/\eta}$, no matter how $S$ is chosen. 

Second, in Section~\ref{sec:specconst} we will construct a special class of data generation processes. This construction will be utilized in Section~\ref{sec:failorpo} to prove that the oracle rules in Corollary~\ref{th:orlipo} and Corollary~\ref{th:orloc} fail if $S \gtrsim N^{1/2}$ or $S \gtrsim N^{1/2}\xi_m^{-1}$, respectively. 

Third, in Section~\ref{sec:failorpr} we will prove that, provided $G \gg N^{1/(2\eta_\tau)}$ and $S \gtrsim N^{1/2}$, the oracle rule in Corollary~\ref{th:orlipr} fails and that, for $G \gg N^{1/(2\eta)}\|\ZZ(x_0)\|^{-1/\eta_\tau}$ and $S \gtrsim N^{1/2}\xi_m^{-1}$, the oracle rule in Corollary~\ref{th:orlocpr} fails.

Fourth, in Section~\ref{sec:prlob} we will derive the lower bounds in Theorem~\ref{th:lilob} and Theorem~\ref{th:loclob} from the necessity of the conditions in Corollary~\ref{th:orloc} and \ref{th:orlocpr}. 

\subsubsection{Sharpness of the conditions on $G$ in Corollary \ref{th:orlipr} and \ref{th:orlocpr}}\label{sec:sharpG}

Fix an integer $\eta_\tau \geq 2$ and assume that $G = G_N \lesssim N^{1/(2\eta_\tau)}$. We will now prove that weak convergence in~\eqref{eq:fixedpro} fails for some $(P,\ZZ) \in \Ps_1(d^{1/2},1,\overline{f},\overline{f'},f_{min})$, independently of the choice of $S$.  

We begin by noting that it is possible to construct a function $g$ on $\R$ with the following properties: $g$ has support $[-1,1]$, $g$ is $\eta_\tau$ times continuously differentiable on $\R$ with absolute values of all the function and derivatives bounded uniformly by 1, $g$ does not vanish on $\R$. From the properties listed here it is clear that $g$ is not a polynomial and hence, $d_{r_\tau} := \inf_{p\in \Ps_{r_\tau}} \sup_{x\in [-1,1]}|g(x)-p(x)|>0$ for any fixed $r_\tau\geq \eta_\tau$ and $r_\tau\in\N$, where $\Ps_{r_\tau}$ is the space of polynomials of degree $r_\tau$. 

We now construct a function $Q$ on $\Tc= [\tau_L,\tau_U]$ that is a proper quantile function with the following properties: the corresponding distribution function $F$ is twice continuously differentiable, its density $f$ satisfies \hyperref[A2]{(A2)}-\hyperref[A3]{(A3)} and there exists a subsequence $N_k$ and $\eps > 0$ such that for all $k$
\begin{align}
G_{N_k}^{\eta_\tau} \inf_{p \in \Ps\Ps_{r_\tau}(G_{N_k})} \sup_{\tau\in \Tc}|Q(\tau) - p(\tau)| \geq \eps,\label{eq:lowbdd}
\end{align}
where $\Ps\Ps_{r_\tau}(G_N)$ is the set of piecewise polynomials of degree $r_\tau$ defined on a the equally spaced grid $\tau_L = t_{1,N} < ... < t_{G_N,N} = \tau_U$. Begin by considering the case $\limsup_{N \to \infty} G_N = \infty$. Pick a subsequence $N_k$ of $N$ and a sequence of integers $\nu_k \leq G_{N_k}$ such that
\begin{multline*}
I_k := \Big[\tau_L + \frac{\nu_k(\tau_U - \tau_L)}{G_{N_k}}, \tau_L + \frac{(\nu_k+1)(\tau_U - \tau_L)}{G_{N_k}} \Big]
\\
\subset \Big(\tau_L + (\tau_U - \tau_L)\sum_{j = 1}^{k-1} 2^{-j}, \tau_L + (\tau_U - \tau_L)\sum_{j = 1}^{k} 2^{-j} \Big).
\end{multline*} 
By construction, the intervals $I_k$ are mutually disjoint and each interval is contained between two adjacent knots in the grid $t_{1,N_k},...,t_{G_{N_k},N_k}$.

Define for some $a,c>0$ (specified later): 
\begin{align}
Q(\tau) := a\tau + c \sum_{k=2}^\infty (2G_{N_k})^{-\eta_\tau} g(2G_{N_k}(\tau - s_k)/(\tau_U - \tau_L))\label{eq:Q1}
\end{align}
where $s_k$ denotes the midpoint of $I_k$. Note that by construction the supports of $\tau \mapsto g(2G_{N_k}(\tau - s_k)/(\tau_U - \tau_L))$ do not overlap for different $k$, so for each fixed $\tau$ only one summand in the sum above does not vanish. Moreover, again by construction, the support of $\tau \mapsto g(2G_{N_k}(\tau - s_k)/(\tau_U - \tau_L))$ is contained between two knots of the form $(\tau_U - \tau_L)\nu_k/G_{N_k}, (\tau_U - \tau_L)(\nu_k+1)/G_{N_k}$ for some $\nu_k\in\N$. Hence, \eqref{eq:Q1} implies that
\begin{align}
\inf_{p \in \Ps\Ps_{r_\tau}(G_{N_k})} \sup_{\tau\in \Tc}|Q(\tau) - p(\tau)| \geq c(2G_{N_k})^{-\eta_\tau} d_{r_\tau}.\label{eq:lowbd1}
\end{align}

Suppose now that $\limsup_{N\to\infty}G_N < \infty$. Since in this case $G_N$ is a bounded sequence of natural numbers, it must take at least one value, say $N_0$, infinitely often. Let $s := \tau_L + (\tau_U - \tau_L)/(2G_{N_0})$. Define for $a,c>0$,
\begin{align}
Q(\tau) := a\tau + c (2G_{N_0})^{-\eta_\tau} g(2G_{N_0}(\tau - s)/(\tau_U - \tau_L)). \label{eq:Q2}
\end{align}
Similar to \eqref{eq:lowbd1}, we have for $N_k \equiv N_0$ 
\begin{align}
\inf_{p \in \Ps\Ps_{r_\tau}(G_{N_k})} \sup_{\tau\in \Tc}|Q(\tau) - p(\tau)| \geq c(2G_{N_k})^{-\eta_\tau} d_{r_\tau}.\label{eq:lowbd2}
\end{align}

Simple computations show that for the $Q$ defined in~\eqref{eq:Q1} or~\eqref{eq:Q2} $\sup_{\tau} |Q'(\tau)| \leq a + c, \sup_{\tau} |Q''(\tau)| \leq c, \inf_{\tau} Q'(\tau) \geq a - c$. Hence $a,c>0$ can be chosen such that \hyperref[A2]{(A2)}-\hyperref[A3]{(A3)} hold for $F = Q^{-1}$. Thus we have established~\eqref{eq:lowbdd}.  

Now let $P(X_1 = m^{1/2}\eee_j) = 1/m$ for $j=1,...,m$ and let $\ZZ(x) \equiv x$ and assume $X_1,...,X_N$ are i.i.d. Let $Y_1,...,Y_N$ be independent of $\{X_i\}_{1=1,...,N}$ and have quantile function $Q$ for which~\eqref{eq:lowbdd} holds and \hyperref[A2]{(A2)}-\hyperref[A3]{(A3)} hold for $F = Q^{-1}$. To simplify notation, assume that~\eqref{eq:lowbdd} holds with $N \equiv N_k$. Observe that by definition $\tau \mapsto \hat\beta_j(\tau) \in \Ps\Ps_{r_\tau}(G_N)$ for any $j=1,...,d$, where we recall that $\BB(\tau)$ is a vector of B-spline basis with degree $r_\tau$. Thus we have almost surely for all $N$
\begin{align}
G_{N}^{\eta_\tau} \sup_{\tau\in \Tc}|\hat\beta_j(\tau) - \beta_j(\tau)| \geq G_{N}^{\eta_\tau} \inf_{p \in \Ps\Ps_{r_\tau}(G_{N_k})}\sup_{\tau\in \Tc}| p(\tau)-Q(\tau)| \geq \eps > 0. \label{eq:lowboundbias1}
\end{align}
Note that \eqref{eq:lowboundbias1} holds regardless of the number of groups $S$. Now, since $G_N \lesssim N^{1/(2\eta_\tau)}$, we have $C := \sup_{N\geq 1} G_{N}^{\eta_\tau}/N^{1/2} < \infty$ and thus~\eqref{eq:lowboundbias1} implies that for all $N$
\begin{align}\label{eq:lowboundbias2}
N^{1/2} \sup_{\tau\in \Tc}|\hat\beta_j(\tau) - \beta_j(\tau)| \geq \eps/C > 0 \quad a.s.
\end{align}
Now assume that the weak convergence \eqref{eq:fixedpro} holds. We have by the continuous mapping theorem 
\begin{align}
\sqrt{N} \sup_{\tau\in \Tc}|\hat\beta_j(\tau) - \beta_j(\tau)| \weak \sup_{\tau\in \Tc} |\Gb(\tau)|, \label{eq:cmt}
\end{align}
where $\Gb$ is a tight, centered Gaussian process with continuous sample paths. By the arguments given on pages 60-61 of \cite{leta1991}, it follows that $\sup_{\tau\in \Tc} |\Gb(\tau)|$ has a continuous distribution with left support point zero, and thus $P(\sup_{\tau\in \Tc} |\Gb(\tau)| < \delta) > 0$ for all $\delta > 0$. This also implies that $[\eps/(2C),\infty)$ is a continuity set of the distribution of $\sup_{\tau\in \Tc} |\Gb(\tau)|$, and thus the weak convergence together with the Portmanteau Theorem implies that 
\[
\lim_{N \to \infty} P\Big(\sqrt{N} \sup_{\tau\in \Tc}|\hat\beta_j(\tau) - \beta_j(\tau)| \geq \eps/(2C) \Big) = P\Big(\sup_{\tau\in \Tc} |\Gb(\tau)| \geq \eps/(2C) \Big) < 1.
\]
This contradicts~\eqref{eq:lowboundbias2}, and thus the weak convergence \eqref{eq:fixedpro} can not hold. Arguments for the case in Corollary \ref{th:orlocpr} with increasing $m$ are similar as above, except that the rate of weak convergence in \eqref{eq:cmt} is $\sqrt{N}\|\ZZ(x_0)\|^{-1}$, and $G_N\lesssim N^{1/(2\eta_\tau)}\|\ZZ(x_0)\|^{-1/\eta_\tau}$. \hfill $\qed$

\subsubsection{A special class of data generation processes}\label{sec:specconst}

Consider a family of quantile functions $Q_{a,b}(\tau) = a \tau^2 + b\tau$, $\tau \in (0,1)$ indexed by $a\geq 0,b>0$. The corresponding distribution functions $F_{a,b}$ have support $[0,a+b]$ and take the form 
\begin{align}\label{eq:Fab}
F_{a,b}(y) = 
\begin{cases}
\frac{1}{2a}(-b + (b^2+4ay)^{1/2}),&\mbox{ }a>0,b>0;\\
\frac{y}{b},&\mbox{ }a=0,b>0.
\end{cases}
\end{align}
Note that the first two derivatives of the function $F_{a,b}$ defined in \eqref{eq:Fab} are 
\[
F_{a,b}'(y) = (b^2+4ay)^{-1/2}, \quad F_{a,b}''(y) = -2a(b^2+4ay)^{-3/2}
\]
and in particular 
\[
\inf_{y\in[0,a+b]} F_{a,b}'(y) \geq (b+2a)^{-1} \quad \sup_{y\in[0,a+b]} |F_{a,b}'(y)| \leq b^{-1},\quad  \sup_{y\in[0,a+b]} |F_{a,b}''(y)| \leq 2ab^{-3}.
\] 
Choosing $b = 1/\overline{f}$ and $a$ such that
\begin{equation}\label{eq:condamax}
0 \leq 2a < \Big(1/f_{min} - 1/\overline{f}\Big) \wedge \Big(\overline{f'} (\overline{f})^{-3}\Big)
\end{equation}
ensures that
\begin{equation}\label{eq:assa1a3}
\inf_{y\in[0,a+b]} F_{a,b}'(y) \geq f_{min}, \quad \sup_{y\in[0,a+b]} |F_{a,b}'(y)| \leq \overline{f}, \quad \sup_{y\in[0,a+b]} |F_{a,b}''(y)| \leq \overline{f'}.
\end{equation}
Let $Y_1,...,Y_n$ be i.i.d. with distribution function $F_{a,b}$ and $X_1,...,X_n$ i.i.d., independent of $Y_1,...,Y_n$, and such that $P(X_1 = \eee_j\sqrt{m}) = 1/m, j=1,...,m$ where $\eee_j \in \R^m$ denotes the $j$th unit vector. In this model, we have
\[
\sum_{i=1}^n \rho_\tau(Y_i - \zb^\top X_i) = \sum_{j=1}^m \sum_{i: X_i = \eee_j}\rho_\tau(Y_i - \sqrt{m}\beta_j).
\] 
Define $\Ac_j := \{i: X_i = \eee_j\}$, $n_j := \# \Ac_j$. Letting 
\begin{equation}\label{eq:tildebeta}
\tilde\zb(\tau) = \argmin_{\zb \in \R^m} \sum_{i=1}^n\rho_\tau(Y_i - \zb^\top X_i)
\end{equation}
	we find that $\tilde\beta_j(\tau) = \argmin_{b \in \R}\sum_{i: X_i = \eee_j}\rho_\tau(Y_i - \sqrt{m}b)$, and by p.9 in \cite{K2005} the solution of this minimization problem with smallest absolute value is $\tilde \beta_j(\tau) = m^{-1/2}Y_{(\lceil n_j\tau\rceil)}^{j}\IF\{n_j> 0\}$ where $Y_{(k)}^{j}$ is the $k$-th order statistic of the sample $\{Y_i\}_{i\in\Ac_j}$ and $\lceil u\rceil$ denotes the smallest integer greater or equal $u$. 
	
Now assume that $m,n,K,G$ are sequences of positive integers indexed by $N \to\infty$ such that $n \geq 2, m\leq n, n = N/S$ for an integer $S$.

In Section~\ref{sec:prooftilbeta} of the online Appendix we shall prove the following facts about $\tilde \beta_j(\tau)$ defined above 
\begin{enumerate}
\item For any fixed $\tau \in\Tc$, $b > 0$ and any $a_{max} > 0$ there exists $a \in [0,a_{max}]$ such that
\begin{align}
\limsup_{N\to\infty} \frac{n}{m}\Big| \E[m^{1/2}\tilde\beta_j(\tau)] - Q_{a,b}(\tau) \Big| > 0 \label{`EQ:TILBET1}
\\
\frac{n}{m}\Var(m^{1/2}\tilde\beta_j(\tau)) \lesssim 1 \label{`EQ:TILBET2}
\end{align}
\item Assume that $G \gg (N/m)^\alpha$ for some $\alpha > 0$ and $N/m \geq 1$.
Let $\check \beta_j(\tau) := \sum_{k=1}^K A_k(\tau_k) \tilde\beta_j(\tau_k)$ where 
\begin{align}
A_k(\tau) := \BB(\tau)^\top \bigg(\sum_{k=1}^K \BB(\tau_k)\BB(\tau_k)^\top \bigg)^{-1} \BB(\tau_k),\label{eq:Ak}
\end{align}
and we recall that $\BB = (B_1,...,B_q)^\top$ where $\{B_1,...,B_q\}$ B-spline basis with equidistant knots $\tau_L = t_1 < ... < t_G = \tau_U$ and degree $r_\tau\geq\eta_\tau$. For any closed interval $\Tc_0 \subset \Tc$ with non-empty interior and any $a_{max}, b>0$ there exists $a \in [0,a_{max}]$ and a sequence $\tau_N$ in $\Tc_0$ such that 
\begin{align}
\limsup_{N\to\infty} \frac{n}{m}\Big| \E[m^{1/2}\check\beta_j(\tau_N)] - Q_{a,b}(\tau_N) \Big| > 0 \label{`EQ:HATBET1}
\\
\frac{n}{m}\Var(m^{1/2}\check\beta_j(\tau_N)) \lesssim 1. \label{`EQ:HATBET2}
\end{align} 
\end{enumerate} 


\subsubsection{Necessity of conditions of oracle rules in Corollary~\ref{th:orlipo} and Corollary~\ref{th:orloc}} \label{sec:failorpo}

Fix an arbitrary $\tau \in \Tc$. Let $Y_1,...,Y_N$ be i.i.d. with distribution function $F_{a,b}$ defined in \eqref{eq:Fab} and independent of $X_1,...,X_N$ where $X_i$ i.i.d. with $P(X_1 = \eee_j\sqrt{m}) = 1/m, j=1,...,m$. Let $\ZZ(x) := x$. Define $b = 1/\overline{f}$ and pick $a$ satisfying~\eqref{eq:condamax} such that \eqref{`EQ:TILBET1}-\eqref{`EQ:TILBET2} hold. 

We begin by considering the case where $m = d$ is fixed. Denote by $P$ the measure corresponding to $X_1, Y_1$. Due to the choice of $a,b$ and the distribution of $X$ the pair $(\ZZ,P)$ is an element of $\Ps_1(d^{1/2},1,\overline f, \overline{f'},f_{min})$. Moreover, by construction, $\{\hat \beta_j^s(\tau)\}_{s=1,...,S}$ are i.i.d. and $\hat \beta_j^1(\tau)$ has the same distribution as $\tilde \beta_j(\tau)$ defined in~\eqref{eq:tildebeta}. 

Next, we shall prove that weak convergence of $\overline\zb$ in Corollary~\ref{th:orlipo} fails for $\bu = \eee_1$. Given~\eqref{`EQ:TILBET2}, a simple computation shows that 
\begin{equation}\label{eq:varbb_h1}
\E[(\sqrt{N}(\overline\beta_1(\tau) - \E[\overline\beta_1(\tau)]))^2] = \frac{N}{S} \Var(\tilde\beta_1(\tau)) = n \Var(\tilde\beta_1(\tau)) \lesssim 1
\end{equation}
for any fixed $\tau \in \Tc$. Thus the sequence $\sqrt{N}(\overline\beta_1(\tau) - \E[\overline\beta_1(\tau)])$ is uniformly integrable.

Now assume that the that weak convergence of $\overline\zb$ in Corollary~\ref{th:orlipo} holds. By~\eqref{`EQ:TILBET1} we know that for $S \gtrsim N^{1/2}$ and any fixed $\tau \in \Tc$
\begin{equation}\label{eq:biasbb_h1}
\big|\E[\overline\beta_1(\tau)] - \beta_1(\tau)\big| \gtrsim 1/n = S/N \gtrsim N^{-1/2}.
\end{equation} 

If $\big|\E[\overline\beta_1(\tau)] - \beta_1(\tau)\big| \asymp N^{-1/2}$, uniform integrability of $\sqrt{N}(\overline\beta_1(\tau) - \E[\overline\beta_1(\tau)])$ implies uniform integrability of $\sqrt{N}(\overline\beta_1(\tau) - \beta_1(\tau))$. In that case weak convergence of $\overline\zb$ implies $\E[\sqrt{N}(\overline\beta_1(\tau) - \beta_1(\tau))] \to 0$ by Theorem 25.12 on p.338 of \cite{B95}, which contradicts~\eqref{eq:biasbb_h1}.

If $\big|\E[\overline\beta_1(\tau)] - \beta_1(\tau)\big| \gg N^{-1/2}$, \eqref{eq:varbb_h1} implies that $N\E[(\overline\beta_1(\tau)-\beta_1(\tau))^2]\to\infty$, and thus $\sqrt{N}\big|\overline\beta_1(\tau) - \beta_1(\tau)\big|$ diverges to infinity in probability, which contradicts weak convergence of $\overline\zb$.

Now consider the setting of Section~\ref{sec:loc} where $m \to \infty$. With the choice of $X$ described in Section \ref{sec:specconst} we have $\zg_N(\tau) = m^{-1/2} Q_{a,b}(\tau)(1,...,1)^\top \in \R^m$, $\xi_m = m^{1/2}$, $c_m(\zg_N) = 0$. The matrix $\E[\ZZ\ZZ^\top]$ is the $m\times m$ identity matrix while $J_m(\tau), \tilde J_m(\zg_N(\tau))$ are diagonal with entries bounded away from zero and infinity. A simple computation shows that~\hyperref[L]{(L)} holds with $r=1$. Denote by $P_N$ the sequence of measures corresponding to $X_i, Y_i$. Picking $m$ such that $m^2(\log N)^6 = o(N)$ ensures that $(\ZZ,P_N)$ lies in $\Ps_L(1,\overline f,\overline{f'},f_{min},R)$ for any $R \geq 1$. Letting $x_0 = \ZZ(x_0) = \eee_1 \in \R^m$ the weak convergence in~\eqref{eq:weakloc} takes the form
\[
\sqrt{\frac{N}{m}}(m^{1/2}\overline{\beta}_1(\tau) - Q_{a,b}(\tau)) \weak \Nc(0,\sigma(\tau))
\]
for some constant $\sigma(\tau)$. Moreover,~\eqref{`EQ:TILBET2} yields
\begin{multline*}
\E[(\sqrt{N}m^{-1/2}(m^{1/2}\overline\beta_1(\tau) - m^{1/2}\E[\overline\beta_1(\tau)]))^2]
\\
 = \frac{N}{Sm} \Var(m^{1/2}\tilde\beta_1(\tau)) = \frac{n}{m} \Var(m^{1/2}\tilde\beta_1(\tau)) \lesssim 1
\end{multline*}
while~\eqref{`EQ:TILBET1} implies that for $S \gtrsim N^{1/2}/m^{1/2} = N^{1/2}\xi_m^{-1}$
\[
\big|\E[m^{1/2}\overline\beta_1(\tau)] - Q_{a,b}(\tau)\big| \gtrsim m/n = mS/N \gtrsim m^{1/2}N^{-1/2}.
\] 
Now failure of weak convergence in~\eqref{eq:weakloc} can be established for $x_0 = \ZZ(x_0) = \eee_1 \in \R^m$ by exactly the same arguments as given in the first part of the proof for the fixed $m$ case, details are omitted for the sake of brevity.   

\subsubsection{Necessity of conditions for oracle rules in Corollary~\ref{th:orlipr} and Corollary~\ref{th:orlocpr}} \label{sec:failorpr}

Fix a subset $\Tc_0 \subset \Tc$ which is closed and has non-empty interior. Let $Y_1,...,Y_N$ be i.i.d. with distribution function $F_{a,b}$ defined in \eqref{eq:Fab} and independent of $X_1,...,X_N$ where $X_i$ i.i.d. with $P(X_1 = \eee_j\sqrt{m}) = 1/m, j=1,...,m$. Let $b = 1/\overline{f}$ and pick $a$ satisfying~\eqref{eq:condamax} and a sequence $\tau_N$ in $\Tc_0$ such that \eqref{`EQ:HATBET1}-\eqref{`EQ:HATBET2} hold. Since $\tau_N$ is a sequence in a compact set, it must have a convergent subsequence $\tau_{N_k} \to \tau_0 \in \Tc_0$. From now on, we will without loss of generality assume that $N = N_k$.

Assume that~\eqref{eq:fixedpro} holds. Process convergence in~\eqref{eq:fixedpro}, along with the continuity of the sample paths of the limiting processes, implies that the sequence of random variables $\sqrt{N}(m^{1/2}\hat\beta_{1}(\tau_{N}) - Q_{a,b}(\tau_{N}))$ converges to a centered normal random variable. This can be led to a contradiction with \eqref{`EQ:HATBET1}-\eqref{`EQ:HATBET2} by similar arguments as in Section~\ref{sec:failorpo} after observing that by definition $\hat \beta_1(\tau) = S^{-1} \sum_{s=1}^S \check \beta_1(\tau)$.  

The proof that~\eqref{eq:prloc} fails can be done similarly, and details are omitted for the sake of brevity. \hfill $\Box$

\subsubsection{Proof of Theorem~\ref{th:loclob} and Theorem~\ref{th:lilob}}
\label{sec:prlob}

The proofs of both results follow by the same type of arguments, and thus we only give details for the proof of Theorem~\ref{th:lilob}. This proof will be based on the necessary conditions in Corollary~\ref{th:orlipo} and Corollary~\ref{th:orlipr}. 

First we derive~\eqref{eq:lilobpo} by contradiction. Assume that for some $\tau\in \Tc$ we have for any $(P,\ZZ)\in\Pc_1$ and $C>0$ that $P(\|\bar\zb(\tau) - \hat\zb_{or}(\tau)\| > CS/N) \to 0$. Then $\bar\zb(\tau) - \hat\zb_{or}(\tau) = o_P(S/N)$ and thus $\sqrt{N}(\bar\zb(\tau) - \zb(\tau)) \weak \Nc(0,H(\tau,\tau))$ even if $S = N^{1/2}$. However, this contradicts the necessity statement in Corollary~\ref{th:orlipo} (in the proof of Corollary~\ref{th:orlipo} we show that for any $\tau \in \Tc$ if $S/N^{1/2} \gtrsim 1$ then there exists $(P,\ZZ)\in\Pc_1$ such that this weak convergence fails). 

The claim in~\eqref{eq:lilobpr} will be derived from the 'necessary' part of Corollary~\ref{th:orlipr} by similar arguments. Assume that~\eqref{eq:lilobpr} does not hold, i.e. there exist $c,\eta>0$ such that for any $(P,\ZZ)\in\Pc_1$ and any $x_0 \in \Xc$ with $\tau\mapsto Q(x_0;\tau) \in \Lambda_c^\eta$ we have for any $C>0$
\[
P\Big(\sup_{\tau \in \Tc}\|\widehat{\zb}(\tau) - \hat\zb_{or}(\tau)\| \geq \frac{CS}{N} + CG^{-\eta}\Big) \to 0. 
\]
This implies 
\[
\sup_{\tau \in \Tc}\|\widehat{\zb}(\tau) - \hat\zb_{or}(\tau)\| = o_P\Big( \frac{S}{N} + G^{-\eta}\Big),
\] 
and the process convergence $\sqrt{N}(\hat\zb_{or} - \zb) \weak \Gb$ implies $\sqrt{N}(\hat\zb - \zb) \weak \Gb$. This contradicts the necessity of $S = o(N^{1/2})$ \textit{and} $G \gg N^{1/(2\eta)}$ in Corollary~\ref{th:orlipr}. Hence~\eqref{eq:lilobpr} follows.  \hfill $\Box$


\subsection{Proofs for Section \ref{SEC:GEN}}

Recall the definition of $g_N(\bb_N)$ in \eqref{eq:defgn} and $c_m(\bb_N)$ in \eqref{def:cn}. As mentioned in Remark~\ref{rem:center}, the results in Section~\ref{SEC:GEN} hold for general 'centerings' $\zb_N$ (which includes $\zg_N$ as a special case) provided that certain technical conditions are satisfied. The precise form of those conditions is discussed in section~\ref{sec:zbN}. There, we show that $\zg_N$ satisfies those general conditions. In sections~\ref{sec:proof_dcolin} and~\ref{sec:proof_orgenpr}, all theoretical results will be proved for $\zb_N$ instead of $\zg_N$ under the general conditions described in section~\ref{sec:zbN}. The results with $\zg_N$ in Section~\ref{SEC:GEN} follow as a special case.   



\bigskip

\subsubsection{Details of Remark~\ref{rem:center}}\label{sec:zbN}
Assume that \hyperref[A1]{(A1)}-\hyperref[A3]{(A3)} hold. Corollary~\ref{th:orgenpo} holds for any sequence of centering coefficients $\zb_N$ under the following conditions: $S = o(N^{1/2}/(m\xi_m^2(\log N)^2))$, $S c_m^2(\zb_N) m \log N = o(1)$, and
\begin{enumerate}
\item[(B1)]\label{B1} $g_N(\zb_N) = o(N^{-1/2})$ and $m^4(\log N)^{10} + m^2\xi_m^4 = o(N)$. 
\item[(B2)]\label{B2} $m c_m(\zb_N) \log N = o(1)$ and the sequence $\bu_N$ (in Corollary~\ref{th:orgenpo}) satisfies
\end{enumerate}
\begin{align*}
\chi_N(\bu_N ,\zb_N)&:= \sup_{\tau \in \Tc} \Big| \bu_N^\top J_m(\tau)^{-1} \E\big[\ZZ_i\big(\IF\{Y_i \leq Q(X_i;\tau)\} - \IF\{Y_i \leq \ZZ^\top_i \zb_N(\tau) \}\big)\big] \Big|
\\
& = o(\|\bu_N\|N^{-1/2}).
\end{align*}

Furthermore, Corollary~\ref{th:orgenpr} holds under the additional restriction $c_m(\zb_N) = o(\|\ZZ(x_0)\|N^{-1/2})$ and the other conditions on $Q, G, K$ stated in Corollary~\ref{th:orgenpr}.

In the sequel, we verify that, assuming the conditions stated in Corollary~\ref{th:orgenpo} and Corollary~\ref{th:orgenpr} respectively, $\zg_N$ satisfies the conditions given above. Note that some conditions in these theorems overlap with the conditions stated above, so we only need to verify that the conditions $c_m^2(\zg_N)\xi_m = o(N^{-1/2})$ and $S = o(N^{1/2}/(m\xi_m^2(\log N)^2))$ in both theorems imply $g_N(\zg_N) = o(N^{-1/2})$, $\chi_N(\bu_N ,\zg_N)=o(\|\bu_N\|N^{-1/2})$, $S c_m^2(\zg_N) m \log N = o(1)$ and moreover $m c_m(\zg_N) \log N = o(1)$.

Under the definition of $\zg_N$ in \eqref{eq:gamman} and \hyperref[A1]{(A1)}-\hyperref[A3]{(A3)}, we have $g_N(\zg_N) = O(\xi_mc_m^2(\zg_N))$. To see this, note that under \hyperref[A2]{(A2)},
\begin{align}
\ZZ(F_{Y|X}(\ZZ^\top \zg_N(\tau)|X) - \tau) &= \ZZ f_{Y|X}(Q(X;\tau)|X)(\ZZ^\top \zg_N(\tau) - Q(X;\tau)) \notag  
\\
& + \frac{1}{2}\ZZ f'_{Y|X}(\zeta_n(X,\tau)|X)(\ZZ^\top \zg_N(\tau) - Q(X;\tau))^2,\label{eq:compsplprot}
\end{align}
where $\zeta_n(X,\tau)$ is a value between $\ZZ(X)^\top \zg_N(\tau)$ and $Q(X;\tau)$. Thus
\begin{multline}\label{eq:compsplpro2}
g_N(\zg_N)  = \sup_{\tau \in \Tc}\Big\| \E[\ZZ(F_{Y|X}(\ZZ^\top \zg_N(\tau)|X) - \tau)] \Big\|
\\
\leq \frac{\overline{f'}c_m(\zg_N)^2}{2} \|\ZZ\| = O(\xi_m c_m^2(\zg_N)), 
\end{multline}
where the second inequality follows from the first order condition 
\[
\E[\ZZ f_{Y|X}(Q(X;\tau)|X)(\ZZ^\top \zg_N(\tau) - Q(X;\tau))] = 0
\]
by the definition of $\zg_N$ in \eqref{eq:gamman} and \eqref{eq:compsplprot}. Moreover, 
\[
\chi_N(\bu_N,\zg_N) = \sup_{\tau\in \Tc} \big| \bu_N^\top J_m(\tau)^{-1}\mu(\zg_N(\tau),\tau) \big| \lesssim  \|\bu_N\|_2 g_N(\zg_N).
\]
Hence, $c_m^2(\zg_N)\xi_m = o(N^{-1/2})$ implies $g_N(\zg_N) = o(N^{-1/2})$ and $\chi_N(\bu_N ,\zg_N)= o(\|\bu_N\|N^{-1/2})$. In addition, 
\begin{align*}
	S c_m^2(\zg_N) m \log N = o(S N^{-1/2} \xi_m^{-1} m \log N) = o(1),
\end{align*}
where the first equality is from $c_m^2(\zg_N)\xi_m = o(N^{-1/2})$ and the second from $S = o(N^{1/2}/(m\xi_m^2(\log N)^2))$. Finally 
\[
m c_m(\zg_N) \log N = o(m N^{-1/4} \log N) = o((m^4 N^{-1} (\log N)^4)^{1/4}) = o(1)
\]
where the first bound follows from $c_m^2(\zg_N)\xi_m = o(N^{-1/2}), \xi_m \gtrsim 1$ and the third bound from $m^4 (\log N)^{10}=o(N)$.

\hfill $\qed$

\subsubsection{Proof of Theorem~\ref{th:gendiffup}}\label{sec:proof_dcolin}


We begin by proving of~\eqref{eq:genubpo}. From Theorem~\ref{th:bahsimple} in the supplementary material, we obtain the following representation for the oracle estimator
\begin{align}\label{eq:bahgen1}
\hat\zb_{or}(\tau) - \zb_N(\tau) =& - J_m(\tau)^{-1} \frac{1}{N}\sum_{i=1}^N \ZZ_i(\IF\{Y_i \leq \ZZ_i^\top \zb_N(\tau)\}- \tau)
 + r_n^{or}(\tau) 
\end{align} 
where $\sup_{\tau \in \Tc} \| r_n^{or}(\tau)\| = o_P(R_N^{or})$ and
\begin{multline*}
R_N^{or} : = c_m\Big(\frac{m\log N}{N} \Big)^{1/2} + c_m^2 \xi_m  + \Big(\frac{m\xi_m^2 \log N}{N} + c_m^4\xi_m^4\Big)
\\
 + \frac{1}{N^{1/2}}\Big( \Big(m\xi_m^2 c_m^2 (\log N)^3\Big)^{1/2}+\Big(\frac{m^3\xi_m^2(\log N)^7}{N}\Big)^{1/4}\Big).
\end{multline*}
Next we will prove that under the assumptions of Theorem~\ref{th:gendiffup} we also have
\begin{align}\label{eq:bahgen1_new}
\sup_{\tau \in \Tc_K}\Big\|\overline\zb(\tau) - \zb_N(\tau) + J_m(\tau)^{-1} \frac{1}{N}\sum_{i=1}^N \ZZ_i(\IF\{Y_i \leq \ZZ_i^\top \zb_N(\tau)\}- \tau)\Big\| = O_P(R_N)
\end{align}
where 
\begin{align}
R_N := ~&  c_m\Big(\frac{m\log N}{N} \Big)^{1/2} + c_m^2 \xi_m  + \Big(\frac{Sm\xi_m^2 \log N}{N} + c_m^4\xi_m^4\Big)\Big(1 + \frac{\log N}{S^{1/2}} \Big) \nonumber
\\
& + \frac{1}{N^{1/2}}\Big( \Big(m\xi_m^2 c_m^2 (\log N)^3\Big)^{1/2}+\Big(\frac{Sm^3\xi_m^2(\log N)^7}{N}\Big)^{1/4}\Big). \nonumber
\end{align}
To this end, apply Theorem~\ref{th:aggbah}. A simple computation shows that the quantities $R_{1,n},R_{2,n}$ defined in \eqref{eq:th1RN1} and \eqref{eq:th1RN2} of Theorem~\ref{th:aggbah} satisfy 
\[
R_{1,n}(A\log n) + R_{2,n}(A,A\log n) \leq  C R_N 
\]  
for any fixed $A>0$. Thus we obtain from \eqref{eq:th1rn1} and \eqref{eq:th1rn2} that there exists $N_0$ such that for all $N \geq N_0$ and any $\tau \in \Tc$
\begin{align}
P\big(\big\|r_N^{(1)}(\tau) + r_N^{(2)}(\tau)\big\| > \tilde CR_N \big) \leq 2n^{-A}(3S+1), 
\end{align}
where
\[
\overline\zb(\tau) - \zb_N(\tau) + J_m(\tau)^{-1} \frac{1}{N}\sum_{i=1}^N \ZZ_i(\IF\{Y_i \leq \ZZ_i^\top \zb_N(\tau)\}- \tau) = r_N^{(1)}(\tau) + r_N^{(2)}(\tau).
\]
Choosing $A$ large enough ensures that
\begin{align}
\sup_{\tau \in \Tc_K} \big\|r_N^{(1)}(\tau) + r_N^{(2)}(\tau)\big\| = o_P(R_N), \label{eq:rmd_sqrtn}
\end{align}
and we obtain~\eqref{eq:bahgen1_new}.

\medskip

Next we prove~\eqref{eq:genubpr}. Recall the definition of the projection operator $\Pi_K$ right after~\eqref{eq:hatbeta}.

\bigskip

\noindent\textbf{Step 1: $\|\Pi_K\|_\infty := \sup_{\|f\|_\infty = 1} \|\Pi_K f\|_\infty<C$ for a constant $C>0$ independent of $K$.}\\ 
We will apply Theorem A.1 of \cite{huang2003}, and for that we need to validate conditions A.1-A.3 of \cite{huang2003}. Note that \emph{none} of these conditions is associated with the functions being projected on the spline space -- all conditions are associated with the design and the selection of basis functions.

We first show that Condition A.1 of \cite{huang2003} holds in our context. By the definition of grids $t_1,...,t_G$ and $\tau_1,...,\tau_K$, each interval $[t_k,t_{k+1}]$ contains $c_{k,K}$ spaced grid points of the form $\tau_U + l(\tau_L-\tau_U)/K$. Hence, applying Lemma~\ref{lem:poly}, we obtain for $j=0,...,G-1$ and any $f \in \Theta_G$
\begin{align*}
 K^{-1} \sum_{k:\tau_k\in(t_j,t_{j+1}]} f^2(\tau_k) \leq \Big(1+\frac{C_{r_\tau}}{c_{j,K}-1}\Big) \int_{t_j}^{t_{j+1}} f^2(x)dx,
\end{align*}
and
\begin{align*}
 K^{-1} \sum_{k:\tau_k\in(t_j,t_{j+1}]} f^2(\tau_k) \geq \Big(1-\frac{C_{r_\tau}}{c_{j,K}-1}\Big) \int_{t_j}^{t_{j+1}} f^2(x)dx,
\end{align*}
for a constant $C_{r_\tau}$ that depends only on $r_\tau$. Since $\min_{k=1,...,K}c_{j,K} \to \infty$ it follows that for sufficiently large $K$ Condition A.1 of \cite{huang2003} holds with $\gamma_1= 1/2$ and $\gamma_2=2$.

Conditions A.2 (i) and (ii) of \cite{huang2003} easily follow from the fact that $\Tc$ is bounded. Finally, the definition of the knots ensures that A.2 (iii) in \cite{huang2003} holds. Condition A.3 holds since B-spline basis is used in the projection procedure, cf. the discussion following Condition A.3 on page 1630 of \cite{huang2003}. 

Thus conditions A.1-A.3 in \cite{huang2003} are satisfied, and Theorem A.1 of \cite{huang2003} shows the existence of a constant $C$ such that $\|\Pi_K f\|_\infty \leq C \|f\|_\infty$ for all functions $f\in\ell^\infty(\Tc)$. This completes the proof of step 1.\\

From the definition of $\hat\zb$ and linearity of $\Pi_K$ we obtain 
\[
\ZZ(x_0)^\top\hat \zb(\tau) = (\Pi_K \ZZ(x_0)^\top \overline{\zb}(\cdot))(\tau)
\]
and thus
\begin{align*}
\ZZ(x_0)^\top (\hat \zb(\tau) - \hat \zb_{or}(\tau)) =~& (\Pi_K \{\ZZ(x_0)^\top [\overline{\zb}(\cdot)-\hat \zb_{or}(\cdot)]\})(\tau)
\\
& + (\Pi_K \{\ZZ(x_0)^\top \hat\zb_{or}(\cdot)\})(\tau) - \ZZ(x_0)^\top \hat \zb_{or}(\tau).
\end{align*}
Observe that
\[
\sup_{\tau\in\Tc} \Big|(\Pi_K \{\ZZ(x_0)^\top [\overline{\zb}(\cdot)-\hat \zb_{or}(\cdot)]\})(\tau)\Big| \lesssim \sup_{\tau\in \Tc_K} \Big| \ZZ(x_0)^\top [\overline{\zb}(\tau)-\hat \zb_{or}(\tau)] \Big|.
\]
Hence it remains to prove that
\begin{multline*}
\sup_{\tau\in \Tc} \Big|(\Pi_K \ZZ(x_0)^\top \hat\zb_{or}(\cdot))(\tau) - \ZZ(x_0)^\top \hat \zb_{or}(\tau)\Big|
\\
 \leq \sup_{\tau\in \Tc} \Big|(\Pi_K Q(x_0;\cdot))(\tau) - Q(x_0;\tau)\Big| + o_P(\|\ZZ(x_0)\|N^{-1/2}).
\end{multline*}

From Theorem 2.1 of \cite{ChaVolChe2016} we obtain
\beq \label{eq:barh}
\ZZ(x_0)^\top\hat\zb_{or}(\tau) = Q(x_0;\tau) + \ZZ(x_0)^\top \bU_N(\tau) + r_{N}(\tau)
\eeq
where 
\begin{align}
\bU_N(\tau) := -N^{-1} J_m(\tau)^{-1} \sum_{i=1}^N \ZZ_i (\IF\{Y_i \leq Q(X_i;\tau)\}-\tau) \label{eq:bU}
\end{align}
and $\sup_{\tau \in \Tc} |r_{N}(\tau)| \leq \sup_{\tau \in \Tc} |\ZZ(x_0)^\top\zg_N(\tau) - Q(x_0;\tau)| + o_P(\|\ZZ(x_0)\|N^{-1/2})$. This implies
\begin{multline*}
\sup_{\tau \in \Tc} \Big|(\Pi_K \ZZ(x_0)^\top \hat\zb_{or}(\cdot))(\tau) - \ZZ(x_0)^\top \hat \zb_{or}(\tau)\Big| \leq o_P\Big(\frac{\|\ZZ(x_0)\|}{N^{1/2}}\Big)
\\ 
+ \sup_{\tau \in \Tc} \Big|[\Pi_K Q(x_0;\cdot)](\tau) - Q(x_0;\tau)\Big| + \sup_{\tau \in \Tc} |\ZZ(x_0)^\top\zg_N(\tau) - Q(x_0;\tau)|
\\
+ \sup_{\tau \in \Tc} \Big|[\Pi_K \ZZ(x_0)^\top \bU_N](\tau) - \ZZ(x_0)^\top \bU_N(\tau)\Big|. 
\end{multline*}
and it remains to bound the last term. From Lemma~5.1 of \cite{huang2003} we obtain that 
\begin{align}
\sup_{\tau \in \Tc}\big|(\Pi_K \ZZ(x_0)^\top \bU_N)(\tau)-\ZZ(x_0)^\top \bU_N(\tau)|
\leq C \inf_{\tilde f \in \Theta_G} \sup_{\tau \in \Tc}|\ZZ(x_0)^\top \bU_N(\tau)-\tilde f(\tau)|, \label{huanglemma5.1}
\end{align}
where $C$ is an absolute constant and $\Theta_G$ is the space of splines of degree $r_\tau$ with knots $t_1,...,t_G$ on $\Tc=[\tau_L,\tau_U]\subset(0,1)$. The right hand side of \eqref{huanglemma5.1} can be related to the \emph{modulus of continuity}. Indeed, using Theorem~6.27 of \cite{schumaker:81} we have for an absolute constant $C>0$,
\begin{multline}
\inf_{\tilde f \in \Theta_G} \sup_{\tau \in \Tc}\big|\ZZ(x_0)^\top \bU_N(\tau)-\tilde f(\tau)\big|
\\
 \leq C \sup_{\tau,\tau' \in \Tc, |\tau - \tau'| \leq \delta}|\ZZ(x_0)^\top \bU_N(\tau)-\ZZ(x_0)^\top \bU_N(\tau')|=o_P\Big(\frac{\|\ZZ(x_0)\|}{N^{1/2}}\Big), \label{eq:interbaha1}
\end{multline}
where $\delta = \max_{0 \leq j \leq K-1}(t_{j+1}-t_j) = o(1)$, and the last equality follows from Lemma~\ref{lem:equicont}. Thus the proof of~\eqref{eq:genubpr} is complete. \hfill $\Box$

\subsubsection{Proof of Corollary~\ref{th:orgenpo}}

From Theorem 2.1 of \cite{ChaVolChe2016} we obtain the representation
\begin{multline*}
\ZZ(x_0)^\top(\hat\zb_{or}(\tau) - \zg_N(\tau))
\\
 = - \ZZ(x_0)^\top J_m(\tau)^{-1} \frac{1}{N}\sum_{i=1}^N \ZZ_i(\IF\{Y_i \leq Q(X_i;\tau)\}- \tau)
 + o_P\Big(\frac{\|\ZZ(x_0)\|}{N^{1/2}}\Big). 
\end{multline*}
Moreover, by straightforward calculations bounding the terms in~\eqref{eq:genubpo} we obtain the representation $\ZZ(x_0)^\top(\hat\zb_{or}(\tau)-\overline\zb(\tau)) = o_P(\|\ZZ(x_0)\|N^{-1/2})$. Hence to prove the weak convergence result for both $\hat\zb_{or}$ and $\overline{\zb}$, it suffices to prove that 
\[
\frac{1}{\sqrt{N}} \sum_{i=1}^N \frac{\ZZ(x_0)^\top J_m(\tau)^{-1}\ZZ_i(\IF\{Y_i \leq Q(X_i;\tau)\} -\tau)}{(\ZZ(x_0)^\top J_m(\tau)^{-1}\E[\ZZ\ZZ^\top]J_m(\tau)^{-1}\ZZ(x_0))^{1/2}} \weak \Nc(0,\tau(1-\tau)).
\] 
This follows by an application of the Lindeberg CLT. The verification of the Lindeberg condition here is a simple modification from finite dimensional joint convergence in the proof of (2.4) on page 3292 in \cite{ChaVolChe2016} to pointwise convergence. \hfill $\qed$

\subsubsection{Details for Remark~\ref{rem:compspl}} \label{sec:compsplpro} 

In the proof below, we will use the notations in Example \ref{ex:spline}: polynomial spline $\tilde B_j$, $j$'th knot $t_j$ and degree $q$. 

We first verify \hyperref[A1]{(A1)}. The bounded spectrum $\E[\ZZ\ZZ^\top]$ follows by the argument in Example \ref{ex:spline}. We note that every univariate polynomial spline function satisfies $|\tilde B_j| \leq C$ almost surely for some constant $C>0$ for all $j \leq m$. Moreover, the support of $\tilde B_j$ and $\tilde B_k$ is disjoint unless $|j-k|\leq q+1$. Hence, we have almost surely
\begin{multline}
\|\ZZ_i\| = m^{1/2} \bigg(\sum_{j=1}^{k-q-1} \tilde B_j(X_i)^2\bigg)^{1/2}\\
 \lesssim m^{1/2} \bigg(\sum_{j=1}^{k-q-1} \IF\{t_j \leq X_i \leq t_{j+q+1}\}\bigg)^{1/2} \lesssim m^{1/2}. \label{eq:compsplpro1}
\end{multline}
On the other hand, since each $\tilde B_j(X_i)$ is nonzero almost surely, $\|\ZZ_i\|\gtrsim m^{1/2}$ and $\xi_m \asymp m^{1/2}$.

We now verify the sufficient conditions of Corollary~\ref{th:orloc} and Corollary~\ref{th:orgenpo} under the conditions made in this remark. Condition \hyperref[L]{(L)} holds given the discussion in Example~\ref{ex:spline}. \hyperref[L1]{(L1)} and the conditions $\xi_m^4(\log N)^6 = o(N), c_m^2(\zg_N) = o(N^{-1/2})$ hold under the assumptions of this remark. On the other hand, the sufficient conditions of Corollary~\ref{th:orgenpo} can be easily verified by the assumptions of this remark since $\xi_m \asymp m^{1/2}$. 

\subsubsection{Proof of Corollary~\ref{th:orgenpr}}\label{sec:proof_orgenpr}

Similar to Section \ref{sec:proof_dcolin}, we will prove the weak convergence of $\hat\zb(\tau)$ in models with centering coefficient $\zb_N(\tau)$ from a class specified in Section \ref{sec:zbN}.

The process convergence of $\frac{\sqrt{N}}{\|\ZZ(x_0)\|}\big(\ZZ(x_0)^\top \hat\zb_{or}(\cdot) - Q(x_0;\cdot)\big)$ is exactly Theorem 2.1 of \cite{ChaVolChe2016}. A straightforward computation shows that under the assumptions of Corollary~\ref{th:orgenpr} the results in~\eqref{eq:genubpo} and~\eqref{eq:genubpr} yield
\[
\sup_{\tau\in \Tc} \frac{\sqrt{N}}{\|\ZZ(x_0)\|}\Big|\ZZ(x_0)^\top \hat\zb_{or}(\tau) - \ZZ(x_0)^\top \hat\zb(\tau)\Big| = o_P(1).
\] 
This yields the claim in the first part. The proof of the second part follows by similar arguments as the proof of Corollary~4.1 in \cite{ChaVolChe2016}, and is therefore omitted for brevity. \hfill $\qed$


\section{Proofs for Section~\ref{SEC:INF}}

\subsection{Proof of Theorem~\ref{th:simpleinf}} \label{sec:prsimple}

Since the proofs under the assumptions of Corollary~\ref{th:orlipr}, Corollary~\ref{th:orlocpr} and Corollary~\ref{th:orgenpr} are similar, we will only give a proof under the conditions of Corollary~\ref{th:orgenpr}. Define 
\[
W_N := \sqrt{n}(\ZZ(x_0)^\top \bar\zb(\tau) - Q(x_0;\tau)), \quad \hat \sigma^2 := n \ZZ(x_0)^\top \hat\Sigma^D \ZZ(x_0) 
\]
and
\[
p_{N} := P\Big( S^{1/2}|W_N| \leq \hat\sigma t_{S-1,1-\alpha/2}\Big). 
\]

Throughout this proof, we will write $S_N$ instead of $S$ in order to emphasize that $S$ depends on $N$. First, observe that along any sub-sequence $N_k$ such that $S_{N_k}$ only takes the constant value $S_0>2$ we have
\[
\frac{\sqrt{n}}{\sigma_n}(\ZZ(x_0)^\top \hat\zb^1(\tau) - Q(x_0;\tau),...,\ZZ(x_0)^\top \hat\zb^{S_0}(\tau)- Q(x_0;\tau)) \weak \Nc(0, I_{S_0})
\]
where $I_{S_0}$ denotes the $S_0\times S_0$ identity matrix and
\[
\sigma_n^2 = \tau(1-\tau) \frac{\ZZ(x_0)^\top J_m(\tau)^{-1}\E[\ZZ\ZZ^\top] J_m(\tau)^{-1}\ZZ(x_0)}{\|\ZZ(x_0)\|^2}.
\] 
Observe that $\ZZ(x_0)^\top \hat\Sigma^D \ZZ(x_0)$ is the sample variance of $\{\ZZ(x_0)^\top \hat\zb^s\}_{s=1,...,S}$. As a direct consequence of the continuous mapping theorem, we obtain $p_{N_k} \to (1-\alpha)$. By the definition of convergence of a sequence of real numbers to a limit, this also implies that for any fixed $S_0 > 2$
\[
\lim_{N_0 \to \infty} \sup_{N \geq N_0: S_N = S_0} |p_N - (1-\alpha)| = 0,
\]  
which in turn yields that for any fixed $\bar S$
\[
\lim_{N_0 \to \infty} \sup_{N \geq N_0: 2 < S_N \leq \bar S} |p_N - (1-\alpha)| = 0.
\]
Next observe that for any fixed $\eps>0$
\begin{align*}
\Big| p_N - (1-\alpha) \Big|
\leq &~ 
P\Big(\Big|\frac{\hat\sigma t_{S-1,1-\alpha/2}}{\sigma_N \Phi^{-1}(1-\alpha/2)} - 1\Big| \geq \eps \Big)
\\
&~+ P\Big( \Big|S^{1/2}\sigma_N^{-1}|W_N| - \Phi^{-1}(1-\alpha/2) \Big| \leq \eps \Phi^{-1}(1-\alpha/2) \Big)
\\
&~+\Big| P\Big(  S^{1/2}\sigma_N^{-1}|W_N| \leq  \Phi^{-1}(1-\alpha/2)\Big) - (1-\alpha) \Big|.
\end{align*}
Under the assumptions of Corollary~\ref{th:orgenpr}, we have the decomposition
\[
\ZZ(x_0)^\top\hat \zb^s(\tau) = Q(x_0;\tau) + L_n^s(\tau) + r_n^s(\tau)
\]
where $\sup_{s=1,...,S} |r_n^s(\tau)| = o_P(\|\ZZ(x_0)\|n^{-1/2})$ and 
\[
L_n^s(\tau) := \frac{1}{n}\sum_{i=1}^n \ZZ(x_0)^\top J_m(\tau)^{-1}\ZZ(X_{is})(\IF\{Y_{is} \leq Q(X_{is};\tau)\} - \tau).
\]
This implies
\begin{align*}
\hat \sigma^2 &= \frac{n}{S_N-1}\sum_{s=1}^{S_N} \Big(\ZZ(x_0)^\top \hat\zb^s(\tau) - \ZZ(x_0)^\top \bar\zb(\tau)\Big)^2
\\
& = \frac{n}{S_N-1}\sum_{s=1}^{S_N} \Big(L_n^s(\tau) - S_N^{-1}\sum_{s=1}^S L_n^s(\tau) \Big)^2 + o_P(1)
\\
& = \frac{1}{S_N-1}\sum_{s=1}^{S_N} (\sqrt{n}L_n^s(\tau))^2 - \frac{1}{S_N(S_N-1)}\Big(\sum_{s=1}^{S_N} \sqrt{n}L_n^s(\tau) \Big)^2 + o_P(1)
\\
& =: \hat \sigma_{L,N}^2 + o_P(\|\ZZ(x_0)\|^2)
\end{align*} 
where the $o_P(1)$ term is for $N\to \infty$ for arbitrary sequences $S_N$. Noting that $\E[(\sqrt{n}L_n^s(\tau))^4/\sigma_N^4]$ is bounded uniformly in $n,s$ and that $\E[L_n^s(\tau)] = 0$, $\E[(\sqrt{n}L_n^s(\tau))^2] = \sigma_N^2$, an application of the Markov inequality shows that there exists a constant $C$ such that for any fixed $\eps>0$ and any $S_N >2$
\[
P\Big(\Big|\frac{\hat \sigma_{L,N}^2}{\sigma_N^2 } - 1\Big| \geq \eps \Big) \leq C S_N^{-1} \eps^{-2}. 
\] 
Next, observe that $|\hat \sigma_{L,N}^2/\sigma_N^2 - 1| \geq |\hat \sigma_{L,N}/\sigma_N - 1|$ and thus
\begin{multline*}
\Big|\frac{\hat\sigma t_{S-1,1-\alpha/2}}{\sigma_N \Phi^{-1}(1-\alpha/2)} - 1\Big|
\leq
\Big|\frac{\hat\sigma}{\sigma_N} - 1\Big|\frac{t_{1,1-\alpha/2}}{\Phi^{-1}(1-\alpha/2)} + \Big|\frac{t_{S-1,1-\alpha/2}}{\Phi^{-1}(1-\alpha/2)} - 1\Big|
\\
\leq
\Big|\frac{\hat\sigma^2}{\sigma_N^2} - 1\Big|\frac{t_{1,1-\alpha/2}}{\Phi^{-1}(1-\alpha/2)} + \Big|\frac{t_{S-1,1-\alpha/2}}{\Phi^{-1}(1-\alpha/2)} - 1\Big|.
\end{multline*}
It follows that for any fixed $\eps>0$ and $\bar S$ with $|t_{\bar S - 1,1-\alpha/2}/\Phi^{-1}(1-\alpha/2) - 1| <\eps/2$ we have for some constant $\tilde C$
\begin{multline*}
\limsup_{N\to\infty: S_N \geq \bar S} P\Big(\Big|\frac{\hat\sigma t_{S-1,1-\alpha/2}}{\sigma_N \Phi^{-1}(1-\alpha/2)} - 1\Big| \geq \eps\Big)
\leq \tilde C \bar S^{-1} \Big(\frac{t_{1,1-\alpha/2}}{\Phi^{-1}(1-\alpha/2)}\Big)^2 \eps^{-2} 
\end{multline*}
Finally, noting that $S^{1/2}\sigma_N^{-1}|W_N| \weak |\Nc(0,1)|$, we have
\begin{multline*}
\limsup_{N\to\infty} P\Big( \Big|S^{1/2}\sigma_N^{-1}|W_N| - \Phi^{-1}(1-\alpha/2) \Big| \leq \eps \Phi^{-1}(1-\alpha/2) \Big)
\\
= P\Big( \Big||\Nc(0,1)| - \Phi^{-1}(1-\alpha/2) \Big| \leq \eps \Phi^{-1}(1-\alpha/2) \Big)
\end{multline*}
and
\[
\Big| P\Big(  S^{1/2}\sigma_N^{-1}|W_N| \leq  \Phi^{-1}(1-\alpha/2)\Big) - (1-\alpha) \Big| \to 0.
\]
Combining all results so far, we obtain for any fixed $\eps>0$ and $\bar S$ with $|t_{\bar S - 1,1-\alpha/2}/\Phi^{-1}(1-\alpha/2) - 1| <\eps/2$
\begin{align*}
\limsup_{N\to\infty} \Big| p_N - (1-\alpha) \Big| \leq & \tilde C \bar S^{-1} \Big(\frac{t_{1,1-\alpha/2}}{\Phi^{-1}(1-\alpha/2)}\Big)^2 \eps^{-2}
\\
& + P\Big( \Big||\Nc(0,1)| - \Phi^{-1}(1-\alpha/2) \Big| \leq \eps \Phi^{-1}(1-\alpha/2) \Big).
\end{align*}
The right-hand side can be made arbitrarily small by the choice of fixed values for $\eps, \bar S$, and thus $p_N \to (1-\alpha)$. This completes the proof of the first part of Theorem~\ref{th:simpleinf}. The second part follows easily since $t_{S_N - 1,1-\alpha/2} \to \Phi^{-1}(1-\alpha/2) $ if $S \to \infty$. \hfill $\Box$


\subsection{Proof of Theorem~\ref{th:boot}} \label{sec:prboot}

Since the arguments under the assumptions of Corollary~\ref{th:orlocpr} are very similar to the arguments under the assumptions of Corollary~\ref{th:orgenpr}, we will only give a proof in the latter case. Note also that Corollary~\ref{th:orlipr} follows from Corollary~\ref{th:orgenpr}. Define
\[
\Gb_N^{(1)}(\tau) := - \frac{1}{N} \sum_{s=1}^S (\omega_{s,1}-1) \ZZ(x_0)^\top J_m^{-1}(\tau)\sum_{i=1}^n \ZZ(X_{is})(\IF\{Y_{is} \leq Q(X_{is};\tau)\} - \tau)
\]
and
\[
\Gb_N(\tau) := - \frac{1}{N} \sum_{s=1}^S \ZZ(x_0)^\top J_m^{-1}(\tau)\sum_{i=1}^n \ZZ(X_{is})(\IF\{Y_{is} \leq Q(X_{is};\tau)\} - \tau).
\]
Throughout the proof, let $C_\omega := 1 + \sqrt{2}$ and note that $\omega_{s,b} \leq C_\omega$ almost surely. The proof will rely on the following two key steps:

\textbf{Step 1}: prove that for any fixed $B \in \N$
\begin{equation} \label{eq:wjboot}
\frac{\sqrt{N}}{\|\ZZ(x_0)\|}(\Gb_N(\cdot),\Gb_N^{(1)}(\cdot),\dots,\Gb_N^{(B)}(\cdot)) \weak (\Gb_{x_0}(\cdot),\Gb_{x_0}^{(1)}(\cdot),...,\Gb_{x_0}^{(B)}(\cdot))
\end{equation}
as elements in $(\ell^\infty(\Tc))^{B+1}$ where the processes $\Gb_{x_0}^{(1)},...,\Gb_{x_0}^{(B)}$ iid copies of the process $\Gb_{x_0}$ defined in~\eqref{eq:weakorgen} and that $\Gb_N^{(1)}(\cdot)$ is asymptotically uniformly equicontinuous in probability.

\textbf{Step 2}: prove that uniformly in $\tau \in \Tc_k$ we have for $j=1,...,B$  
\begin{equation} \label{eq:bahgen2_boot}
\ZZ(x_0)^\top ( \overline \zb^{(j)}(\tau) - \overline \zb(\tau) ) 
= \Gb_N^{(j)}(\tau) + o_P\Big(\frac{\|\ZZ(x_0)\|}{N^{1/2}}\Big).
\end{equation}

Given the results in step 1 and step 2, exactly the same arguments as those in the proof of Corollary~\ref{th:orgenpr} (replacing~\eqref{eq:barh} by~\eqref{eq:bahgen2_boot} and~\eqref{eq:interbaha1} by uniform asymptotic equicontinuity of $\Gb_{x_0}^{(1)}$; note that $\Gb_{x_0}^{(1)}$ in the present setting corresponds to $\ZZ(x_0)^\top \bU_N(\tau)$ in the proof of Corollary~\ref{th:orgenpr}) show that for $j=1,...,B$ 
\[
\sup_{\tau \in \Tc} \Big|\ZZ(x_0)^\top ( \widehat \zb^{(j)}(\tau) - \hat \zb(\tau) ) - \Gb_N^{(j)}(\tau)\Big| = o_P\Big(\frac{\|\ZZ(x_0)\|}{N^{1/2}}\Big),
\]
and together with Step 2,~\eqref{eq:wjboot} and Lemma 3.1 from \cite{buko2017} the assertion follows.

It remains to establish the claims in step 1 and step 2. 

\textbf{Proof of step 1} Finite-dimensional convergence follows by a simple application of the triangular array Lyapunov CLT with $\delta = 2$ after observing that $\frac{\sqrt{N}}{\|\ZZ(x_0)\|}\Gb_N^{(1)}(\tau) = \sum_{s=1}^S V_{N,s}^{(1)}(\tau)$
where
\[
V_{N,s}^{(1)}(\tau) := \frac{(\omega_{s,1}-1)}{\sqrt{N}\|\ZZ(x_0)\|}  \ZZ(x_0)^\top J_m^{-1}(\tau)\sum_{i=1}^n \ZZ(X_{is})(\IF\{Y_{is} \leq Q(X_{is};\tau)\} - \tau)
\] 
are independent summands. Hence it suffices to establish asymptotic uniform equicontinuity, which entails tightness. The proof is based on similar ideas as the proof of Lemma A.3 in \cite{ChaVolChe2016}, and we will only stress the main differences. Following the latter paper, define
\begin{align*}
U_N(\tau) &:= \frac{1}{S} \sum_{s=1}^S (\omega_{s,1} - 1) J_m(\tau)^{-1} \frac{1}{n}\sum_{s=1}^S\ZZ(X_{is})(\IF\{Y_{is} \leq Q(X_{is};\tau)\} - \tau)
\\
\Delta_{is}(\tau,\eta) &:= \IF\{Y_{is} \leq Q(X_{is};\tau)\} - \IF\{Y_{is} \leq Q(X_{is};\eta)\} - (\tau - \eta).
\end{align*}
Similarly to the proof of equations (A.34) and (A.35) in the latter paper we find that for any vector $\bu_N \in \R^m$ with $\|\bu_N\| = 1$ 
\[
\E\Big[|(\omega_{s,1} - 1)\bu_N^\top J_m(\tau)^{-1} \ZZ(X_{is}) \Delta_{is}(\tau,\eta)|^4\Big] \lesssim \xi_m^2 |\tau - \eta|
\]
and
\begin{align*}
& \E\Big[|(\omega_{s,1} - 1)\bu_N^\top (J_m^{-1}(\tau) - J_m^{-1}(\eta)) \ZZ(X_{is})(\IF\{Y_{is} \leq Q(X_{is};\tau)\} - \tau)|^4\Big]
\\
 & \lesssim \xi_m^2 |\tau - \eta|^2
\end{align*}
where the constants on the right-hand side are independent of $\tau,\eta, n$. Next, similar computations as on page 3307 (above the derivation of equation (A.37)) in \cite{ChaVolChe2016} yield for any $\tau, \eta \in \Tc$
\begin{equation}\label{eq:booth2}
\E\Big[|N^{1/2}\bu_N^\top( U_N(\tau) - U_N(\eta))|^4\Big] \lesssim |\tau-\eta|^2 + \frac{\xi_m^2}{N} |\tau - \eta|. 
\end{equation}
In particular we have for $|\tau - \eta| \geq \xi_m^6/N^3$
\begin{equation}\label{eq:booth3}
\E\Big[|N^{1/2}\bu_N^\top( U_N(\tau) - U_N(\eta))|^4\Big] \lesssim |\tau-\eta|^{4/3}. 
\end{equation}
Apply Lemma A.1 of \cite{kvdh2015} with $T = \Tc$, $d(s,t) = |t-s|^{1/3}$, $\bar \eta_n = \xi_m^2/N$, $\Psi(x) = x^4$. Note that with the notation introduced in the latter reference $D(\eps,d) \lesssim \eps^{-3}$ and thus for any $v \geq \xi_m^2/N$
\begin{multline*}
\sup_{|\tau-\eta|\leq \delta } |N^{1/2}\bu_N^\top( U_N(\tau) - U_N(\eta))|
\\
\leq S_1(v) + 2 \sup_{|\tau-\eta|^{1/3} \leq \xi_m^2/N, \tau \in \tilde T} |N^{1/2}\bu_N^\top( U_N(\tau) - U_N(\eta))|
\end{multline*}
where the set $\tilde T$ contains at most $O(\xi_m^{-6}N^3)$ points and the random variable $S_1(v)$ satisfies
\[
\|S_1(v)\|_4 \lesssim \int_0^v \eps^{-3/4} d\eps + (\delta^{1/3} + 2\xi_m^2/N)v^{-6/4}.
\]
Later we will prove that 
\begin{equation}\label{eq:booth1}
\sup_{|\tau-\eta|^{1/3} \leq \xi_m^2/N, \tau \in \tilde T} |N^{1/2}\bu_N^\top( U_N(\tau) - U_N(\eta))| = o_P(1).
\end{equation} 
Given~\eqref{eq:booth1}, it follows that for any $c,v > 0$
\begin{align*}
&\lim_{\delta \downarrow 0}\limsup_{N\to\infty}P\Big(\sup_{|\tau-\eta|\leq \delta } |N^{1/2}\bu_N^\top( U_N(\tau) - U_N(\eta))| \geq c N^{-1/2} \Big)
\\
\leq &\limsup_{N\to\infty}P\Big(\sup_{|\tau-\eta|^{1/3} \leq \xi_m^2/N, \tau \in \tilde T} |N^{1/2}\bu_N^\top( U_N(\tau) - U_N(\eta))| \geq c N^{-1/2}/2 \Big)
\\
& + \lim_{\delta \downarrow 0}\limsup_{N\to\infty} P(S_1(v) \geq c/2)
\\
\lesssim & \lim_{\delta \downarrow 0}\limsup_{N\to\infty} \frac{16}{c^4}\Big[\int_0^v \eps^{-3/4} d\eps + (\delta^{1/3} + 2\xi_m^2/N)v^{-6/4} \Big]^{4} 
\\ 
= & \frac{16}{c^4}\Big[\int_0^v \eps^{-3/4} d\eps\Big]^{4}. 
\end{align*}
Since the right-hand side can be made arbitrarily small by choosing $v$ small, it follows that 
\[
\lim_{\delta \downarrow 0}\limsup_{N\to\infty}P\Big(\sup_{|\tau-\eta|\leq \delta } |N^{1/2}\bu_N^\top( U_N(\tau) - U_N(\eta))| \geq c \Big) = 0,
\]
and hence it remains to prove~\eqref{eq:booth1}. To do so, consider the decomposition
\begin{align*}
& \bu_N^\top( U_N(\tau) - U_N(\eta)) 
\\
= & \frac{1}{N} \bu_N^\top (J_m^{-1}(\tau) - J_m^{-1}(\eta))\sum_{s,i} (\omega_{s,1} - 1) \ZZ(X_{is})(\IF\{Y_{is} \leq Q(X_{is};\tau)\} - \tau)
\\ 
& + \frac{1}{N} \bu_N^\top J_m^{-1}(\eta) \sum_{s,i} (\omega_{s,1} - 1) \ZZ(X_{is}) \Delta_{is}(\tau,\eta).
\end{align*}
Similarly as in the proof of Lemma A.3 in \cite{ChaVolChe2016} (in the calculation below (A.34) on page 3306) we obtain 
\[
\|J_m^{-1}(\tau) - J_m^{-1}(\eta)\|_{op} \lesssim |\tau - \eta|
\]
where $\|\cdot\|_{op}$ denotes the operator norm, and thus for $|\tau-\eta| \leq \xi_m^6N^{-3}$  
\begin{align}
&\Big|\frac{1}{N} \bu_N^\top (J_m^{-1}(\tau) - J_m^{-1}(\eta))\sum_{s,i} (\omega_{s,1} - 1) \ZZ(X_{is})(\IF\{Y_{is} \leq Q(X_{is};\tau)\} - \tau)\Big| \nonumber
\\
\lesssim & |\tau-\eta|\xi_m = O(\xi_m^7 N^{-3}) = o(N^{-1}) \quad a.s. \label{eq:pbh10}
\end{align}
Next observe that for $\tau \leq \eta$
\[
|\Delta_{is}(\tau,\eta)| \leq \IF\{Q(X_{is};\tau) \leq Y_{is} \leq Q(X_{is};\eta)\}+ |\tau-\eta|.
\]
This implies that for any fixed $\tau$ we have a.s. for a constant $C$ independent of $\tau, \eta, N$
\begin{align*}
& \sup_{|\tau-\eta|\leq \eps_N} \Big| \bu_N^\top J_m^{-1}(\eta) N^{-1}\sum_{s,i} (\omega_{s,1} - 1) \ZZ(X_{is}) \Delta_{is}(\tau,\eta)\Big|
\\
\leq & C \frac{\xi_m}{N}\sum_{i,s} \Big[\eps_N + \IF\{Q(X_{is};\tau-\eps_N) \leq Y_{is} \leq Q(X_{is};\tau+\eps_N)\}\Big].
\end{align*}
Thus 
\begin{align*}
& \sup_{\tau \in \tilde T}\sup_{|\tau-\eta|\leq \xi_m^6 N^{-3}} \Big| \bu_N^\top J_m^{-1}(\eta) N^{-1/2}\sum_{s,i} (\omega_{s,1} - 1) \ZZ(X_{is}) \Delta_{is}(\tau,\eta)\Big|
\\
\lesssim & \frac{\eps_N \xi_m}{N} + \xi_m \sup_{\tau \in \tilde T}B_{N}(\tau,\eps_N)
\end{align*}
where 
\[
B_{N}(\tau,\eps_N) := \sum_{i,s} \IF\{Q(X_{is};\tau-\eps_N) \leq Y_{is} \leq Q(X_{is};\tau+\eps_N)\}.
\] 
Since $B_{N}(\tau,\eps_N) \sim Bin(N,2\eps_N)$ distribution and by applying the multiplicative Chernoff bound for Binomial random variables 
we find that for any $t \geq 1$
\[
P\Big(B_{N}(\tau,\eps_N) > N (1+t) \eps_N \Big) \leq \exp(-2tN\eps_N/3).
\]
Plugging in $\eps_N = \xi_m^6N^{-3}$ and $t = 3 A N^2 \log N / (2\xi_m^6) \gg 1$ we obtain
\[
P\Big(B_{N}(\tau,2\xi_m^6N^{-3}) > 2\xi_m^6N^{-2} + 3 A \log N \Big) \leq N^{-A}.
\]
By choosing $A = 3$ and applying the union bound for probabilities we finally obtain for $N$ sufficiently large and a constant $0<c<\infty$
\begin{align*}
&P\Big(\sup_{\tau \in \tilde T}\sup_{|\tau-\eta|\leq \xi_m^6 N^{-3}} \Big| \bu_N^\top J_m^{-1}(\eta) \sum_{s,i} (\omega_{s,1} - 1) \ZZ(X_{is}) \Delta_{is}(\tau,\eta)\Big| \geq c \xi_m\log N \Big)
\\
\leq & |\tilde T| N^{-3} = o(1).
\end{align*}
Combining this with~\eqref{eq:pbh10}, \eqref{eq:booth1} follows.

\textbf{Proof of step 2}  
Observe that
\begin{align*}
&\ZZ(x_0)^\top ( \overline \zb^{(1)}(\tau) - \overline \zb(\tau) )
\\
=~& \ZZ(x_0)^\top \frac{1}{S} \sum_{s=1}^S \Big( \frac{\omega_{s,1}}{\bar\omega_{\cdot,1}} - 1 \Big) \widehat \zb^s(\tau)
\\ 
=~& \frac{1}{\bar\omega_{\cdot,1}}\ZZ(x_0)^\top \frac{1}{S} \sum_{s=1}^S \Big( \omega_{s,1} - \bar\omega_{\cdot,1} \Big) \Big(\widehat \zb^s(\tau) - \zb_N(\tau)\Big) 
\\ 
=~& \frac{1}{\bar\omega_{\cdot,1}}\ZZ(x_0)^\top \frac{1}{S} \sum_{s=1}^S \Big( \omega_{s,1} - 1\Big) \Big(\widehat \zb^s(\tau) - \zb_N(\tau)\Big) 
\\ 
&- \frac{\bar\omega_{\cdot,1} - 1}{\bar\omega_{\cdot,1}}\ZZ(x_0)^\top \frac{1}{S} \sum_{s=1}^S\Big(\widehat \zb^s(\tau) - \zb_N(\tau)\Big) 
\\
=~& \frac{1}{\bar\omega_{\cdot,1}}\ZZ(x_0)^\top \frac{1}{S} \sum_{s=1}^S \Big( \omega_{s,1} - 1\Big) \Big(\widehat \zb^s(\tau) - \zb_N(\tau)\Big)  + o_P\Big(\frac{\|\ZZ(x_0)\|}{N^{1/2}}\Big).
\end{align*}
Note that under the assumptions on the weights $\bar\omega_{\cdot,1} = 1 + o_P(1)$. Thus~\eqref{eq:bahgen2_boot} will follow from step 1 once we prove that, uniformly in $\tau \in \Tc_K$,
\begin{equation}\label{eq:bahgen2_boot2}
\ZZ(x_0)^\top \frac{1}{S} \sum_{s=1}^S \Big( \omega_{s,1} - 1\Big) \Big(\widehat \zb^s(\tau) - \zb_N(\tau)\Big) = \Gb_N^{(1)}(\tau) + o_P\Big(\frac{\|\ZZ(x_0)\|}{N^{1/2}}\Big).
\end{equation}
To this end apply Theorem~\ref{th:bahsimple}, which implies that
\begin{multline*}
\ZZ(x_0)^\top \frac{1}{S} \sum_{s=1}^S ( \omega_{s,1} - 1) \Big(\widehat \zb^s(\tau) - \zb_N(\tau)\Big) - \tilde \Gb_N^{(1)}(\tau)
\\
= \ZZ(x_0)^\top \frac{1}{S} \sum_{s=1}^S ( \omega_{s,1} - 1)\Big(r_{n,1}^s(\tau) + r_{n,2}^s(\tau) + r_{n,3}^s(\tau) + r_{n,4}^s(\tau)\Big)
\end{multline*}
where
\[
\tilde \Gb_N^{(1)}(\tau) := - \frac{1}{N} \sum_{s=1}^S (\omega_{s,1}-1) \ZZ(x_0)^\top J_m^{-1}(\tau)\sum_{i=1}^n \psi(Y_{is},\ZZ(X_{is});\zb_N(\tau),\tau)
\] 
and the remainder terms $r_{n,j}^s(\tau)$ correspond to $r_{n,j}(\tau)$ defined in Theorem~\ref{th:bahsimple}. We will prove that
\begin{align}
&\sup_{\tau \in \Tc_K} \Big| \ZZ(x_0)^\top \frac{1}{S} \sum_{s=1}^S ( \omega_{s,1} - 1)\Big(r_{n,1}^s(\tau) + r_{n,2}^s(\tau) + r_{n,3}^s(\tau) + r_{n,4}^s(\tau)\Big)\Big| \label{eq:booth4}
\\ 
& = o_P(\|\ZZ(x_0)\|N^{-1/2}), \nonumber
\\
&\sup_{\tau \in \Tc_K} \Big| \tilde \Gb_N^{(1)}(\tau) - \Gb_N^{(1)}(\tau)\Big| = o_P(\|\ZZ(x_0)\|N^{-1/2}). \label{eq:booth5}
\end{align}
We begin by a proof of~\eqref{eq:booth4}. Applying Theorem~\ref{th:bahsimple} with $\kappa_n = A \log n$ and $A$ sufficiently large we find that
\begin{align*}
& \Big| \ZZ(x_0)^\top \frac{1}{S} \sum_{s=1}^S ( \omega_{s,1} - 1)\Big(r_{n,1}^s(\tau) + r_{n,2}^s(\tau) + r_{n,4}^s(\tau)\Big) \Big|
\\
\leq &\sup_s \sup_{\tau \in \Tc} \Big|\ZZ(x_0)^\top(r_{n,1}^s(\tau) + r_{n,2}^s(\tau) + r_{n,4}^s(\tau))\Big|  \frac{1}{S}\sum_{j=1}^S |\omega_{j,1} - 1|
\\
= & o_P\Big(\frac{\|\ZZ(x_0)\|}{N^{1/2}}\Big).
\end{align*}
Next, note that by independence between $\omega_{s,1}$ and the sample we have $\E[( \omega_{s,1} - 1)r_{n,3}^s(\tau)] = 0$ for all $s,\tau$. Let 
\[
D_n(A) := C\Re_3(A\log n) 
\]
where $\Re_3(\cdot)$ is defined in~\eqref{eq:re3}, and $C$ is the constant appearing in~\eqref{eq:retail}. 
Consider the decomposition
\begin{align*}
(\omega_{s,1} - 1)\ZZ(x_0)^\top r_{n,3}^s(\tau) = &\ZZ(x_0)^\top (\omega_{s,1} - 1)r_{n,3}^s(\tau)\IF\{|r_{n,3}^s(\tau)| \leq D_n(A)C_\omega\}
\\
&+ \ZZ(x_0)^\top (\omega_{s,1} - 1)r_{n,3}^s(\tau)\IF\{|r_{n,3}^s(\tau)| > D_n(A)C_\omega\}
\\
=: & R_{n,1}^{(s)}(\tau) + R_{n,2}^{(s)}(\tau)
\end{align*}
By \eqref{eq:retail}, \eqref{eq:re3bias} in Theorem~\ref{th:bahsimple}, we see that
\begin{align}
P\Big(\sup_{s=1,...,S}\sup_\tau \| R_{n,2}^{(s)}(\tau) \| \neq 0\Big) \leq 2Sn^{-A}.\label{eq:th1h1_boot}
\end{align}
Moreover, by independence of the weights $\omega_{s,1}$ from the sample $\E[R_{n,1}^{(s)}(\tau)] = 0$ and by construction $|R_{n,1}^{(s)}(\tau)| \leq \|\ZZ(x_0)\| C_\omega D_N(A)$ a.s. Apply Hoeffding's inequality to obtain 
\[
P\Big( \Big|\frac{1}{S} \sum_{s=1}^S R_{n,1}^{(s)}(\tau) \Big| \geq t \|\ZZ(x_0)\| \Big) \leq 2\exp\Big( - \frac{2St^2}{4C_\omega^2 D_N^2(A)}\Big)
\]
Letting $t = 4 \sqrt{\log n} S^{-1/2} D_N(6) C_\omega$ and combining this with~\eqref{eq:th1h1_boot} shows that
\begin{align*}
& P\Big( \sup_{\tau \in \Tc_K} \Big|\ZZ(x_0)^\top \frac{1}{S} \sum_{s=1}^S (\omega_{s,1} - 1)r_{n,3}^s(\tau)\Big| \geq 4 \|\ZZ(x_0)\| \sqrt{\log n} S^{-1/2} D_N(6) C_\omega \Big)
\\
\leq & N^2\Big(2 n^{-8} + 2S n^{-6}\Big) = o(1). 
\end{align*}
Straightforward but tedious computations show that
\[
\sqrt{\log n} S^{-1/2} D_N(6) = o(N^{-1/2}),
\]
and thus the proof of~\eqref{eq:booth4} is complete. Next we prove~\eqref{eq:booth5}. Note that
\begin{align*}
~&\Big| \tilde \Gb_N^{(1)}(\tau) - \Gb_N^{(1)}(\tau)\Big| 
\\
\leq ~& C_\omega \frac{1}{N}\sum_{i,s} |\ZZ(x_0)^\top J_m^{-1}(\tau) \ZZ(X_{is})| \IF\{|Y_{is} - Q(X_{is};\tau)| \leq c_m(\zg_N)\}
\\
=: ~& C_\omega \|\ZZ(x_0)\|\frac{1}{N}\sum_{i,s} W_{is}(\tau).
\end{align*}
We have 
\[
\sup_{\tau \in \Tc} \E[W_{is}(\tau)] = O(c_m(\zg_N)) = o(N^{-1/2}\|\ZZ(x_0)\|)
\]
and $\sup_{\tau\in\Tc}|W_{is}(\tau)| \leq C \xi_m$ almost surely for some constant $C<\infty$ . Moreover, the $W_{is}(\tau)$ are iid across $i,s$ and 
\[
\sup_{\tau \in \Tc} \Var(W_{11}(\tau)) = O(c_N(\zg_N)) = O(\|\ZZ(x_0)\|N^{-1/2}) =  o(N^{-1/4}).
\]
Hence Bernstein's inequality and the union bound for probabilities yields
\begin{align*}
&P\Big(\sup_{\tau \in \Tc_K} \Big| \sum_{i,s} W_{is}(\tau) - \E[W_{is}(\tau)]\Big| \geq N^{1/2} (\log N)^{-1} \Big)
\\
\leq & 2 |\Tc_K| \exp\Big( - \frac{N(\log N)^{-2}/2}{N \sup_{\tau \in \Tc} \Var(W_{11}(\tau)) + C \xi_m N^{1/2}(\log N)^{-1}/3} \Big) = o(1).
\end{align*}
Thus the proof of~\eqref{eq:booth5} is complete.

 \hfill $\Box$
 

\subsection{Proof of Theorem~\ref{th:powell}} \label{sec:prpowell}

Define the class of functions
\[
\Gc_3(\delta) := \Big\{ (z,y) \mapsto \|zz^\top\|\IF \Big\{|y - z^\top \bb + t | \leq \delta \Big\}\IF\{\|z\|\leq \xi\}  \Big| t\in \R, \bb\in \R^m \Big\}.
\]
Let 
\[
\tilde \zb^{s,-i}(\tau) := \zb(\tau) - (n-1)^{-1/2} J_m(\tau)^{-1}\Gb_n^{s,-i} \psi(\cdot;\hat \zb^{s,-i}(\tau),\tau) 
\]
where $\Gb_n^{s,-i}$ is the empirical process corresponding to sub-sample $s$ with i'th observation removed and
\[
\hat \zb^{s,-j}(\tau) := \arg\min_{\bb \in \R^m} \sum_{i=1, i\neq j}^n \rho_\tau(Y_{is} - \ZZ_{is}^\top \bb).
\]
By Lemma~\ref{lem:leaveout} we obtain for $A \geq 1$ and some constant $C$
\[
P\Big( \sup_{s,\tau,j} \|\hat \zb^{s,-j}(\tau) - \hat \zb^s(\tau)\| \geq C A \frac{\log n}{n}\Big) \leq 2Sn^{1-A}.
\]
Following the arguments in the proof of Lemma~\ref{lem:leaveout} we also find that for some constant $C_1$ we have almost surely
\[
\Big\| \tilde \zb^{s,-j}(\tau) - \hat \zb^{s,-j}(\tau) \Big\| \leq C_1\Big(\frac{1}{n} + \|\hat \zb^{s,-j}(\tau) - \zb(\tau)\|^2\Big),
\]
so that by~\eqref{eq:omega1} there exists a constant $C_2$ such that
\begin{equation} \label{eq:powhe1}
P\Big( \sup_{s,\tau,j} \|\tilde \zb^{s,-j}(\tau) - \hat \zb^s(\tau)\| \geq C_2 A \frac{\log n}{n}\Big) \leq Sn^{1-A}.
\end{equation}
Moreover, the arguments in the proof of Theorem~\ref{th:bahsimple} show that 
\begin{align} \nonumber
&\sup_{s,i,\tau} \Big\| \E[\tilde \zb^{s,-i}(\tau)] - \zb(\tau) \Big\|
\\
 =& \sup_{s,i,\tau} \Big\| \E\Big[ n^{-1/2} J_m(\tau)^{-1}\Gb_n^{s,-i} \psi(\cdot;\hat \zb^{s,-i}(\tau),\tau)\Big] \Big\| = O\Big(\frac{\log n}{n}\Big). \label{eq:powhe2}
\end{align}
Define the events
\[
\Omega_{1,n}^{s,j} := \Big\{ \|\tilde \zb^{s,-j}(\tau) - \hat \zb^s(\tau)\| \leq C_2 4 \frac{\log n}{n} \Big\}, \quad \Omega_{1,n} := \cap_{s,j} \Omega_{1,n}^{s,j}. 
\] 
and observe that for $\delta_n := 4 C_2 \frac{\log n}{n}$
\begin{align*}
& \frac{1}{2nh_n} \Big\|\sum_{i=1}^n \ZZ_{is}\ZZ_{is}^\top \Big[\IF \Big\{|Y_{is} - \ZZ_{is}^\top \hat \zb^s(\tau)| \leq h_n \Big\} - \IF \Big\{|Y_{is} - \ZZ_{is}^\top \tilde \zb^{s,-i}(\tau)| \leq h_n \Big\}\Big] \Big\|
\\
\leq~& \frac{1}{h_n} \sup_{\bb \in \R^m, t\in \R} \frac{1}{n} \sum_{i=1}^n \|\ZZ_{is}\ZZ_{is}^\top \| \IF \Big\{|Y_{is} - \ZZ_{is}^\top \bb - t| \leq \xi\delta_n \Big\} + \frac{\xi^2}{h_n} \IF_{\Omega_{1,n}^C} 
\\
\leq~&  \frac{1}{h_n} \sup_{\bb \in \R^m, t\in \R} \E\Big[\|\ZZ_{1s}\ZZ_{1s}^\top \| \IF \Big\{|Y_{1s} - \ZZ_{1s}^\top \bb - t| \leq \xi\delta_n \Big\} \Big]
\\
& + \frac{1}{h_n} \|\Pb_{n,s}-P\|_{\Gc_3(\delta_n)} + \frac{\xi^2}{h_n} \IF_{\Omega_{1,n}^C} 
\\
=~& \frac{1}{h_n} \|\Pb_{n,s}-P\|_{\Gc_3(\xi\delta_n)} + O(h_n^{-1}\delta_n) + \frac{\xi^2}{h_n} \IF_{\Omega_{1,n}^C}.  
\end{align*}
Standard arguments show that $\Gc_3(\xi\delta_n)$ satisfies~\eqref{eq:entr} with fixed constant $V$ and thus by~\eqref{eq:Gexpect} we have
\[
\sup_s \E\|\Pb_{n,s}-P\|_{\Gc_3(\xi\delta_n)} = O\Big(\frac{\log n}{n}\Big).
\]
This implies
\[
\sup_\tau \Big\| \overline J_{m}^P (\tau) - \frac{1}{S}\sum_{s=1}^S \frac{1}{2nh_n}\sum_{i=1}^n \ZZ_{is}\ZZ_{is}^\top \IF \Big\{|Y_{is} - \ZZ_{is}^\top \tilde \zb^{s,-i}(\tau)| \leq h_n \Big\}\Big\| = O_P\Big(\frac{\log n}{n h_n}\Big).
\] 
Next note that by a Taylor expansion we have
\begin{align*}
F_{Y|X}(a + h_n|x) - F_{Y|X}(a - h_n|x) = 2h_n f_{Y|X}(a|x) + \frac{h_n^3}{3}f''_{Y|X}(a|x) + o(h_n^3)
\end{align*}
where the remainder term is uniform in $a \in \R$. Thus we have by~\eqref{eq:powhe2}
\begin{align*}
&\E\Big[\frac{1}{2h_n} \ZZ_{1s}\ZZ_{1s}^\top \IF \Big\{|Y_{1s} - \ZZ_{1s}^\top \tilde \zb^{s,-1}(\tau)| \leq h_n \Big\}\Big]
\\
=~& \int \frac{1}{2h_n} \ZZ(x)\ZZ(x)^\top \E\Big[\Big(F_{Y|X}(\ZZ(x)^\top \tilde \zb^{s,-1}(\tau) + h_n|x)
\\
& \quad\quad\quad\quad - F_{Y|X}(\ZZ(x)^\top \tilde \zb^{s,-1}(\tau) - h_n|x) \Big)\Big] dP^X(x) 
\\
=~& \int \ZZ(x)\ZZ(x)^\top \E\Big[\Big(f_{Y|X}(\ZZ(x)^\top \tilde \zb^{s,-1}(\tau)|x)
\\
& \quad\quad\quad\quad + \frac{h_n^2}{6}f''_{Y|X}(\ZZ(x)^\top \tilde \zb^{s,-1}(\tau)|x) \Big)\Big] dP^X(x) + o(h_n^2)
\\
=~& \int \ZZ(x)\ZZ(x)^\top \Big[\Big(f_{Y|X}(\ZZ(x)^\top \zb(\tau)|x) + \frac{h_n^2}{6}f''_{Y|X}(\ZZ(x)^\top \zb(\tau)|x) \Big)\Big] dP^X(x)
\\
&  + O\Big(\frac{\log n}{n}\Big) + o(h_n^2).
\end{align*}
Similar but simpler computations show that for 
\[
\WW_s(\bb) := \mbox{vec}\Big( \frac{1}{2h_n} \ZZ_{1s}\ZZ_{1s}^\top \IF \Big\{|Y_{1s} - \ZZ_{1s}^\top \bb| \leq h_n \Big\} \Big)
\] 
we have uniformly in $s$
\[
\Var(\WW_s(\tilde \zb^{s,-1}(\tau))) = \Var(\WW_s(\zb(\tau))) + o(h_n^{-1}).
\]
This completes the proof. \hfill $\Box$


\vskip 2em

%

\vskip 2em

\section{Proofs for Appendix \ref{SEC:APP_ABR}}\label{sec:s_pfappA}

\setcounter{subsection}{0}
\setcounter{equation}{0}
\setcounter{theo}{0}

%

\subsection{Proof of Theorem~\ref{th:aggbah}}\label{sec:proof_aggbah}

From Theorem~\ref{th:bahsimple} we obtain the representation
\begin{multline*}
\widehat\zb^s(\tau) - \zb_N(\tau) = -  n^{-1/2}J_m(\tau)^{-1}\Gb_n^s(\psi(\cdot;\zb_N(\tau),\tau)) 
\\
+ r_{n,1}^s(\tau) + r_{n,2}^s(\tau) + r_{n,3}^s(\tau) + r_{n,4}^s(\tau).
\end{multline*}
where the quantities $r_{n,j}^1,...,r_{n,j}^S$ are i.i.d. for each $j=1,...,4$ and 
\[
\Gb_n^s(\psi(\cdot;\zb_N(\tau),\tau)) = n^{-1/2}\sum_{i=1}^n \psi(Y_{is},\ZZ(X_{is});\zb_N(\tau),\tau).
\]
Letting 
\[
r_{N}^{(1)}(\tau) := \frac{1}{S} \sum_{s=1}^S \Big(r_{n,1}^s(\tau) + r_{n,2}^s(\tau) + r_{n,4}^s(\tau)\Big) 
\]
we find that the bound in~\eqref{eq:th1rn1} is a direct consequence of Theorem~\ref{th:bahsimple}, the union probability bound and the observation that $\xi_m g_N^2 = o(g_N)$ by the assumption $g_N=o(\xi_m^{-1})$. 

Next, let 
\[
r_{N}^{(2)}(\tau) := \frac{1}{S} \sum_{s=1}^S r_{n,3}^s(\tau).
\]
Define $\tilde r_{n,3}^s(\tau) := r_{n,3}^s(\tau) - \E[r_{n,3}^s(\tau)]$ and consider the decomposition
\begin{multline*}
r_{n,3}^s(\tau) = \tilde r_{n,3}^s\IF\big\{\|\tilde r_{n,3}^s(\tau)\| \leq D_n(A)\big\}
\\
 + \tilde r_{n,3}^s(\tau)\IF\big\{\|\tilde r_{n,3}^s(\tau)\| > D_n(A)\big\} + \E[r_{n,3}^s(\tau)] =: \sum_{j=1}^3 R_{n,j}^{(s)}(\tau)
\end{multline*}
where, for a sufficiently large constant $C$, 
\[
D_n(A) := C \Re_3(A\log n) + C_{3,2} \xi_m\Big(g_N^2\xi_m + \frac{m \xi_m \log n}{n}\Big), 
\]
$\Re_3(\cdot)$ is defined in \eqref{eq:re3} and $C_{3,2}$ denotes the constant appearing in \eqref{eq:re3bias}. By \eqref{eq:retail}, \eqref{eq:re3bias} in Theorem~\ref{th:bahsimple}, we see that
\begin{align}
P\Big(\sup_{s=1,...,S}\sup_\tau \| R_{n,2}^{(s)}(\tau) \| \neq 0\Big) \leq 2Sn^{-A};\label{eq:th1h1}
\end{align}
and by \eqref{eq:re3bias} in Theorem~\ref{th:bahsimple},
\begin{align}
\sup_\tau \| R_{n,3}^{(s)}(\tau) \| \leq  C_{3,2} \xi_m\Big(g_N^2\xi_m + \frac{m \xi_m \log n}{n}\Big).\label{eq:th1h2}
\end{align}
Next we deal with $R_{n,1}^{(s)}(\tau)$. Consider the function $g: (\R^m)^S \to \R$, 
\[
g(x_1,...,x_S) := \Big\|\frac{1}{S}\sum_{s=1}^S x_s\Big\|. 
\]
Direct computations show that 
\[
|g(x_1,...,x_S) - g(x_1,...,x_i',...,x_S)| \leq \frac{\|x_i'-x_i\|}{S}.
\]
Note that by construction $\|R_{n,1}^{(s)}(\tau)\| \leq D_n(A)$ almost surely. Apply McDiarmid's inequality (Lemma 1.2 on page 149 in \cite{MD89}) to $g(R_{n,1}^{(1)}(\tau),..., R_{n,1}^{(S)}(\tau))$ to obtain for all $t>0$
\[
\sup_\tau P\Big(\Big\|\frac{1}{S}\sum_{s=1}^S R_{n,1}^{(s)}(\tau)\Big\| \geq t \Big) \leq 2\exp\Big(\frac{-St^2}{2D_n^2(A)} \Big)
\]
Letting $t = \sqrt{2}\kappa_n S^{-1/2} D_n(A)$ and combining this with \eqref{eq:th1h1} and \eqref{eq:th1h2} shows that
\[
\sup_\tau P\Big(\big\| r_N^{(2)}(\tau)\big\| > \sqrt{2}\kappa_n S^{-1/2} D_n(A) + C_{3,2} \xi_m\Big(g_N^2\xi_m + \frac{m \xi_m \log n}{n} \Big)\Big) \leq 2Sn^{-A} + 2e^{-\kappa_n^2}.
\]
Finally, from an elementary but tedious calculation and the definition of $\Re_3(A\log n)$ in \eqref{eq:re3}, we obtain
\[
\kappa_n S^{-1/2} \Re_3(A \log n) \lesssim A \frac{\kappa_n}{S^{1/2}} \Big(\frac{m}{n}\Big)^{1/2} \Big( (\xi_m g_N \log n)^{1/2} + \Big(\frac{m\xi_m^2 (\log n)^3}{n} \Big)^{1/4}\Big) .
\]
Thus the proof of Theorem~\ref{th:aggbah} is complete. \hfill $\qed$

\subsection{Proof of Theorem \ref{th:aggbah_loc}}\label{sec:proof_aggbah_loc} 

The proof proceeds similarly to the proof of Theorem~\ref{th:aggbah}. From Theorem~\ref{th:bahsimple_bspl} we obtain the representation
\begin{align*}
& \bu_N^\top(\widehat\zb^s(\tau) - \zg_N(\tau))
\\
= & -  n^{-1}\bu_N^\top J_m(\zg_N(\tau))^{-1} \sum_{i=1}^n \ZZ_i^s (\IF\{Y_i^s \leq (\ZZ_i^s)^\top \zg_N(\tau)\}- \tau)
+ \bu_N^\top \sum_{k=1}^4 r_{n,k}^s(\tau)
\\
& + \bu_N^\top J_m(\zg_N(\tau))^{-1}\E[\ZZ(F_{Y|X}(\ZZ^\top \zg_N(\tau)|X) - \tau)].
\end{align*}
where the quantities $r_{n,j}^1(\tau),...,r_{n,j}^S(\tau)$ are i.i.d. for each $j=1,...,4$ and 
\[
\Gb_n^s(\psi(\cdot;\zg_N(\tau),\tau)) = n^{-1/2}\sum_{i=1}^n \psi(Y_{is},\ZZ(X_{is});\zg_N(\tau),\tau).
\]
Next, note that
\begin{multline*}
\ZZ(F_{Y|X}(\ZZ^\top \zg_N(\tau)|X) - \tau) = \ZZ f_{Y|X}(Q(X;\tau)|X)(\ZZ^\top \zg_N(\tau) - Q(X;\tau)) 
\\
+ \frac{1}{2}\ZZ f'_{Y|X}(\zeta_n(X,\tau)|X)(\ZZ^\top \zg_N(\tau) - Q(X;\tau))^2
\end{multline*}
where $\zeta_n(X,\tau)$ is a value between $\ZZ(X)^\top \zg_N(\tau)$ and $Q(X;\tau)$. From the definition of $\zg_N$ as \eqref{eq:gamman}, we obtain $\E[\ZZ f_{Y|X}(Q(X;\tau)|X)(\ZZ^\top \zg_N(\tau) - Q(X;\tau))] = 0$ so that 
\begin{align*}
\sup_{\tau \in \Tc}\Big| \bu_N^\top \tilde J_m(\zg_N(\tau))^{-1}\E[\ZZ(F_{Y|X}(\ZZ^\top \zg_N(\tau)|X) - \tau)] \Big|
\lesssim \tilde \Ec(\bu_N,\zg_N) c_n^2(\zg_N).
\end{align*}
Letting 
\begin{multline*}
r_{N}^{(1)}(\tau,\bu_N) := \Big|\frac{1}{S} \bu_N^\top \sum_{s=1}^S \Big(r_{n,1}^s(\tau)
+ r_{n,2}^s(\tau) + r_{n,4}^s(\tau)\Big)
\\
 + \bu_N^\top J_m(\zg_N(\tau))^{-1}\E[\ZZ(F_{Y|X}(\ZZ^\top \zg_N(\tau)|X) - \tau)] \Big| 
\end{multline*}
we find that the bound on $r_{N}^{(1)}(\tau,\bu_N)$ is a direct consequence of Theorem~\ref{th:bahsimple_bspl}. Next, let 
\[
r_{N}^{(2)}(\tau,\bu_N) := \frac{1}{S} \sum_{s=1}^S \bu_N^\top r_{n,3}^s(\tau).
\]
Define $\tilde r_{n,3}^s(\tau) := r_{n,3}^s(\tau) - \E[r_{n,3}^s(\tau)]$ and consider the decomposition
\begin{align*}
r_{n,3}^s(\tau)  =~& \tilde r_{n,3}^s\IF\Big\{\sup_{\bu_N \in \Sc_\Ic^{m-1}} |\bu_N^\top \tilde r_{n,3}^s(\tau)| \leq D_n^{(L)}(A)\Big\}
\\
& + \tilde r_{n,3}^s(\tau)\IF\Big\{\sup_{\bu_N \in \Sc_\Ic^{m-1}} |\bu_N^\top \tilde r_{n,3}^s(\tau)| > D_n^{(L)}(A)\Big\} + \E[r_{n,3}^1(\tau)]
\\
=:~& \sum_{j=1}^3 R_{n,j}^{(s)}(\tau)
\end{align*}
where for a constant $A>0$, and a sufficiently large constant $C$
\[
D_n^{(L)}(A) := C\Re^{(L)}_3(A\log n) + C_{3,2} \sup_{\bu_N \in \Sc_\Ic^{m-1}}\tilde{\Ec}(\bu_N,\zg_N)\Big(c_m^4(\zg_N) + \frac{\xi_m^2 (\log n)^2}{n}\Big),
\] 
$\Re^{(L)}_3$ is defined in \eqref{eq:re3_a}, $C_{3,2}$ denotes the constant appearing in \eqref{eq:re3bias_a}.

By \eqref{eq:retail_a} in Theorem \ref{th:bahsimple_bspl},
\begin{align}
P\Big(\sup_{s=1,...,S}\sup_\tau \sup_{\bu_N \in \Sc_\Ic^{m-1}} | \bu_N^\top R_{n,2}^{(s)}(\tau) | \neq 0\Big) \leq 2Sn^{-A}.\label{eq:th11h1}
\end{align}
Applying \eqref{eq:re3bias_a} in Theorem~\ref{th:bahsimple_bspl} yields
\begin{align}
\sup_\tau \sup_{\bu_N \in \Sc_\Ic^{m-1}} |\bu_N^\top R_{n,3}^{(s)}(\tau)| \leq C_{3,2} \sup_{\bu_N \in \Sc_\Ic^{m-1}}\tilde{\mathcal{E}}(\bu_N,\zg_N)\Big(c_m^4(\zg_N) + \frac{\xi_m^2 (\log n)^2}{n}\Big).\label{eq:th11h2}
\end{align}
Next consider he function $g: (\R^m)^S \to \R$, 
\[
g(x_1,...,x_S) := \sup_{\bu_N \in \Sc_\Ic^{m-1}} \Big|\bu_N^\top \frac{1}{S}\sum_{s=1}^S x_s\Big| = \Big\|\Big(\frac{1}{S}\sum_{s=1}^S x_s\Big)^{(\Ic)}\Big\|.
\]
The reverse triangle inequality shows that  
\[
|g(x_1,...,x_S) - g(x_1,...,x_i',...,x_S)| \leq \sup_{\bu_N \in \Sc_\Ic^{m-1}} \frac{|\bu_N^\top(x_i'-x_i)|}{S}.
\]
Note that by construction $\sup_{\bu_N \in \Sc_\Ic^{m-1}} |\bu_N^\top R_{n,1}^{(s)}(\tau)| \leq D_n^{(L)}(A)$ almost surely. Apply McDiarmid's inequality (Lemma 1.2 on page 149 in \cite{MD89}) to $g(R_{n,1}^{(1)}(\tau),..., R_{n,1}^{(S)}(\tau))$ to obtain for all $t>0$
\[
\sup_\tau P\Big(\sup_{\bu_N \in \Sc_\Ic^{m-1}} \Big|\frac{1}{S}\bu_N^\top \sum_{s=1}^S R_{n,1}^{(s)}(\tau)\Big| \geq t \Big) \leq 2\exp\bigg(\frac{-St^2}{2(D_n^{(L)}(A))^2} \bigg).
\]
Letting $t = \sqrt{2}\kappa_n S^{-1/2} D_n^{(L)}(A)$ and combining this with \eqref{eq:th11h1}-\eqref{eq:th11h2} shows that
\begin{multline*}
\sup_\tau P\Big(\sup_{\bu_N \in \Sc_\Ic^{m-1}}\Big| r_N^{(2)}(\tau,\bu_N)\Big| > \sqrt{2}\kappa_n S^{-1/2} D_n^{(L)}(A)
\\
+ C_{3,2} \sup_{\bu_N \in \Sc_\Ic^{m-1}}\tilde{\mathcal{E}}(\bu_N,\zg_N)\Big(c_m^4(\zg_N) + \frac{\xi_m^2 (\log n)^2}{n} \Big)\Big)
\\
\leq 2Sn^{-A} + 2e^{-\kappa_n^2}.
\end{multline*}
Thus the proof of Theorem~\ref{th:aggbah_loc} is complete. \hfill $\qed$


\setcounter{subsection}{0}
\setcounter{equation}{0}
\setcounter{theo}{0}

\section{Refined Bahadur representations for sub-samples and their proofs}\label{sec:s_br}

In this section, we consider a triangular array $\{(X_{is},Y_{is})\}_{i=1}^{n}$ from the $s$-th group. In order to keep the notation simple, the sample will be denoted by $\{(X_i,Y_i)\}_{i=1}^{n}$. In addition to the notation introduced in the main part of the manuscript, we will make use of the following notation throughout the appendix. Denote the empirical measure of $n$ samples $(Y_i,\ZZ_i)$ by $\Pb_n$. For a function $x \mapsto f(x)$ define $\Gb_n(f) := n^{1/2}\int f(x) (d\Pb_n(x)-dP(x))$ and $\|f\|_{L_p(P)} = (\int |f(x)|^p dP(x))^{1/p}$ for $0<p < \infty$. For a class of functions $\Gc$, let $\|\Pb_n-P\|_\Gc := \sup_{f \in \Gc} |\Pb_nf - Pf|$. For any $\epsilon>0$, the covering number $N(\epsilon,\Gc,L_p)$ is the minimal number of balls of radius $\epsilon$ (under $L_p$-norm) that is needed to cover $\Gc$. The bracketing number $N_{[ ~]}(\epsilon,\Gc,L_p)$ is the minimal number of $\epsilon$-brackets that is needed to cover $\Gc$. An $\epsilon$-bracket refers to a pair of functions within an $\epsilon$ distance: $\|u-l\|_{L_p} < \epsilon$. Throughout the proofs, $C,C_1,C_j$ etc. will denote constants which do not depend on $n$ but may have different values in different lines. For $\tau\in\Tc$, define $\psi(Y_i,\ZZ_i;\bb,\tau) := \ZZ_i(\IF\{Y_i \leq \ZZ_i^\top  \bb\}- \tau)$ and $\mu(\bb,\tau) := \E\big[\psi(Y_i,\ZZ_i;\bb,\tau)\big]=\E\big[\ZZ_i\big\{F_{Y|X}(\ZZ_i^\top  \bb|X)-\tau\big\}\big]$.

The aim of this section is to provide precise bounds on the remainder terms in the Bahadur representation of the estimator $\breve\zb(\tau)$ which is defined as
\[
\breve\zb(\tau) := \arg\,\min_{\bb\in\R^m} \sum_{i=1}^n \rho_\tau(Y_{i} - \bb^\top\ZZ(X_{i})), \quad \tau \in \Tc.
\]\\
\begin{theo} \label{th:bahsimple} Suppose Conditions \hyperref[A1]{(A1)}-\hyperref[A3]{(A3)} hold. Assume that additionally $m \xi_m^2 \log n = o(n)$. Then, for any $\zb_N(\cdot)$ satisfying $g_N(\zb_N) = o(\xi_m^{-1})$ and $c_m(\zb_N) = o(1)$, we have
\begin{multline}
\breve\zb(\tau) - \zb_N(\tau) = -  n^{-1/2}J_m(\tau)^{-1}\Gb_n(\psi(\cdot;\zb_N(\tau),\tau))
\\
+ r_{n,1}(\tau) + r_{n,2}(\tau) + r_{n,3}(\tau) + r_{n,4}(\tau).\label{eq:bahadur_bn}
\end{multline}
The remainder terms $r_{n,j}$'s can be bounded as follows:
\beq \label{eq:brn1}
\sup_\tau \|r_{n,1}(\tau)\| \leq  C_1 \frac{m\xi_m}{n} \quad a.s.
\eeq
for a constant $C_1$ independent of $n$. Moreover, we have for any $\kappa_n \ll n/\xi_m^2$, all sufficiently large $n$, and a constant $C$ independent of $n$
\begin{align}
P\Big( \sup_\tau \|r_{n,j}(\tau)\| \leq C \Re_j(\kappa_n) \Big) \geq 1 - 2e^{-\kappa_n}, \quad j=2,3,4,
\label{eq:retail}
\end{align}
where
\begin{align}
\Re_2(\kappa_n) &:= \xi_m \Big( \Big(\frac{m}{n} \log n\Big)^{1/2}  + n^{-1/2}\kappa_n^{1/2} + g_N \Big)^2, \label{eq:re2}
\\
\Re_3(\kappa_n) &:= \Big(\Big(\frac{m \log n}{n}\Big)^{1/2} + \Big(\frac{\kappa_n}{n}\Big)^{1/2} + g_N \Big)^{1/2} \bigg(\Big(\frac{m \xi_m \log n}{n}\Big)^{1/2} + \Big(\frac{\xi_m \kappa_n}{n} \Big)^{1/2}\bigg),\quad\quad \label{eq:re3}
\\
\Re_4(\kappa_n) &:= c_m(\zb_N) \Big( \Big(\frac{m}{n} \log n\Big)^{1/2}  + \frac{\kappa_n^{1/2}}{n^{1/2}}\Big) + g_N . \label{eq:re4}
\end{align}
Additionally, 
\begin{align} 
\sup_\tau \|r_{n,3}(\tau)\| \leq C_{3,1}\xi_m, \quad a.s., \label{eq:re3asbound}
\end{align}
and it holds that for sufficiently large $n$,
\begin{align} 
\sup_{\tau} \| \E[r_{n,3}(\tau)] \| \leq C_{3,2} \xi_m^2\Big(g_N^2 + \frac{m \log n}{n}\Big) \label{eq:re3bias}
\end{align}
where $C_{3,1},C_{3,2}$ are constants that are independent of $n$.
\end{theo}

The statements in \eqref{eq:brn1}-\eqref{eq:re4} are from Theorem 5.1 of \cite{ChaVolChe2016} and reproduced here for the sake of completeness. The bound in \eqref{eq:re3bias} is new and crucial for the ABR in Theorem~\ref{th:aggbah} in the main manuscript. See Section \ref{sec:proofbr} in the supplementary material for a proof.

The next result provides a refined version of the previous theorem for local basis functions that satisfy Condition \hyperref[L]{(L)}. 

\begin{theo} \label{th:bahsimple_bspl} Suppose Conditions \hyperref[A1]{(A1)}-\hyperref[A3]{(A3)} and \hyperref[L]{(L)} hold. Assume that additionally $m \xi_m^2 \log n = o(n)$. Assume that the set $\Ic$ consists of at most $L \geq 1$ \textit{consecutive} integers. Then, for $\zg_N$ defined in~\eqref{eq:gamman} 
	we have
\[
\breve\zb(\tau) - \zg_N(\tau) = -  n^{-1/2} \tilde J_m(\zg_N(\tau))^{-1}\Gb_n(\psi(\cdot;\zg_N(\tau),\tau))
+ \sum_{k=1}^4 r_{n,k}^{(L)}(\tau).
\]
where remainder terms $r_{n,j}$'s can be bounded as follows ($C_1,C_4$ are constants that do not depend on $n$):
\begin{align} \label{eq:brn1_a}
\sup_{\ba \in \Sc_\Ic^{m-1}}\sup_\tau |\ba^\top r_{n,1}^{(L)}(\tau)| &\leq C_{1}\frac{\xi_m\log n}{n} 
\quad a.s.,
\\
\sup_{\ba \in \Sc_\Ic^{m-1}} \sup_\tau |\ba^\top r_{n,4}^{(L)}(\tau)| &\leq \frac{1}{n} + C_4 c_m^2(\zg_N) \sup_{\ba \in \Sc_\Ic^{m-1}}\tilde{\mathcal{E}}(\ba,\zg_N) \quad a.s. \label{eq:re4_a}
\end{align}
Moreover, we have for any $\kappa_n \ll n/\xi_m^2$, all sufficiently large $n$, and a constant $C$ independent of $n$
\begin{align}
P\Big( \sup_{\ba \in \Sc_\Ic^{m-1}} \sup_\tau |\ba^\top r_{n,j}^{(L)}(\tau)| \leq C \Re_j^{(L)}(\kappa_n) \Big) \geq 1 - n^2 e^{-\kappa_n}, \quad j=2,3\label{eq:retail_a}
\end{align}
where 
\begin{align} 
\Re_2^{(L)}(\kappa_n) &:=  \sup_{\ba \in \Sc_\Ic^{m-1}}\tilde{\mathcal{E}}(\ba,\zg_N)\Big(\frac{\xi_m (\log n + \kappa_n^{1/2})}{n^{1/2}} + c_m^2(\zg_N)\Big)^2, \label{eq:re2_a}
\\
\Re_3^{(L)}(\kappa_n) &:=  \Big(c_m(\zg_N)\frac{\kappa_n^{1/2} \vee \log n}{n^{1/2}} + \frac{\xi_m^{1/2}(\kappa_n^{1/2} \vee \log n)^{3/2}}{n^{3/4}}\Big),  \label{eq:re3_a}
\end{align}
and $\tilde{\mathcal{E}}(\ba,\zg_N):=\sup_{\tau\in\Tc}\E[\ba^\top \tilde J_m(\zg_N(\tau))^{-1} \ZZ]$. Additionally, 
\begin{equation}  
\sup_\tau \sup_{\ba \in \Sc_\Ic^{m-1}}|\ba^\top r_{n,3}^{(L)}(\tau)| \leq C_{3,1}\xi_m \quad a.s.  \label{eq:re3asbound_a}
\end{equation}
 and for sufficiently large $n$ we have
\begin{equation} \label{eq:re3bias_a}
\sup_{\ba \in \Sc_\Ic^{m-1}}\sup_{\tau} | \ba^\top\E[r_{n,3}^{(L)}(\tau)] | \leq C_{3,2} \sup_{\ba \in \Sc_\Ic^{m-1}}\tilde{\mathcal{E}}(\ba,\zg_N)\Big(c_m^4(\zg_N) + \frac{\xi_m^2 (\log n)^2}{n}\Big). 
\end{equation}
where $C_{3,1},C_{3,2}$ are constants that do not depend on $n$. 
\end{theo}

Similarly to the setting in Theorem~\ref{th:bahsimple}, the statements in \eqref{eq:brn1_a}-\eqref{eq:re3_a} are proved in Theorem 5.2 of \cite{ChaVolChe2016} and reproduced here for the sake of completeness. The bound in \eqref{eq:re3bias_a} is new and crucial for the ABR in Theorem~\ref{th:aggbah_loc}. 

In both proofs, we will repeatedly use the following leave-one-out estimator 
\beq \label{eq:betaminusj}
\breve\zb^{(-j)}(\tau) := \arg\,\min_{\bb\in\R^m} \sum_{i \neq j} \rho_\tau(Y_{i} - \bb^\top\ZZ(X_{i})), \quad \tau \in \Tc, j=1,...,n
\eeq

\subsection{Proof of Theorem \ref{th:bahsimple}}\label{sec:proofbr}

Observe the representation
\begin{multline*}
\breve\zb(\tau) - \zb_N(\tau) = -  n^{-1/2} J_m(\tau)^{-1}\Gb_n(\psi(\cdot;\zb_N(\tau),\tau))
\\
+ r_{n,1}(\tau) + r_{n,2}(\tau) + r_{n,3}(\tau) + r_{n,4}(\tau)
\end{multline*}
where
\begin{eqnarray*}
r_{n,1}(\tau) &:=& \tilde J_m(\zb_N(\tau))^{-1}\Pb_n \psi(\cdot;\breve\zb(\tau),\tau),
\\
r_{n,2}(\tau) &:=& - \tilde J_m(\zb_N(\tau))^{-1}\Big(\mu(\breve\zb(\tau),\tau) - \mu(\zb_N(\tau),\tau)
\\
&& \quad\quad\quad - \tilde J_m(\zb_N(\tau))(\breve\zb(\tau) - \zb_N(\tau))\Big),
\\
r_{n,3}(\tau) &:=& -n^{-1/2}\tilde J_m(\zb_N(\tau))^{-1}\Big(\Gb_n(\psi(\cdot;\breve\zb(\tau),\tau)) - \Gb_n(\psi(\cdot;\zb_N(\tau),\tau))\Big),
\\
r_{n,4}(\tau) &:=& -n^{-1/2}(J_m(\tau)^{-1} - \tilde J_m(\zb_N(\tau))^{-1})\Gb_n(\psi(\cdot;\zb_N(\tau),\tau))
\\
&& - \tilde J_m(\zb_N(\tau))^{-1} \mu(\zb_N(\tau),\tau).
\end{eqnarray*}
The bounds in \eqref{eq:brn1}-\eqref{eq:re4} follow similarly as in the proof of Theorem 5.1 in \cite{ChaVolChe2016}. The bound in \eqref{eq:re3asbound} follows from the definition of $r_{n,3}$. It remains to prove \eqref{eq:re3bias}. Note that the expectation of $r_{n,3}(\tau)$ is not zero in general. Write
\begin{align*}
& r_{n,3}(\tau) 
\\
 =  & - \tilde J_m(\zb_N(\tau))^{-1}\Big[ - n^{-1/2}\Gb_n \psi(\cdot;\zb_N(\tau),\tau)
\\ 
& \quad\quad + n^{-1} \sum_{j=1}^n\Big( \psi(Y_j,\ZZ_j;\breve\zb^{(-j)}(\tau),\tau) - \mu(\breve\zb^{(-j)}(\tau),\tau) \Big)
\\ 
& - n^{-1} \sum_{j=1}^n \Big(\psi(Y_j,\ZZ_j;\breve\zb^{(-j)}(\tau),\tau) - \psi(Y_j,\ZZ_j;\breve\zb(\tau),\tau)
\\
& \quad\quad \quad\quad \quad\quad - \mu(\breve\zb^{(-j)}(\tau),\tau) + \mu(\breve\zb(\tau),\tau)\Big) \Big].
\end{align*}
By definition of $\breve\zb^{(-j)}(\tau)$ in \eqref{eq:betaminusj}, the expectation of the first line in the above representation is zero. To bound the expectation of the second line, observe that for 
\[
\delta_n := \tilde C \xi_m\Big(g_N^2(\zb_N) + \frac{m  \log n}{n}\Big)
\] 
with $\tilde C$ sufficiently large we have
\begin{align*}
& \sup_{\tau\in\Tc}\Big\| \E\Big[n^{-1} \sum_{j=1}^n \psi(Y_j,\ZZ_j;\breve\zb^{(-j)}(\tau),\tau) - \psi(Y_j,\ZZ_j;\breve\zb(\tau),\tau)\Big] \Big\|
\\
= & \sup_{\tau\in\Tc}\sup_{\|\ba\|=1} \Big|\E\Big[n^{-1} \sum_{j=1}^n \ba^\top \ZZ_j\big(\IF\{Y_j \leq \ZZ_j^\top \breve\zb^{(-j)}(\tau)\} - \IF\{Y_j \leq \ZZ_j^\top \breve\zb(\tau)\}\big) \Big]\Big|
\\
\leq & \sup_{\tau\in\Tc}\sup_{\|\ba\|=1} n^{-1} \sum_{j=1}^n \E\Big[ |\ba^\top \ZZ_j| \IF\big\{ |Y_j - \ZZ_j^\top \breve\zb^{(-j)}(\tau) |\leq \xi_m\delta_n \big\} 
\\
& \quad\quad\quad\quad\quad\quad\quad\quad\quad\quad \times \IF\big\{\sup_j \big|\ZZ_j^\top \breve\zb(\tau) - \ZZ_j^\top \breve\zb^{(-j)}(\tau)\big| \leq \xi_m\delta_n\big\} \Big]
\\
& + \xi_m \sup_{\tau\in\Tc}P\Big(\sup_j \big|\ZZ_j^\top \breve\zb(\tau) - \ZZ_j^\top \breve\zb^{(-j)}(\tau)\big| > \xi_m\delta_n\Big)
\\
\leq & \sup_{\tau\in\Tc}\sup_{\|\ba\|=1} n^{-1} \sum_{j=1}^n \E\Big[ |\ba^\top \ZZ_j| \big|F_{Y|X}(\ZZ_j^\top \breve\zb^{(-j)}(\tau) + \xi_m\delta_n |X_j)
\\
& \quad\quad\quad\quad \quad\quad\quad\quad\quad\quad\quad\quad- F_{Y|X}(\ZZ_j^\top \breve\zb^{(-j)}(\tau) - \xi_m\delta_n |X_j)\big| \Big]
\\
& + \xi_m P\Big(\sup_{\tau\in\Tc}\sup_j \big|\ZZ_j^\top \breve\zb(\tau) - \ZZ_j^\top \breve\zb^{(-j)}(\tau)\big| > \xi_m\delta_n\Big)
\\
\leq & \xi_m P\Big(\sup_{\tau\in\Tc}\sup_j \big|\ZZ_j^\top \breve\zb(\tau) - \ZZ_j^\top \breve\zb^{(-j)}(\tau)\big| > \xi_m\delta_n\Big)
\\
& + 2 \xi_m \delta_n \bar f \sup_{\|\ba\|=1} n^{-1} \sum_{j=1}^n \E\big[ |\ba^\top \ZZ_j| \big]
\\
\leq & \ 2 \xi_m \delta_n \bar f \lambda_{\max}(\E[\ZZ\ZZ^\top])^{1/2}
 + \xi_m P\Big(\sup_{\tau\in\Tc}\sup_j \big|\ZZ_j^\top \breve\zb(\tau) - \ZZ_j^\top \breve\zb^{(-j)}(\tau)\big| > \xi_m\delta_n\Big)
\\
\leq & C_{3,2} \xi_m^2\Big(g_N^2(\zb_N) + \frac{m \log n}{n}\Big) 
\end{align*}
for all sufficiently large $n$ and a constant $C_{3,2}$ independent of $\tau,n$. Here, the last line follows from Lemma~\ref{lem:leaveout} and the definition of $\delta_n$ provided that $\tilde C$ in the definition of $\delta_n$ is sufficiently large. Thus it remains to bound the expectation of $\mu(\breve\zb^{(-j)}(\tau),\tau) - \mu(\breve\zb(\tau),\tau)$. A Taylor expansion shows that
\begin{align*}
\sup_\tau \|\mu(\bb_1,\tau) - \mu(\bb_2,\tau)\| & \leq \sup_{\|\ba\|=1}\E\big[ |\ba^\top \ZZ| \IF\big\{ |Y - \ZZ^\top \bb_1 |\leq |\ZZ^\top(\bb_1-\bb_2)| \big\} \big]
\\
& \leq 2\bar f \sup_{\|\ba\|=1}\E\big[ |\ba^\top \ZZ| |\ZZ^\top(\bb_1-\bb_2)|\big]\\
&\leq 2\bar f \lambda_{\max}(\E[\ZZ\ZZ^\top])\|\bb_1 - \bb_2\|,
\end{align*}
where the last inequality follows by Cauchy-Schwarz inequality. On the other hand $\|\mu(\bb_1,\tau) - \mu(\bb_2,\tau)\| \leq 2\xi_m$. Thus we have for any $\delta_n > 0$,
\begin{multline*}
\sup_{\tau,j} \Big\|\E\big[\mu(\breve\zb^{(-j)}(\tau),\tau) - \mu(\breve\zb(\tau),\tau)\big]\Big\|
\\ 
\leq 2\bar f \lambda_{\max}(\E[\ZZ\ZZ^\top])\delta_n + 2\xi_m P\Big(\sup_{\tau,j} \|\breve\zb^{(-j)}(\tau)-\breve\zb(\tau)\| > \delta_n\Big).
\end{multline*}
Choosing $\delta_n$ as above completes the proof.
\hfill $\qed$

\begin{lemma} \label{lem:leaveout}

Under the assumptions of Theorem \ref{th:bahsimple} we have for any $\kappa_n \ll n/\xi_m^2$, all sufficiently large $n$, and a constant $C$ independent of $n$ we have for $\breve\zb^{(-j)}$ defined in~\eqref{eq:betaminusj}
\[
P\Big( \max_{j=1,...,n}\sup_{\tau} \|\breve\zb^{(-j)}(\tau) - \breve\zb(\tau)\| \geq C\Big(g_N^2(\zb_N)\xi_m + \frac{\xi_m\kappa_n}{n} + \frac{m \xi_m \log n}{n}\Big) \Big) \leq 2ne^{-\kappa_n}. 
\]
\end{lemma}
\noindent
\textbf{Proof of Lemma \ref{lem:leaveout}} \\
Let $\Pb_n^{(-j)}(f) := (n-1)^{-1}\sum_{i \neq j} f(X_i,Y_i)$, $\Gb_n^{(-j)}(f) := \sqrt{n-1}(\Pb_n^{(-j)}(f) - \E f(X_1,Y_1))$. Similarly to the decomposition considered previously \eqref{eq:bahadur_bn}, we have
\begin{equation} \label{eq:bahadurohne}
\breve\zb^{(-j)}(\tau) - \zb_N(\tau) = -  (n-1)^{-1/2} J_m(\tau)^{-1}\Gb_n^{(-j)}(\psi(\cdot;\zb_N(\tau),\tau))
+ \sum_{k=1}^4 r_{n,k}^{(-j)}(\tau)
\end{equation}
where
\begin{eqnarray*}
r_{n,1}^{(-j)}(\tau) &:=& \tilde J_m(\zb_N(\tau))^{-1}\Pb_n^{(-j)} \psi(\cdot;\breve\zb^{(-j)}(\tau),\tau),
\\
r_{n,2}^{(-j)}(\tau) &:=& - \tilde J_m(\zb_N(\tau))^{-1}\Big(\mu(\breve\zb^{(-j)}(\tau),\tau) - \mu(\zb_N(\tau),\tau)
\\
&& \quad\quad\quad\quad\quad\quad\quad\quad - \tilde J_m(\zb_N(\tau))(\breve\zb^{(-j)}(\tau) - \zb_N(\tau))\Big),
\\
r_{n,3}^{(-j)}(\tau) &:=& \frac{-\tilde J_m(\zb_N(\tau))^{-1}}{\sqrt{n-1}}\Big(\Gb_n^{(-j)}(\psi(\cdot;\breve\zb^{(-j)}(\tau),\tau)) - \Gb_n^{(-j)}(\psi(\cdot;\zb_N(\tau),\tau))\Big),
\\
r_{n,4}^{(-j)}(\tau) &:=& -(n-1)^{-1/2}(J_m(\tau)^{-1} - \tilde J_m(\zb_N(\tau))^{-1})\Gb_n^{(-j)}(\psi(\cdot;\zb_N(\tau),\tau))
\\
&& - \tilde J_m(\zb_N(\tau))^{-1} \mu(\zb_N(\tau),\tau).
\end{eqnarray*}
By similar arguments as in the proof for $r_{n,2}$ in Theorem 5.1 in \cite{ChaVolChe2016} there exists 
$C$ independent of $n$ such that $P(\Omega_{1,n}^{(j)}(C)) \leq e^{-\kappa_n}, j=1,...,n$ where we defined the event
\begin{align}
\Omega_{1,n}^{(j)}(C) := \Big\{ \sup_{\tau} \|\breve\zb^{(-j)}(\tau) - \zb_N(\tau)\| \geq C \Big( g_N(\zb_N) + \frac{\kappa_n^{1/2}}{n^{1/2}} + \Big(\frac{m \log n}{n}\Big)^{1/2}\Big) \Big\}.\label{eq:omega1}
\end{align}
Combining \eqref{eq:bahadur_bn} and \eqref{eq:bahadurohne} we find that, almost surely,
\[
\sup_j\sup_\tau \Big\| \breve\zb^{(-j)}(\tau) - \breve\zb(\tau) - \sum_{k=1}^4 \Big(r_{n,k}^{(-j)}(\tau) - r_{n,k}(\tau) \Big) \Big\| \lesssim \frac{\xi_m}{n}. 
\]
As in the proof of Lemma 34 on page 106 in \cite{bechfe2011}, arguments from linear optimization lead to the bound
\[
\sup_j \sup_\tau \|r_{n,1}^{(-j)}(\tau) - r_{n,1}(\tau)\| \leq  \frac{2}{\inf_\tau \lambda_{min}(\tilde J_m(\zb_N(\tau)))}\frac{m\xi_m}{n-1} \quad a.s.
\]
Direct calculations show that
\begin{eqnarray*}
\sup_j \sup_\tau \|r_{n,4}^{(-j)}(\tau) - r_{n,4}(\tau)\| \leq C_1 \frac{\xi_m}{n},
\end{eqnarray*}
and from Lemma~\ref{lem:taylormu} we obtain for a constant $C$ independent of $j,\tau,n$
\[
\|r_{n,2}^{(-j)}(\tau) - r_{n,2}(\tau)\| \leq C\xi_m \Big(\|\breve\zb^{(-j)}(\tau) - \zb_N(\tau)\|^2 + \|\breve\zb(\tau) - \zb_N(\tau)\|^2\Big).
\]
The probability of the event
\begin{multline*}
\Omega_{1,n}(C) := \Big\{ \sup_{\tau,j} \Big(\|\breve\zb^{(-j)}(\tau) - \zb_N(\tau)\|^2 + \|\breve\zb(\tau) - \zb_N(\tau)\|^2\Big)
\\
\geq C\Big( g_N^2(\zb_N) + \frac{\kappa_n}{n} + \frac{m \log n}{n}\Big)\Big\}
\end{multline*}
can be bounded by $(n+1)e^{-\kappa_n}$ if $C$ is chosen suitably, this follows from the bound $P(\Omega_{1,n}^{(j)}(C)) \leq e^{-\kappa_n}$ in \eqref{eq:omega1}. Finally, define
\begin{multline}
\label{eq:defgc2}
\Gc_2(\delta) := \Big\{ (z,y) \mapsto \ba^\top z(\IF\{y \leq z^\top \bb_1\} - \IF\{y \leq z^\top \bb_2\})\IF\{\|z\| \leq \xi_m\} \Big|
\\
\bb_1,\bb_2 \in \R^m, \|\bb_1 - \bb_2\| \leq \delta, \ba\in\Sc^{m-1}\Big\}.
\end{multline}
Letting $\delta_n := \sup_j\sup_\tau\|\breve\zb(\tau)^{(-j)} - \breve\zb(\tau)\|$, we have 
\[
\sup_j \sup_\tau \|r_{n,3}^{(-j)}(\tau) - r_{n,3}(\tau)\| \leq C_3 \Big(\frac{\xi_m}{n} + \|\Pb_n - P\|_{\Gc_2(\delta_n)} \Big).
\]
Summarizing the bounds obtained so far we find for a constant $C_4$ independent of $n$
\[
\delta_n \leq C_4\Big(\frac{m\xi_m}{n} + \|\Pb_n - P\|_{\Gc_2(\delta_n)} + \sup_{\tau,j}\xi_m \big(\|\breve\zb^{(-j)}(\tau) - \zb_N(\tau)\|^2 + \|\breve\zb(\tau) - \zb_N(\tau)\|^2\big) \Big).
\]
Consider the event 
\[
\Omega_{2,n}(C) := \Big\{ \sup_{0 \leq t\leq 1} \frac{\|\Pb_n - P\|_{\Gc_2(t)}}{\chi_n(t,\kappa_n)} \geq C \Big\},
\]
where 
$$
\chi_n(t,\kappa_n) = \xi_m^{1/2}t^{1/2} \Big(\frac{m}{n} \log (\xi_m \vee n)\Big)^{1/2} + \frac{m\xi_m}{n} \log (\xi_m \vee n) + \bigg(\frac{\xi_m t \kappa_n}{n}\bigg)^{1/2} + \frac{\xi_m \kappa_n}{n}
$$
defined as \eqref{eq:chi_n} in Lemma \ref{lem:gc}. It follows from \eqref{eq:bgn23} that $P(\Omega_{2,n}(C)) \leq e^{-\kappa_n} \log_2 n$ if we choose $C$ suitably. Thus, $\delta_n$ satisfies the inequality 
\begin{align*}
\delta_n  \leq & \ C_8 \xi_m \Big( g_N^2(\zb_N) + \frac{\kappa_n}{n} + \frac{m \log n}{n}\Big) + C_8\chi_n(\delta_n,\kappa_n) 
\\
\leq & \ C_8 \xi_m \Big( g_N^2(\zb_N) + \frac{\kappa_n}{n} + \frac{m \log n}{n}\Big) + C_8\frac{m\xi_m}{n} \log (\xi_m \vee n) + C_8 \frac{\xi_m \kappa_n}{n}
\\
& \ + C_9 \delta_n^{1/2}\Big(\frac{\xi_m m}{n} \log (\xi_m \vee n) + \frac{\xi_m \kappa_n }{n}\Big)^{1/2}
\\
:=& \ a_n + \delta_n^{1/2}b_n
\end{align*}
with probability at least $1 - 2ne^{-\kappa_n}$ (note that $(n+1)e^{-\kappa_n}+\log_2 n e^{-\kappa_n} \leq 2n e^{-\kappa_n}$).
A simple calculation shows that the inequality $0 \leq \delta_n \leq a_n +b_n\delta_n^{1/2}$ implies $0 \leq \delta_n \leq 4 \max(a_n,b_n^2)$. Straightforward calculations show that under the assumption of the theorem we have
\[
a_n + b_n^2 \lesssim g_N^2(\zb_N)\xi_m + \frac{\xi_m\kappa_n}{n} + \frac{m \xi_m \log n}{n}
\]
provided that $\kappa_n \to \infty$. Thus we have for some constant $C$ independent of $n$
\[
P\Big(\delta_n \geq C\Big(g_N^2(\zb_N)\xi_m + \frac{\xi_m\kappa_n}{n} + \frac{m \xi_m \log n}{n}\Big) \Big) \leq 2ne^{-\kappa_n}. 
\]
This completes the proof. \hfill $\qed$

%

\begin{lemma} \label{lem:gc}
Consider the class of functions $\Gc_2(\delta_n)$ defined in~\eqref{eq:defgc2}. Under assumptions \hyperref[A1]{(A1)}-\hyperref[A3]{(A3)} and $\xi_m \delta_n \gg n^{-1}, \xi_m = O(n^b)$ for some $b<\infty$, we have for some constant $C$ independent of $n$, sufficiently large $n$ and arbitrary $\kappa_n>0$,
\begin{equation} \label{eq:bgn22}
P\Big( \|\Pb_n-P\|_{\Gc_2(\delta_n)} \geq C \chi_n(\delta_n,\kappa_n) \Big) \leq e^{-\kappa_n}
\end{equation}
and for $\kappa_n \geq 1$
\begin{equation} \label{eq:bgn23}
P\Big( \sup_{0 \leq t \leq 1} \frac{\|\Pb_n-P\|_{\Gc_2(t)}}{\chi_n(t,\kappa_n)} \geq 2 C  \Big) \leq  e^{-\kappa_n} \log_2 n
\end{equation}
where
\begin{align}
\chi_n(t,\kappa_n) := \Big(\frac{m\xi_m}{n} \log (\xi_m \vee n)t\Big)^{1/2} + \frac{m\xi_m}{n} \log (\xi_m \vee n) + \Big(\frac{\xi_m \kappa_n t}{n}\Big)^{1/2} +\frac{\xi_m \kappa_n}{n}.\label{eq:chi_n}
\end{align}
\end{lemma}

\noindent \textbf{Proof of Lemma \ref{lem:gc}.} 
For a proof of \eqref{eq:bgn22} see Lemma C.3 in the Appendix of \cite{ChaVolChe2016}. For a proof of \eqref{eq:bgn23}, observe that for $\kappa_n \geq 1$ we have $\chi_n(t/2,\kappa_n) \geq \chi_n(t,\kappa_n)/2, \chi_n(0,\kappa_n) \geq \chi_n(n^{-1},\kappa_n)/2$, and thus
\begin{align*}
& \sup_{0 \leq t \leq 1} \frac{\|\Pb_n-P\|_{\Gc_2(t)}}{\chi_n(t,\kappa_n)}
\\
& \leq \Big(\sup_{0 \leq t \leq n^{-1}}\frac{\|\Pb_n-P\|_{\Gc_2(t)}}{\chi_n(t,\kappa_n)}\Big) \vee \Big(\max_{k: n^{-1} \leq 2^{-k} \leq 1} \sup_{2^{-k-1} \leq t \leq 2^{-k}}\frac{\|\Pb_n-P\|_{\Gc_2(t)}}{\chi_n(t,\kappa_n)} \Big)
\\
&\leq  \Big(\max_{k: n^{-1} \leq 2^{-k} \leq 1} \frac{ 2\|\Pb_n-P\|_{\Gc_2(2^{-k})}}{\chi_n(2^{-k},\kappa_n)}\Big) \vee \Big(\frac{ 2\|\Pb_n-P\|_{\Gc_2(n^{-1})}}{\chi_n(n^{-1},\kappa_n)}\Big)
\end{align*}
Now the set $\{k: n^{-1} \leq 2^{-k} \leq 1\}$ contains not more than $\log_2 n$ elements, and \eqref{eq:bgn23} thus follows from \eqref{eq:bgn22}.\hfill $\qed$

\begin{lemma} \label{lem:taylormu}
Under assumptions \hyperref[A1]{(A1)}-\hyperref[A3]{(A3)} we have
\begin{eqnarray*}
&&\sup_{\tau \in \Tc} \sup_{\|\bb - \zb_N(\tau)\| \leq \omega} \| \mu(\bb,\tau) - \mu(\zb_N(\tau),\tau) - \tilde J_m(\zb_N(\tau))(\bb - \zb_N(\tau)) \|
\\
&\leq& \lambda_{\max}(\E[\ZZ\ZZ^\top]) \overline{f'} \omega^{2} \xi_m,
\end{eqnarray*}
\end{lemma}

\noindent
\textbf{Proof of Lemma \ref{lem:taylormu}} see the proof of Lemma C.1 in the Appendix of \cite{ChaVolChe2016}. \hfill $\qed$

\begin{lemma} \label{lem:cons1}
Let assumptions \hyperref[A1]{(A1)}-\hyperref[A3]{(A3)} hold. Then, for any $t > 1$
\begin{multline*}
\Big\{ \sup_{\tau \in \Tc} \|\breve\zb(\tau) - \zb_N(\tau)\| \leq \frac{2 t (s_{n,1} + g_N(\zb_N))}{\inf_{\tau\in\Tc}\lambda_{min}(\tilde J_m(\zb_N(\tau)))} \Big\}
\\
\supseteq \Big\{ (s_{n,1} + g_N(\zb_N)) < \frac{\inf_{\tau\in\Tc}\lambda_{min}^2(\tilde J_m(\zb_N(\tau)))}{4 t \xi_m \overline{f'}\lambda_{\max}^2(\E[\ZZ \ZZ^\top])}\Big\}.
\end{multline*}
where $s_{n,1} := \|\Pb_n-P\|_{\Gc_1}$ and
\[
\Gc_1 := \Big\{ (z,y) \mapsto \ba^\top z(\IF\{y \leq z^\top \bb\} - \tau)\IF\{\|z\| \leq \xi_m\} \Big| \tau \in \Tc, \bb \in \R^m, \ba \in \Sc^m \Big\}
\]
\end{lemma}

\noindent
\textbf{Proof of Lemma \ref{lem:cons1}.} See the proof of Lemma C.2 in the Appendix of \cite{ChaVolChe2016}. \hfill $\qed$

\subsection{Proof of Theorem~\ref{th:bahsimple_bspl}}\label{sec:proofbr_bspl}
We begin by introducing some notation and useful preliminary results which will be used throughout this section. For a vector $\bv \in \R^m$ and a set $\Ic \subset \{1,...,m\}$ let $\bv^{(\Ic)} \in \R^m$ denote the vector that has entries $v_i$ for $i \in \Ic$ and zero otherwise. For a vector $\ba \in \R^m$ let $\|\ba\|_0$ denote the number of non-zero entries in this vector, let $k_\ba$ denote the position of the first non-zero entry and define
\begin{align}
	\Ic_\ba(D):=\big\{i\in\{1,...,m\}:|i-k_\ba| \leq \|\ba\|_0+D\big\}.
\end{align}

Under \hyperref[A1]{(A1)}-\hyperref[A3]{(A3)} and \hyperref[L]{(L)} we obtain by similar arguments as in Lemma A.1 in \cite{ChaVolChe2016} 
\begin{equation}\label{eq:locbas2_a}
\|\ba^\top \tilde J_m^{-1}(\zg_N(\tau)) - (\ba^\top \tilde J_m^{-1}(\zg_N(\tau)))^{(\Ic_\ba(D))} \| \leq C \|\ba\|_\infty \|\ba\|_0 \lambda^{D}
\end{equation}
for constants $\lambda\in(0,1), C > 0$ independent of $n, \tau$. 

In the sequel, we will make use of the following decomposition which holds for any $\bb \in \R^m$ such that $\tilde J_m(\bb)$ is invertible
\beq \label{eq:bahadur_gen}
\breve\zb(\tau) - \bb = -n^{-1/2}\tilde J_m^{-1}(\bb)\Gb_n(\psi(\cdot; \bb,\tau))  + \tilde J_m^{-1}(\bb) \sum_{k=1}^4 R_{n,k}(\tau,\bb)
\eeq
where
\begin{eqnarray}
\quad R_{n,1}(\tau,\bb) &:=& \Pb_n \psi(\cdot;\breve\zb(\tau),\tau), \label{eq:Rn1}
\\
\quad R_{n,2}(\tau,\bb) &:=& - \Big(\mu(\breve\zb(\tau),\tau) - \mu(\bb,\tau) - \tilde J_m(\bb)(\breve\zb(\tau) - \bb)\Big), \label{eq:Rn2}
\\
\quad R_{n,3}(\tau,\bb) &:=& -n^{-1/2} \Big(\Gb_n(\psi(\cdot;\breve\zb(\tau),\tau)) - \Gb_n(\psi(\cdot;\bb,\tau))\Big), \label{eq:Rn3} 
\\
\quad R_{n,4}(\tau,\bb) &:=&  - \mu(\bb,\tau). \label{eq:Rn4}
\end{eqnarray}

Let $r_{n,k}^{(L)}(\tau) := \tilde J_m^{-1}(\zg_N(\tau)) R_{n,k}(\tau,\zg_N(\tau)), k=1,...,4$. The bounds in \eqref{eq:brn1_a}-\eqref{eq:re4_a} follow similarly to the bounds (5.6)-(5.9) in Theorem 5.2 in \cite{ChaVolChe2016}. The bound in \eqref{eq:re3asbound_a} follows from the definition of $r_{n,3}$. Thus it remains to establish \eqref{eq:re3bias_a}. Write
\begin{align*}
R_{n,3}(\tau,\zg_N(\tau)) =~&  - \Big[ - n^{-1/2}\Gb_n \psi(\cdot;\zg_N(\tau),\tau)
\\
& \quad\quad\quad + n^{-1} \sum_{j=1}^n\Big( \psi(Y_j,\ZZ_j;\breve\zb^{(-j)}(\tau),\tau) - \mu(\breve\zb^{(-j)}(\tau),\tau) \Big)
\\ 
&  - n^{-1} \sum_{j=1}^n \Big(\psi(Y_j,\ZZ_j;\breve\zb^{(-j)}(\tau),\tau) - \psi(Y_j,\ZZ_j;\breve\zb(\tau),\tau)
\\
& \quad\quad\quad\quad \quad\quad\quad\quad - \mu(\breve\zb^{(-j)}(\tau),\tau) + \mu(\breve\zb(\tau),\tau)\Big) \Big].
\end{align*}
Define
$$
\tilde\delta_n := \tilde C\Big(c_m(\zg_N)^4 + \frac{\xi_m^2(\log n)^2}{n}\Big).
$$
and observe that for any $\ba \in \Sc_\Ic^{m-1}$, $\tau\in\Tc$,
\begin{align*}
& \Big| (\ba^\top \tilde J_m(\zg_N(\tau))^{-1})\E\Big[n^{-1} \sum_{j=1}^n \psi(Y_j,\ZZ_j;\breve\zb^{(-j)}(\tau),\tau) - \psi(Y_j,\ZZ_j;\breve\zb(\tau),\tau)\Big] \Big|
\\
= & \Big|\E\Big[n^{-1} \sum_{j=1}^n (\ba^\top \tilde J_m(\zg_N(\tau))^{-1}) \ZZ_j(\IF\{Y_i \leq \ZZ_j^\top \breve\zb^{(-j)}(\tau)\} - \IF\{Y_i \leq \ZZ_j^\top \breve\zb(\tau)\}) \Big]\Big|
\\
\lesssim &\ n^{-1} \sum_{j=1}^n \E\Big[ |(\ba^\top \tilde J_m(\zg_N(\tau))^{-1}) \ZZ_j| \IF\Big\{ |Y_j - \ZZ_j^\top \breve\zb^{(-j)}(\tau) |\leq \tilde\delta_n \Big\}
\\
& \quad\quad\quad\quad\quad\quad\quad\quad\quad\quad\quad\quad \times \IF\Big\{\sup_j \big|\ZZ_j^\top \breve\zb(\tau) - \ZZ_j^\top \breve\zb^{(-j)}(\tau)\big| \leq \tilde\delta_n\Big\} \Big]
\\
& + \xi_m P\Big(\sup_j \big|\ZZ_j^\top \breve\zb(\tau) - \ZZ_j^\top \breve\zb^{(-j)}(\tau)\big| > \tilde\delta_n\Big)
\\
\leq &\ n^{-1} \sum_{j=1}^n \E\Big[ |(\ba^\top \tilde J_m(\zg_N(\tau))^{-1}) \ZZ_j| \Big(F_{Y|X}(\ZZ_j^\top \breve\zb^{(-j)}(\tau) + \tilde\delta_n |X_j)
\\
& \quad\quad\quad\quad\quad\quad\quad\quad\quad\quad\quad\quad\quad\quad\quad\quad\quad  - F_{Y|X}(\ZZ_j^\top \breve\zb^{(-j)}(\tau) - \tilde\delta_n |X_j)\Big) \Big]
\\
& + \xi_m P\Big(\sup_j \big|\ZZ_j^\top \breve\zb(\tau) - \ZZ_j^\top \breve\zb^{(-j)}(\tau)\big| > \tilde\delta_n\Big)
\\
\leq & \ 2 \tilde\delta_n \bar f n^{-1} \sum_{j=1}^n \E\Big[ |(\ba^\top \tilde J_m(\zg_N(\tau))^{-1}) \ZZ_j| \Big]
 + \xi_m P\Big(\sup_j \big|\ZZ_j^\top \breve\zb(\tau) - \ZZ_j^\top \breve\zb^{(-j)}(\tau)\big| > \tilde\delta_n\Big)
\\
\leq & C_{3,2} \tilde{\mathcal{E}}(\ba,\zg_N) \Big(c_m^4(\zg_N) + \frac{\xi_m^2(\log n)^2}{n}\Big) 
\end{align*}
for all sufficiently large $n$ and a constant $C_{3,2}$ independent of $\tau,n,\ba$. Here, the last line follows from Lemma~\ref{lem:leaveout_bspl}
with $\tilde C$ in the definition of $\tilde\delta_n$ chosen sufficiently large. Thus it remains to bound $(\ba^\top \tilde J_m(\zg_N(\tau))^{-1})\E[\mu(\breve\zb^{(-j)}(\tau),\tau) - \mu(\breve\zb(\tau),\tau)]$. A Taylor expansion shows that for $\bb_1,\bb_2 \in \R^m$
\begin{align*}
&\sup_\tau |(\ba^\top \tilde J_m(\zg_N(\tau))^{-1})\mu(\bb_1,\tau) - \mu(\bb_2,\tau)|
\\
 & \leq \E\Big[ |(\ba^\top \tilde J_m(\zg_N(\tau))^{-1})\ZZ_j| \IF\Big\{ |Y_j - \ZZ_j^\top \bb_1 |\leq |\ZZ_j^\top(\bb_1-\bb_2)| \Big\} \Big]
\\
& \leq 2\bar f \tilde{\mathcal{E}}(\ba,\zg_N) \sup_{x\in\Xc} |\ZZ(x)^\top(\bb_1-\bb_2)|.
\end{align*}
On the other hand $\|\mu(\bb_1,\tau) - \mu(\bb_2,\tau)\| \leq 2\xi_m$. Thus we have for any $\tilde\delta_n > 0, \ba \in \Sc_\Ic^{m-1}$
\begin{multline*}
\sup_{\tau,j} \Big\|(\ba^\top \tilde J_m(\zg_N(\tau))^{-1})\E[\mu(\breve\zb^{(-j)}(\tau),\tau) - \mu(\breve\zb(\tau),\tau)]\Big\|
\\
\lesssim \tilde{\mathcal{E}}(\ba,\zg_N) \tilde\delta_n + \xi_m P\Big(\sup_{\tau,j} \|\ZZ_j^\top \breve\zb(\tau) - \ZZ_j^\top \breve\zb^{(-j)}(\tau)\| > \tilde\delta_n\Big).
\end{multline*}
Choosing $\tilde\delta_n$ as above completes the proof.
\hfill $\qed$

\begin{lemma} \label{lem:leaveout_bspl}
Under the assumptions of Theorem \ref{th:bahsimple_bspl} we have for any $\kappa_n \ll n/\xi_m^2$, all sufficiently large $n$, and a constant $C$ independent of $n$
\begin{align*}
& P\Big(\sup_{j,\tau,x} \|\ZZ(x)^\top \breve\zb^{(-j)}(\tau) - \ZZ(x)^\top\breve\zb(\tau)\| \geq C\Big(c_m^4(\zg_N) + \frac{\xi_m^2 (\kappa_n + (\log n)^2)}{n}\Big) \Big)
\\
\leq &3m(n+1)e^{-\kappa_n}. 
\end{align*}
\end{lemma}

\noindent\textbf{Proof of Lemma \ref{lem:leaveout_bspl}} \\
Recall that for a vector $\ba \in \R^m$ $\|\ba\|_0$ denotes the number of non-zero entries in $\ba$ and $k_\ba$ denotes the position of the first non-zero entry. Consider the following additional notation which is used exclusively in this proof:
\begin{align}
		\Ic(x) &:= \{i \in \{1,...,m\}: |i - k_{\ZZ(x)}| \leq r+ D\log n\},\\
		\Ic'(x) &:= \{i \in \{1,...,m\}: \exists\,j\in\Ic(x) \mbox{ such that }|i - j| \leq r\},
\end{align}
where $r$ is defined in \hyperref[L]{(L)}, and we suppress the dependence of $\Ic(x)$ and $\Ic'(x)$ on $D$ for the sake of a simpler notation.

Let $\Pb_n^{(-j)} f := (n-1)^{-1}\sum_{i \neq j} f(X_i,Y_i)$, $\Gb_n^{(-j)} f := \sqrt{n-1}(\Pb_n^{(-j)} f - \E f(X_1,Y_1))$. We have
\begin{multline} \label{eq:bahadurohne_bspl}
\ZZ(x)^\top (\breve\zb^{(-j)}(\tau) - \zg_N(\tau))
\\
= -  (n-1)^{-1/2} (\ZZ(x)^\top \tilde J_m(\zg_N(\tau))^{-1})^{(\Ic(x))}\Gb_n^{(-j)}(\psi(\cdot;\zg_N(\tau),\tau))
+ \sum_{k=1}^5 r_{n,k}^{(-j)}(\tau,x)
\end{multline}
where $r_{n,k}^{(-j)}(\tau,x) := (\ZZ(x)^\top \tilde J_m(\zg_N(\tau))^{-1})^{(\Ic(x))} R_{n,k}^{(-j)}(\tau,\zg_N(\tau)), k=1,...,4$, $R_{n,k}^{(-j)}$ is defined exactly as $R_{n,k}$ in \eqref{eq:Rn1}-\eqref{eq:Rn4} with $n,\breve\zb, \Pb_n,\Gb_n$ replaced by $n-1,\breve\zb^{(-j)}, \Pb_n^{(-j)},\Gb_n^{(-j)}$, respectively, and
\begin{align*}
& r_{n,5}^{(-j)}(\tau,x) 
\\
:= & - \Big[(\ZZ(x)^\top \tilde J_m(\zg_N(\tau))^{-1})^{(\Ic(x))} - (\ZZ(x)^\top \tilde J_m(\zg_N(\tau))^{-1})\Big] 
\\
& \times  - \Big[ - (n-1)^{-1/2}\Gb_n^{(-j)}(\psi(\cdot;\zg_N(\tau),\tau)) + \sum_{k=1}^4 R_{n,k}^{(-j)}(\tau,\zg_N(\tau))\Big]\\
= &\Big[(\ZZ(x)^\top \tilde J_m(\zg_N(\tau))^{-1})^{(\Ic(x))} - (\ZZ(x)^\top \tilde J_m(\zg_N(\tau))^{-1})\Big](\breve\zb^{(-j)}(\tau)-\zg_N(\tau)),
\end{align*}
where the last expression follows from \eqref{eq:bahadur_gen}.

Similarly we have
\begin{multline*}
\ZZ(x)^\top (\breve\zb(\tau) - \zg_N(\tau))
\\
= -n^{-1/2} (\ZZ(x)^\top \tilde J_m(\zg_N(\tau))^{-1})^{(\Ic(x))}\Gb_n(\psi(\cdot;\zg_N(\tau),\tau))
+ \sum_{k=1}^5 r_{n,k}(\tau)
\end{multline*}
where
\begin{align*}
& r_{n,5}(\tau,x) 
\\
:= & - \Big[(\ZZ(x)^\top \tilde J_m(\zg_N(\tau))^{-1})^{(\Ic(x))} - (\ZZ(x)^\top \tilde J_m(\zg_N(\tau))^{-1})\Big] 
\\
& \times - \Big[ - n^{-1/2}\Gb_n(\psi(\cdot;\zg_N(\tau),\tau)) + \sum_{k=1}^4 R_{n,k}(\tau,\zg_N(\tau))\Big].\\
=&\Big[(\ZZ(x)^\top \tilde J_m(\zg_N(\tau))^{-1})^{(\Ic(x))} - (\ZZ(x)^\top \tilde J_m(\zg_N(\tau))^{-1})\Big](\breve\zb(\tau)-\zg_N(\tau)),
\end{align*}
where the last expression follows from \eqref{eq:bahadur_gen}.

We obtain the representation
\begin{multline*}
\sup_{j,\tau,x} \Big| \ZZ(x)^\top\breve\zb^{(-j)}(\tau) - \ZZ(x)^\top\breve\zb(\tau) - (r_{n,5}^{(-j)}(\tau,x) - r_{n,5}(\tau,x)) 
\\
-  \sum_{k=1}^{3} \Big(r_{n,k}^{(-j)}(\tau,x) - r_{n,k}(\tau,x) \Big) \Big| \lesssim \frac{\xi_m^2}{n} \quad a.s. 
\end{multline*}
Observe that $r_{n,4}^{(-j)}(\tau) - r_{n,4}(\tau)=0$ for all $\tau$ because $r_{n,4}$ does not depend on $\breve\zb(\tau)$ or $\breve\zb^{(-j)}(\tau)$.

The results in Lemma~\ref{lem:subgr} show that
\[
\sup_{j,\tau,x} \|r_{n,1}^{(-j)}(\tau,x) - r_{n,1}(\tau,x)\| \lesssim  \frac{\xi_m^2 \log n}{n} \quad a.s.
\]
From Lemma~\ref{lem:taylormu_bspl} and condition (L) we obtain for a constant $C$ independent of $j,\tau,n,x$
\begin{multline*}
\|r_{n,2}^{(-j)}(\tau) - r_{n,2}(\tau)\|
\\
\leq C\Big(|\ZZ(x)^\top\breve\zb^{(-j)}(\tau) - \ZZ(x)^\top\zg_N(\tau)|^2 + |\ZZ(x)^\top\breve\zb(\tau) - \ZZ(x)^\top\zg_N(\tau)|^2\Big).
\end{multline*}
Let $w_n := \Big(c_m^4(\zg_N) + \frac{\xi_m^2(\log n)^2 + \xi_m^2\kappa_n}{n} \Big)$ and define the event 
\[
\tilde\Omega_{1,n}(C) := \Big\{ \sup_{\tau,j,x} \Big(|\ZZ(x)^\top (\widehat\zb^{(-j)}(\tau) - \zg_N(\tau))|^2 + |\ZZ(x)^\top(\breve\zb(\tau) - \zg_N(\tau))|^2\Big) \geq Cw_n \Big\}.
\]
By an application of Lemma~\ref{lem:bsplrate}, $P(\tilde\Omega_{1,n}(C))$ can be bounded by $ m(n+1)e^{-\kappa_n}$ if $C$ is chosen suitably. 

Next, let $\tilde\delta_n := \sup_{j,\tau,x}\|\ZZ(x)^\top\breve\zb(\tau)^{(-j)} - \ZZ(x)^\top\breve\zb(\tau)\|$. We obtain for the class of functions $\Gc_2$ defined in \eqref{eq:defgc2_bspl} and using (L)
\[
\sup_{j,\tau,x} |r_{n,3}^{(-j)}(\tau,x) - r_{n,3}(\tau,x)| \leq C_3 \Big(\frac{\xi_m^2}{n} + \xi_m \sup_x \|\Pb_n - P\|_{\Gc_2(\tilde\delta_n,\Ic(x),\Ic'(x))} \Big).
\]
Consider the event [here $\tilde\chi_n(t,\Ic_1,\Ic_j,\kappa_n)$ is defined as \eqref{eq:chi2} in Lemma \ref{lem:rn3_bspl}]
\[
\tilde\Omega_{2,n}(C) := \Big\{ \sup_x \sup_{0 \leq t\leq 1} \frac{\|\Pb_n - P\|_{\Gc_2(t,\Ic(x),\Ic'(x))}}{\tilde \chi_n(t,\Ic(x),\Ic'(x),\kappa_n)} \geq C \Big\}.
\]
Since for $x \in \Xc$ the pair of sets $(\Ic(x),\Ic'(x))$ takes at most $m$ distinct values, it follows from \eqref{eq:bgn23_bspl} that $P(\tilde\Omega_{2,n}(C)) \leq m e^{-\kappa_n} \log_2 n$ if we choose $C$ suitably. 

Finally, apply \eqref{eq:locbas2_a} to find that
\[
\sup_{j,\tau,x} |r_{n,5}(\tau,x)| + |r_{n,5}^{(-j)}(\tau,x)| \lesssim \xi_m \lambda^{D\log n} \Big(\|\breve\zb(\tau) - \zg_N(\tau)\| + \|\breve\zb^{(-j)}(\tau) - \zg_N(\tau)\|\Big),
\]
almost surely. Applying Theorem \ref{th:bahsimple} with $\zb_N = \zg_N$ allows to bound the probability of the event $\Omega_{3,n}(C) := \{\|\breve\zb(\tau) - \zg_N(\tau)\| + \|\breve\zb^{(-j)}(\tau) - \zg_N(\tau)\| > C\}$ by $e^{-\kappa_n}$, and by choosing $D$ sufficiently large it follows that $\sup_{j,\tau,x} |r_{n,5}(\tau,x)| + |r_{n,5}^{(-j)}(\tau,x)| \leq n^{-1}$ with probability at least $1 - e^{-\kappa_n}$. 

Summarizing the bounds obtained so far, there is a constant $C_j$ independent of $n$ such that $P(\tilde\Omega_{1,n}(C_2)\cap \tilde\Omega_{2,n}(C_2)\cap \Omega_{3,n}(C_2)) \leq m(n+1)e^{-\kappa_n}+m \log_2 n e^{-\kappa_n} +e^{-\kappa_n}\leq 3m(n+1)e^{-\kappa_n}$. On $\tilde\Omega_{1,n}(C_2)\cap \tilde\Omega_{2,n}(C_2)\cap \Omega_{3,n}(C_2)$,
\begin{align*}
\delta_n  \leq & \ C_2 \Big( c_m^4(\zg_N) + \frac{\xi_m^2\kappa_n}{n} + \frac{\xi_m^2 (\log n)^2}{n}\Big) + C_2 \sup_x \xi_m\chi_n(\delta_n,\Ic(x),\Ic'(x),\kappa_n) 
\\
\leq & \ C_2 \Big( c_m^4(\zg_N) + \frac{\xi_m^2\kappa_n}{n} + \frac{\xi_m^2 (\log n)^2}{n}\Big)  + C_2\Big(\frac{\xi_m^2\log^2 (\xi_m \vee n)}{n}  + \frac{\xi_m^2 \kappa_n }{n}\Big)
\\
& \ + C_2 \delta_n^{1/2}\Big(\frac{\xi_m^2\log^2 (\xi_m \vee n)}{n}  + \frac{\xi_m^2 \kappa_n }{n}\Big)^{1/2}
\\
:=& \ a_n + \delta_n^{1/2}b_n.
\end{align*}
A simple calculation shows that the inequality $0 \leq \delta_n \leq a_n +b_n\delta_n^{1/2}$ implies $0 \leq \delta_n \leq 4 \max(a_n,b_n^2)$. Straightforward calculations show that under the assumption of the theorem we have
\[
a_n + b_n^2 \lesssim c_m^4(\zg_N) + \frac{\xi_m^2\log^2 (\xi_m \vee n)}{n}  + \frac{\xi_m^2 \kappa_n }{n}.
\]
Thus we have for some constant $C_3$ independent of $n$
\[
P\Big(\delta_n \geq C_3\Big(c_m^4(\zg_N) + \frac{\xi_m^2\log^2 (\xi_m \vee n)}{n}  + \frac{\xi_m^2 \kappa_n }{n}\Big) \Big) \leq 3m(n+1)e^{-\kappa_n}.
\]
This completes the proof. \hfill $\qed$

\begin{lemma}\label{lem:bsplrate}
Under the assumptions of Theorem~\ref{th:bahsimple_bspl} we have for sufficiently large $n$ and any $\kappa_n \ll n/\xi_m^2 $
\[
P\Big( \sup_{\tau,x} |\ZZ(x)^\top(\breve\zb(\tau) - \zg_N(\tau))| \geq C\Big(\frac{\xi_m \log n + \xi_m\kappa_n^{1/2}}{n^{1/2}} + c_m^2(\zg_N)\Big)\Big) \leq (m+1) e^{-\kappa_n}.
\] 
where the constant $C$ does not depend on $n$. 

\end{lemma}

\noindent\textbf{Proof of Lemma~\ref{lem:bsplrate}.} See the proof for Lemma C.4 in the Appendix of \cite{ChaVolChe2016}.
\hfill $\qed$

\begin{lemma} \label{lem:rn3_bspl}
Let $\Zc := \{\ZZ(x)|x\in\Xc\}$ where $\Xc$ is the support of $X$. For arbitrary index sets $\Ic,\Ic'\subset \{1,...,m\}$, define the classes of functions
\begin{multline}\label{eq:defgc1_bspl}
\Gc_1(\Ic,\Ic') := \Big\{ (z,y) \mapsto \ba^\top z^{(\Ic)}(\IF\{y \leq z^\top \bb^{(\Ic')}\} - \tau)\IF\{\|z\| \leq \xi_m\}
\\
 \Big| \tau \in \Tc, \bb \in \R^m, \ba \in \Sc^{m-1} \Big\},
\end{multline}
\begin{multline}
\label{eq:defgc2_bspl}
\Gc_2(\delta,\Ic,\Ic') := \Big\{ (z,y) \mapsto \ba^\top z^{(\Ic)}(\IF\{y \leq z^\top \bb_1^{(\Ic')}\} - \IF\{y \leq z^\top \bb_j^{(\Ic')}\})\IF\{z \in \Zc\} \Big|
\\
\bb_1,\bb_j \in \R^m, \sup_{v\in\Zc}\|v^\top \bb_1 - v^\top \bb_j\| \leq \delta, \ba \in \Sc^{m-1} \Big\}.
\end{multline}
Under assumptions \hyperref[A1]{(A1)}-\hyperref[A3]{(A3)} we have
\begin{multline} \label{eq:bgn21_bspl}
P\Big( \|\Pb_n-P\|_{\Gc_1} \geq C\Big[\Big(\frac{\max(|\Ic|,|\Ic'|)}{n} \log \xi_m\Big)^{1/2}
\\
+ \frac{\max(|\Ic|,|\Ic'|) \xi_m}{n} \log \xi_m  + \frac{\kappa_n^{1/2}}{n^{1/2}} + \frac{\xi_m\kappa_n}{n} \Big] \Big) \leq e^{-\kappa_n}
\end{multline}
and for $\xi_m \delta_n \gg n^{-1}$ we have for sufficiently large $n$ and arbitrary $\kappa_n>0$
\begin{equation} \label{eq:bgn22_bspl}
P\Big( \|\Pb_n-P\|_{\Gc_2(\delta,\Ic,\Ic')} \geq C \tilde\chi_n(\delta_n,\Ic,\Ic',\kappa_n) \Big) \leq e^{-\kappa_n}
\end{equation}
and for $\kappa_n \geq 1$
\begin{equation} \label{eq:bgn23_bspl}
P\Big( \sup_{0 \leq t \leq 1} \frac{\|\Pb_n-P\|_{\Gc_2(t,\Ic,\Ic')}}{\tilde\chi_n(t,\Ic,\Ic',\kappa_n)} \geq 2 C  \Big) \leq  e^{-\kappa_n} \log_2 n
\end{equation}
where
\begin{multline}
\tilde\chi_n(t,\Ic,\Ic',\kappa_n) := t^{1/2} \Big(\frac{\max(|\Ic|,|\Ic'|)}{n} \log (\xi_m \vee n)\Big)^{1/2} 
\\
+ \frac{\max(|\Ic|,|\Ic'|)\xi_m}{n} \log (\xi_m \vee n) + \bigg(\frac{t \kappa_n}{n}\bigg)^{1/2} + \frac{\xi_m\kappa_n}{n} .\label{eq:chi2}
\end{multline}
\end{lemma}

\noindent\textbf{Proof of Lemma \ref{lem:rn3_bspl}. } See the proof for Lemma C.5 in the Appendix of \cite{ChaVolChe2016}. The proof of \eqref{eq:bgn23_bspl} is similar to the proof for \eqref{eq:bgn23}. \hfill $\qed$

\begin{lemma} \label{lem:taylormu_bspl}
Under assumptions \hyperref[A1]{(A1)}-\hyperref[A3]{(A3)} we have for any $\ba \in \R^m$
\begin{eqnarray*}
&&\Big|\ba^\top\tilde J_m(\zg_N(\tau))^{-1}\mu(\bb,\tau) - \ba^\top\tilde J_m(\zg_N(\tau))^{-1}\mu(\zg_N(\tau),\tau) - \ba^\top(\bb - \zg_N(\tau)) \Big|
\\
&\leq& \overline{f'} \sup_{x} |\ZZ(x)^\top \bb - \ZZ(x)^\top \zg_N(\tau)|^2 \E[|\ba^\top\tilde J_m(\zg_N(\tau))^{-1}\ZZ|].
\end{eqnarray*}
\end{lemma}

\noindent\textbf{Proof of Lemma \ref{lem:taylormu_bspl}} See the proof for Lemma C.6 in the Appendix of \cite{ChaVolChe2016}. \hfill $\qed$

\begin{lemma} \label{lem:subgr}
Under assumptions \hyperref[A1]{(A1)}-\hyperref[A3]{(A3)} and (L) we have for any $\ba \in \R^m$ having zero entries everywhere except at $L$ consecutive positions:
\begin{eqnarray*}
|\ba^\top \Pb_n \psi(\cdot;\breve\zb(\tau),\tau) | \leq \frac{(L +2r)\|\ba\|\xi_m}{n}.
\end{eqnarray*}
\end{lemma}
\noindent\textbf{Proof of Lemma \ref{lem:subgr}.} See the proof for Lemma C.7 in the Appendix of \cite{ChaVolChe2016}. \hfill $\qed$


\setcounter{subsection}{0}
\setcounter{equation}{0}
\setcounter{theo}{0}

\section{Proof of \eqref{`EQ:TILBET1}-\eqref{`EQ:TILBET2} and \eqref{`EQ:HATBET1}-\eqref{`EQ:HATBET2}} \label{sec:prooftilbeta}

Let $U_i := F_{a,b}(Y_i)$. By strict monotonicity of $F_{a,b}$ on its support we have $U_i \sim U[0,1]$. Moreover $Y_{(k)}^{j} = Q_{a,b}(U_{(k)}^{j})$ where $U_{(k)}^{j}$ denotes the $k$th order statistic of the sample $\{U_i\}_{i\in\Ac_j}$ if $\Ac_j$ is non-empty and $U_{(k)}^{j} := 0$ if $\Ac_j$ is empty. By independence of $\{Y_1,...,Y_n\}$ and $\{X_1,...,X_n\}$ we have conditionally on $\{X_1,...,X_n\}$, 
\begin{equation} \label{eq:udistr}
U_{(k)}^{j} \sim \IF\{n_j > 0\} Beta(k, n_j+ 1 - k).
\end{equation}
Throughout this section we use the notation introduced in section~\ref{sec:specconst}.

\subsection{Proof of \eqref{`EQ:TILBET1} and \eqref{`EQ:HATBET1} (bias)}\label{sec:pfbias}

\textbf{Proof of~\eqref{`EQ:TILBET1}}: note that 
\[
m^{1/2}\tilde\beta_j(\tau)=Y_{(\lceil n_j\tau\rceil)}^{j}\IF\{n_j> 0\}=Q_{a,b}(U_{(\lceil n_j\tau\rceil)}^{j})\IF\{n_j> 0\} = Q_{a,b}(U_{(\lceil n_j\tau\rceil)}^{j})
\] 
where $n_j \sim Bin(n,1/m)$ and the last equality uses the definition of $U_{(\lceil n_j\tau\rceil)}^{j}$ and the fact that $Q_{a,b}(0) = 0$. Now we distinguish two cases. 

If $\limsup_{N\to\infty}(n/m)\big|\E[U_{(\lceil n_j\tau\rceil)}^{j}- \tau]\big|>0$, using the definition of $Q_{a,b}$: 
\begin{align}
\big|\E[Q_{a,b}(U_{(\lceil n_j\tau\rceil)}^{j}) - Q_{a,b}(\tau)]\big| =\big|a (\E[(U_{(\lceil n_j\tau\rceil)}^{j})^2 - \tau^2]) + b(\E[U_{(\lceil n_j\tau\rceil)}^{j} - \tau])\big|, \label{eq:sharpStmp1_new}
\end{align}
so \eqref{`EQ:TILBET1} holds for any $b>0$ and $a=0$. 

If $\limsup_{N\to\infty}(n/m)\big|\E[U_{(\lceil n_j\tau\rceil)}^{j}- \tau]\big|=0$ apply \eqref{eq:udistr} and \eqref{eq:bin1} in Lemma \ref{lem:bin} to obtain
\begin{align}
\limsup_{N\to\infty}\frac{n}{m}\Big| \E[Q_{a,b}(U_{(\lceil n_j\tau\rceil)}^{j})] - Q_{a,b}(\tau) \Big|
&\geq \frac{a \tau(1-\tau)}{200} > 0 
\label{eq:tilbet1tmp1}
\end{align}
since $\Tc=[\tau_L,\tau_U]$ with $0<\tau_L<\tau_U<1$. Thus \eqref{`EQ:TILBET1} holds for any $b>0$ and $0<a<a_{max}$. 

\bigskip
\noindent
\textbf{Proof of~\eqref{`EQ:TILBET2}}: note that $Q_{a,b}$ is a polynomial of degree 2 and 
\[
\bigg\|\sum_{k=1}^K A_k(\cdot)Q_{a,b}(\tau_k)-Q_{a,b}(\cdot)\bigg\|_\infty = \big\|\Pi_K Q_{a,b} - Q_{a,b}\big\|_\infty
\]
where $\Pi_K f$ denotes the projection, with respect to the inner product $\langle f,g \rangle := \sum_k f(\tau_k)g(\tau_k)$, of a function $f$ onto the space of functions spanned by $B_1,...,B_q$. Since $B_1,...,B_q$ are splines of degree  $r_\tau\geq\eta_\tau\geq 2$, $Q_{a,b}$ lies in the space spanned by $B_1,...,B_q$ (see p.111 of \cite{schumaker:81}). In other words, $Q_{a,b} \in \Theta_G$ where $\Theta_G$ is defined in the beginning of Section \ref{sec:proof_orgenpr}. Therefore, by similar arguments as in Section \ref{sec:proof_orgenpr}
\begin{align}
\big\|\Pi_K Q_{a,b} - Q_{a,b}\big\|_\infty \lesssim \inf_{p\in\Theta_G}\|p-Q_{a,b}\|_\infty = 0. 
\label{eq:biaspj}
\end{align}
We have from \eqref{eq:biaspj},
\begin{align}
& \E[m^{1/2}\check\beta_j(\tau) - Q_{a,b}(\tau)] \notag
\\
&= \sum_{k=1}^K A_k(\tau)\E[m^{1/2}\tilde\beta_j(\tau_k) - Q_{a,b}(\tau_k)] 
\notag\\
&= \sum_{k: |\tau_k -\tau|\leq (\log G)^2/G}  A_k(\tau)\E[Q_{a,b}(U_{(\lceil n_j\tau_k\rceil)}^{j}) - Q_{a,b}(\tau_k)] + o(m/n),\label{eq:hatbiasdec_new}
\end{align}
where the $o(m/n)$ is uniform in $\tau\in\Tc$. The first equality above is from \eqref{eq:biaspj}, and the second follows from the fact that $\sup_\tau\big|\E[m^{1/2}\tilde\beta_j(\tau) - Q_{a,b}(\tau)]\big|\leq 2(a+b)$ and 
\begin{multline}
\bigg|\sum_{k: |\tau_k -\tau|> (\log G)^2/G} A_k(\tau)\E\big[m^{1/2}\tilde\beta_j(\tau_k) - Q_{a,b}(\tau_k)\big]\bigg|
\\
\leq 2(a+b)\sum_{k: |\tau_k -\tau|> (\log G)^2/G} |A_k(\tau)| \lesssim \gamma^{(\log G)^2/2} = o(m/n),\label{eq:trunc}
\end{multline}
where $0<\gamma<1$; in the second above inequality we apply the bound \eqref{eq:Slin2} in Lemma \ref{lem:Ak}.

Hence, by \eqref{eq:hatbiasdec_new}, it suffices to prove that for $b > 0$ and any $a_{max} > 0$, there exists $a \in [0,a_{max}]$ and a sequence $\tau_N$ such that
\begin{align}
\limsup_{N\to\infty} \frac{n}{m}\bigg| \sum_{k: |\tau_k -\tau_N|\leq (\log G)^2/G}  A_k(\tau_N)\E\big[Q_{a,b}(U_{(\lceil n_j\tau_k\rceil)}^{j}) - Q_{a,b}(\tau_k)\big] \bigg| > 0. \label{eq:tilbet1check}
\end{align}
In the following, we will distinguish the cases $\limsup_{N\to\infty}(n/m)< \infty$ and $\limsup_{N\to\infty}(n/m)=\infty$. 

First consider the case $\limsup_{N\to\infty}(n/m) < \infty$. Invoke~\eqref{eq:localize01} in Lemma \ref{lem:localize} to see that there exist a sequence $\tau_N$ in $\Tc$ such that
\begin{align*}
		\limsup_{N\to\infty} \frac{n}{m}\bigg|&\sum_{k: |\tau_k -\tau_N|\leq (\log G)^2/G}  A_k(\tau_N)\E\big[Q_{a,b}(U_{(\lceil n_j\tau_k\rceil)}^{j}) - Q_{a,b}(\tau_k)\big]\bigg| \\
		&\geq\limsup_{N\to\infty}\frac{n}{m}\bigg|\E\big[Q_{a,b}(U_{(\lceil n_j\tau_N\rceil)}^{j}) - Q_{a,b}(\tau_N)\big]\bigg|
\end{align*}

If $\limsup_{N\to\infty}(n/m) \big|\E[U_{(\lceil n_j\tau_N\rceil)}^{j}- \tau_N]\big| \neq 0$, use \eqref{eq:sharpStmp1_new}, to see that the right-hand side in the equation above is non-zero for any $b>0$ and $a=0$. 

If $\limsup_{N\to\infty}(n/m)\big|\E[U_{(\lceil n_j\tau_N\rceil)}^{j}- \tau_N]\big|=0$ use \eqref{eq:bin1} in Lemma \ref{lem:bin} to obtain
$$
\limsup_{N\to\infty}\frac{n}{m}\bigg|\E\big[Q_{a,b}(U_{(\lceil n_j\tau_N\rceil)}^{j}) - Q_{a,b}(\tau_N)\big]\bigg| \geq \limsup_{N\to\infty}\frac{a \tau_N(1-\tau_N)}{200} 
> 0
$$
where the last bound follows since $0 < \tau_L \leq \tau_N \leq \tau_U < 1$ for all $N$. Thus we have established \eqref{eq:tilbet1check} for the case $\limsup_{N\to\infty}(n/m) > 0$.

\bigskip

Next consider the case $\limsup_{N\to\infty}(n/m)=\infty$. By~\eqref{eq:localize02} in Lemma \ref{lem:localize} there exist a sequence $\tau_N \in \Tc$ such that
\begin{multline*}
\limsup_{N\to\infty} \frac{n}{m}\bigg|\sum_{k: |\tau_k -\tau_N|\leq (\log G)^2/G}  A_k(\tau_N)\E\big[Q_{a,b}(U_{(\lceil n_j\tau_k\rceil)}^{j}) - Q_{a,b}(\tau_k)\big]\bigg| 
\\
\geq \limsup_{N\to\infty} \frac{n}{m} \bigg| a\frac{\tau_N(1-\tau_N)}{(n/m)+1}
+ (b+2a\tau_N) \sum_{k: |\tau_k -\tau_N|\leq (\log G)^2/G} A_k(\tau_N)\Big(\E[U_{(\lceil n_j\tau_k\rceil)}^{j}] - \tau_N\Big)\bigg|.
\end{multline*}
We distinguish two cases. 

If $\limsup_{N\to\infty}(n/m) \Big|\sum_{k: |\tau_k -\tau_N|\leq (\log G)^2/G} A_k(\tau_N)\Big(\E[U_{(\lceil n_j\tau_k\rceil)}^{j}] - \tau_N\Big)\Big| > 0$, \eqref{eq:tilbet1check} holds for $\tau_N$ with any $b>0$ and $a=0$. 

If $\limsup_{N\to\infty}(n/m)\Big|\sum_{k: |\tau_k -\tau_N|\leq (\log G)^2/G} A_k(\tau_N)\Big(\E[U_{(\lceil n_j\tau_k\rceil)}^{j}] - \tau_N\Big)\Big| = 0$, we can pick an arbitrary $0<a\leq a_{max}$ since $\tau_N\in\Tc=[\tau_L,\tau_U]$ with $0<\tau_L<\tau_U<1$ and thus
\begin{multline*}
\limsup_{N\to\infty}\frac{n}{m}\bigg|a\frac{\tau_N(1-\tau_N)}{(n/m)+1}\bigg| \geq \frac{1}{2}\limsup_{N\to\infty}a \tau_N(1-\tau_N)
\\
\geq \frac{1}{2} a\min\big\{\tau_L(1-\tau_L),\tau_U(1-\tau_U)\big\}>0.
\end{multline*}
Thus, \eqref{eq:tilbet1check} holds when $\limsup_{N\to\infty}(n/m)=\infty$. Combining the two cases above completes the proof of \eqref{eq:tilbet1check}. \hfill $\Box$


\subsection{Proof of \eqref{`EQ:TILBET2} and \eqref{`EQ:HATBET2}}

We begin by proving a slightly stronger version of \eqref{`EQ:TILBET2} for any $j$:
\begin{equation} \label{eq:tilbet2_un}
\sup_{\tau \in \Tc} \Var(\tilde\beta_j(\tau)) \lesssim 1/n.
\end{equation}
Observe that, as discussed in te beginning of Section~\ref{sec:pfbias}, $m^{1/2}\tilde\beta_j(\tau)=Q_{a,b}(U_{(\lceil n_j\tau\rceil)}^{j})$, and
\begin{align*}
	&\Var\big(m^{1/2}\tilde\beta_j(\tau)\big)
	\\
	&=\Var\big(Q_{a,b}(U_{(\lceil n_j\tau\rceil)}^{j})\big)\\
	&=\E\big[\big\{a\big((U_{(\lceil n_j\tau\rceil)}^{j})^2-\E[(U_{(\lceil n_j\tau\rceil)}^{j})^2]\big)+b\big(U_{(\lceil n_j\tau\rceil)}^{j}-\E[U_{(\lceil n_j\tau\rceil)}^{j}]\big)\big\}^2\big]\\
	&=\E\big[\big\{a\big((U_{(\lceil n_j\tau\rceil)}^{j})^2-\E[U_{(\lceil n_j\tau\rceil)}^{j}]^2-\Var(U_{(\lceil n_j\tau\rceil)}^{j})\big)+b\big(U_{(\lceil n_j\tau\rceil)}^{j}-\E[U_{(\lceil n_j\tau\rceil)}^{j}]\big)\big\}^2\big]\\
	&\leq 9\big\{(4a^2+b^2) \Var\big(U_{(\lceil n_j\tau\rceil)}^{j}\big)+a^2\Var\big(U_{(\lceil n_j\tau\rceil)}^{j}\big)^2\big\}.
\end{align*}
Therefore, it is enough to show that 
\begin{align}
	\sup_{\tau \in \Tc}\Var\big(U_{(\lceil n_j\tau\rceil)}^{j}\big) \lesssim m/n.\label{eq:varu}
\end{align}
Distinguish two cases. If $m/n$ is bounded away from zero, this follows since $|U_{(\lceil n_j\tau\rceil)}^{j}| \leq 1$ and thus $\Var\big(U_{(\lceil n_j\tau\rceil)}^{j}\big)$ is bounded by a constant independent of $\tau$. We will thus without loss of generality assume that $m/n \to 0$. Observe the decomposition
\begin{align*}
\Var(U_{(\lceil n_j\tau\rceil)}^j)&=\E\Big[\Big(U_{(\lceil n_j\tau\rceil)}^j - \E[U_{(\lceil n_j\tau\rceil)}^j]\Big)^2\Big] \\
&=	\E\Big[\Big(U_{(\lceil n_j\tau\rceil)}^j - \E[U_{(\lceil n_j\tau\rceil)}^j|n_j]\Big)^2\Big]+\E\Big[\Big(\E[U_{(\lceil n_j\tau\rceil)}^j|n_j]-\E[U_{(\lceil n_j\tau\rceil)}^j]\Big)^2\Big].
\end{align*}
First observe that
\begin{align}
	&\hspace{-0.5cm}\E\Big[\Big(U_{(\lceil n_j\tau\rceil)}^j - \E[U_{(\lceil n_j\tau\rceil)}^j|n_j]\Big)^2\Big] \notag
	\\
	&= \E\Big[\big(U_{(\lceil n_j\tau\rceil)}^j - \E[U_{(\lceil n_j\tau\rceil)}^j|n_j]\big)^2\IF\Big\{\Big|n_j-\frac{n}{m}\Big|\leq 6\sqrt{ \frac{n}{m} \log \Big(\frac{n}{m}\Big)}\Big\}\Big]+O\big((n/m)^{-2}\big)\notag
	\\
	&\leq \max_{l:|l-(n/m)|\leq 6\sqrt{(n/m) \log (n/m)}} \E\Big[\big(U_{(\lceil l\tau\rceil)}^j - \E[U_{(\lceil l\tau\rceil)}^j]\big)^2\Big]
	+O\big((n/m)^{-2}\big)\notag
	\\
	&= \max_{l:|l-(n/m)|\leq 6\sqrt{(n/m) \log (n/m)}}\Var\big(U_{(\lceil l\tau\rceil)}\big)+O\big((n/m)^{-2}\big),\notag
	\\
	&\asymp m/n, \label{eq:U4momtmp1new_varu}
\end{align}
uniformly in $\tau\in\Tc$. Here, the first equality follows since, $\Big|\big(U_{(\lceil n_j\tau\rceil)}^j - \E[U_{(\lceil n_j\tau\rceil)}^j|n_j]\big)\Big|^2 \leq 1$ a.s., $n_j\sim Bin(n,1/m)$ and the following result which is a consequence of Bernstein's inequality
\begin{align}
	P\Big(\Big|n_j - \frac{n}{m}\Big| \geq C\sqrt{\frac{n}{m}\log \Big(\frac{n}{m}\Big)}\Big) \leq 2\Big(\frac{n}{m}\Big)^{-C/3} \quad \mbox{for all }C\geq 1 
	\label{eq:hoef}
\end{align}
The final inequality is from \eqref{eq:udistr}.

Next, observe that, almost surely, 
\[
\Big|\E[U_{(\lceil n_j\tau\rceil)}^j|n_j] - \tau\Big| \leq \frac{1}{n_j + 1}
\]
by~\eqref{eq:udistr}. This implies $|\E[U_{(\lceil n_j\tau\rceil)}^j] - \tau| \leq \E[(1+n_j)^{-1}]$, and thus
\begin{equation} \label{eq:conduncond}
	\E\Big[\Big(\E[U_{(\lceil n_j\tau\rceil)}^j|n_j] - \E[U_{(\lceil n_j\tau\rceil)}^j]\Big)^2\Big] \lesssim \E\big[(1+n_j)^{-2}\big] + (\E[(1+n_j)^{-1}])^2 = O((m/n)^2)
	\end{equation} 
by an application of~\eqref{eq:hoef}. Combining \eqref{eq:U4momtmp1new_varu} and \eqref{eq:conduncond} finishes the proof of \eqref{eq:varu}.

To prove~\eqref{`EQ:HATBET2}, note that by H\"older's inequality and Lemma~\ref{lem:Ak}
\begin{align*}
\E[|\check\beta_j(\tau) - \E[\check\beta_j(\tau)]|^2] &= \E\Big[ \Big(\sum_{k=1}^K A_k(\tau)\{\tilde\beta_j(\tau_k) - \E[\tilde\beta_j(\tau_k)] \} \Big)^2\Big]
\\
& \leq \E\Big[\Big(\sum_{k=1}^K |A_k(\tau)| \Big) \Big(\sum_{k=1}^K |A_k(\tau)| |\tilde\beta_j(\tau_k) - \E[\tilde\beta_j(\tau_k)]|^2 \Big)\Big]
\\
& \lesssim 
\sup_{k=1,...,K} \Var(\tilde\beta_j(\tau_k)) \\
&\lesssim 1/n,
\end{align*}
where the last inequality is an application of \eqref{eq:tilbet2_un}. \hfill $\Box$

\subsection{Auxiliary Lemmas}
\begin{lemma}\label{lem:bin}
Assume that $W_n \sim Bin(n,1/m)$ and that $U_\tau|W_n=w \sim Beta(\lceil w\tau\rceil, w + 1 - \lceil w\tau\rceil)$ if $w > 0$ and $U_\tau = 0$ a.s. conditional on $W_n = 0$. Let $Q_{a,b}(x) := ax^2 + bx$. Then for any $a,b \geq 0$ and $n,m$ such that $n \geq 2$, $n/m \geq 1$, we have for all $\tau\in (0,1)$ that 
\begin{align}
\frac{n}{m} \big|\E[Q_{a,b}(U_\tau) - Q_{a,b}(\tau)]\big| &\geq \frac{a \tau(1-\tau)}{200} - \frac{n}{m} b\big|\E[U_{\tau}- \tau]\big| - 2a\frac{n}{m} \big|\E[U_\tau- \tau]\big|. \label{eq:bin1}
\end{align}
\end{lemma}
\noindent
\textbf{Proof of Lemma~\ref{lem:bin}.} 
For a proof of \eqref{eq:bin1}, note that 
\[
|(\E[U_\tau])^2 - \tau^2| \leq 2 |\E[U_\tau] - \tau|.
\]
Note also that
\[
\Var(U_\tau) = \E[\Var(U_\tau|W)] + \Var(\E[U_\tau|W]) \geq \E[\Var(U_\tau|W)].
\]
Thus with an application of triangle inequality, for $a,b \geq 0$, 
\begin{align}
\big|\E[Q_{a,b}(U_\tau) - Q_{a,b}(\tau)]\big| &=\big|a (\E[U_\tau^2 - \tau^2]) + b(\E[U_\tau - \tau])\big|\notag\\
&\geq a\E[\Var(U_\tau|W)] - b\big|\E[U_{\tau}- \tau]\big| - 2a \big|\E[U_\tau- \tau]\big|. \label{eq:hhh7}
\end{align}
Next we will prove a lower bound on $\E[\Var(U_\tau|W)]$. The following tail bound on binomial probabilities follows from the Bernstein inequality
\[
P\Big(W_n \geq (1+x)\frac{n}{m} \Big) \leq \exp\Big(-\frac{x^2}{2+x}\frac{n}{m}\Big) \quad \forall x\geq 0
\]
In particular for $n/m\geq 1$ setting $C = x-1$ yields for $C\geq 1$ 
\begin{equation} \label{eq:hhh5}
P\Big( W_n \geq C \frac{n}{m} \Big) \leq \exp\Big(-\frac{(C-1)^2}{1+C}\Big)
\end{equation}
Next note that 
\[
\E[\Var(U_\tau|W)] = \sum_{w = 1}^\infty P(W_n = w) \frac{\lceil w\tau \rceil(w + 1 - \lceil w\tau \rceil)}{(w+1)^2(w+2)},
\]
and moreover for any $w > 0$ 
\[
\frac{\tau(1-\tau)}{4(w+2)} \leq \frac{\lceil w\tau \rceil(w + 1 - \lceil w\tau \rceil)}{(w+1)^2(w+2)} \leq \frac{1}{w+2}.
\]
Thus from~\eqref{eq:hhh5} we obtain for $n/m \geq 1$ by setting $C=3$
\begin{align}
\E[\Var(U_\tau|W)] &\geq \frac{\tau(1-\tau)}{4(3(n/m)+2)}\Big(P(W_n > 0)- e^{-1}\Big) \geq \frac{m}{n}\frac{\tau(1-\tau)}{20}\Big( \frac{1}{2} - e^{-1}\Big) \notag
\\
&\geq \frac{m}{n}\frac{\tau(1-\tau)}{200}. \label{eq:hhh6}
\end{align}
where we used the fact that for $n\geq 2, n/m \geq 1$ we have $P(W_n > 0) \geq 1/2$. The result follows by combining this with~\eqref{eq:hhh7}. 

\hfill $\qed$


\begin{lemma}[Localization]\label{lem:localize}
Let all conditions in Lemma \ref{lem:bin} hold and assume that additionally $N \geq m$, $G \gg (N/m)^\alpha$ for some $\alpha > 0$ and $S \gg (N/m)^{1/2}$. Recall the definitions of $A_k, Q_{a,b}$ given in Section~\ref{sec:specconst} and assume that $r_\tau \geq 2$.
\begin{enumerate}
\item If $\limsup_{N\to\infty} n/m < \infty$ there exists a sequence $\tau_N$ in $\Tc$ such that
\begin{multline}
\limsup_{N\to\infty} \frac{n}{m}\bigg|\sum_{k: |\tau_k -\tau_N|\leq (\log G)^2/G}  A_k(\tau_N)\E[Q_{a,b}(U_{\tau_k}) - Q_{a,b}(\tau_k)]\bigg|
\\
\geq \underset{N\to\infty}{\limsup}\,\, (n/m)\big|\E[Q_{a,b}(U_{\tau_N}) - Q_{a,b}(\tau_N)]\big| \label{eq:localize01}
\end{multline}
\item If $\limsup_{N\to\infty} n/m = \infty$ there exists a sequence $\tau_N$ in $\Tc$ such that
\begin{multline}
	\hspace{-0.5cm}\limsup_{N\to\infty} \frac{n}{m}\bigg|\sum_{k: |\tau_k -\tau_N|\leq (\log G)^2/G}  A_k(\tau_N)\E[Q_{a,b}(U_{\tau_k}) - Q_{a,b}(\tau_k)]\bigg|
\\
\geq \limsup_{N\to\infty} \frac{n}{m} \bigg| a\frac{\tau_N(1-\tau_N)}{(n/m)+1} + (b+2a\tau_N) \sum_{k: |\tau_k -\tau_N|\leq (\log G)^2/G} A_k(\tau_N)\Big(\E[U_{\tau_k}] - \tau_N\Big)\bigg| \label{eq:localize02}
	\end{multline}
\end{enumerate}

\end{lemma}

\begin{proof}[Proof of Lemma \ref{lem:localize}] For a proof of~\eqref{eq:localize01}, for arbitrary $w_0\in\N$, observe that the sum
	\begin{align}
		S_{w_0}(\tau):=\sum_{w=1}^{w_0} P(W_n = w) \big( a \E[U_\tau^2 |W_n = w] 
		+ b\E[U_\tau|W_n = w]\big) \label{eq:defSw}
	\end{align}
	is a piecewise constant function of $\tau$. This function is constant on intervals of the form $[l_1/w_1,l_2/w_2]$ for some positive integers $l_1,l_2,w_1,w_2 \leq w_0$ satisfying $\tau_L \leq l_1/w_1<l_2/w_2\leq \tau_U$. Thus, the length of each interval where $S_{w_0}$ is constant can be bounded from below as follows
	\begin{align}
		\frac{l_1}{w_1}-\frac{l_2}{w_2} = \frac{l_1 w_2 - l_2w_1}{w_1w_2} \geq w_0^{-2}, \label{eq:intervallength}
	\end{align}
	since $l_1 w_2 - l_2w_1 \geq 1$ because both $l_1 w_2$ and $l_2 w_1$ are positive integers, and $w_1,w_2 \leq w_0$. Let $w_N = \log G$. Then
	\begin{align}
	& a (\E[U_\tau^2 - \tau^2]) + b(\E[U_\tau - \tau])
	=  S_{w_N}(\tau)- P(W_n\leq w_N)(b\tau + a\tau^2)  + r_G(\tau) \label{eq:hhh3}
	\end{align}
	where the remainder satisfies $\sup_{\tau\in\Tc}|r_G(\tau)| \leq 2(a+b) P(W_n > w_N)$, and by \eqref{eq:hhh5}
	\begin{align}
	  P(W_n > w_N) = o(1), \quad \sup_{\tau\in\Tc}|r_G(\tau)| = o(m/n) \label{eq:poitail}
	\end{align}
	since $w_N \to \infty$ and we are in the case $n/m$ bounded. For each $N$, select $\tau_N$ as the midpoint of any open interval $\Jc_N\subset \Tc$ on which $S_{w_N}(\tau)$ is a constant. From the choice of $\tau_N$ and $w_N$, we have when $N$ is sufficiently large, 
	\begin{align}
	S_{w_N}(\tau_k)=S_{w_N}(\tau_N) \mbox{ for all $k:|\tau_k -\tau_N|\leq 1/(3w_N^2)$}
	\label{eq:int}
	\end{align}
and in particular $S_{w_N}(\tau_k)=S_{w_N}(\tau_N)$ for $k:|\tau_k -\tau_N|\leq (\log G)^2/G$ since $G \to \infty$.	Hence, for each $N$, we get from \eqref{eq:hhh3} that
	\begin{align}
		&\sum_{k: |\tau_k -\tau_N|\leq (\log G)^2/G}  A_k(\tau_N)\E[Q_{a,b}(U_{\tau_k}) - Q_{a,b}(\tau_k)]\notag\\
	& = \sum_{k: |\tau_k -\tau_N|\leq (\log G)^2/G}  A_k(\tau_N) \big[a (\E[U_{\tau_k}^2 - \tau_k^2]) + b(\E[U_{\tau_k} - \tau_k])\big]
	\notag\\
	& = \sum_{k: |\tau_k -\tau_N|\leq (\log G)^2/G}  A_k(\tau_N)\Big\{  S_{w_N}(\tau_k)+ \Big(- P(W_n\leq w_N)(b\tau_k + a\tau_k^2) + r_G(\tau_k)\Big)\Big\} .\label{eq:pjmain}
\end{align}
	Note that by similar argument as \eqref{eq:biaspj} and \eqref{eq:trunc}, we have 
	\begin{align}
		&\hspace{-1cm}\bigg|\sum_{k: |\tau_k -\tau_N|\leq (\log G)^2/G}  \Big\{A_k(\tau_N) (b\tau_k + a\tau_k^2)\Big\}-(b\tau_N + a\tau_N^2)\bigg|\notag\\
		&\stackrel{\eqref{eq:trunc}}{=}\bigg|\sum_{k=1}^K  \Big\{A_k(\tau_N) (b\tau_k + a\tau_k^2)\Big\}-(b\tau_N + a\tau_N^2)\bigg|+o(m/n) \stackrel{\eqref{eq:biaspj}}{=} o(m/n), \label{eq:binpj1}
	\end{align}
	and 
	\begin{align}
	\bigg|\sum_{k: |\tau_k -\tau_N|\leq (\log G)^2/G} A_k(\tau_N) r_G(\tau_k)\bigg|\stackrel{\eqref{eq:trunc}}{=}\bigg|\sum_{k=1}^K A_k(\tau_N) r_G(\tau_k) \bigg|+o(m/n) = o(m/n).\label{eq:binpj2}
	\end{align}
	where the last identity follows from~\eqref{eq:hhh3} and Lemma~\ref{lem:Ak}. Hence, we obtain from \eqref{eq:pjmain} that
	\begin{align*}
		\limsup_{N\to\infty}&\frac{n}{m}\bigg|\sum_{k: |\tau_k -\tau_N|\leq (\log G)^2/G}  A_k(\tau_N)\E[Q_{a,b}(U_{\tau_k}) - Q_{a,b}(\tau_k)]\bigg|
		\\
		\stackrel{\eqref{eq:binpj1}\atop\eqref{eq:binpj2}}{=} &\limsup_{N\to\infty}\frac{n}{m}\bigg|\sum_{k: |\tau_k -\tau_N|\leq (\log G)^2/G}  \Big\{A_k(\tau_N) S_{w_N}(\tau_k)\Big\}-P(W_n\leq w_N)(b\tau_N + a\tau_N^2)\bigg|
		\\
		\stackrel{\eqref{eq:int}}{=}&\limsup_{N\to\infty}\frac{n}{m}\bigg|S_{w_N}(\tau_N) \sum_{k: |\tau_k -\tau_N|\leq (\log G)^2/G}  \Big\{A_k(\tau_N)\Big\} -P(W_n\leq w_N)(b\tau_N + a\tau_N^2)\bigg|
		\\
		\stackrel{\eqref{eq:trunc}}=&\limsup_{N\to\infty}\frac{n}{m}\bigg|S_{w_N}(\tau_N) \sum_{k=1}^K  \Big\{A_k(\tau_N)\Big\} -P(W_n\leq w_N)(b\tau_N + a\tau_N^2)\bigg|
		\\
		= \hspace{0.4cm}&\limsup_{N\to\infty}\frac{n}{m}\big|S_{w_N}(\tau_N) -P(W_n\leq w_N)(b\tau_N + a\tau_N^2) + r_G(\tau_N)\big|.
	\end{align*}
Here the last line follows from Lemma~\ref{lem:Ak}. Apply~\eqref{eq:hhh3} to complete the proof of~\eqref{eq:localize01}.

For a proof of \eqref{eq:localize02}, note that by the law of iterated expectation,
		\begin{align}
			\E\big[Q_{a,b}(U_\tau) - Q_{a,b}(\tau)\big] & = a(\E[v_{2,n}(\tau)] + \E[v_{1,n}^2(\tau)] - \tau^2) + b( \E[v_{1,n}(\tau)] - \tau)\label{eq:sharpStmp0_new}
		\end{align}
		where
		\begin{align}
			\begin{split} \label{eq:defvn_new}
			v_{1,n}(\tau) &:= \E[U_\tau|W_n] = \IF\{W_n > 0\}\frac{\lceil W_n\tau \rceil}{W_n+1};\\
		  v_{2,n}(\tau) &:= \Var(U_\tau|W_n) = \IF\{W_n > 0\}\frac{\lceil W_n\tau \rceil(W_n + 1 - \lceil W_n\tau \rceil)}{(W_n+1)^2(W_n+2)}.
		  \end{split}
		\end{align}
The following bound for binomial random variables can be obtained by Bernstein's inequality for $n/m \geq 1$: 
	\[
	P\Big( |W_n - n /m| \geq C \sqrt{(n/m)\log (n/m)}\Big) \leq 2(n/m)^{-C/3} \quad \mbox{for all } C\geq 1
	\]
	Using this bound, straightforward calculations show that 
	\begin{align*}
	\E[v_{1,n}(\tau)] &=  \tau + O(m/n)
	\\
	\E[v_{2,n}(\tau)] &= \frac{\tau(1-\tau)}{(n/m)+1} + o(m/n)
	\\
	\E[v_{1,n}^2(\tau) -\tau^2] &= 2\tau\big(\E[v_{1,n}(\tau)]-\tau\big) + o(m/n)
\end{align*}
where all remainder terms are uniform in $\tau\in \Tc$. Here, the first equation is a consequence of the fact that $|v_{1,n}(\tau) - \tau| \leq (1+W_n)^{-1}$ while the third equation follows since
\begin{align*}
\E[v_{1,n}^2(\tau) -\tau^2] - 2\tau\big(\E[v_{1,n}(\tau)]-\tau\big) = \E[(v_{1,n}(\tau) - \tau)^2].
\end{align*}
Thus, from \eqref{eq:sharpStmp0_new},
\begin{align}
\E\big[Q_{a,b}(U_\tau) - Q_{a,b}(\tau)\big] = a\frac{\tau(1-\tau)}{(n/m)+1} + (b+2a\tau)\big(\E[v_{1,n}(\tau)] - \tau\big) + r_{1,N}(\tau) \label{eq:Slin1_new}
\end{align}
where $r_{1,N}(\tau)=o((n/m)^{-1})$ uniformly in $\tau\in\Tc$. Pick $\tau_N \equiv \tau$ for an arbitrary $\tau \in \Tc$. From Lemma~\ref{lem:Ak} we obtain
	\begin{align}
	\bigg|\sum_{k: |\tau_k -\tau_N|\leq (\log G)^2/G} A_k(\tau_N) r_{1,N}(\tau_k)\bigg| \leq \sup_{k=1,...,K} |r_{1,N}(\tau_k)| \sum_{k=1}^K |A_k(\tau)|= o(m/n).\label{eq:binpj3}
	\end{align}
Hence, combine \eqref{eq:binpj3} and \eqref{eq:Slin1_new} to get
	\begin{align*}
	&\limsup_{N\to\infty} \frac{n}{m}\bigg|\sum_{k: |\tau_k -\tau_N|\leq (\log G)^2/G}  A_k(\tau_N)\E\big[Q_{a,b}(U_{\tau_k}) - Q_{a,b}(\tau_k)\big]\bigg|\\
	&\stackrel{\eqref{eq:Slin1_new}\atop\eqref{eq:binpj3}}{=} \limsup_{N\to\infty} \frac{n}{m} \bigg|\sum_{k: |\tau_k -\tau_N|\leq (\log G)^2/G}  A_k(\tau_N)\Big\{ a\frac{\tau_k(1-\tau_k)}{(n/m)+1} + (b+2a\tau_k)\big(\E[v_{1,n}(\tau_k)] - \tau_k\big)\Big\}\bigg|\\
	&\stackrel{\eqref{eq:trunc}\atop\eqref{eq:biaspj}}{=} \limsup_{N\to\infty} \frac{n}{m} \bigg| a\frac{\tau_N(1-\tau_N)}{(n/m)+1} + (b+2a\tau_N) \sum_{k: |\tau_k -\tau_N|\leq (\log G)^2/G} A_k(\tau_N)\big(\E[v_{1,n}(\tau_k)] - \tau_N\big)\bigg|
	\end{align*}
	where the last inequality follows since for all $k$ in the sum $|\tau_k -\tau_N| = O((\log G)^2/G) = o(1)$ and $|\E[v_{1,n}(\tau_k)]| = \tau_k + O(m/n)$ uniformly in $\tau_k$. This completes the proof of~\eqref{eq:localize02}.
\end{proof}

\begin{lemma}\label{lem:Ak}
Assume that $K \geq G$. Recall the definition of $A_k(\tau)$ in \eqref{eq:Ak}. There exist constants $c,C>0$ independent of $K,G$ such that 
\[
\sup_{k,\tau\in\Tc}|A_k(\tau)| \leq c, \quad \sup_{\tau \in \Tc}\sum_{k=1}^K |A_k(\tau)| \leq C.
\]
Moreover, for any $\tau \in \Tc$, $\sum_{k=1}^K A_k(\tau) = 1$.  
\end{lemma}
\begin{proof}[Proof of Lemma \ref{lem:Ak}]

We start with proving $\sum_{k=1}^K A_k(\tau) = 1$. Observe that $\sum_{k=1}^K A_k(\tau)=\Pi_K g_1$, where $\Pi_K$ is defined in Section \ref{sec:proof_orgenpr} and $g_1(x)\equiv 1$. Since the degree of the piecewise polynomials in $\BB$ is greater or equal to 2 and $g_1$ is an order 0 polynomial, similar arguments as used in the proof of~\eqref{eq:biaspj} show that $\|\Pi_K g_1 - 1\|_\infty=0$. This implies $\sum_{k=1}^K A_k(\tau)=1$ for all $\tau\in\Tc$.

To obtain a bound for $\sup_{k,\tau\in\Tc}|A_k(\tau)|$, note that
		\begin{align*}
			\sup_{k,\tau\in\Tc}\big|A_k(\tau)\big| = \sup_k \|\Pi_K \overline{e}_k \|_\infty \leq \|\Pi_K\|_\infty \sup_k\|\overline{e}_k \|_\infty,
		\end{align*}
		where $\overline{e}_k(\cdot):\Tc\to [0,1]$ is the function interpolating $\big\{(\tau_1,0)$,..., $(\tau_{k-1},0)$, $(\tau_k,1)$, $(\tau_{k+1},0)$..., $(\tau_K,0)\big\}$, so $\sup_k\|\overline{e}_k \|_\infty=1$. In the first step in Section \ref{sec:proof_orgenpr} we have shown that $\|\Pi_K\|_\infty<C'$ for some constant $C'>0$ independent of $K$. Hence, taking $c=\|\Pi_K\|_\infty$ finishes the first claim.

Next we show that $\sup_{\tau \in \Tc}\sum_{k=1}^K |A_k(\tau)| \leq C$. For any $\tau\in\Tc$ and $\tau_k\in\Tc_K$, define the index sets 
		\begin{align*}
				\Ic(\tau)&:= \big\{j \in \{1,...,q\}: B_j(\tau)\neq 0\big\};\\
				\Ic_k&:=\big\{j \in \{1,...,q\}: B_j(\tau_k)\neq 0\big\}.
		\end{align*}
Note that $|\Ic(\tau)|=|\Ic(\tau_k)|=r_\tau+1$ for all $\tau\in\Tc$ and $k$, and the elements in both sets are consecutive positive integers by the connective support of $B_j$.

				By Lemma 6.3 of \cite{zsw98} (note that their $\mathbf{N}$, $n$ and $G_{k,n}$ are our $\BB$, $K$ and $K^{-1}\sum_{k=1}^K \BB(\tau_k)\BB(\tau_k)^\top$, while their $h$ is of order $1/G$ see page 1761, 1762 and 1765 of \cite{zsw98}), there exist constants $0<\gamma<1$ and $c_1>0$ independent of $K,G,\tau$ such that the $ij$'th element of the matrix satisfies
		\begin{align}
			\Big(\sum_{k=1}^K \BB(\tau_k)\BB(\tau_k)^\top \Big)_{ij}^{-1} \leq c_1 K^{-1} G \gamma^{|j-i|} \leq c_1 \gamma^{|j-i|} \label{eq:invbd}
		\end{align}
		where the second inequality follows from our assumption that $K \geq G$.
		
Recall that our $B_j$ is defined as given in Definition 4.19 on page 124 of \cite{schumaker:81}. Thus, by equation (4.31) in the same reference, 
\[
\sup_{\tau\in\Tc}\max_{j\leq q}|B_j(\tau)|^2\leq 1,
\]
and by \eqref{eq:invbd}, 
\begin{align}
\big|A_k(\tau)\big| &= \bigg|\BB(\tau)^\top \Big(\sum_{k=1}^K \BB(\tau_k)\BB(\tau_k)^\top \Big)^{-1} \BB(\tau_k)\bigg| \leq c_1 \sum_{j\in \Ic(\tau),j'\in \Ic_k} \gamma^{|j-j'|}\label{eq:bddAk}\\
			&\leq c_1 \sum_{l=0}^\infty \big(\big|(\Ic(\tau)+l)\cap \Ic_k\big|+\big|(\Ic(\tau)-l)\cap \Ic_k\big|\big) \gamma^l,\label{eq:Slin2}
		\end{align}
		where $\Ic(\tau)-l=\{j-l:j\in\Ic(\tau)\}$, $\Ic(\tau)+l=\{j+l:j\in\Ic(\tau)\}$. 
		Now, we make a key observation that $\big(\big|(\Ic(\tau)+l)\cap \Ic_k\big|+\big|(\Ic(\tau)-l)\cap \Ic_k\big|\big) \neq 0$ for at most $6(r_\tau+1)$ many $k$ at each fixed $l\in\N$. This is due to the fact that both $\Ic(\tau)$ and $\Ic_k$ consist of $r_\tau+1$ \emph{consecutive} positive integers, and the same is true for both $\Ic(\tau)-l$ and $\Ic(\tau)+l$. Hence, from \eqref{eq:Slin2} and Fubini's theorem,
		\begin{align*}
			\sum_{k=1}^K\big|A_k(\tau)\big| &\leq c_1 \sum_{l=0}^\infty \sum_{k=1}^K\big(\big|(\Ic(\tau)+l)\cap \Ic_k\big|+\big|(\Ic(\tau)-l)\cap \Ic_k\big|\big) \gamma^l
			\\
			&\leq 6c_1 (r_\tau+1) \sum_{l=0}^\infty \gamma^l <\infty.
		\end{align*}
		This completes the proof.
\end{proof}



\vskip 2em 
%

\section{Technical results from empirical process theory}\label{sec:s_mislan}
In this section, we collect some basic results from empirical process theory that we use throughout the proofs. Denote by $\Gc$ a class of functions that satisfies $|f(x)| \leq F(x) \leq U$ for every $f \in \Gc$ and let $\sigma^2 \geq \sup_{f \in \Gc} Pf^2$. Additionally, let for some $A>0,V>0$ and all $\eps > 0$
\beq \label{eq:entr}
N(\eps,\Gc,L_2(\Pb_n)) \leq \Big(\frac{A \|F\|_{L^2(\Pb_n)}}{\eps} \Big)^V.
\eeq
In that case, the symmetrization inequality and inequality (2.2) from \cite{kolt2006} yield
\beq \label{eq:Gexpect}
\E \|\Pb_n - P\|_{\Gc} \leq c_0 \Big[\sigma \Big(\frac{V}{n} \log \frac{A \|F\|_{L^2(P)}}{\sigma}\Big)^{1/2} + \frac{VU}{n} \log \frac{A \|F\|_{L^2(P)}}{\sigma}  \Big]
\eeq
for a universal constant $c_0>0$ provided that $1 \geq \sigma^2 > \text{const} n^{-1}$ [in fact, the inequality in \cite{kolt2006} is for $\sigma^2 = \sup_{f \in \Gc} Pf^2$. However, this is not a problem since we can replace $\Gc$ by $\Gc \sigma / (\sup_{f \in \Gc} Pf^2)^{1/2}$]. The second inequality (a refined version of Talagrand's concentration inequality) states that for any countable class of measurable functions $\Fc$ with elements mapping into $[-M,M]$
\beq \label{eq:Gcontract}
P\Big\{\|\Pb_n - P\|_{\Fc} \geq 2 \E\|\Pb_n - P\|_{\Fc} + c_1 n^{-1/2 }\Big(\sup_{f \in \Fc} Pf^2\Big)^{1/2} \sqrt{t} + n^{-1}c_j M t \Big\} \leq e^{-t},
\eeq
for all $t>0$ and universal constants $c_1,c_j >0$. This is a special case of Theorem 3 in \cite{mass2000} [in the notation of that paper, set  $\eps = 1$].

\begin{lemma} \label{lem:equicont}
Under \hyperref[A1]{(A1)}-\hyperref[A3]{(A3)} and $\xi_m^2(\log n)^2 = o(N)$ we have for any $\delta > 0$ and $\bu_N \in \R^m$,
\begin{equation}\label{eq:equicont}
\lim_{\delta \to 0} \limsup_{N \to \infty} P\Big( \|\bu_N\|_2^{-1} N^{1/2} \sup_{\tau_1,\tau_2\in\Tc,|\tau_1-\tau_2|\leq \delta} \Big|\bu_N^\top \bU_N(\tau_1)- \bu_N^\top \bU_N(\tau_2) \Big| > \eps\Big) = 0,
\end{equation}
where $\bU_N(\tau)$ is defined in \eqref{eq:bU}.
\end{lemma}

\noindent\textbf{Proof of Lemma \ref{lem:equicont}. } See the proof for Lemma A.3 in \cite{ChaVolChe2016}. \hfill $\qed$

\begin{lemma} \label{lem:poly}
Let $p$ be a polynomial of degree $q$. Then for arbitrary intervals $[a,b]$and $\alpha \in [0,1)$ we have
\[
\Big|\int_a^b p^2(x)dx - \frac{b-a}{J}\sum_{j: j + \alpha\leq J} p^2\Big(a + \frac{j + \alpha}{J}(b-a)\Big) \Big| \leq \frac{C_q}{J-1} \int_a^b p^2(x)dx
\]
for a constant $C_q>1$ that depends only on $q$. 
\end{lemma}

\noindent\textbf{Proof of Lemma \ref{lem:poly}. } 
We begin by observing that it suffices to prove that for any polynomial $p$ of degree $q$ we have
\beq\label{eq:poly01}
\Big|\int_0^1 p^2(x)dx - \frac{1}{J}\sum_{j: j + \alpha\leq J} p^2\Big(\frac{j+\alpha}{J}\Big) \Big| \leq \frac{C_q}{J} \int_0^1 p^2(x)dx.
\eeq
To see this, note that for any polynomial $p$ of degree $q$ the function $\tilde p_{a,b}(x) := p(a+x(b-a))(b-a)^{1/2}$ is also a polynomial of degree $q$ and we have $\int_0^1 \tilde p_{a,b}^2(x) dx = \int_a^b p^2(x)dx$ {by a change of variables} and 
\[
\frac{b-a}{J}\sum_{j: j + \alpha\leq J} p^2\Big(a + \frac{j+\alpha}{J}(b-a)\Big) = \frac{1}{J}\sum_{j: j + \alpha\leq J} \tilde p_{a,b}^2\Big(\frac{j+\alpha}{J}\Big).
\]
To prove \eqref{eq:poly01}, consider the bound 
\[
\Big|\int_0^1 p^2(x)dx - \frac{1}{J}\sum_{j:j+\alpha \leq J} p^2\Big(\frac{j+\alpha}{J}\Big) \Big| \leq \frac{1}{J-1}\sup_{x\in [0,1]} \Big|\frac{dp^2(x)}{dx}\Big| + \frac{2}{J}\sup_{x\in [0,1]} p^2(x),
\]
{where the first term on the right is from the difference of the two integrals on the interior partitions, and the second term handles the boundary partitions which are not included in the summation on the left.} Thus we need to bound $\sup_{x\in [0,1]} \Big\{\Big|\frac{dp^2(x)}{dx}\Big| + p^2(x)\Big\}$ by $C_q \int_0^1p^2(x)dx $. To this end, denote by $\Pc_q$ the space of polynomials on $[0,1]$ with degree less or equal to $q$. This is a finite-dimensional vector space. Note that the operator $\Phi: f \mapsto f' + f$ is a linear operator between the finite-dimensional, normed spaces $(\Pc_{2q},\|\cdot\|_1)$ and $(\Pc_{2q},\|\cdot\|_\infty)$. Such an operator must be bounded, and hence there exists a constant $C_q$ with $\|\Phi(p^2)\|_\infty \leq C_q \|p^2\|_1$. This completes the proof. \hfill $\qed$



\section{Technical Details for Simulation and additional simulation results}\label{sec:addsim}

This part of the supplement contains additional results and some technical details for the simulation study in Section~\ref{sec:sim}. Section~\ref{sec:simdet} contains some technical details related to the simulation settings in Section~\ref{sec:sim}, and additional figures (the case $m =16$ and coverage probabilities of confidence intervals for $F_{Y|X}(y|x_0)$ and large values of $S$). In Section \ref{sec:silin}, using the same linear model as in Section \ref{sec:sim}, we provide additional simulation results for the oracle properties of the divide and conquer estimator $\overline\zb(\tau)$ and the distribution function estimator $\hat F_{Y|X}(y|x)$; the emphasis of this section is on illustrating our oracle theory. In Section \ref{sec:sinp}, we consider a non-linear model and illustrate oracle properties of the divide and conquer estimator $\overline\zb(\tau)$ and the distribution function estimator $\hat F_{Y|X}(y|x)$. Some practical conclusions from the simulations are summarized in Section~\ref{sec:simprac}


\subsection{Details for Section~\ref{sec:sim} and additional simulation results} \label{sec:simdet}

The simple pooled estimator for the asymptotic covariance matrix of $\zb_{or}(\tau) - \zb(\tau)$ was computed by first using the options \texttt{se='ker'}, \texttt{covariance=TRUE} in the \texttt{rq} function in \texttt{quantreg} \citep{K16_quantreg} for each sub-sample, and then averaging the results over sub-samples. 

The values of the optimal constant $c^*(\tau)$ were determined from the theory developed in~\cite{kato2012}. {More precisely, we use the following formula given on its page 264 of the latter reference,
\begin{align}
	c^*(\tau) = \sigma \bigg(\frac{4.5 \sum_{j,k=1}^m \E[Z_j^2 Z_k^2]}{\alpha(\tau) \sum_{j,k=1}^m (\E[Z_j Z_k])^2}\bigg)^{1/5}, \label{eq:cst}
\end{align}
where $\alpha(\tau) = (1-\Phi^{-1}(\tau)^2)^2\phi(\Phi^{-1}(\tau))$ (note that this is a corrected version of Kato's formula). Note that $c^*(\tau)$ is symmetric around $\tau=0.5$. In order to save computation time in the simulation, we simulated an independent data set to compute $c^*(\tau)$, and used the same $c^*(\tau)$ (based on $m=4$) for different $m$ in all the simulations, because the exact value of $c^*(\tau)$ does not change much with $m$ under our design of $\ZZ$; see Table \ref{tab:cst}.}
\begin{table}[!h] 
\begin{tabular}{rccc}
	\hline\hline
					&$m = 4$		&$m = 16$	&$m = 32$\\
	\hline
$\tau = 0.1$	&0.242		&0.252		&0.254\\
		  0.5		&0.173		&0.179		&0.180\\
	\hline\hline	  
\end{tabular}
\caption{Values of $c^*(\tau)$ based on \eqref{eq:cst}.}\label{tab:cst}
\end{table}

{In the following, we show additional simulation results. Figure \ref{fig:simlecoverageSmS16} shows the coverage probability of the confidence interval for $Q(x_0;\tau)$ based on $x_0^\top \bar\zb(\tau)$ for dimension $m=16$.}
\begin{figure}[!h]
\centering
\includegraphics[width=13cm]{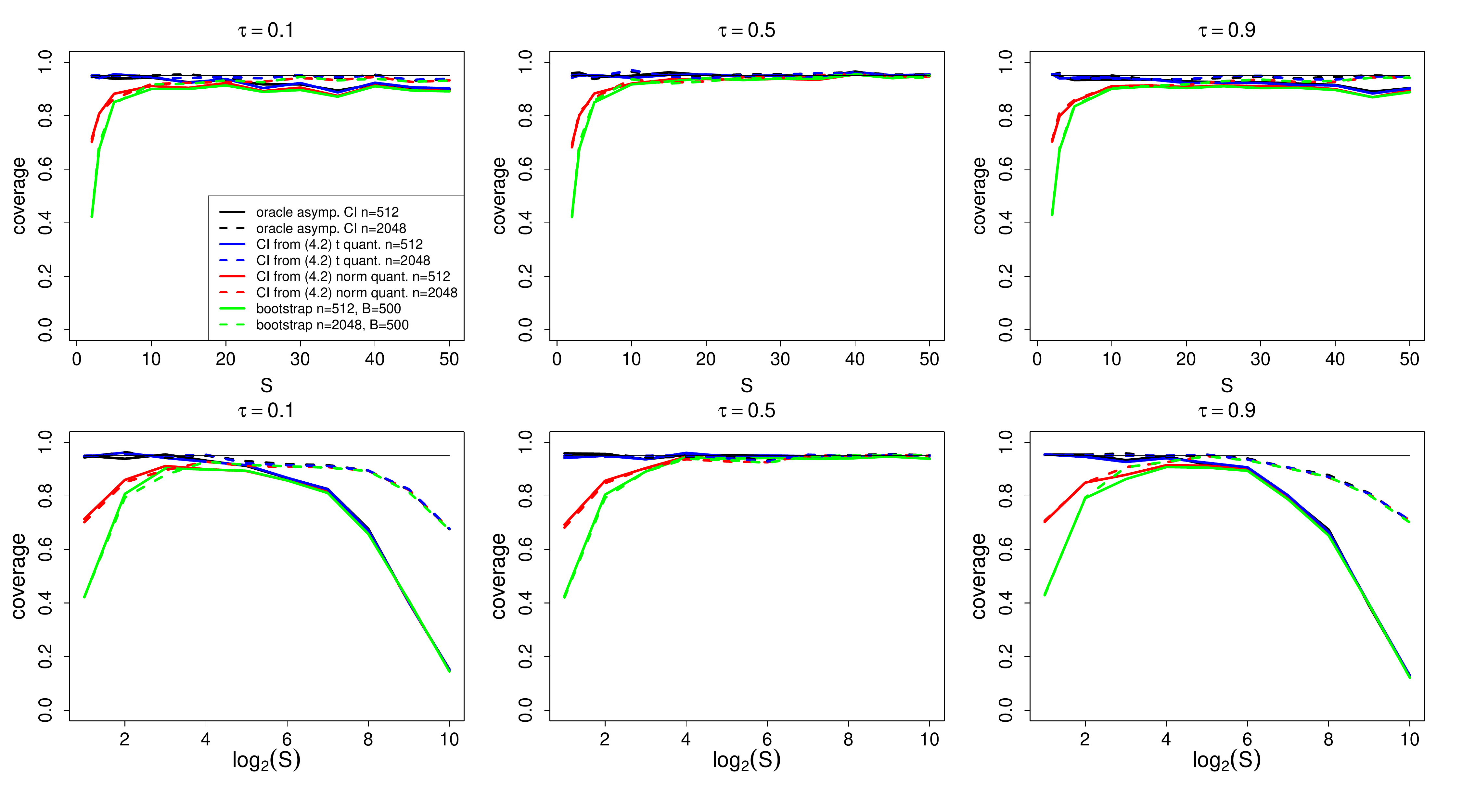}
\caption{Coverage probabilities for $Q(x_0;\tau)$ for $m=16$ in homoskedastic model~\eqref{eq:simlin}. Coverage probabilities for $x_0^\top\zb(\tau)$ for different values of $S$ and $\tau = 0.1, 0.5, 0.9$ (left, middle, right row). Solid lines: $n = 512$, dashed lines: $n = 2048$. Black: asymptotic oracle CI, blue: CI from~\eqref{ci:simple_t} based on t distribution, red: CI from~\eqref{ci:simple} based on normal distribution, green: bootstrap CI.
Throughout $x_0 = (1,...,1)/m^{1/2}$, nominal coverage $0.95$.}\label{fig:simlecoverageSmS16}
\end{figure}
{For additional simulations for $F_{Y|X}(y|x_0)$ under the same settings as that of Figure \ref{fig:Fhatcov} in Section \ref{sec:sim}, Figure \ref{fig:Fhatcov_bigS} shows the coverage probability of the confidence interval with large $S$ for $m=4$ and 32 and Figure \ref{fig:Fhatcov_m16} shows the coverage probability of the confidence interval for $m=16$.}
\begin{figure}[!h]
\centering
\includegraphics[width=13cm]{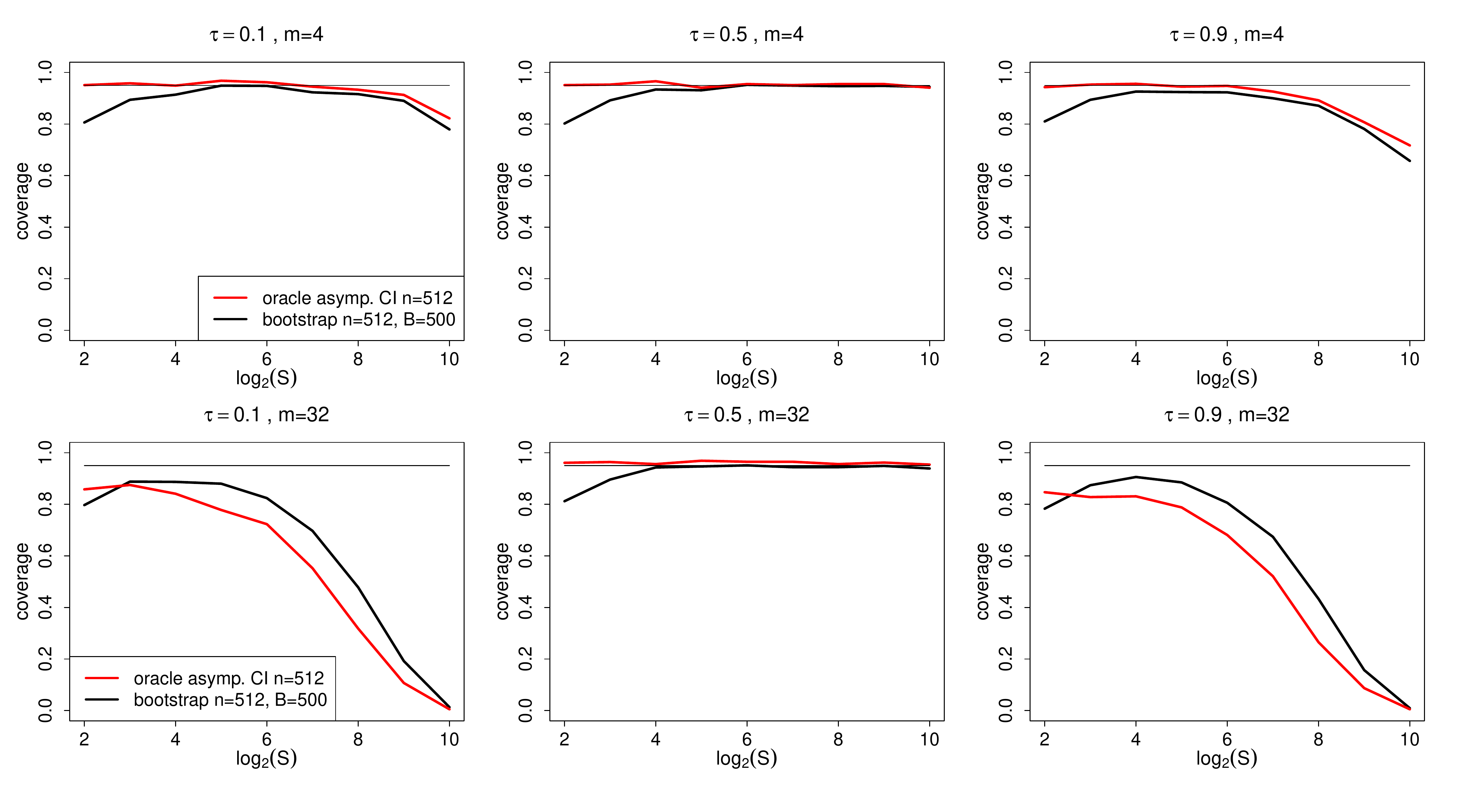}
\caption{Coverage probabilities for $F_{Y|X}(y|x_0)$ in homoskedastic model~\eqref{eq:simlin}  with $m=4,32$ with large number of subsamples $S$, $n = 512$ and nominal coverage $0.95$. Red: oracle asymptotic CI, black: bootstrap CI with $B=500$.}\label{fig:Fhatcov_bigS}
\end{figure}
\begin{figure}[!h]
\centering
\includegraphics[width=13cm]{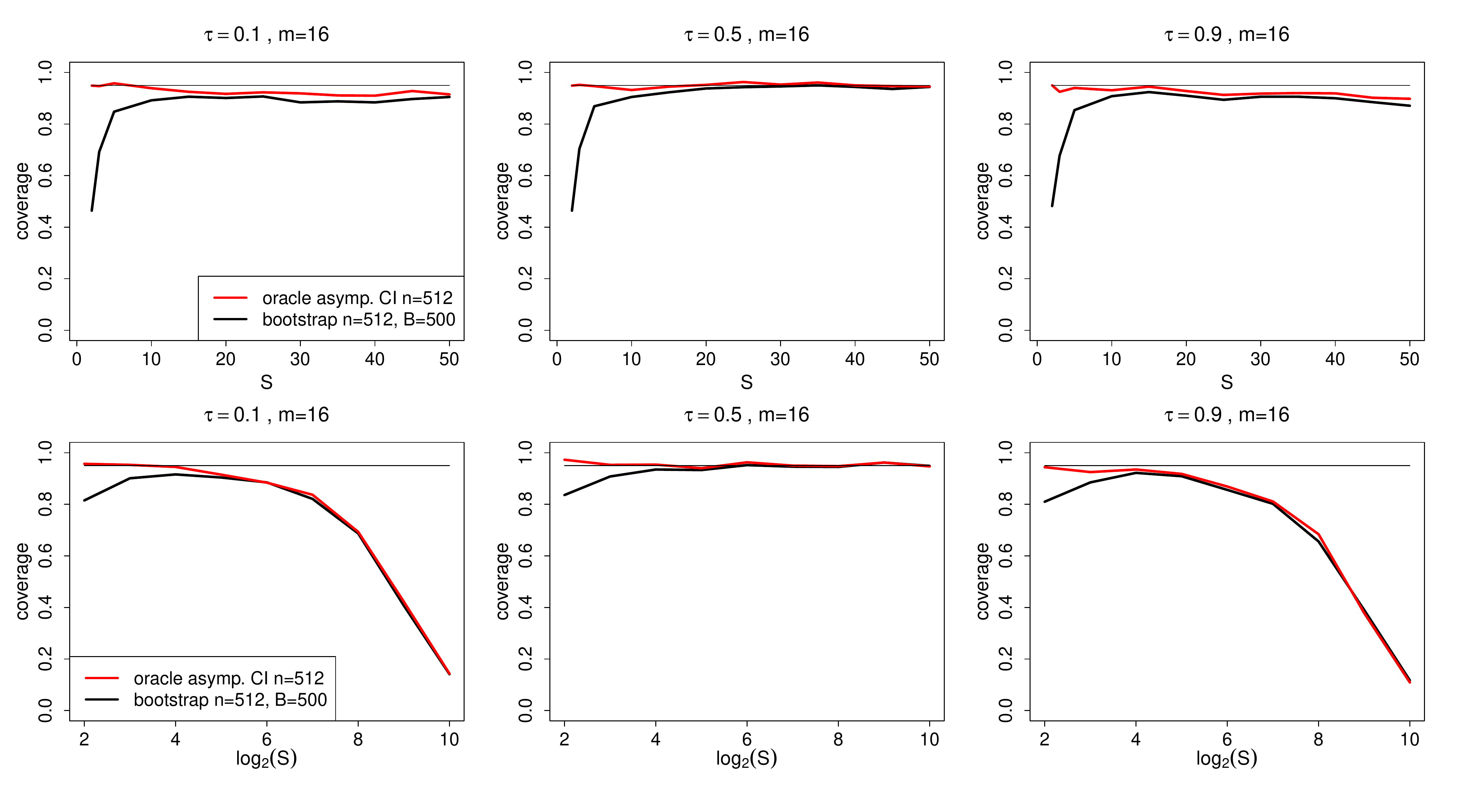}
\caption{Coverage probabilities in homoskedastic model~\eqref{eq:simlin} for $F_{Y|X}(y|x_0)$ for $m=16$, $n = 512$ and nominal coverage $0.95$. Red: oracle asymptotic CI, black: bootstrap CI with $B=500$.} \label{fig:Fhatcov_m16}
\end{figure}

\subsection{Heteroskedastic model}

We consider a linear location-scale shift model,
\begin{align}
Y_i = 0.21 +  \zb_{m-1}^\top X_i+(1+\zz_{m-1}^\top X_i)\varepsilon_i, 
\quad i=1,...,N.\label{eq:simlin_het}
\end{align}
where $m\in\{4,16,32\}$, $\zb_{m-1}$ is chosen as \eqref{eq:betaspeci}, $X_i$ and $\varepsilon_i$ are chosen as model \eqref{eq:simlin}, and 
the vector $\zz_{m-1}$ takes the form
\begin{align}
	\begin{split}\label{eq:betaspeci}
		\zz_{3}&=(0.69,0.56,0.35)^\top;\\
		\zz_{15}&=(\zz_{3}^\top,\zz_{3}^\top,\zz_{3}^\top,\zz_{3}^\top,\zz_{3}^\top)^\top;\\
		\zz_{31}&=(\zb_{15}^\top,\zb_{15}^\top,0.69)^\top.
	\end{split}
\end{align}  

\subsubsection{Results for the divide and conquer estimator $\bar\zb(\tau)$}

We fix the sub-sample size $n$ and consider the impact of $S$ on the coverage probabilities of various 95\% confidence intervals as in Section 5.1. We use the infeasible asymptotic confidence interval~\eqref{eq:lincovp} as benchmark (note that the asymptotic variance is different in the homoskedastic and heteroskedastic model).

In the first step, as in Section~\ref{sec:simdc}, we consider the three types of confidence intervals in Section~\ref{sec:subs}. The coverage probabilities of the corresponding confidence intervals are presented in Figure \ref{fig:simlecoverageSmS_het}. The general patterns are similar to those in Figure \ref{fig:simlecoverageSmS}, but coverage starts to drop earlier for both $m=4,32$. The same conclusion holds for the coverage probabilities for dimension $m=16$ when comparing Figure \ref{fig:simlecoverageSmS16_het} to Figure \ref{fig:simlecoverageSmS16} for the homoskedastic model.

Next, we analyse the asymptotic confidence intervals using the empirical asymptotic variance of $x_0^\top\zb(\tau)$ from data, which is estimated with the three ways described on p.22 in Section 5.1. The constants $c^*(\tau)$ there are replaced by $c_g^*(\tau)$, which is computed by adapting the formula on page 263 of \cite{kato2012} to the $\Nc(0,\sigma^2)$ location-scale shift model \eqref{eq:simlin_het}. This yields the formula:
\begin{align}
	c_g^*(\tau) = \sigma\bigg(\frac{4.5 \sum_{j,k=1}^m \E[(1+\zz_{m-1}^\top X)^{-1}Z_j^2 Z_k^2]}{\alpha(\tau) \sum_{j,k=1}^m (\E[(1+\zz_{m-1}^\top X)^{-3}Z_j Z_k])^2}\bigg)^{1/5}, \label{eq:cst_het}
\end{align}
where $\alpha(\tau) = (1-\Phi^{-1}(\tau)^2)^2\phi(\Phi^{-1}(\tau))$. The values are listed in Table \ref{tab:cst_het}.
\begin{table}[!h] 
\begin{tabular}{rrrr}
	\hline\hline
					&$m = 4$		&$m = 16$	&$m = 32$\\
	\hline
$\tau = 0.1$	&0.4206		&1.2118		&2.3222\\
		  0.5		&0.2990		&0.8614		&1.6508\\
	\hline\hline	  
\end{tabular}
\caption{Values of $c_g^*(\tau)$ based on \eqref{eq:cst_het} of the linear Gaussian location-scale shift model~\eqref{eq:simlin} .}\label{tab:cst_het}
\end{table}
The results for the coverage probabilities are reported in Table \ref{tab:bwcomp_het}. The overall patterns are similar to those in Table \ref{tab:bwcomp}, but the coverage probabilities start to drop earlier, this effect is similar to what we observed in Figure~\ref{fig:simlecoverageSmS_het}.

\begin{figure}[!h]
\centering {\scriptsize (a) $m=4$}\\
\includegraphics[width=13cm]{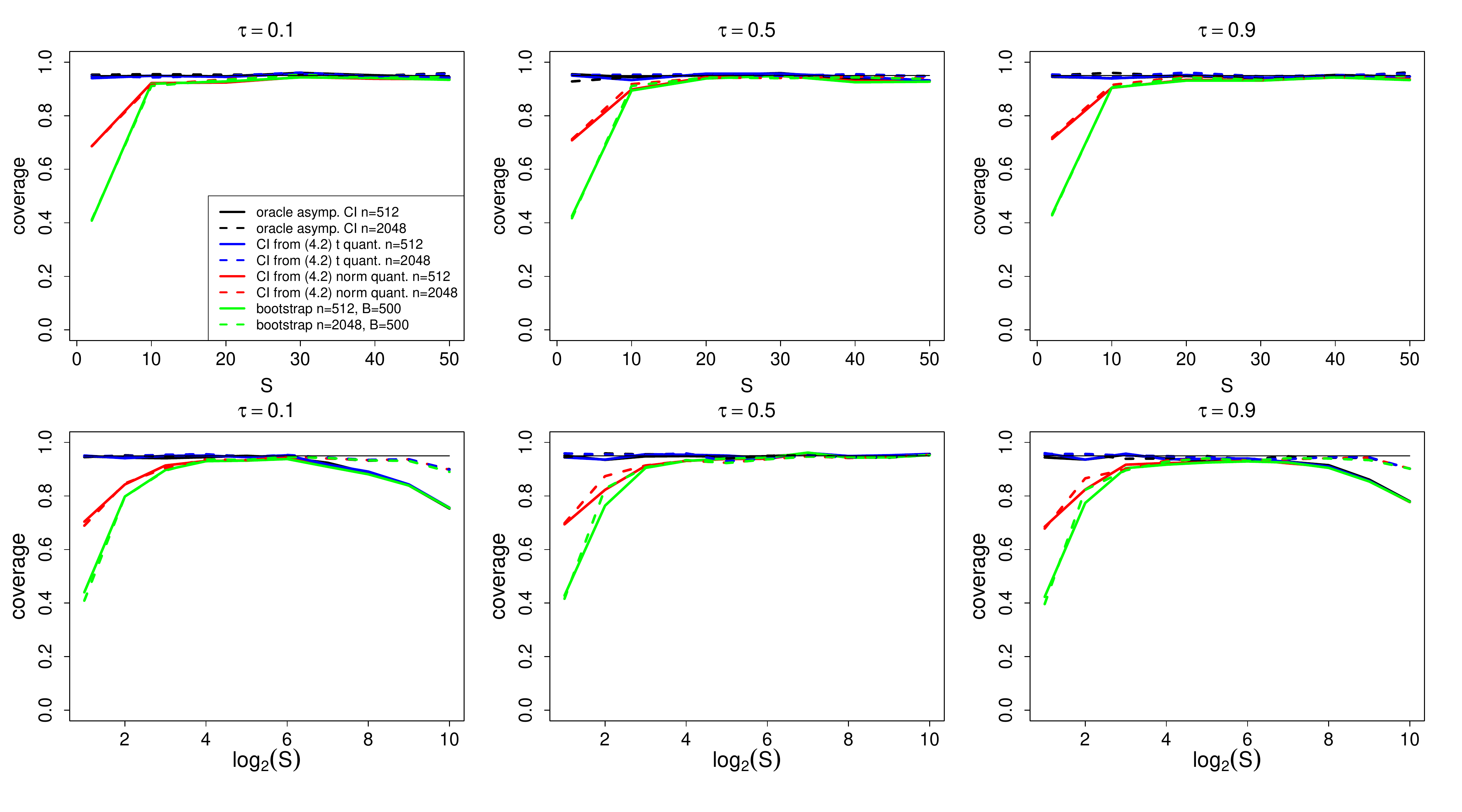}
\\\vspace{-0.2cm}
\centering {\scriptsize (b) $m=32$}\\
\includegraphics[width=13cm]{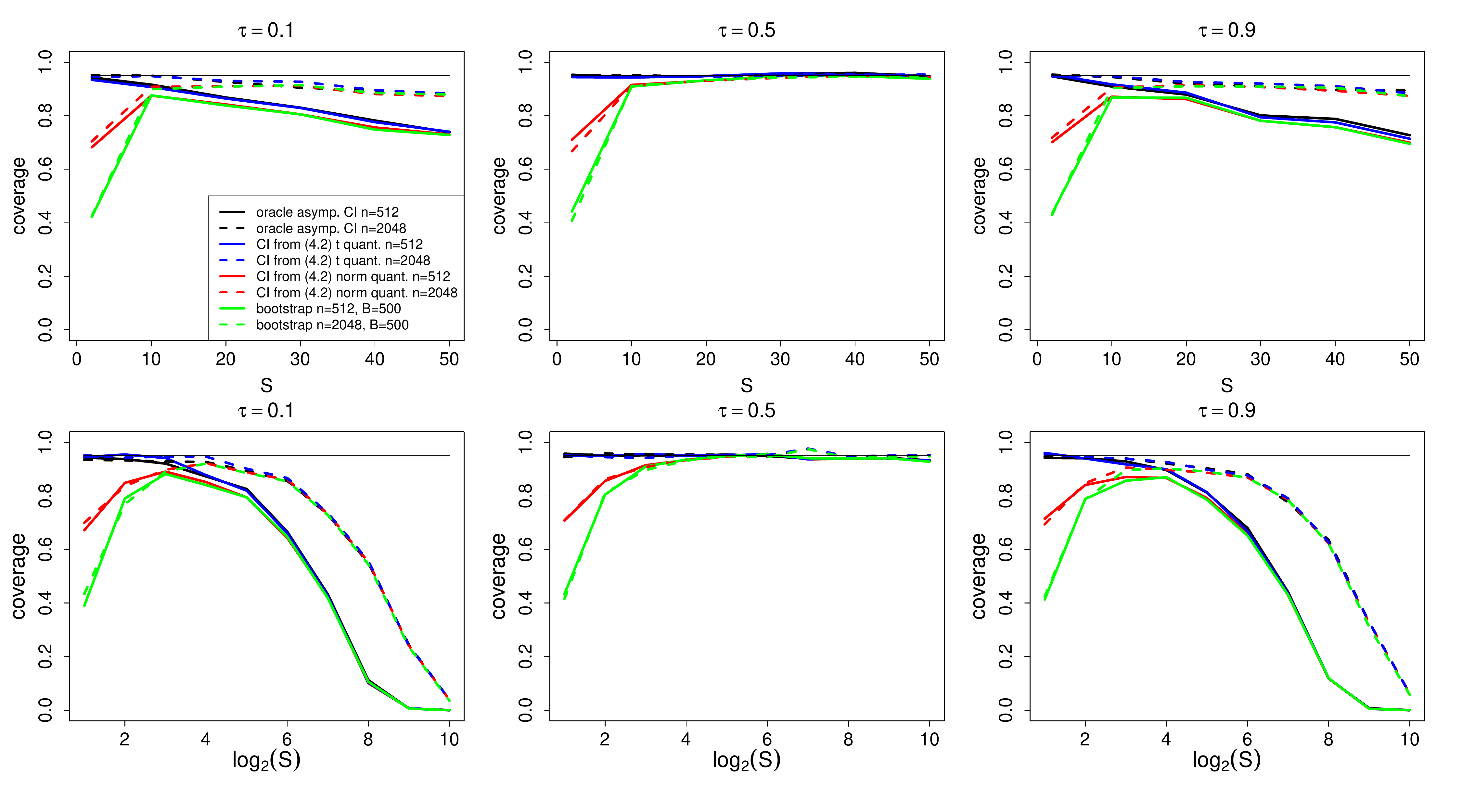}		
\caption{Coverage probabilities under the linear Gaussian location-scale shift model~\eqref{eq:simlin_het} for $x_0^\top\zb(\tau)$ for different values of $S$ and $\tau = 0.1, 0.5, 0.9$ (left, middle, right row). Solid lines: $n = 512$, dashed lines: $n = 2048$. Black: asymptotic oracle CI, blue: CI from~\eqref{ci:simple_t} based on t distribution, red: CI from~\eqref{ci:simple} based on normal distribution, green: bootstrap CI.
Throughout $x_0 = (1,...,1)/m^{1/2}$, nominal coverage $0.95$.}\label{fig:simlecoverageSmS_het}
\end{figure}

\begin{figure}[!h]
\centering
\includegraphics[width=13cm]{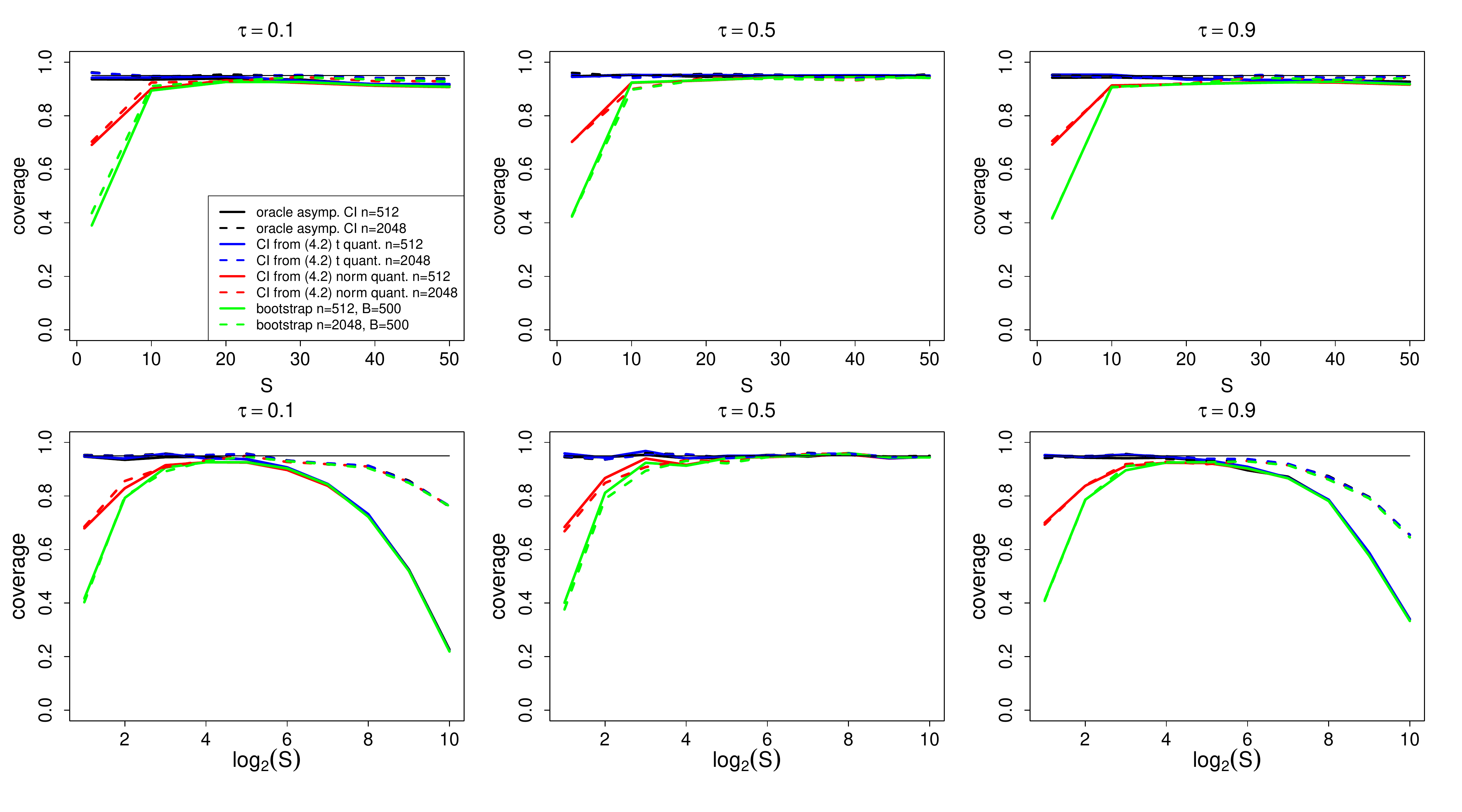}
\caption{Coverage probabilities for $Q(x_0;\tau)$ for $m=16$ under the linear Gaussian location-scale shift model~\eqref{eq:simlin_het}. Coverage probabilities for $x_0^\top\zb(\tau)$ for different values of $S$ and $\tau = 0.1, 0.5, 0.9$ (left, middle, right row). Solid lines: $n = 512$, dashed lines: $n = 2048$. Black: asymptotic oracle CI, blue: CI from~\eqref{ci:simple_t} based on t distribution, red: CI from~\eqref{ci:simple} based on normal distribution, green: bootstrap CI.
Throughout $x_0 = (1,...,1)/m^{1/2}$, nominal coverage $0.95$.}\label{fig:simlecoverageSmS16_het}
\end{figure}

\begin{table}[!h] \scriptsize
\begin{tabular}{c||c|c|c|c||c|c|c|c||c|c|c|c}
\hline
S & 1 & 10 & 30 & 50 & 1 & 10 & 30 & 50 & 1 & 10 & 30 & 50\\
\hline
&\multicolumn{4}{c||}{n = 512, m = 4, $\tau = 0.1$}&\multicolumn{4}{c||}{n = 512, m = 16, $\tau = 0.1$}&\multicolumn{4}{c}{n = 512, m = 32, $\tau = 0.1$}
\\
\hline
or  & 94.9 & 94.1 & 95.0 & 93.9 & 95.1 & 94.0 & 92.7 & 90.5 & 94.3 & 91.6 & 83.0 & 73.8 \\ 
def & 92.8 & 92.7 & 93.8 & 92.6 & 94.9 & 93.5 & 91.9 & 89.8 & 94.6 & 90.4 & 80.2 & 71.3 \\ 
nai & 93.7 & 93.1 & 94.4 & 93.0 & 97.7 & 93.2 & 91.2 & 88.5 & 99.4 & 91.6 & 80.9 & 72.0 \\ 
adj & 93.7 & 93.6 & 94.9 & 93.8 & 97.7 & 94.1 & 92.9 & 90.9 & 99.4 & 93.4 & 83.7 & 75.6 \\ 
\hline
&\multicolumn{4}{c||}{n = 512, m = 4, $\tau = 0.5$}&\multicolumn{4}{c||}{n = 512, m = 16, $\tau = 0.5$}&\multicolumn{4}{c}{n = 512, m = 32, $\tau = 0.5$}
\\
\hline
or  & 95.4 & 94.7 & 95.5 & 94.9 & 95.4 & 96.2 & 95.3 & 94.3 & 95.3 & 94.3 & 95.7 & 94.7 \\ 
def & 97.4 & 97.0 & 97.8 & 97.4 & 99.3 & 99.3 & 98.8 & 99.0 & 98.4 & 98.6 & 98.6 & 98.2 \\ 
nai & 96.0 & 95.6 & 96.1 & 95.8 & 98.6 & 98.0 & 97.3 & 97.1 & 99.0 & 97.5 & 97.2 & 97.4 \\ 
adj & 96.0 & 95.2 & 95.6 & 95.1 & 98.6 & 97.4 & 96.1 & 95.6 & 99.0 & 96.7 & 96.4 & 96.1 \\
\hline
&\multicolumn{4}{c||}{n = 512, m = 4, $\tau = 0.9$}&\multicolumn{4}{c||}{n = 512, m = 16, $\tau = 0.9$}&\multicolumn{4}{c}{n = 512, m = 32, $\tau = 0.9$}
\\
\hline
or  & 95.0 & 95.2 & 93.9 & 94.5 & 95.3 & 94.2 & 92.3 & 91.6 & 94.7 & 90.9 & 80.1 & 72.7 \\ 
def & 92.9 & 94.5 & 92.5 & 93.6 & 94.0 & 93.8 & 91.6 & 90.8 & 94.0 & 88.5 & 79.2 & 69.8 \\ 
nai & 93.8 & 95.0 & 93.0 & 93.7 & 97.1 & 93.3 & 90.8 & 89.8 & 99.4 & 90.0 & 80.1 & 70.6 \\ 
adj & 93.8 & 95.2 & 93.6 & 94.5 & 97.1 & 94.6 & 92.4 & 91.8 & 99.4 & 91.6 & 83.2 & 73.9 \\
\hline
\hline
&\multicolumn{4}{c||}{n = 2048, m = 4, $\tau = 0.1$}&\multicolumn{4}{c||}{n = 2048, m = 16, $\tau = 0.1$}&\multicolumn{4}{c}{n = 2048, m = 32, $\tau = 0.1$}
\\
\hline
or  & 95.1 & 96.3 & 94.9 & 95.1 & 95.6 & 94.1 & 94.8 & 93.8 & 95.2 & 95.0 & 90.5 & 88.2 \\ 
def & 94.0 & 95.8 & 94.2 & 94.8 & 94.6 & 93.5 & 93.7 & 92.8 & 94.7 & 94.3 & 89.3 & 87.0 \\ 
nai & 94.6 & 95.8 & 94.4 & 94.8 & 95.7 & 93.1 & 93.2 & 92.4 & 95.6 & 94.2 & 89.1 & 86.8 \\ 
adj & 94.6 & 96.0 & 94.6 & 95.0 & 95.7 & 93.8 & 94.4 & 93.4 & 95.6 & 94.9 & 90.2 & 87.9 \\
\hline
&\multicolumn{4}{c||}{n = 2048, m = 4, $\tau = 0.5$}&\multicolumn{4}{c||}{n = 2048, m = 16, $\tau = 0.5$}&\multicolumn{4}{c}{n = 2048, m = 32, $\tau = 0.5$}
\\
\hline
or  & 95.1 & 94.4 & 95.6 & 94.6 & 94.8 & 95.0 & 95.4 & 95.1 & 95.2 & 95.1 & 95.2 & 95.0 \\ 
def & 95.9 & 95.5 & 96.6 & 95.7 & 97.6 & 97.8 & 98.0 & 98.0 & 97.5 & 97.1 & 97.6 & 97.4 \\ 
nai & 95.4 & 94.9 & 96.2 & 95.0 & 96.6 & 96.8 & 96.6 & 96.4 & 97.0 & 96.4 & 96.4 & 96.4 \\ 
adj & 95.4 & 94.6 & 95.8 & 94.8 & 96.6 & 95.9 & 95.8 & 95.6 & 97.0 & 96.0 & 95.4 & 95.6 \\
\hline
&\multicolumn{4}{c||}{n = 2048, m = 4, $\tau = 0.9$}&\multicolumn{4}{c||}{n = 2048, m = 16, $\tau = 0.9$}&\multicolumn{4}{c}{n = 2048, m = 32, $\tau = 0.9$}
\\
\hline
or  & 95.1 & 94.9 & 95.6 & 94.6 & 95.0 & 94.8 & 94.2 & 93.2 & 95.3 & 94.5 & 91.0 & 89.4 \\ 
def & 94.2 & 94.4 & 95.1 & 94.1 & 94.0 & 94.2 & 93.4 & 92.1 & 94.8 & 93.8 & 90.1 & 88.2 \\ 
nai & 94.6 & 94.5 & 95.1 & 94.1 & 95.4 & 93.9 & 93.0 & 91.8 & 95.6 & 93.8 & 90.0 & 87.6 \\ 
adj & 94.6 & 94.7 & 95.6 & 94.4 & 95.4 & 94.4 & 94.0 & 93.2 & 95.6 & 94.8 & 91.1 & 89.4 \\ 
\hline
\end{tabular}
\caption{Coverage probabilities under the linear Gaussian location-scale shift mode~\eqref{eq:simlin_het} based on estimating the asymptotic variance. Different rows correspond to different methods for obtaining covariance matrix. or: using true asymptotic variance matrix, def: default choice implemented in quantreg package, nai: asymptotically optimal constant with scaling $h_n \sim n^{-1/5}$, adj: asymptotically optimal constant with scaling $h_n \sim N^{-1/5}$ as suggested by Theorem~\ref{th:powell}.}\label{tab:bwcomp_het}
\end{table}

\subsubsection{Results for the estimator $\hat F_{Y|X}(y|x)$}
In this section, we compare the coverage probabilities of the bootstrap confidence intervals for $F_{Y|X}(Q(x_0;\tau)|x_0)=\tau$ with the oracle asymptotic confidence interval~\eqref{eq:funcovp2}. To estimate $\hat F_{Y|X}$, we use the same number of quantile grid $K=65$ and the same number of equidistant knots $G=32$ for spline interpolation as in Section \ref{sec:simF}, which are sufficiently large to ensure the nominal coverage of the oracle confidence intervals. Coverage probabilities for $m=4,32$ are reported in Figure \ref{fig:Fhatcov_het}. Results for large values of $S$ and dimension $m=16$ are presented in Figure \ref{fig:Fhatcov_bigS_het} and \ref{fig:Fhatcov_m16_het}. The general patterns are similar to the homoskedastic model with overall coverage starting to drop earlier.

\begin{figure}[H]
\includegraphics[width=13cm]{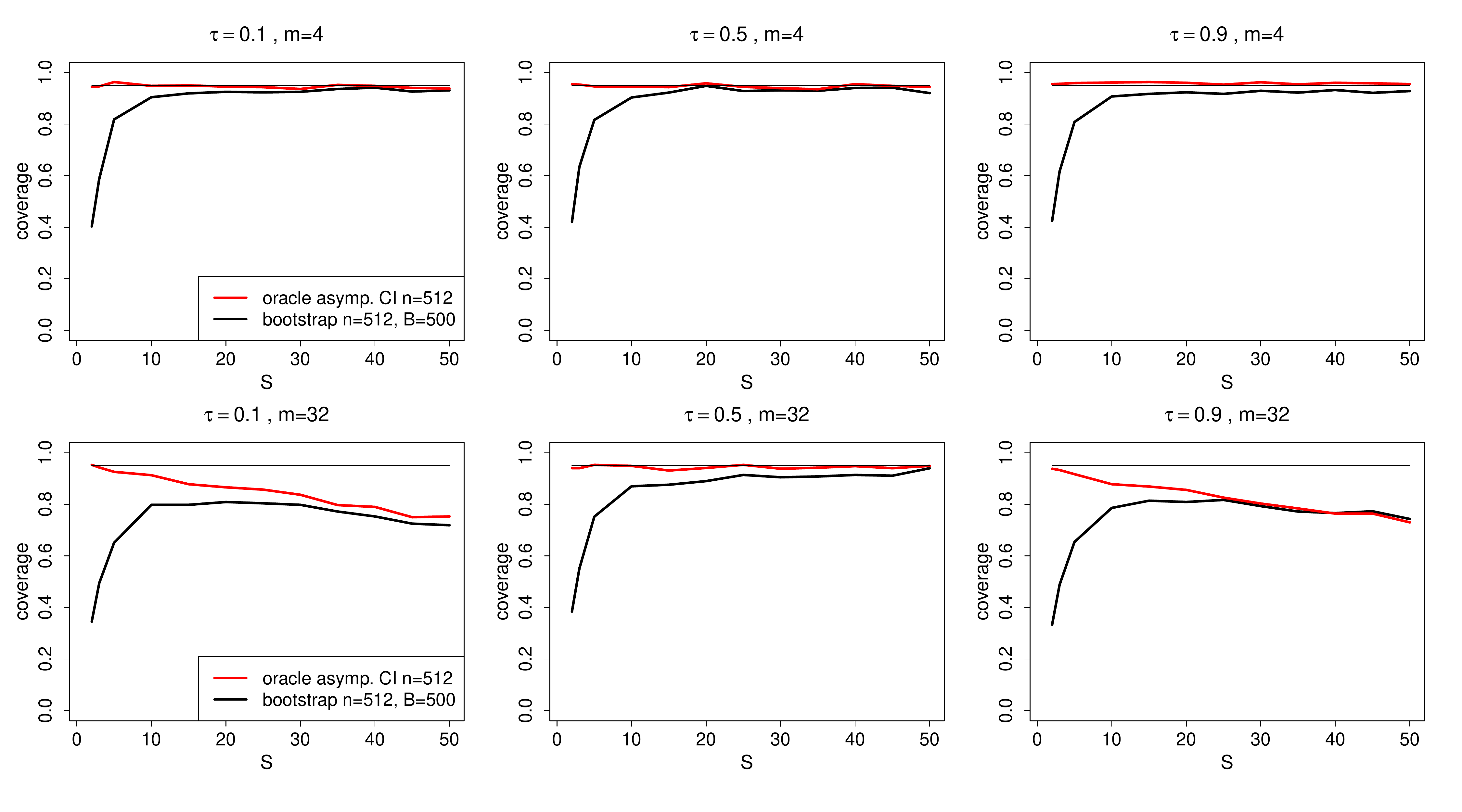}	
\caption{Coverage probabilities for oracle confidence intervals (red) and bootstrap confidence intervals (black) for $F_{Y|X}(Q(x_0;\tau)|x_0)=\tau$ under the linear Gaussian location-scale shift model~\eqref{eq:simlin_het} for $x_0 = (1,...,1)/m^{1/2}$ and $y = Q(x_0;\tau)$, $\tau = 0.1, 0.5, 0.9$. $n = 512$ and nominal coverage $0.95$.}\label{fig:Fhatcov_het}
\end{figure}

\begin{figure}[H]
\centering
\includegraphics[width=13cm]{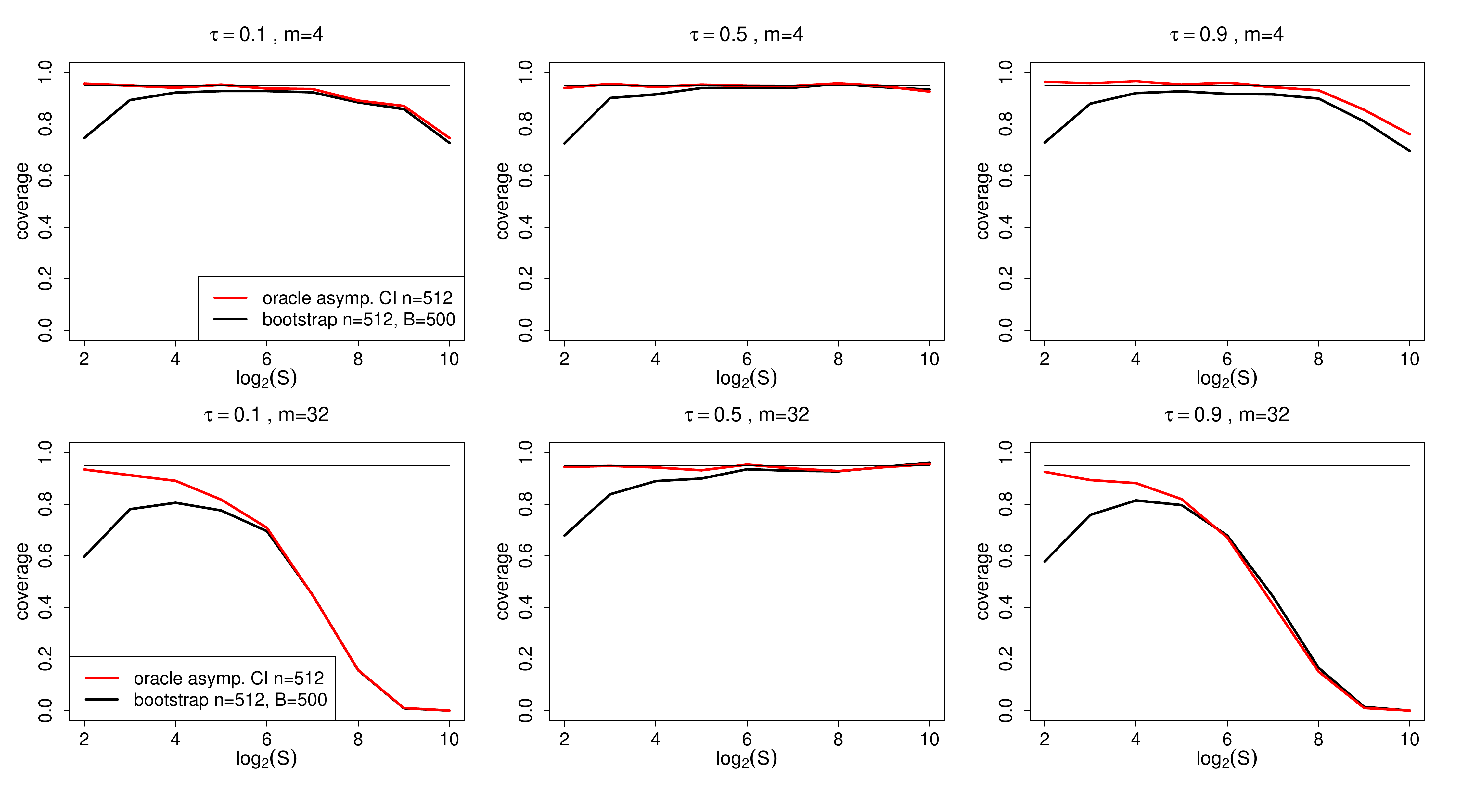}
\caption{Coverage probabilities for $F_{Y|X}(Q(x_0;\tau)|x_0)=\tau$ under the linear Gaussian location-scale shift model~\eqref{eq:simlin_het} for $m=4,32$ with large number of subsamples $S$, $n = 512$ and nominal coverage $0.95$. Red: oracle asymptotic CI, black: bootstrap CI with $B=500$.}\label{fig:Fhatcov_bigS_het}
\end{figure}

\begin{figure}[H]
\centering
\includegraphics[width=13cm]{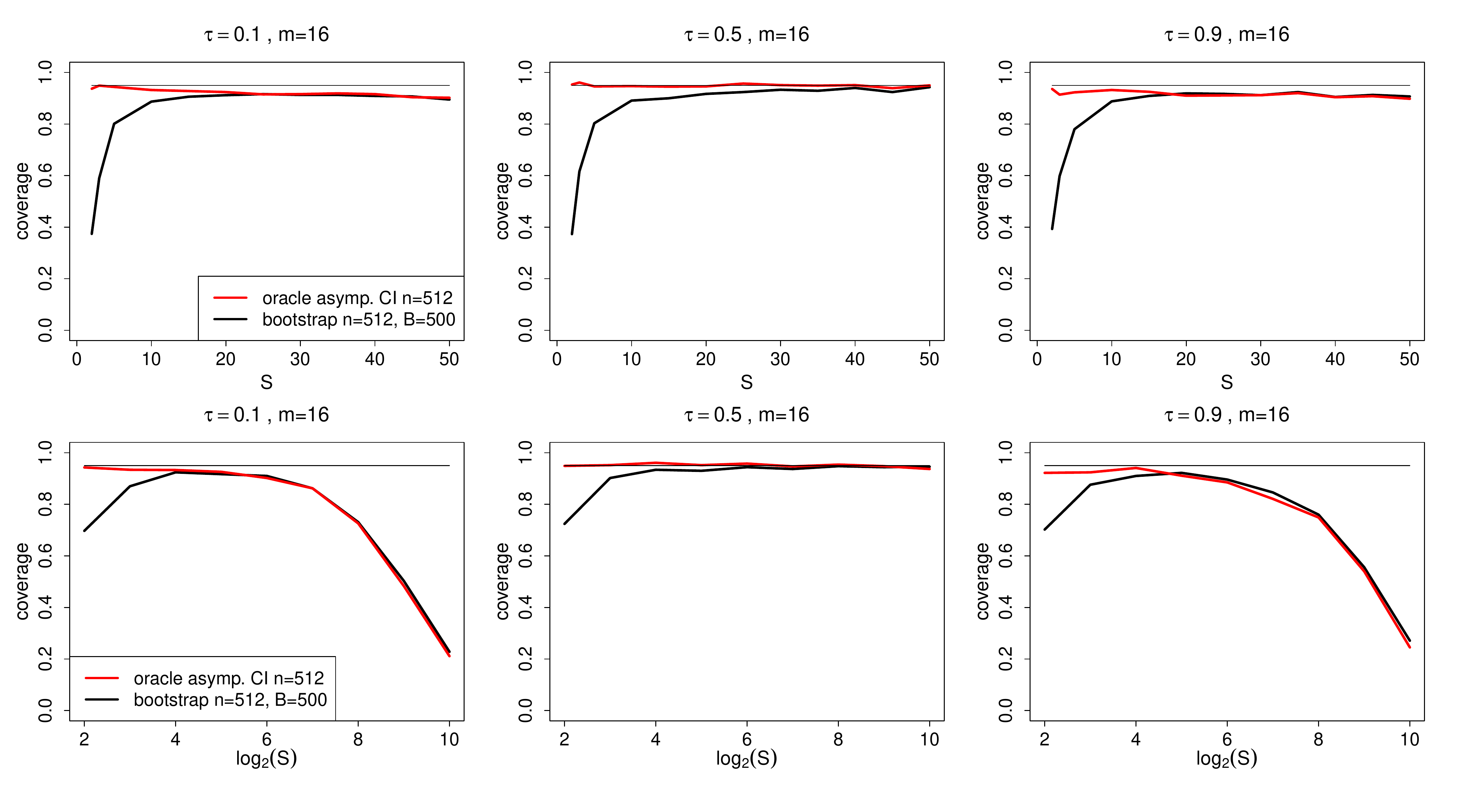}
\caption{Coverage probabilities for $F_{Y|X}(Q(x_0;\tau)|x_0)=\tau$ under the linear Gaussian location-scale shift model~\eqref{eq:simlin_het} for $m=16$, $n = 512$ and nominal coverage $0.95$. Red: oracle asymptotic CI, black: bootstrap CI with $B=500$.} \label{fig:Fhatcov_m16_het}
\end{figure}

\subsection{Additional simulations for the linear model}\label{sec:silin}

The purpose of the simulations in this section is to show the thresholds for the number of sub-samples $S^*$ and the number of quantile grids $K^*$ that guarantee the oracle properties of $\overline\zb(\tau)$ and $\hat\zb(\tau)$ predicted by the sufficient and necessary conditions in the theorems in Section \ref{SEC:THEO}. Besides normal location shift models, we also consider models with exponential error distributions. Using the linear model \eqref{eq:simlin} in Section \ref{sec:sim}, we show the coverage probabilities of the oracle asymptotic confidence intervals, i.e. the confidence intervals based on the true asymptotic variances. Similar to Section \ref{sec:sim}, we select $\tau=0.1,0.5,0.9$, $m=4,16,32$ and $\Tc = [0.05,0.95]$. Section \ref{sec:linfix} contains results for $\overline{\zb}(\tau)$ defined in \eqref{eq:zbbar}, while Section \ref{sec:lincdf} discusses $\hat F_{Y|X}(y|x)$ defined in \eqref{eq:Fhat}. 

%

\subsubsection{Oracle rule for $\overline\zb(\tau)$}\label{sec:linfix}

Recall that from Corollary \ref{th:orlipo}, an asymptotic $1-\alpha$ confidence interval for $x_0^\top \zb(\tau)$ is 
	\begin{align}
	\big[x_0^\top \overline{\zb}(\tau) \pm N^{-1/2} f_{\varepsilon,\tau}^{-1}\sqrt{\tau(1-\tau) x_0^\top \Sigma_X^{-1}x_0}\Phi^{-1}(1-\alpha/2)\big],\label{eq:lincovp2}
	\end{align}
	where $f_{\varepsilon,\tau}=f_\epsilon(F_\epsilon^{-1}(\tau))$, $f_\epsilon$ is the error density and $\Sigma_X = \E[(1,X_i^\top)^\top(1,X_i^\top)]$. Here, we set $x_0 = (1,(m-1)^{-1/2}\bl_{m-1}^\top)^\top$, where $\bl_{m-1}$ is an $(m-1)$ vector with each entry equal to 1. 
	
	We verify the oracle rule by checking whether the empirical coverage probability of \eqref{eq:lincovp} equals $1-\alpha=95\%$. Figure \ref{fig:simlincovp} shows the coverage probabilities for $\varepsilon\sim\Nc(0,\sigma^2)$. In all plots, the coverage probabilities reach the nominal level $95\%$ for $S<S^*$, and then drop quickly to 0 for $S > S^*$ for some $S^*$. When $N$ increases, $S^*$ shifts toward $N^{1/2}$, which is the sharp upper bound in Corollary \ref{th:orlipo}. Also, when $N$ is fixed, a larger dimensionality $m$ leads to a smaller $S^*$. For $\varepsilon\sim \mbox{Exp}(\lambda)$, Figure \ref{fig:simlincovp_exp} shows that the coverage probabilities are no longer symmetric in $\tau$. When $\tau$ is small, $N$ needs to be large enough to attain the same magnitude of $S^*$ due to the skewness of $\mbox{Exp}(\lambda)$. 

\subsubsection{Oracle rule for $\widehat F_{Y|X}(y|x)$}\label{sec:lincdf}

We compute $\hat F_{Y|X}(y_0|x_0)$ defined in \eqref{eq:Fhat} discretely with an equidistant partition of size $1000$ on $[\tau_L,\tau_U]$. The involved $\hat\zb(\tau)$ is computed as in \eqref{eq:hatbeta} with $\BB$ being cubic B-spline with dimension $q$ defined on $G=4+q$ knots. The knots form a partition on $[\tau_L,\tau_U]$ with repetitions on the boundary; see Corollary 4.10 of \cite{schumaker:81}. For simplicity, we set $y_{0} = Q(x_0;\tau)$ so that $F_{Y|X}(y_0|x_0)=\tau$, where $x_0$ is chosen as in Section \ref{sec:linfix}. Recall from Corollary \ref{cor:fhatli} that an asymptotic $1-\alpha$ confidence interval for $F_{Y|X}(Q(x_0;\tau)|x_0) = \tau$ is 
	\begin{align}
	\big[\hat F_{Y|X}(Q(x_0;\tau)|x_0) \pm N^{-1/2} \sqrt{\tau(1-\tau) x_0^\top \Sigma_X^{-1}x_0} \Phi^{-1}(1-\alpha/2)\big].\label{eq:funcovp}
	\end{align}
We fix $N=2^{14}$ in this section.

We demonstrate here the impact of number of subsamples $S$, model dimension $m$, number of basis $q$ and number of quantile grid points $K$ on the coverage probability of \eqref{eq:funcovp}. We note that $q$ and $K$ have similar impact on coverage probabilities given that other parameters are held fixed. 
In Figure \ref{fig:simlincovp_fun} with $\varepsilon\sim N(0,0.1^2)$, we note that an increase in $q=\dim(\BB)$ improves the coverage probabilities given that $m,S$ are held fixed. Also, at a fixed $S$, $q$ can be chosen smaller if $m$ is larger. This is consistent with Corollary \ref{th:orlocpr} that requires $q\gg N^{1/(2\eta)}\|\ZZ(x_0)\|^{-\eta}$. Similar to Section \ref{sec:linfix}, an increase of the dimensionality $m$ leads to a smaller $S^*$ for any fixed $q$. For $\varepsilon\sim\mbox{Exp}(0.8)$ shown in Figure \ref{fig:simlincovp_funexp}, a difference to the normal case is that the performance of the coverage probability is better when $\tau$ is small.

\subsection{Nonparametric model}\label{sec:sinp}
In this section, we consider a nonlinear model
\begin{align}
		Y_i = 2.5 + \sin(2X_i)+2\exp(-16X_i^2)+0.7\epsilon_i, \quad X_i \sim \Uc(-1,1), \quad \epsilon_i \sim \Nc(0,1),\label{eq:npmodel}
\end{align}	
where the function $x\mapsto 2.5 + \sin(2x)+2\exp(-16x^2)$ is plotted in Figure \ref{fig:npfun}. The basis $\ZZ$ is set as cubic B-spline defined at $m+4$ knots. The knots form a partition on $[-1,1]$ with repetitions on the boundary; see Corollary 4.10 of \cite{schumaker:81}. The model \eqref{eq:npmodel} implies the quantile of $Y$ is $Q(x;\tau)=2.5 + \sin(2x)+2\exp(-16x^2)+0.7\Phi^{-1}(\tau)$ for $0<\tau<1$. 
 
Section \ref{sec:npfix} concerns the coverage probabilities of the confidence intervals for $Q(x;\tau)$ for fixed $\tau$. Section \ref{sec:npcdf} deals with the coverage probabilities of the confidence intervals for $F_{Y|X}(y|x)$. In both sections, we fix $N=2^{16}$.

\subsubsection{Oracle rule for $\overline{\zb}(\tau)$}\label{sec:npfix}
According to Corollary \ref{th:orloc}, an asymptotic $1-\alpha$ confidence interval for $Q(x_0;\tau)$ is
\begin{align}
	\big[\ZZ(x_0)^\top \overline{\zb}(\tau) \pm N^{-1/2} \phi_{\sigma}^{-1}(\Phi_\sigma^{-1}(\tau))\sqrt{\tau(1-\tau)} \sigma_0(\ZZ) \Phi^{-1}(1-\alpha/2)\big],\label{eq:npcovp}
	\end{align}
	where $\sigma_0^2(\ZZ) = \ZZ(x_0)^\top \E[\ZZ(X) \ZZ^\top(X)]^{-1}\ZZ(x_0)$ and $\phi_{\sigma}$ is the density of $\Nc(0,\sigma^2)$. 
	
Our simulation results show how $S$ and $m$ (the number of basis functions $m=\dim(\ZZ)$) influence the coverage probabilities of \eqref{eq:npcovp}. For all the plots in Figure \ref{fig:simnpcovp_fix}, the coverage probabilities corresponding to $m=13$ at $x_0=0$ (the solid green curves) performs the worst. This is partly caused by the large bias from the large curvature of $Q(x_0; \tau)$ at $x_0=0$ (see Figure \ref{fig:npfun}). This bias can be reduced by setting greater $m$, but the increase of $m$ results in a smaller $S^*$.

\subsubsection{Oracle rule for $\widehat F_{Y|X}(y|x)$}\label{sec:npcdf}
We compute $\hat F_{Y|X}(y|x)$ in a similar way as Section \ref{sec:lincdf} by taking $y_0 = Q(x_0;\tau)$, which gives $F_{Y|X}(y_0|x_0)=\tau$. By Corollary \ref{th:orlocpr}, an $1-\alpha$ confidence interval for $\tau$ is
	\begin{align}
\big[\hat F_{Y|X}(Q(x_0;\tau)|x_0) \pm N^{-1/2} \sqrt{\tau(1-\tau)} \sigma_0(\ZZ)\Phi^{-1}(1-\alpha/2)\big],\label{eq:npfuncovp}
	\end{align}
where $\sigma_0^2(\ZZ)$ is defined as in \eqref{eq:npcovp}. 

We show the impact of number of subsamples $S$, model dimension $m$, number of basis $q$ and number of quantile grid points $K$ on the coverage probability of \eqref{eq:npfuncovp} in Figure \ref{fig:simnpcovp_inter}. We only show $\tau=0.1$ and omit the results for $\tau=0.5, 0.9$, as they do not show additional insights to Figure \ref{fig:simlincovp_fun}. It can be seen from Figure \ref{fig:simnpcovp_inter} that larger $m$ lead to a smaller estimation bias, and thus improves performance under fixed $S,q$ and $K$. However, such an increase of $m$ also results in smaller $S^*$. Under a fixed $m$, the influence from $S,q$ and $K$ on the coverage probabilities is similar to the linear models with normal errors in Section \ref{sec:lincdf}. These findings are consistent with Corollary \ref{th:orlocpr}.

\subsection{Practical conclusions from simulations} \label{sec:simprac}

In this section, we briefly comment on some practical recommendations regarding the choice of $S,K$. Our general recommendation would be to choose $S$ as small as possible and $K$ as large as possible since the only advantage of choosing large $S$ or small $K$ are computational savings. Both theory and simulations suggest that for models with more predictors, particular care in choosing $S$ is necessary. Giving very specific default recommendations is difficult since, as seen from the simulations, the exact value of $\log_N(S) $ when coverage starts to drop below nominal values depends on the quantile at hand and (unknown) conditional density of response given predictors.

As a rule of thumb, we would recommend to ensure that the 'effective sub-sample size' $n/m$ is at least $50$ and that $S$ does not exceed $n/m$. Additionally, if the spread of conditional quantiles as a function of $\tau$ changes rapidly, extra caution needs to be taken. This recommendation is based on Figure~\ref{fig:simlincovp_exp} where the conditional density of response given predictors is exponential and coverage starts to drop earlier for smaller quantiles, which corresponds to regions where the second derivative of the conditional quantile function is larger. This is also in line with our derivation of lower bonds for the oracle rule, where the second derivative of the conditional quantile function enters as constant $a$ (see Section~\ref{sec:specconst})  

Regarding $K$, $K > N^{1/4}$ worked well in all simulations we considered. Again, the precise constant will depend on the roughness of the function $\tau \mapsto Q(x;\tau)$. The choice of $G$ should depend on $K$. Although the formal statement of our theorem requires $K/G \to \infty$, a look at the proof reveals that for cubic splines $K/G > c$ for some constant $c$ is sufficient. Simulations indicate that $G = K/2$ combined with cubic splines is a reasonable choice.


\newpage
	\begin{figure}[!ht]
	\centering
	\includegraphics[width=4cm, height = 3.2cm]{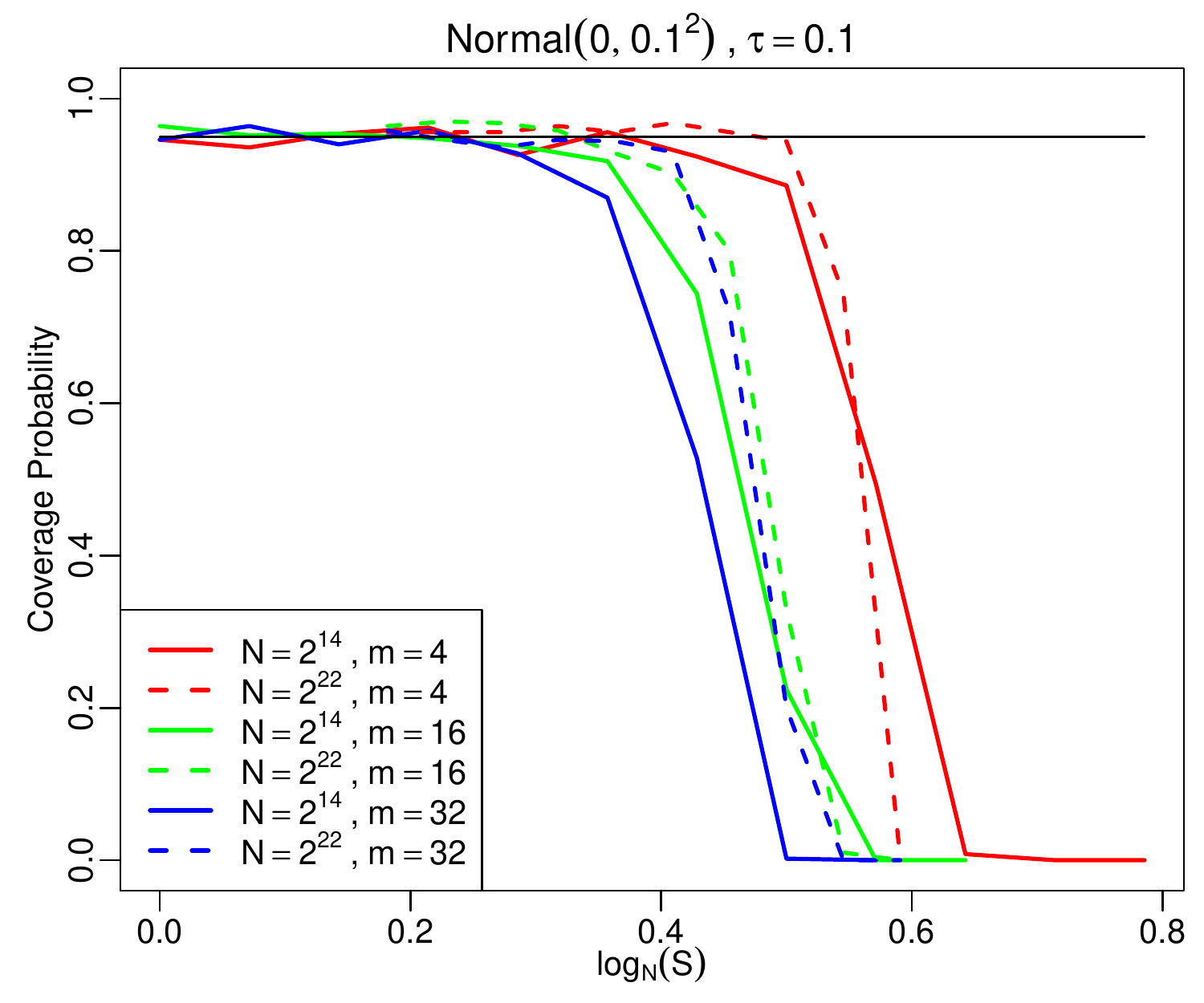}
	\includegraphics[width=4cm, height = 3.2cm]{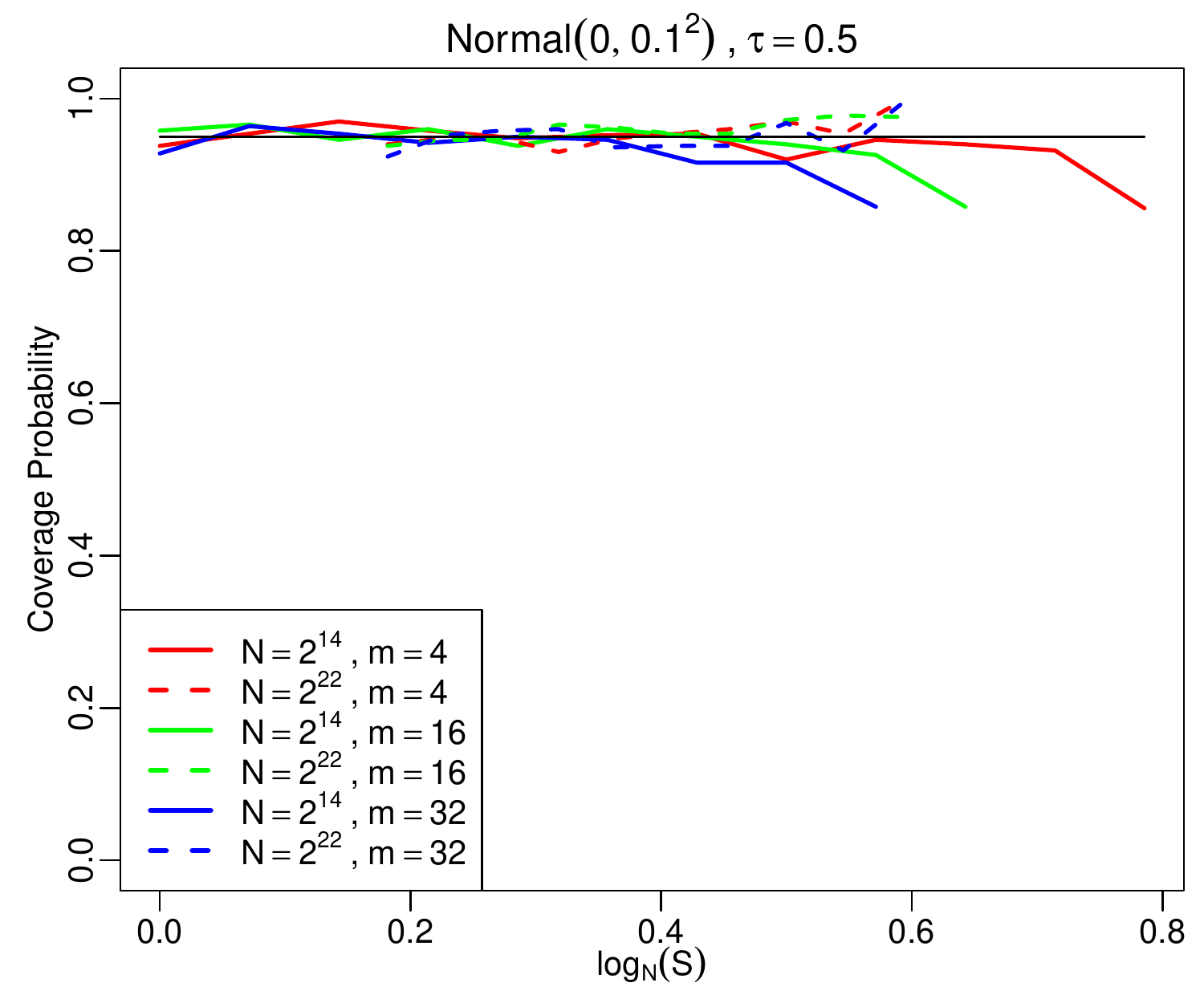}
	\includegraphics[width=4cm, height = 3.2cm]{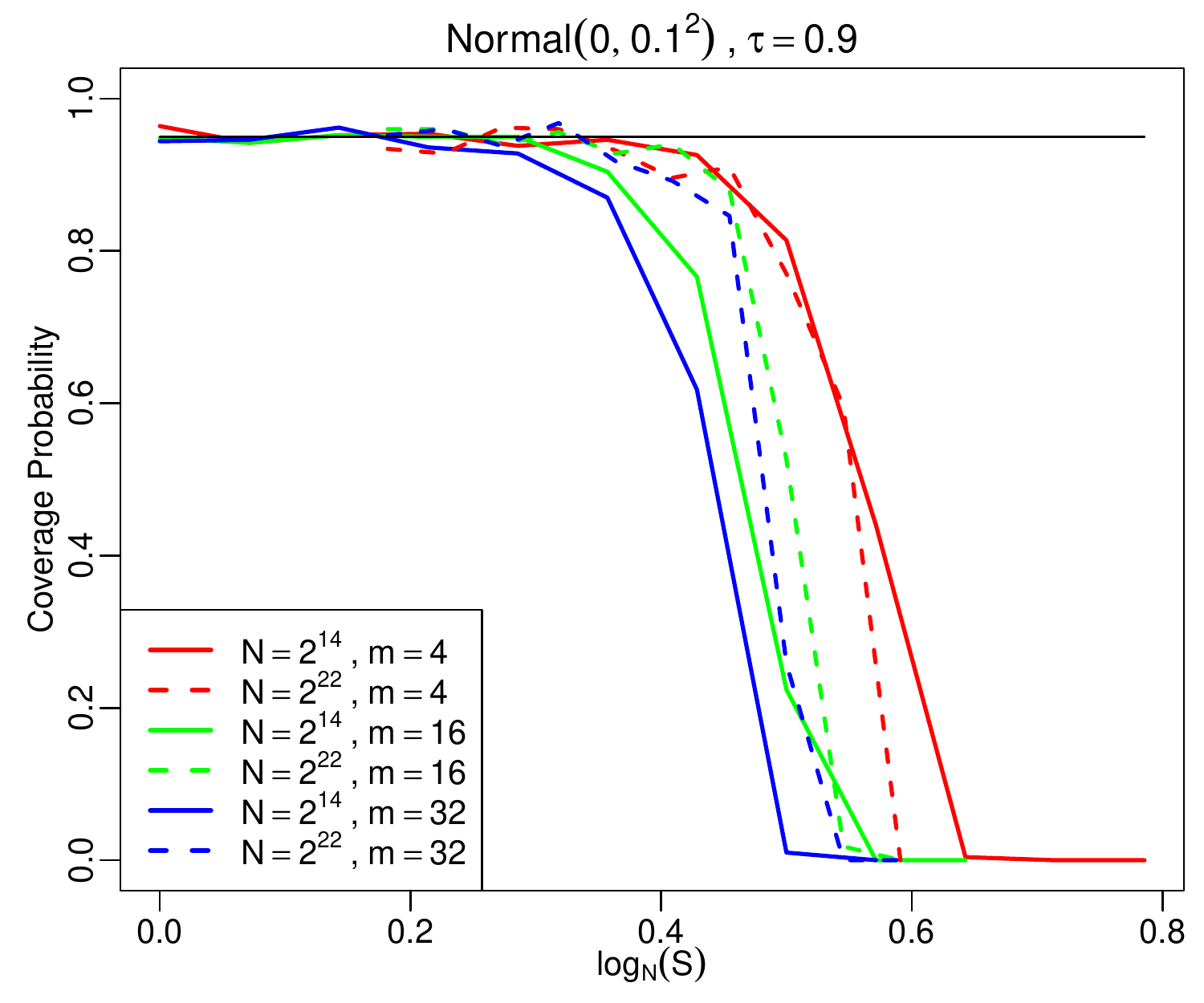}\\
	\includegraphics[width=4cm, height = 3.2cm]{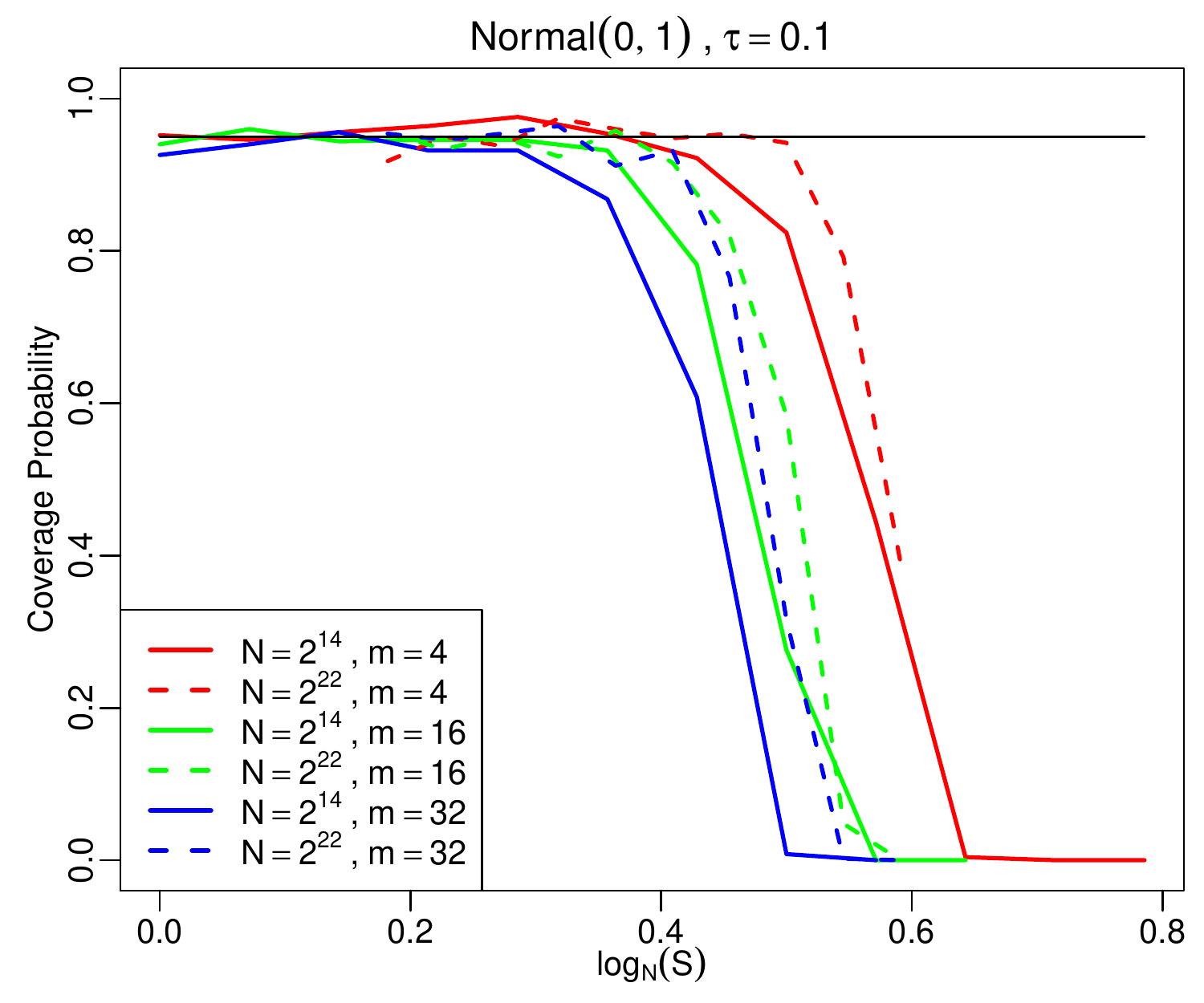}
	\includegraphics[width=4cm, height = 3.2cm]{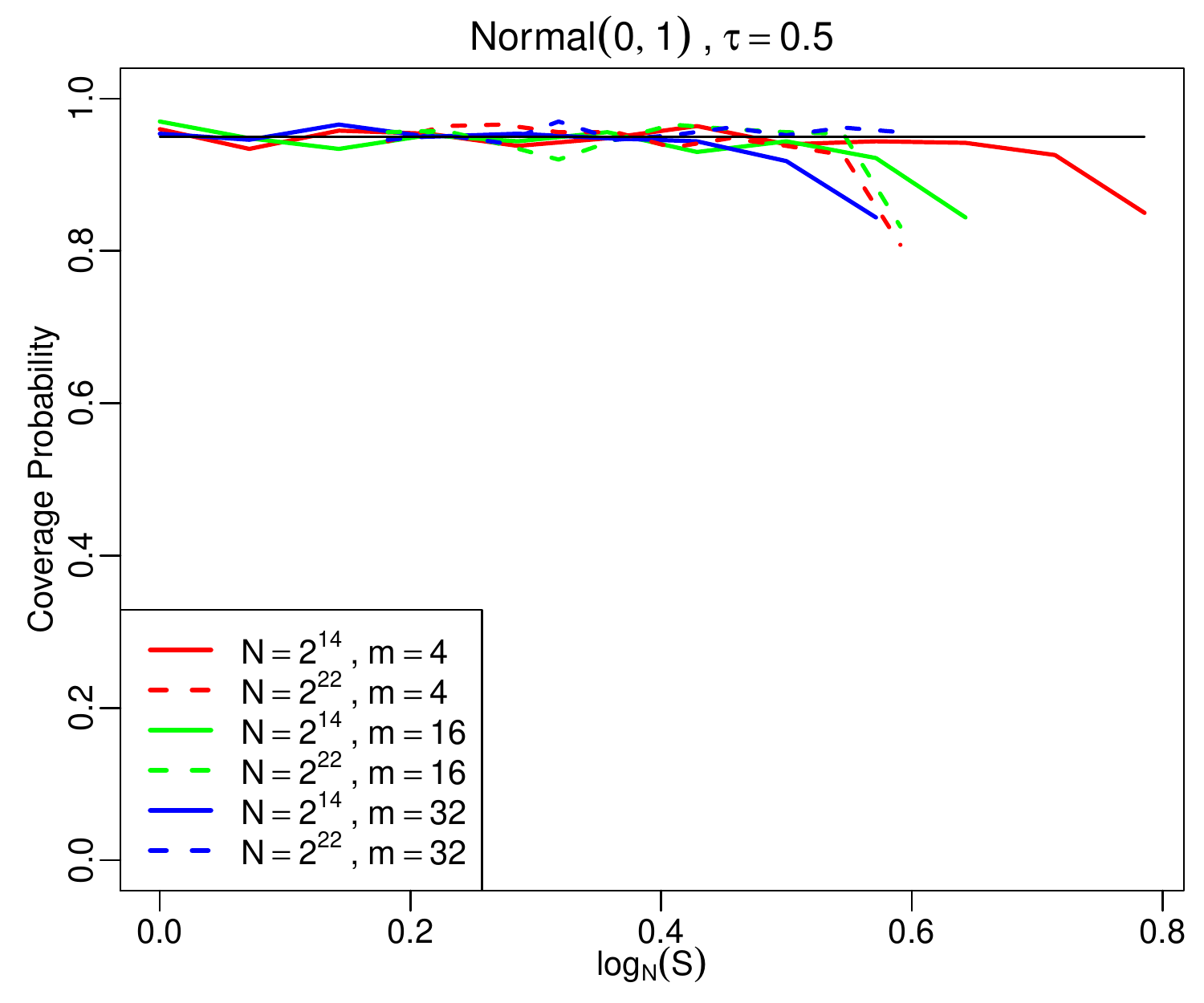}
	\includegraphics[width=4cm, height = 3.2cm]{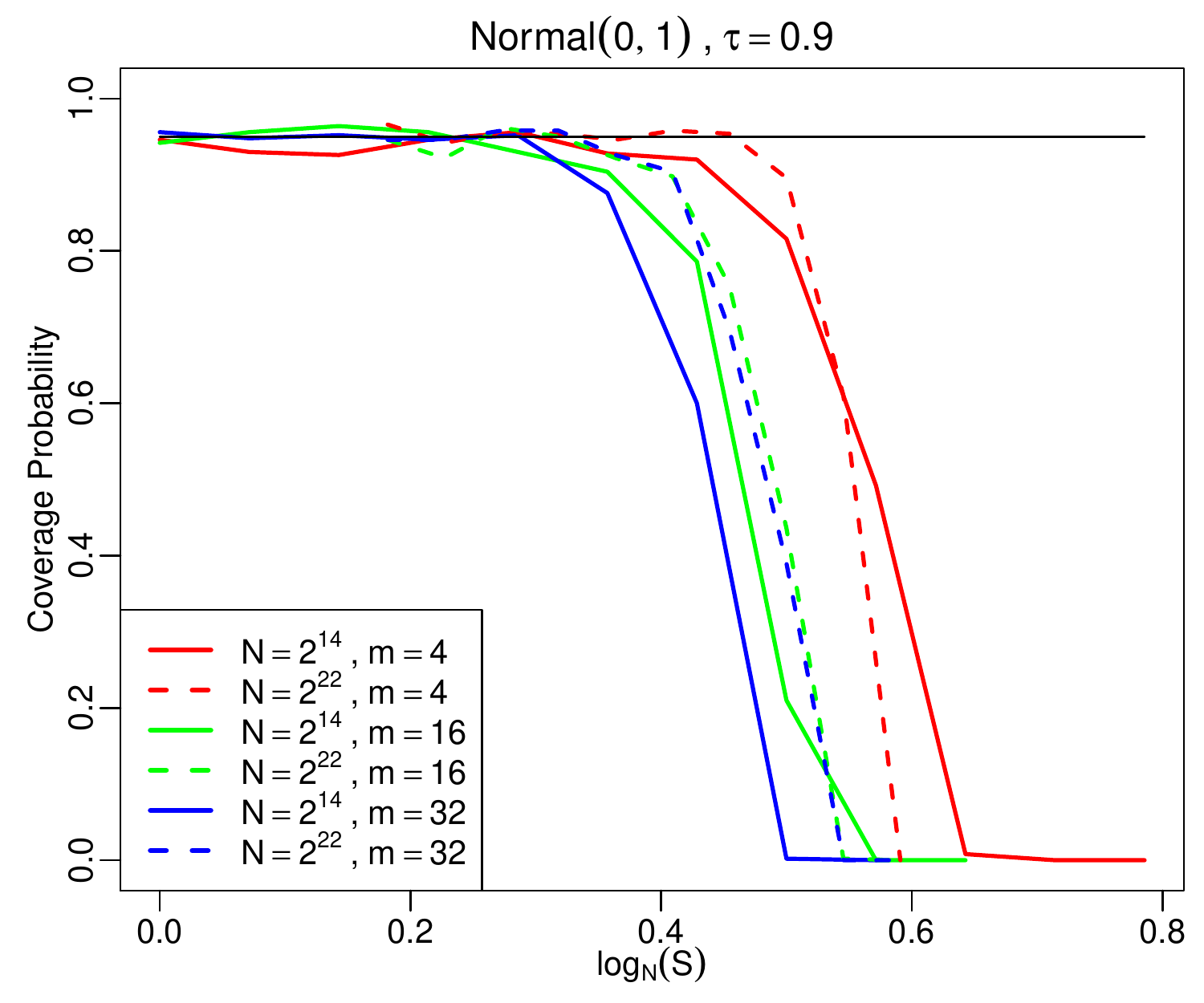}
	\caption{Coverage probability of the $1-\alpha=95\%$ confidence interval for $Q(x_0;\tau)$ in \eqref{eq:lincovp}: $\big[x_0^\top \overline{\zb}(\tau) \pm  f_{\varepsilon,\tau}^{-1}(\tau(1-\tau) x_0^\top \Sigma_X^{-1}x_0/N)^{1/2}\Phi^{-1}(1-\alpha/2)\big]$ under the model $Y_i = 0.21 +  \zb_{m-1}^\top X_i+\varepsilon_i$ with $\zb_{m-1}$ as \eqref{eq:betaspeci} and $\varepsilon\sim\Nc(0,\sigma^2)$, where $N$ is the total sample size, $S$ is the number of subsamples and $m=1+\dim(\zb_{m-1})$.}\label{fig:simlincovp}
	\end{figure}	
	
	\begin{figure}[!ht]
	\centering
	\includegraphics[width=4cm, height = 3.2cm]{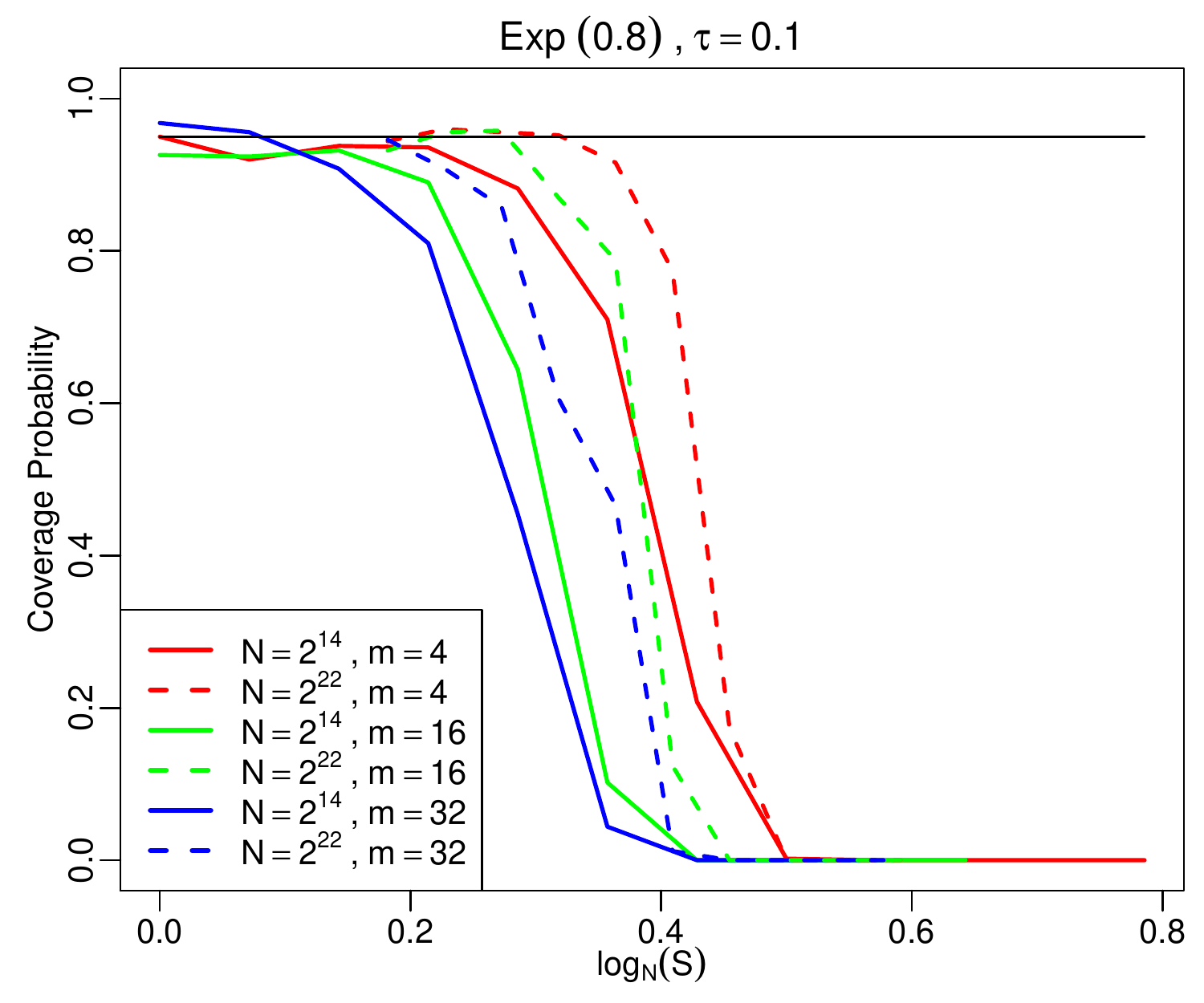}
	\includegraphics[width=4cm, height = 3.2cm]{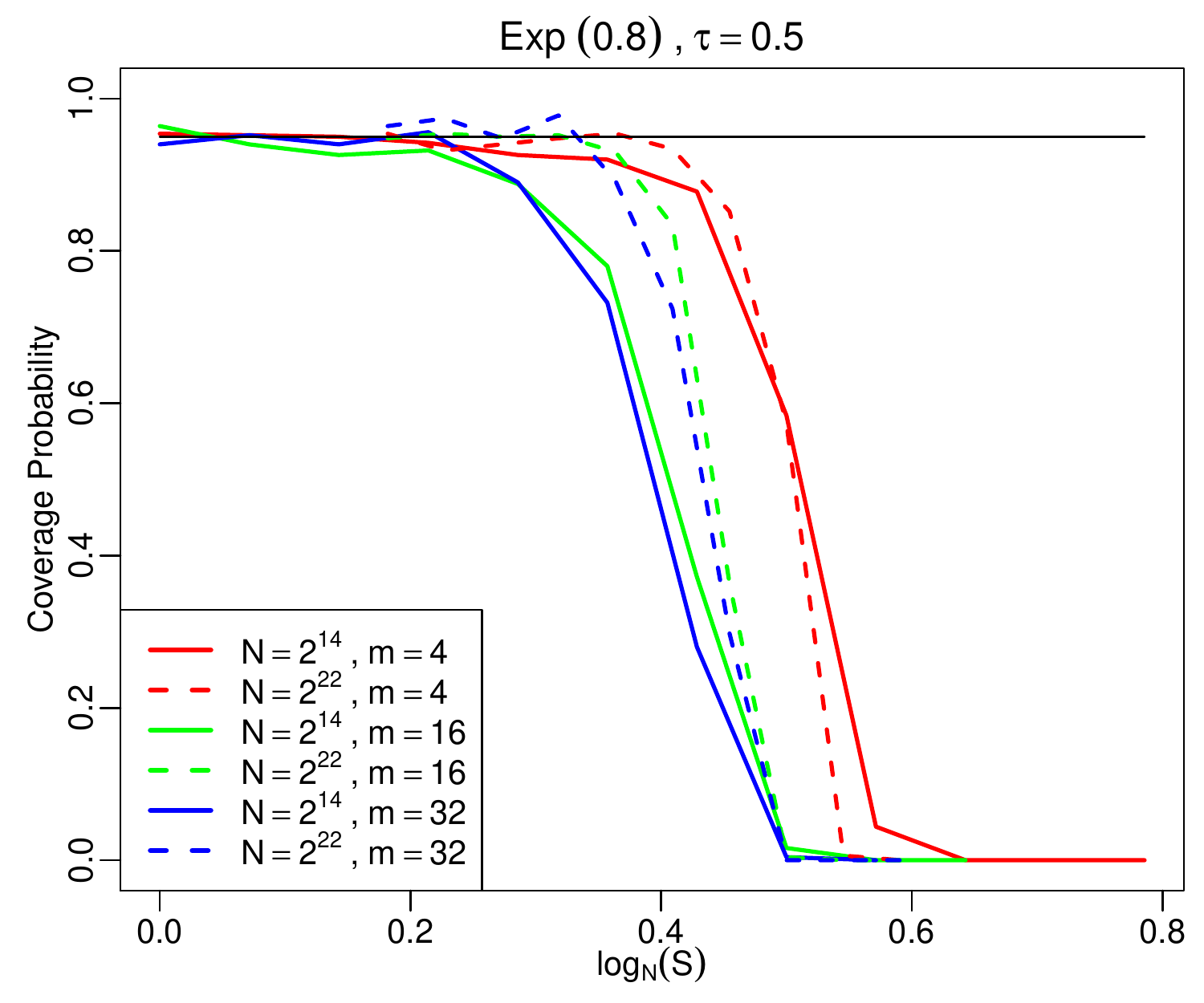}
	\includegraphics[width=4cm, height = 3.2cm]{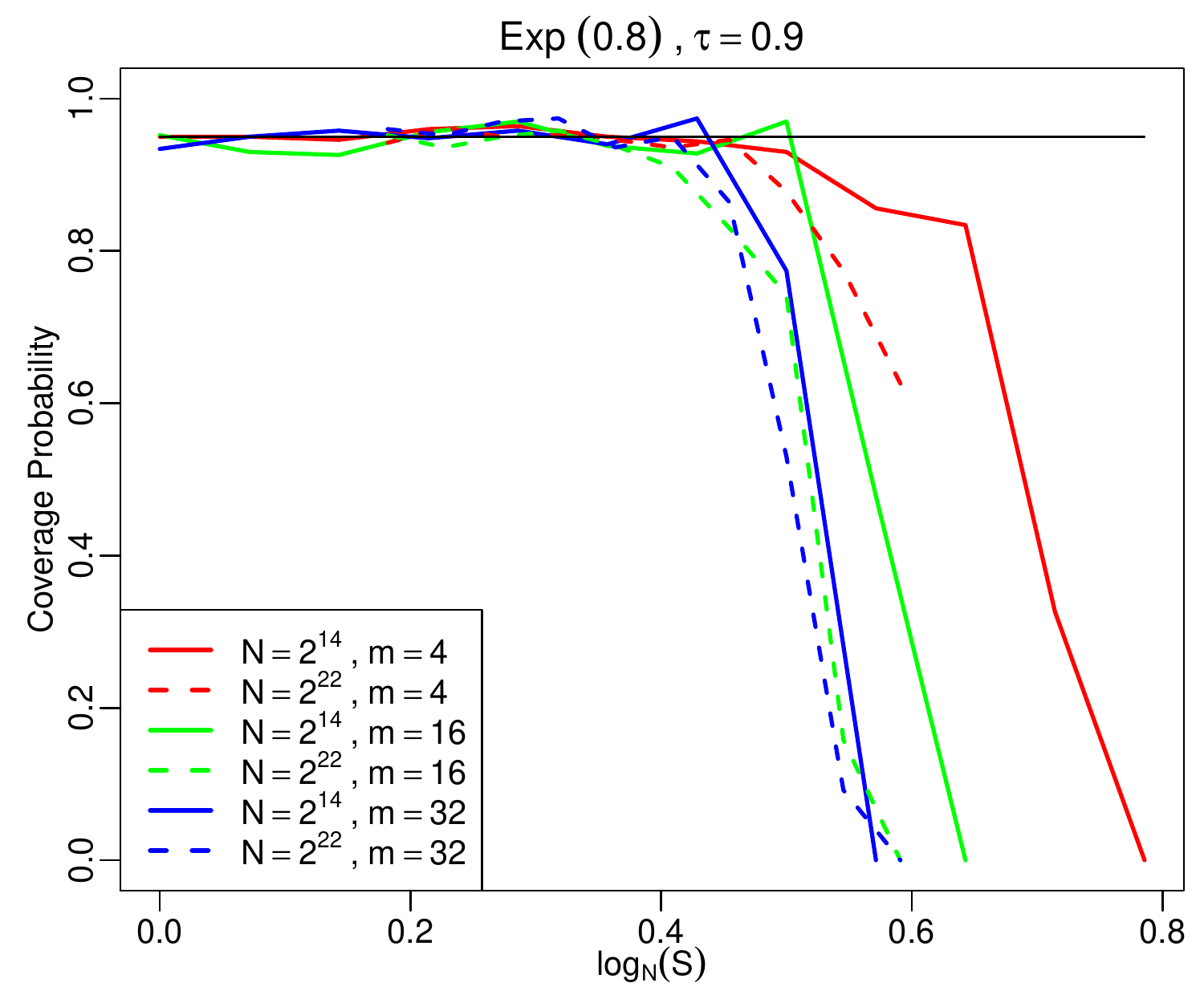}\\
	\includegraphics[width=4cm, height = 3.2cm]{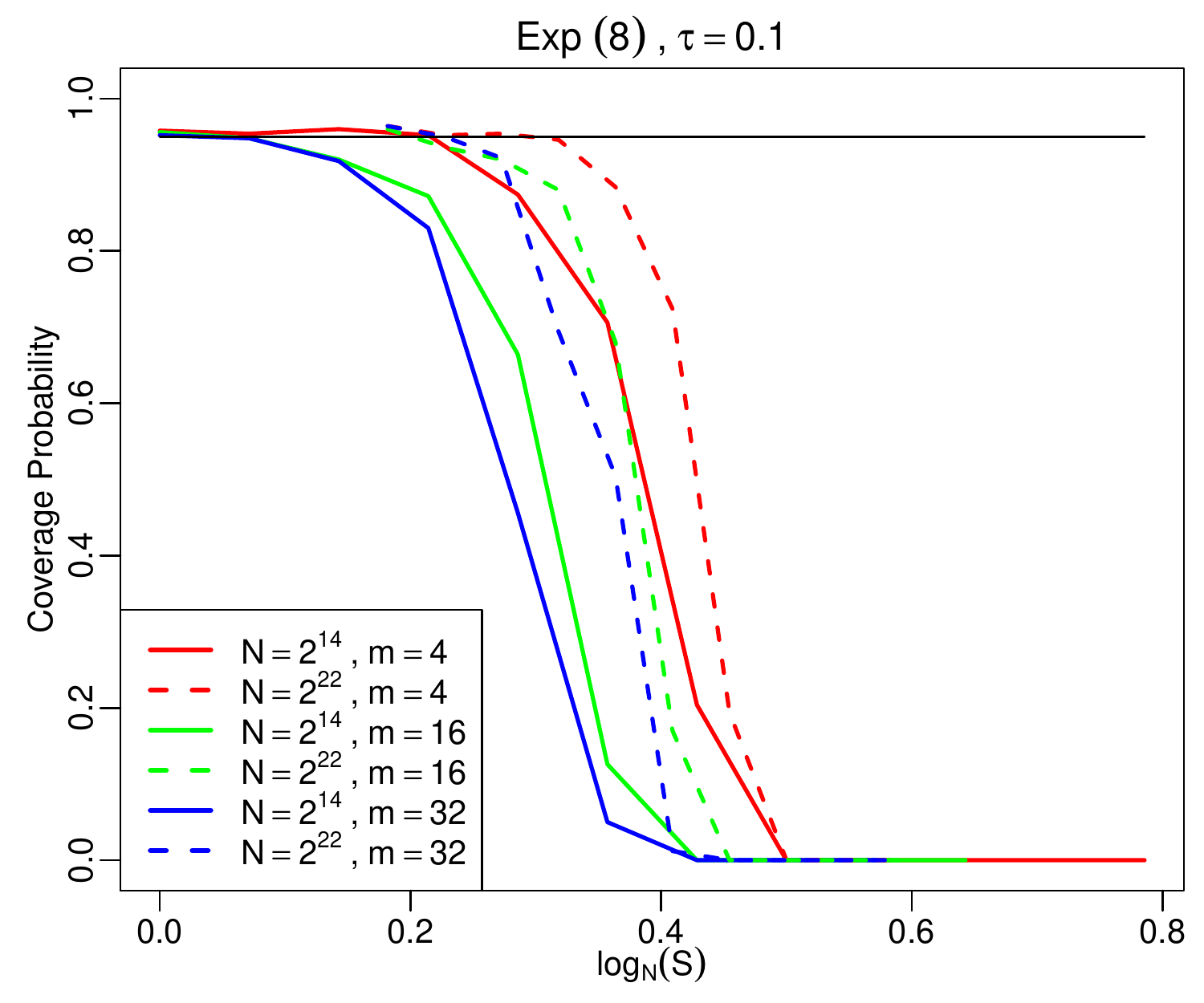}
	\includegraphics[width=4cm, height = 3.2cm]{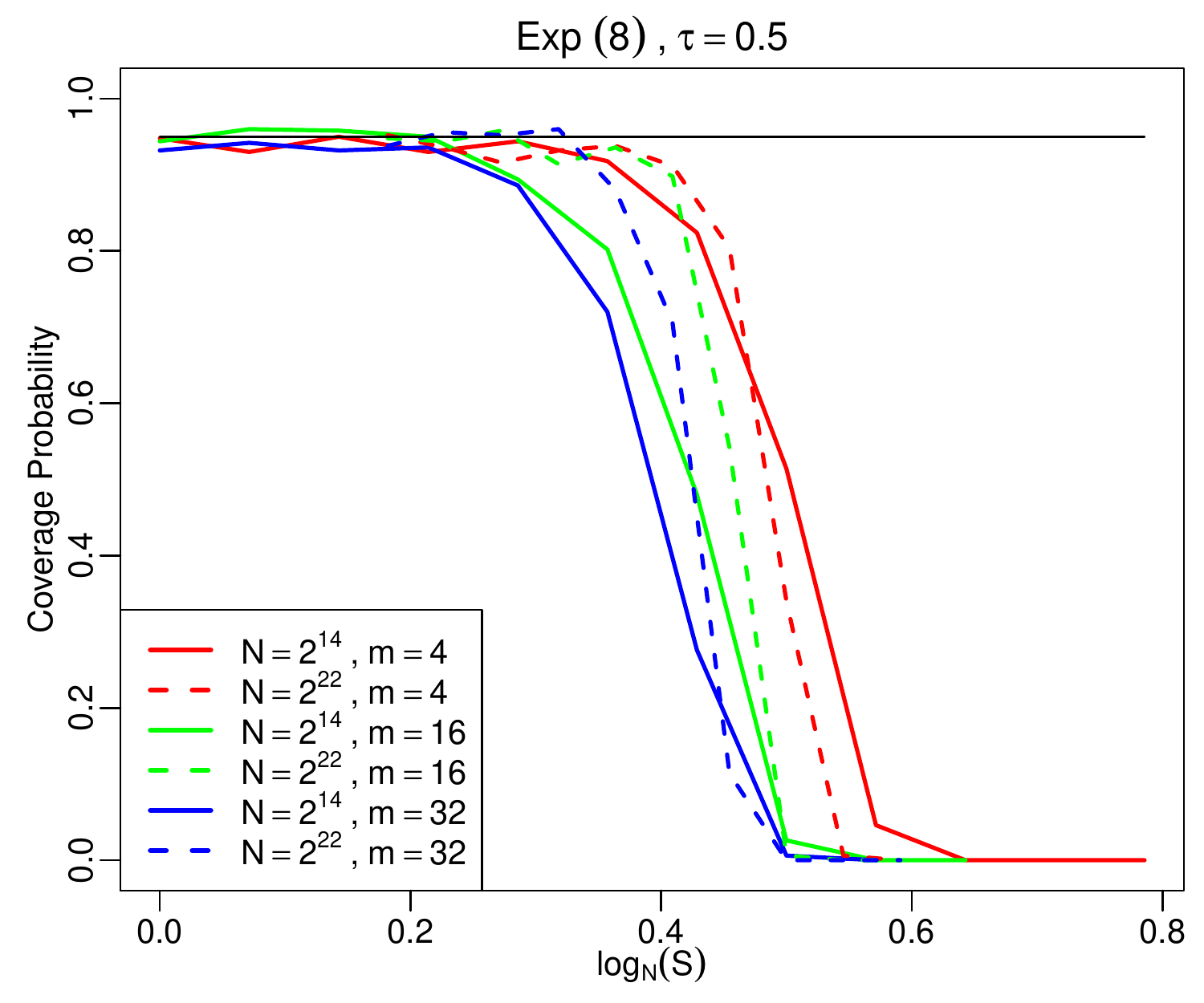}
	\includegraphics[width=4cm, height = 3.2cm]{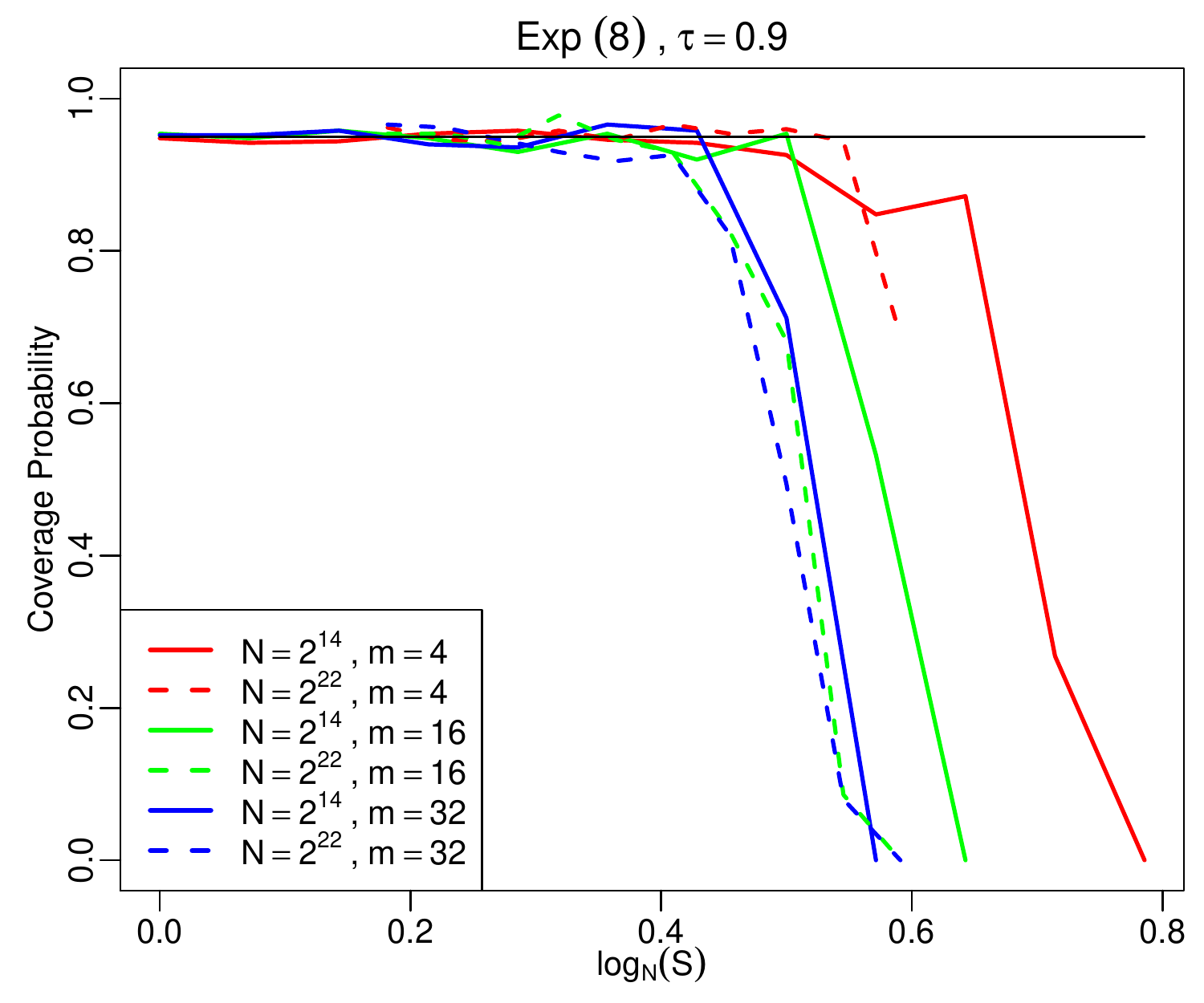}
	\caption{Coverage probability under exactly the same setting as Figure \ref{fig:simlincovp}, only except for $\varepsilon\sim\mbox{Exp}(\lambda)$.}\label{fig:simlincovp_exp}
	\end{figure}
	
\newpage

\begin{figure}[!ht]
\centering
\centering {\scriptsize (a) $m=4$}\\
\includegraphics[width=4cm, height = 3.2cm]{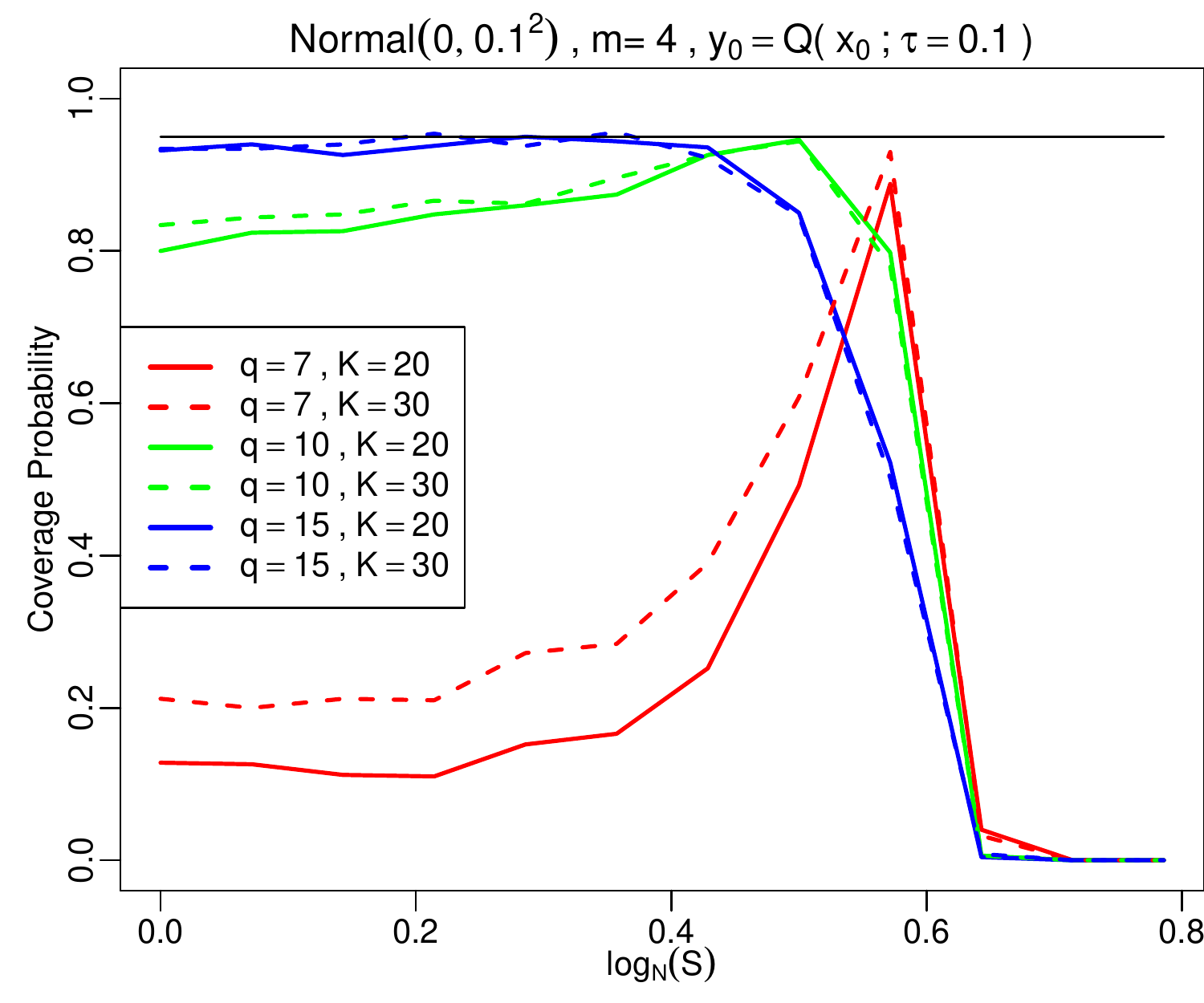}
\includegraphics[width=4cm, height = 3.2cm]{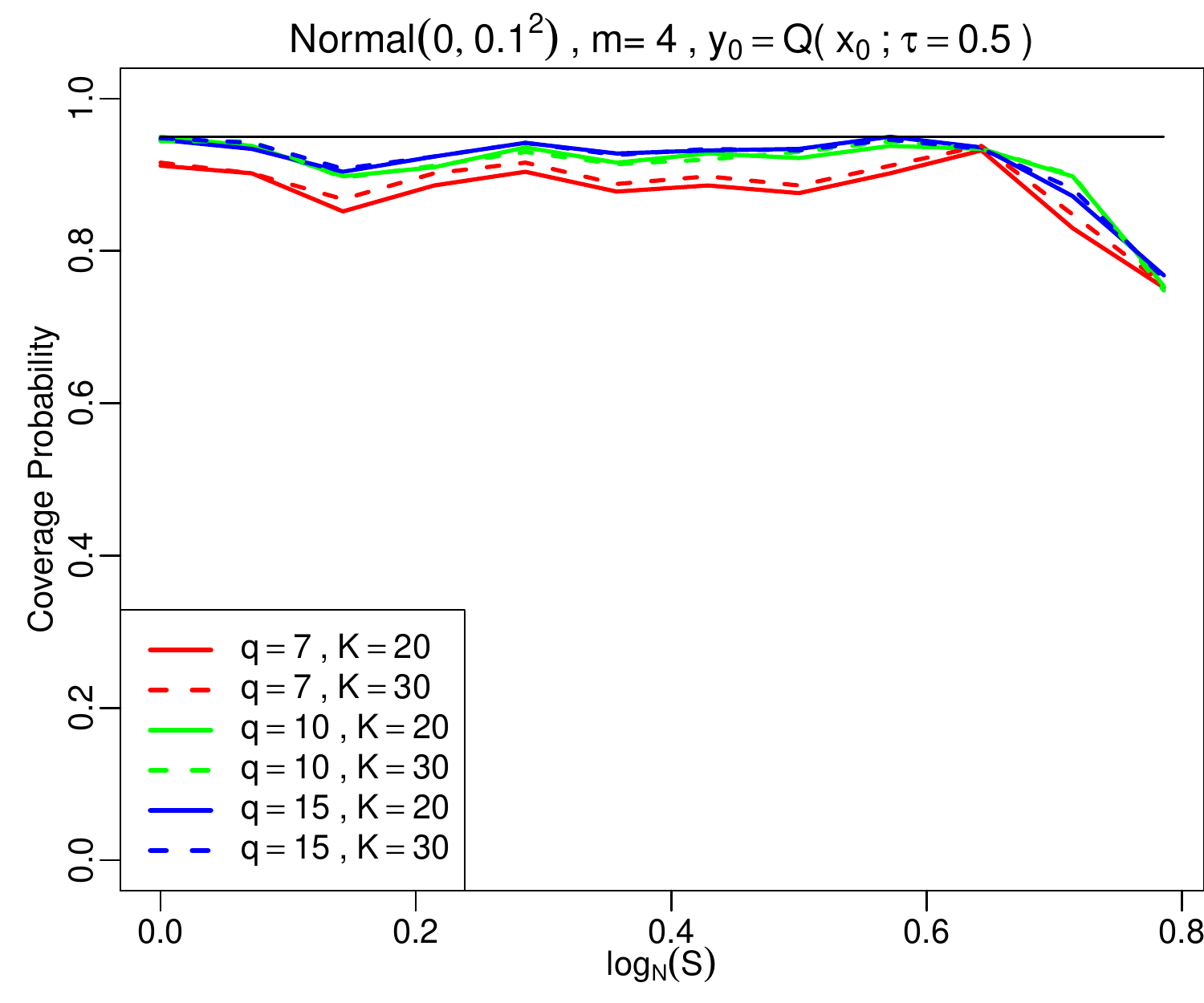}
\includegraphics[width=4cm, height = 3.2cm]{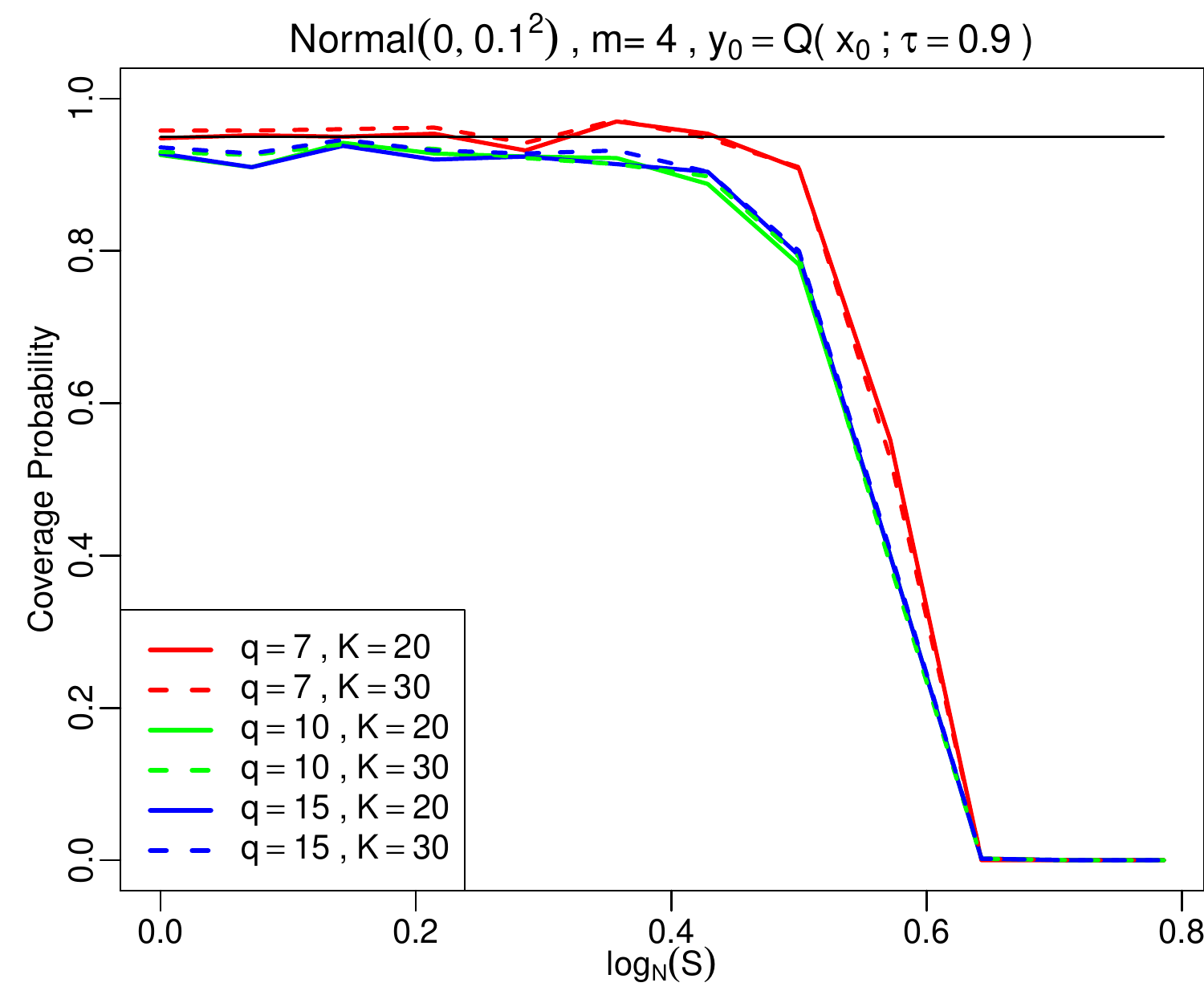}\\
\centering {\scriptsize (b) $m=16$}\\
\includegraphics[width=4cm, height = 3.2cm]{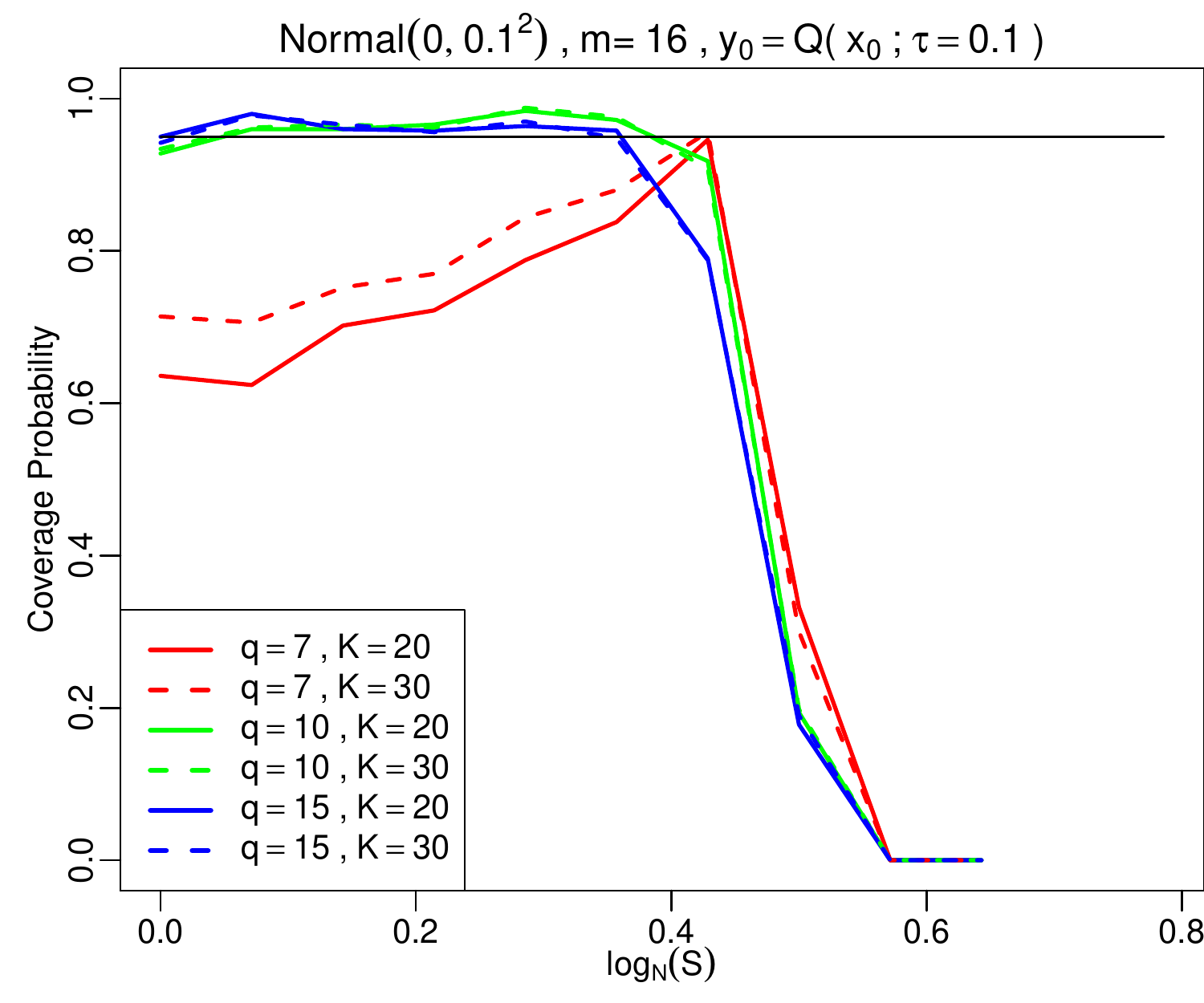}
\includegraphics[width=4cm, height = 3.2cm]{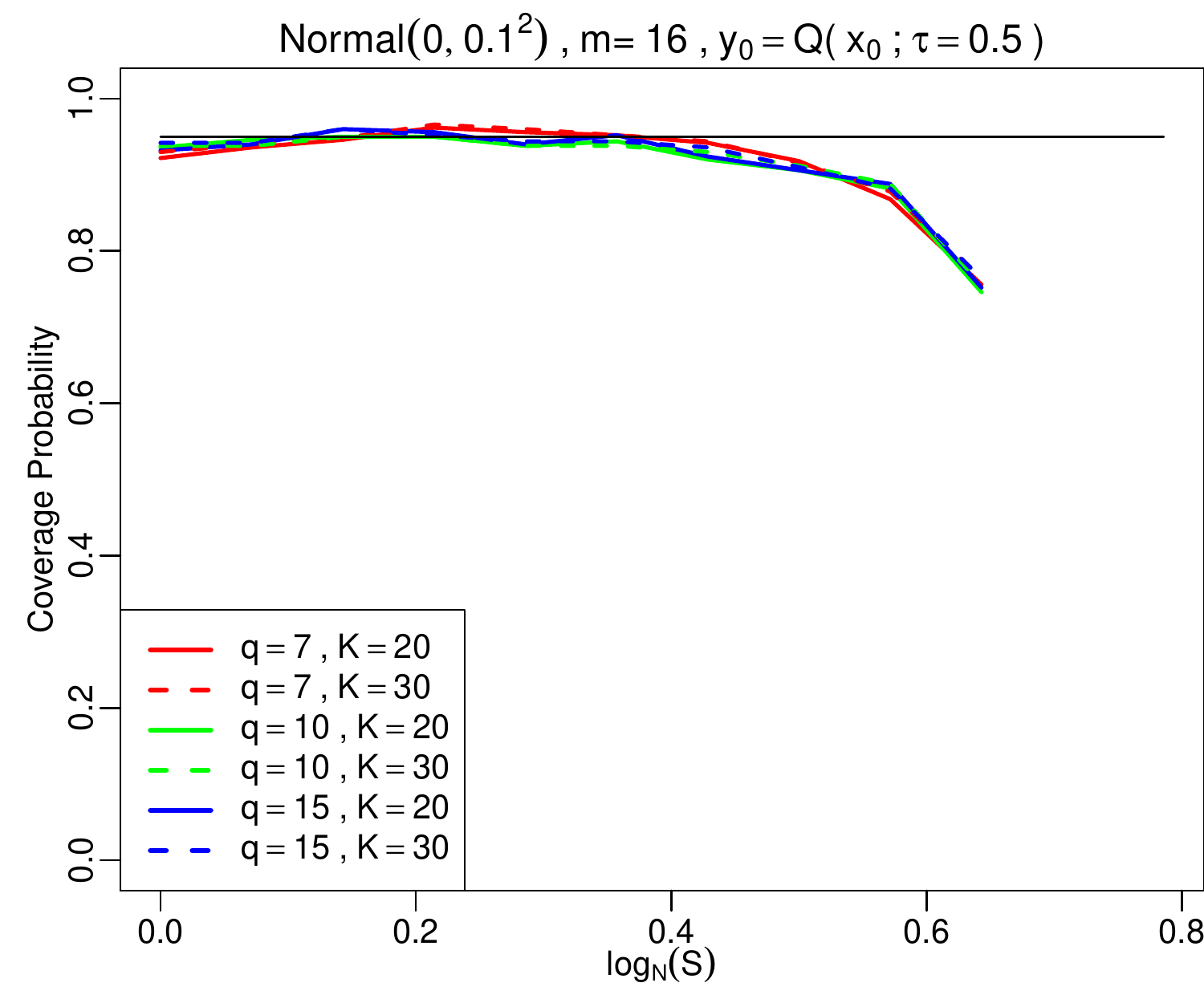}
\includegraphics[width=4cm, height = 3.2cm]{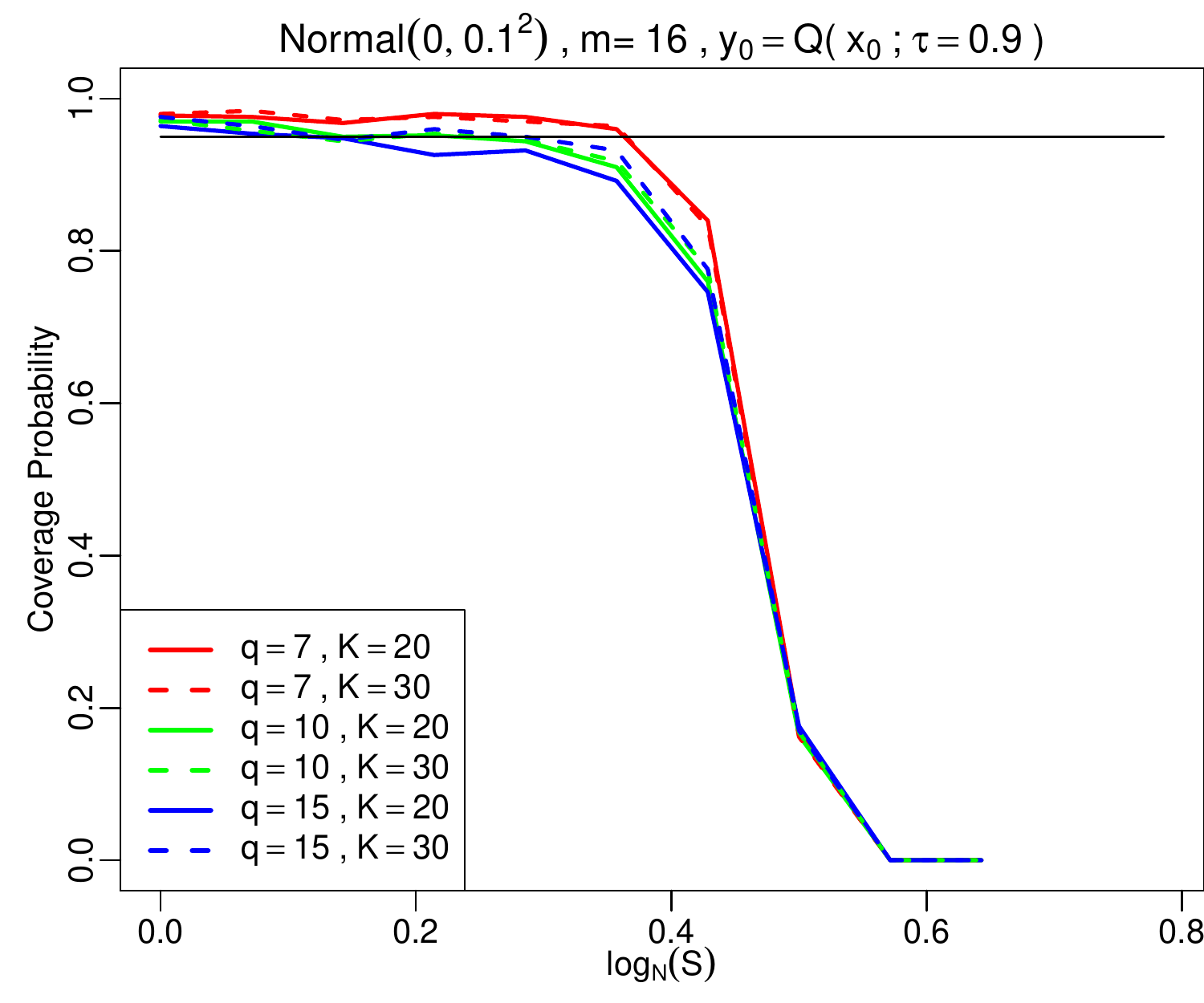}\\
\centering {\scriptsize (c) $m=32$}\\
\includegraphics[width=4cm, height = 3.2cm]{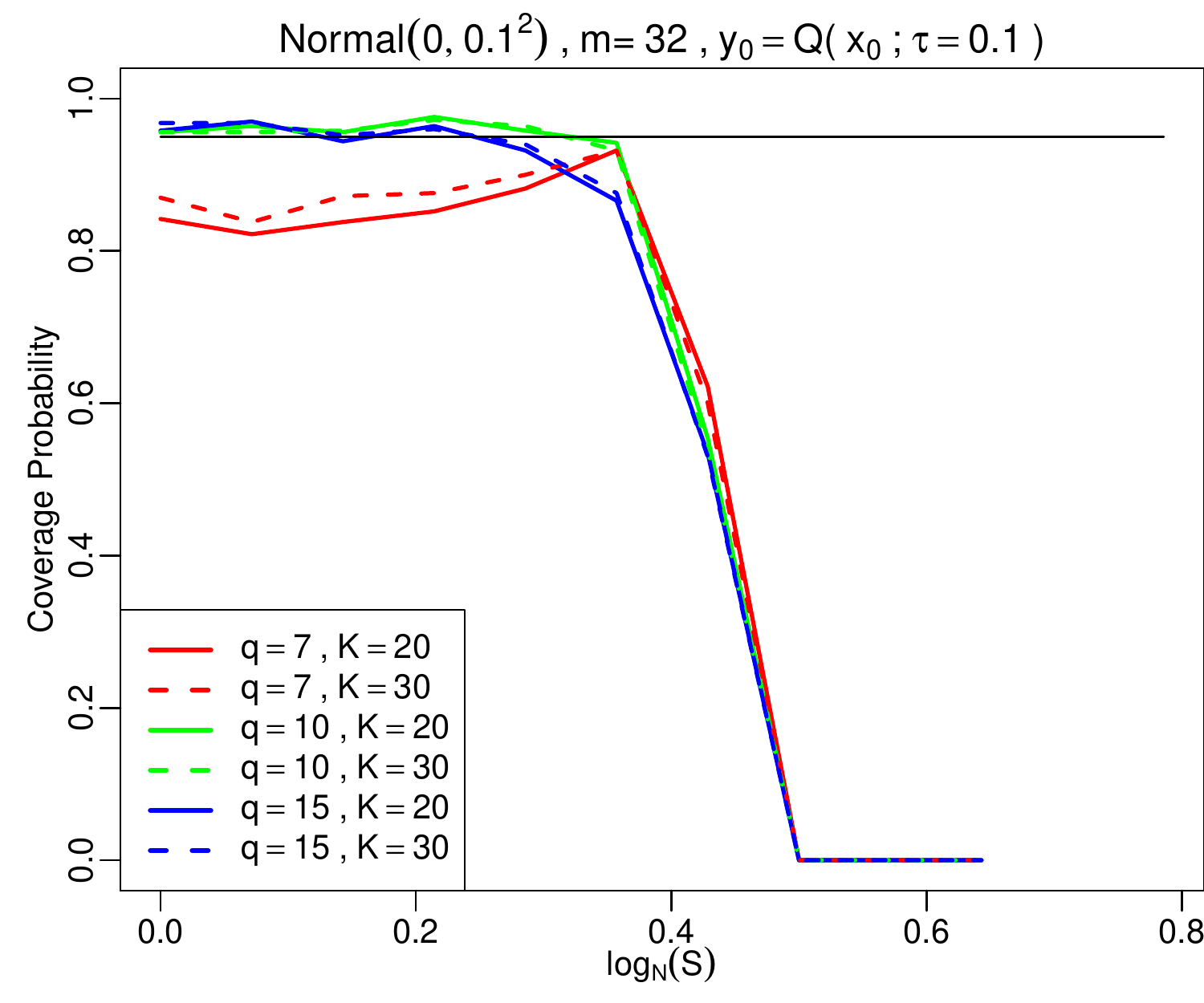}
\includegraphics[width=4cm, height = 3.2cm]{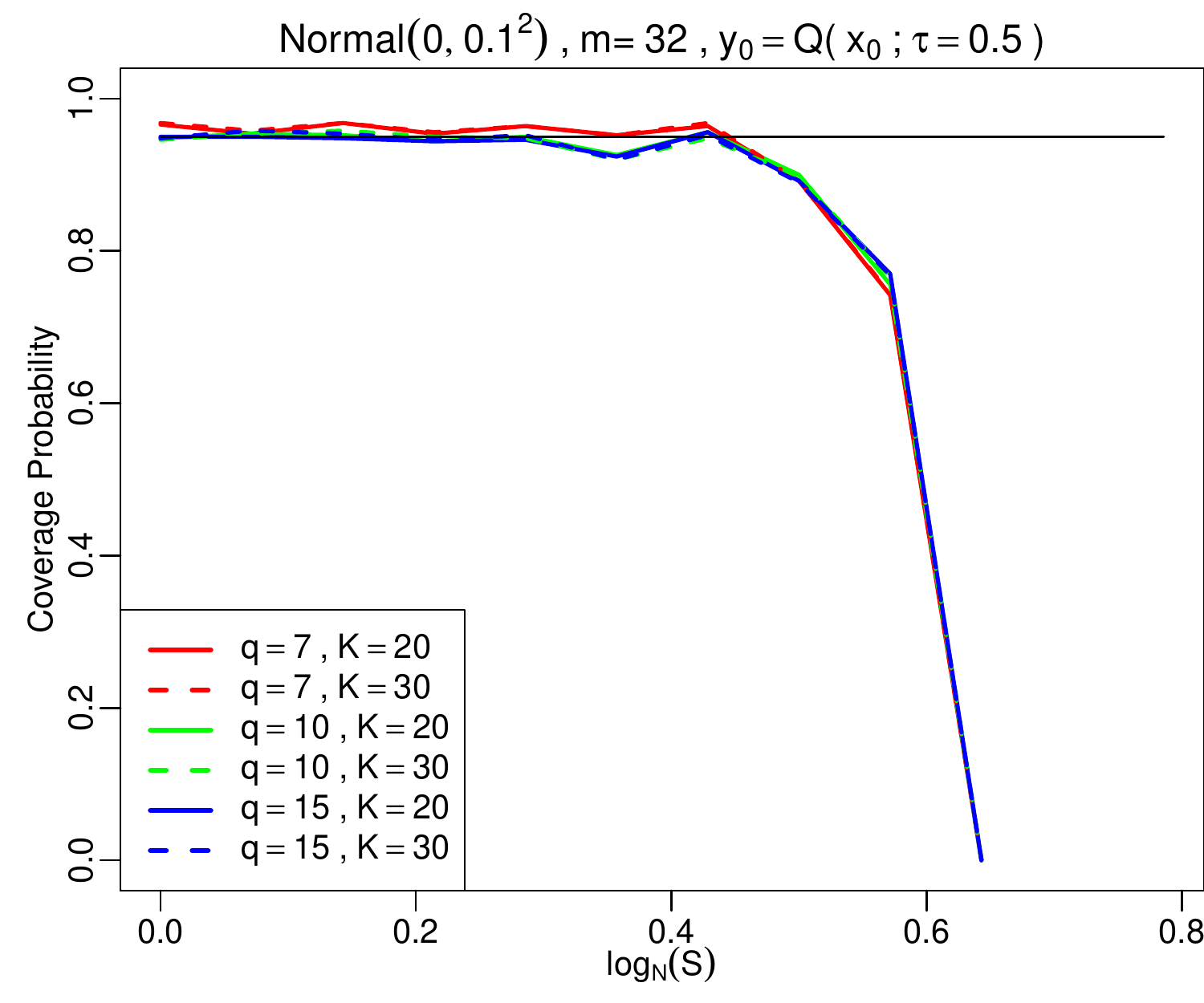}
\includegraphics[width=4cm, height = 3.2cm]{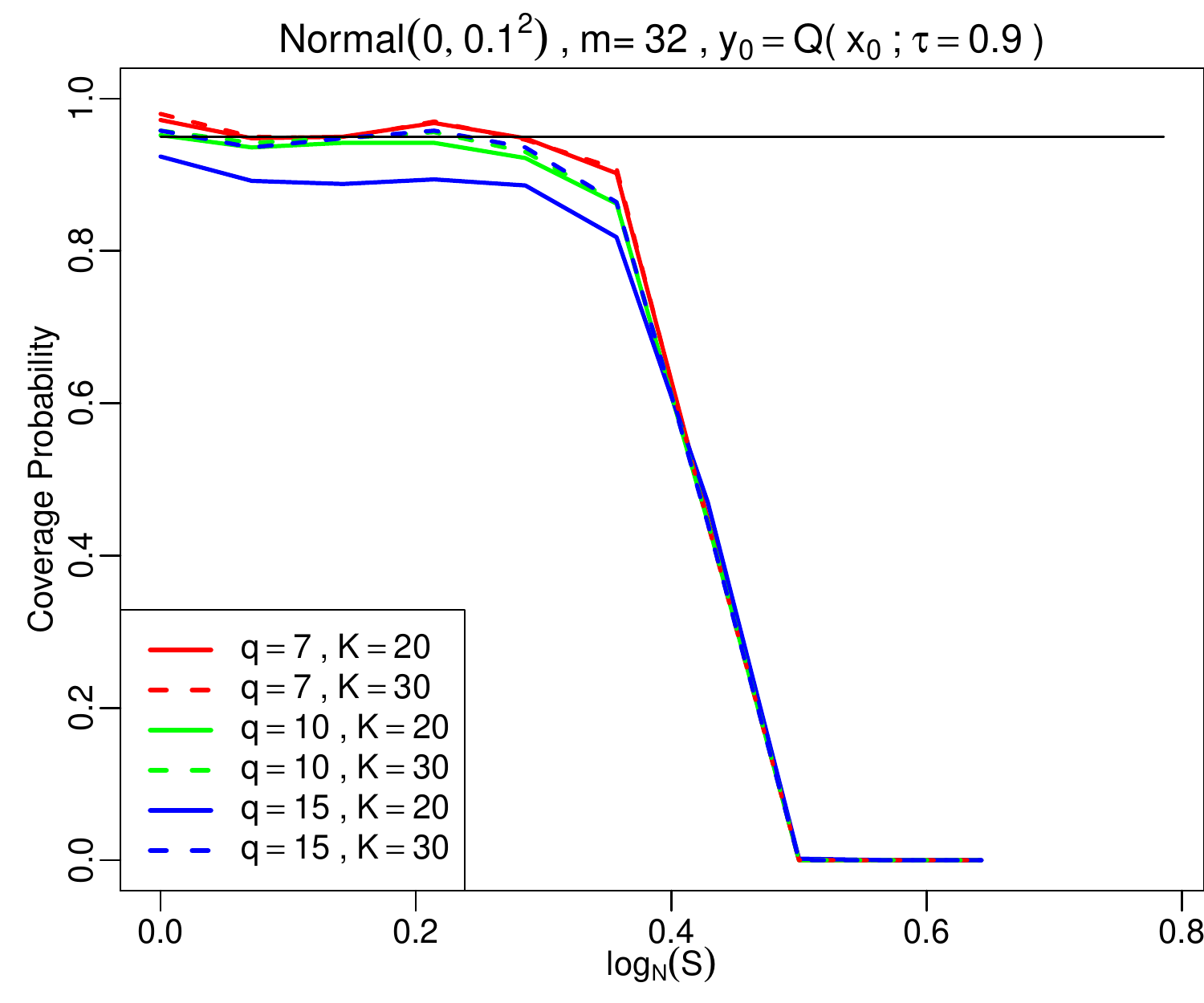}\\
\caption{Coverage probability of the $1-\alpha=95\%$ confidence interval for $F(y_0|x_0)$ in \eqref{eq:funcovp}: $\big[\hat F_{Y|X}(Q(x_0;\tau)|x_0) \pm \big(\tau(1-\tau) x_0^\top \Sigma_X^{-1}x_0/N\big)^{1/2} \Phi^{-1}(1-\alpha/2)\big]$ for $Y_i = 0.21 +  \zb_{m-1}^\top X_i+\varepsilon_i$ with $\zb_{m-1}$ as \eqref{eq:betaspeci} and $\varepsilon\sim\Nc(0,0.1^2)$, where the total number of samples is $N=2^{14}$, $S$ is the number of subsamples, $m=1+\dim(\zb_{m-1})$, $q = \dim(\BB)$ (corresponds to projection in $\tau$ direction) and $K$ is the number of the quantile grid points.}\label{fig:simlincovp_fun}
\end{figure}

\newpage

\begin{figure}[!ht]
\centering
\centering {\scriptsize (a) $m=4$}\\
\includegraphics[width=4cm, height = 3.2cm]{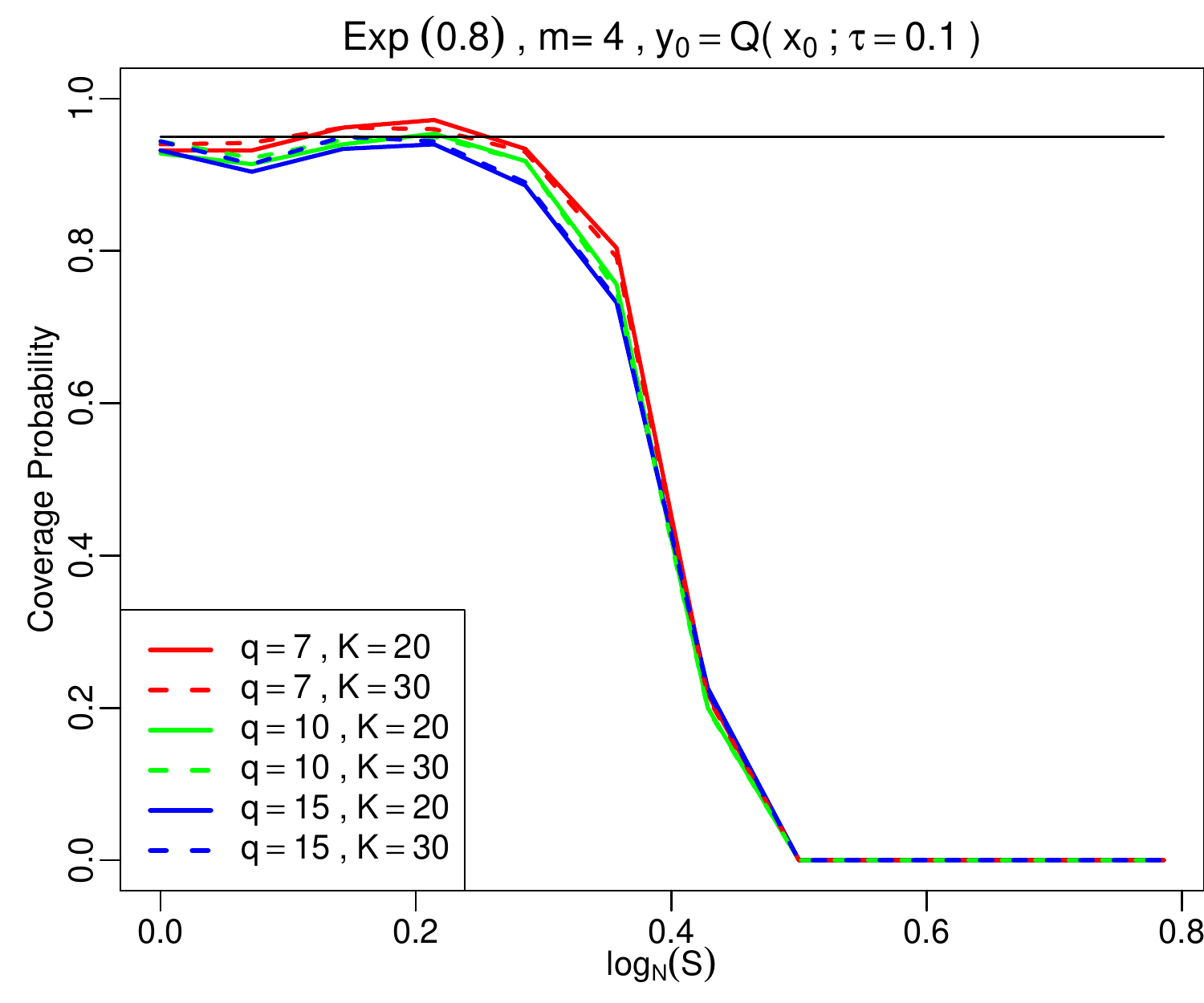}
\includegraphics[width=4cm, height = 3.2cm]{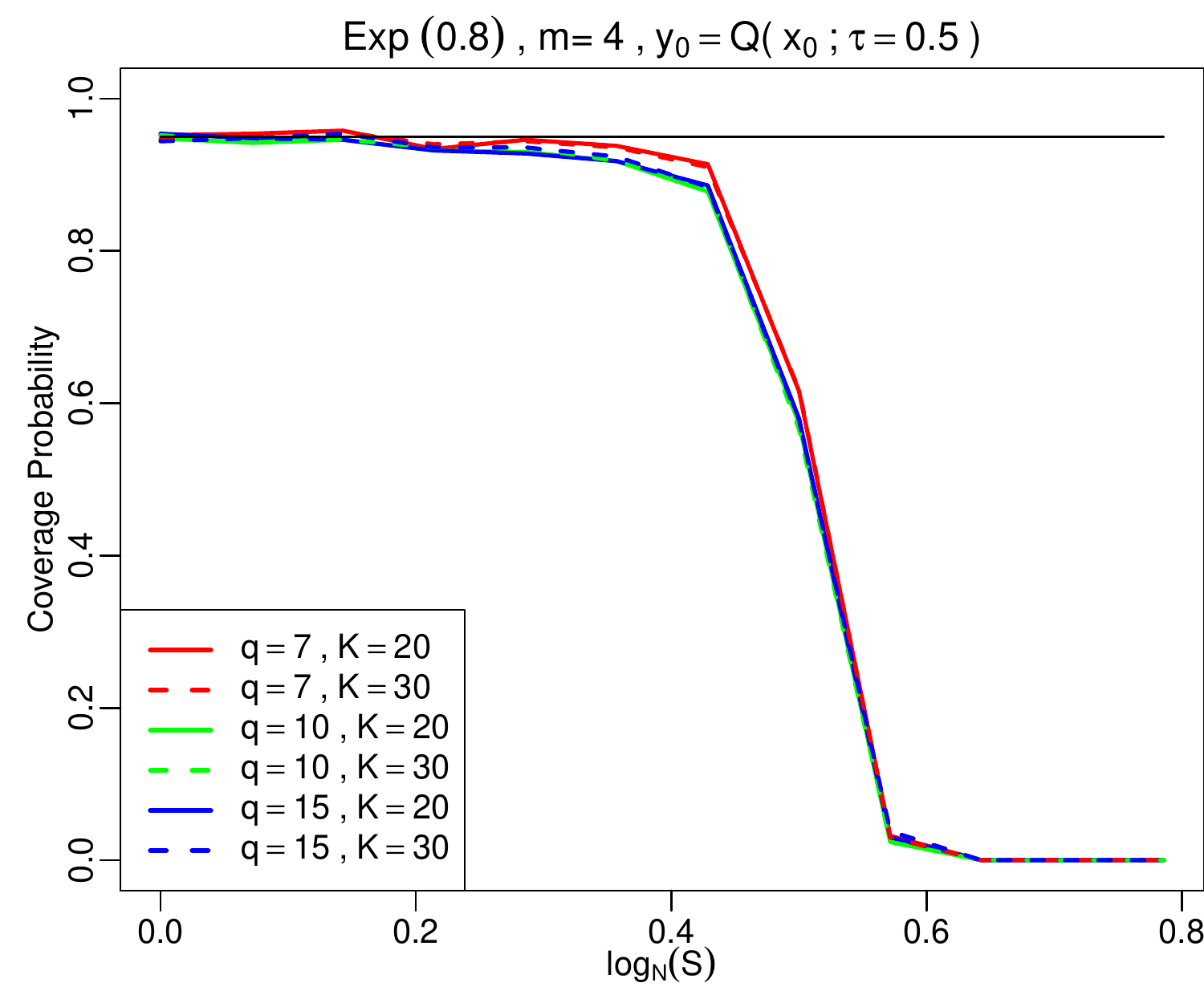}
\includegraphics[width=4cm, height = 3.2cm]{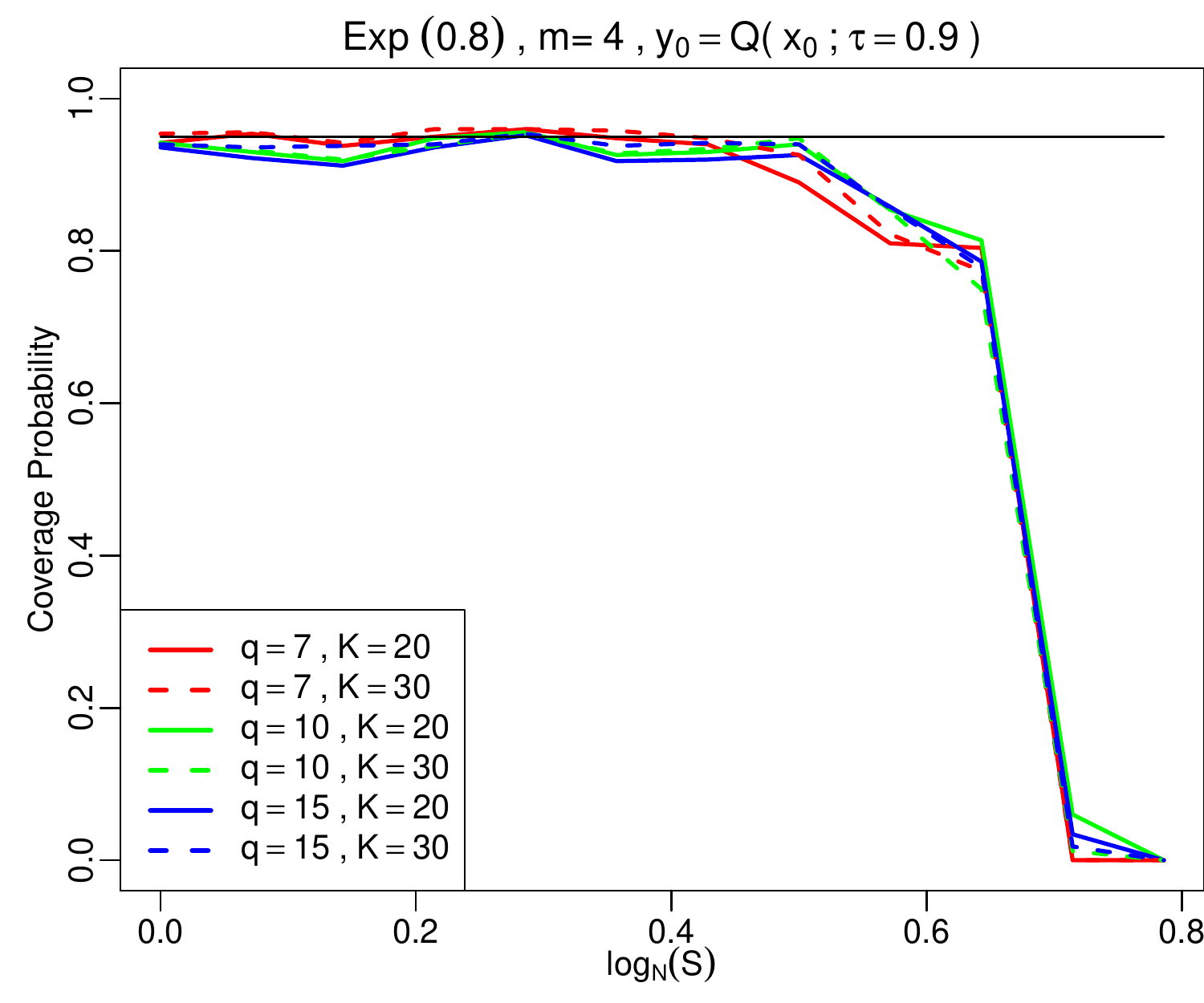}\\
\centering {\scriptsize (b) $m=16$}\\
\includegraphics[width=4cm, height = 3.2cm]{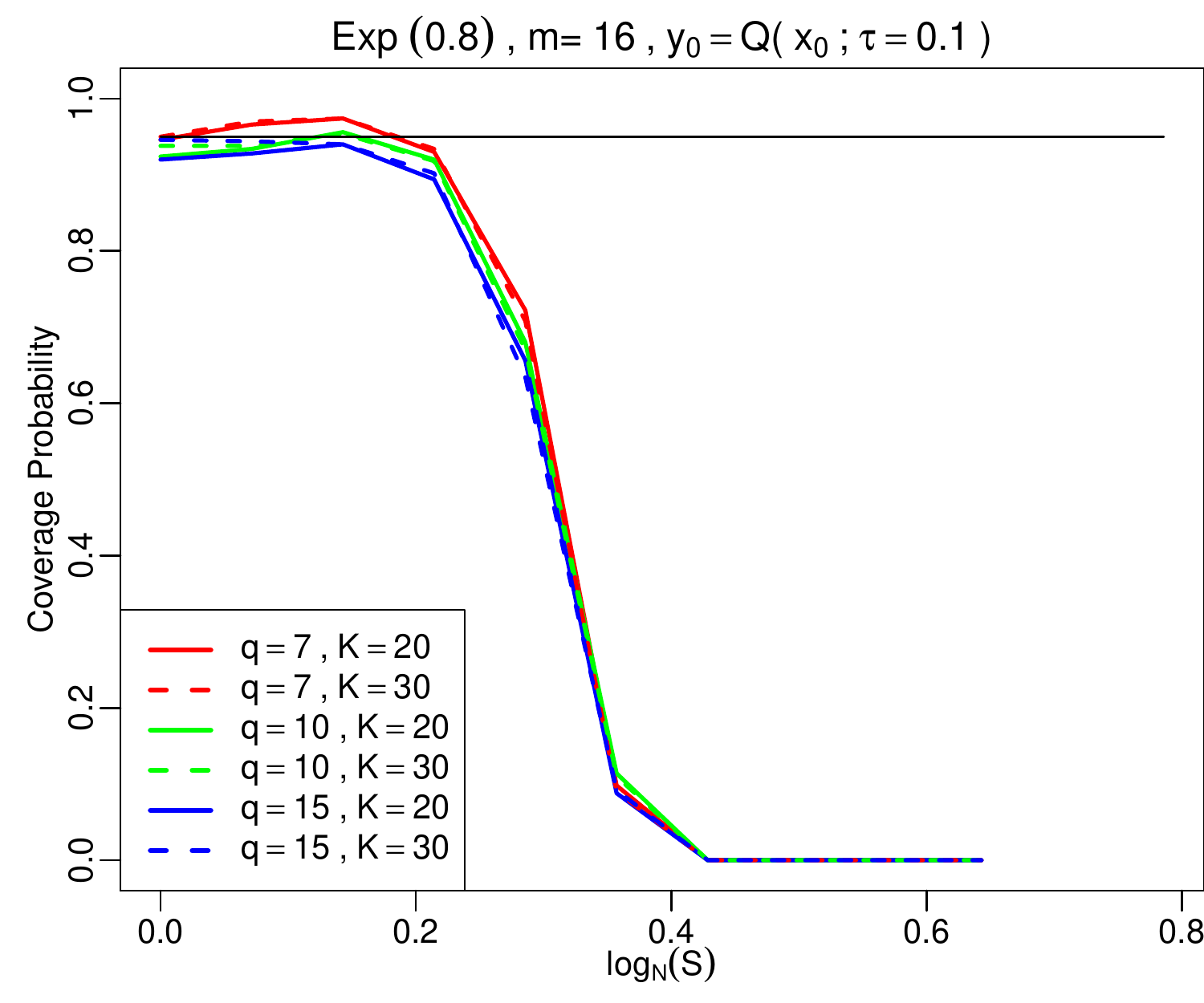}
\includegraphics[width=4cm, height = 3.2cm]{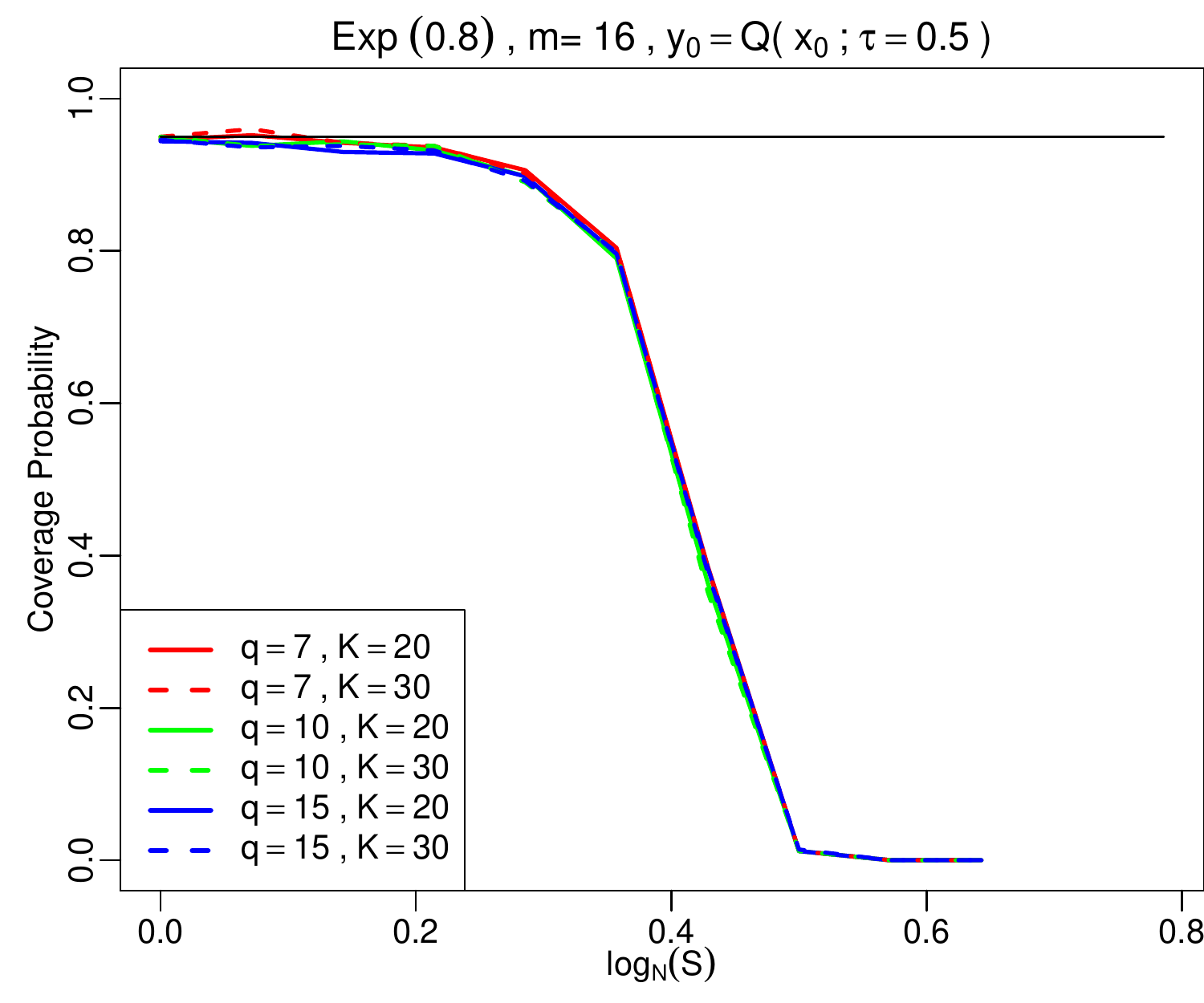}
\includegraphics[width=4cm, height = 3.2cm]{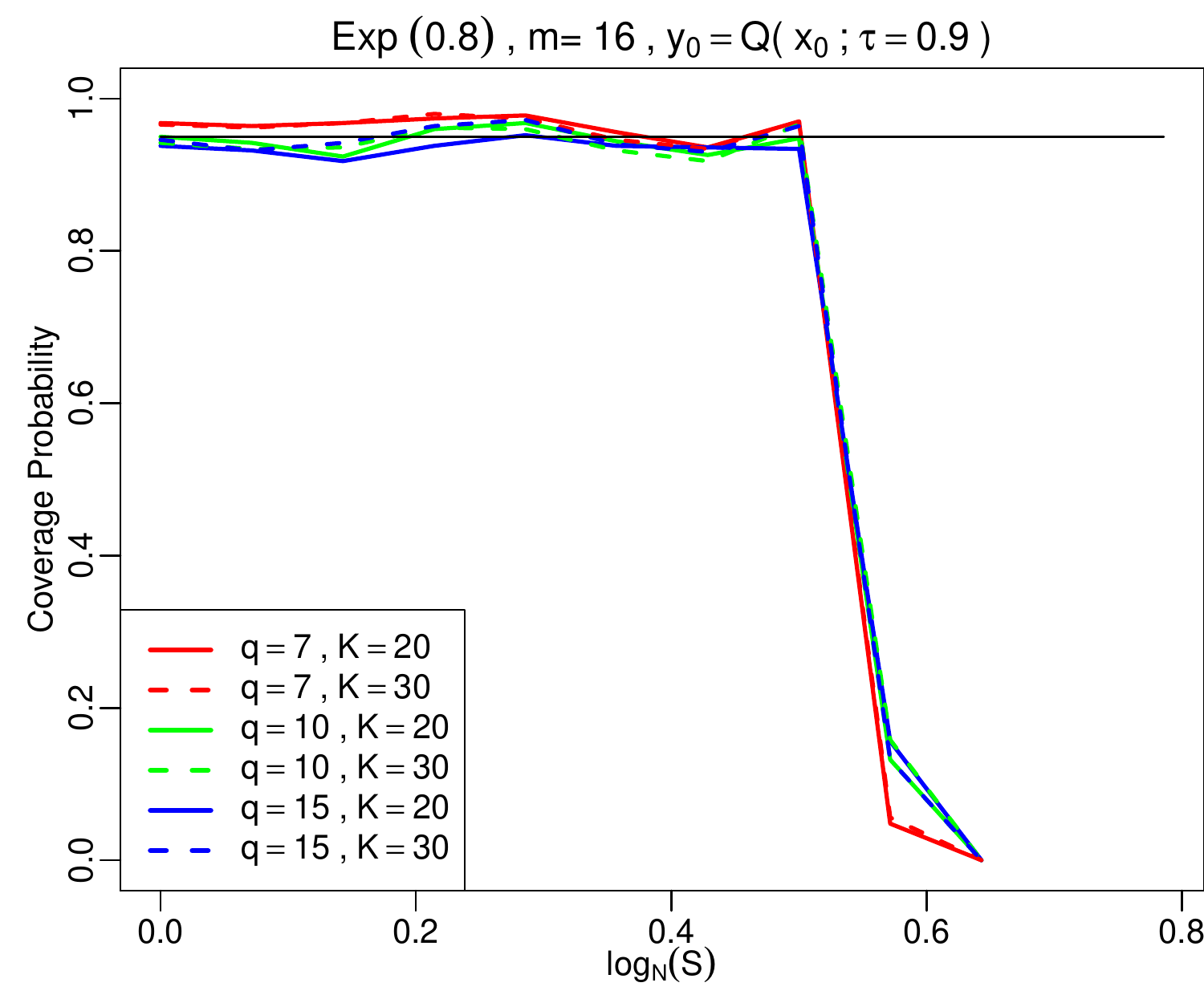}\\
\centering {\scriptsize (c) $m=32$}\\
\includegraphics[width=4cm, height = 3.2cm]{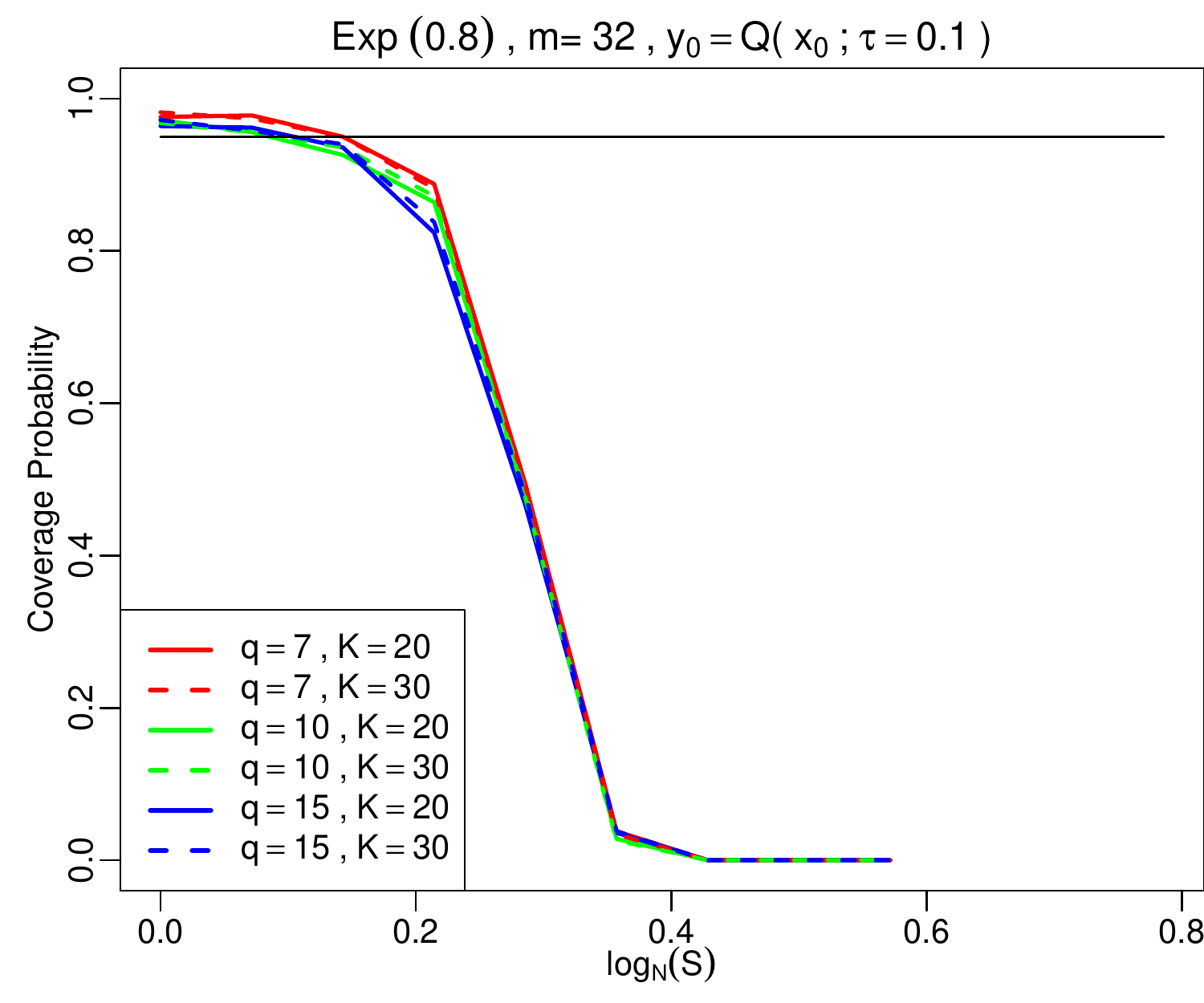}
\includegraphics[width=4cm, height = 3.2cm]{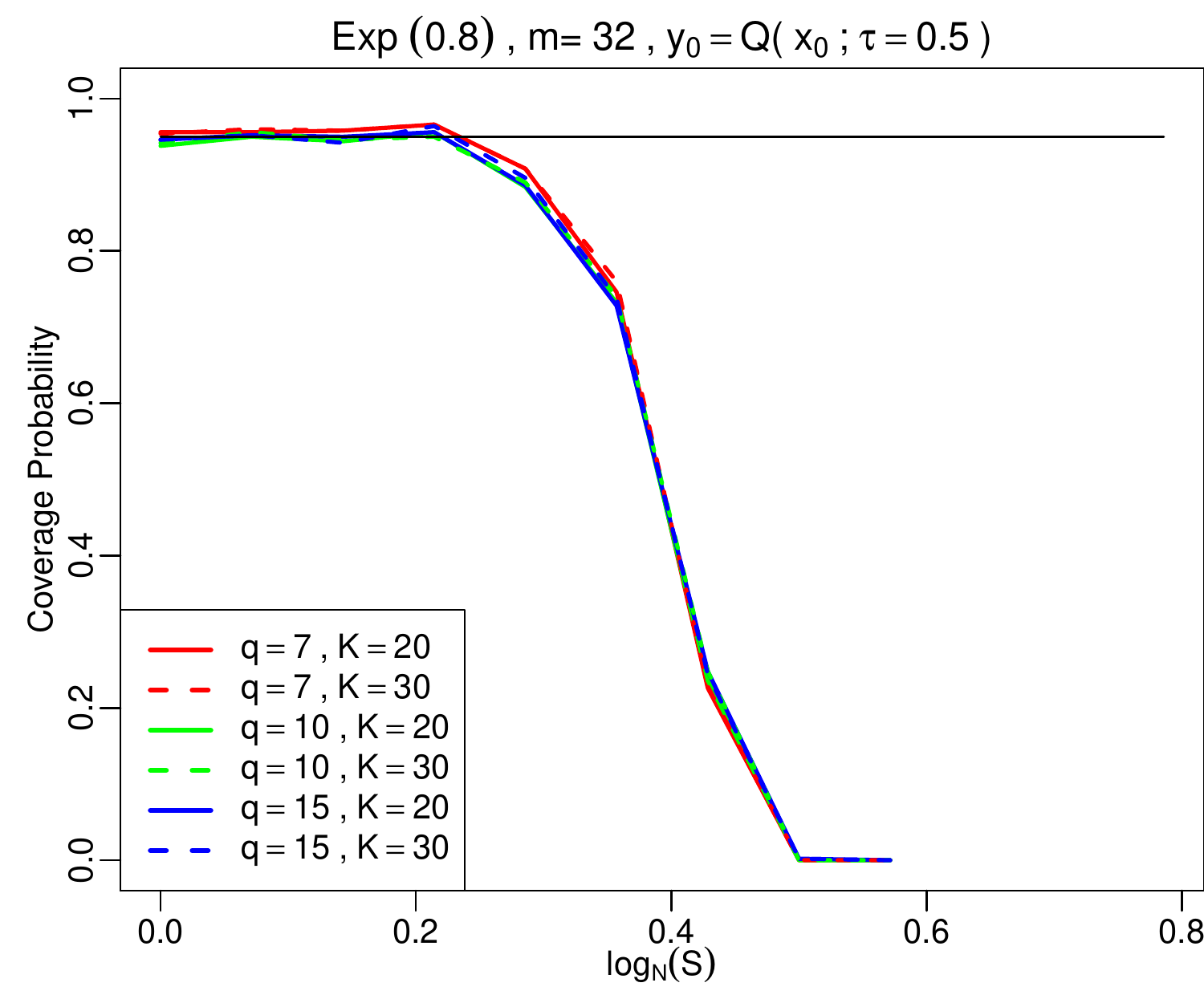}
\includegraphics[width=4cm, height = 3.2cm]{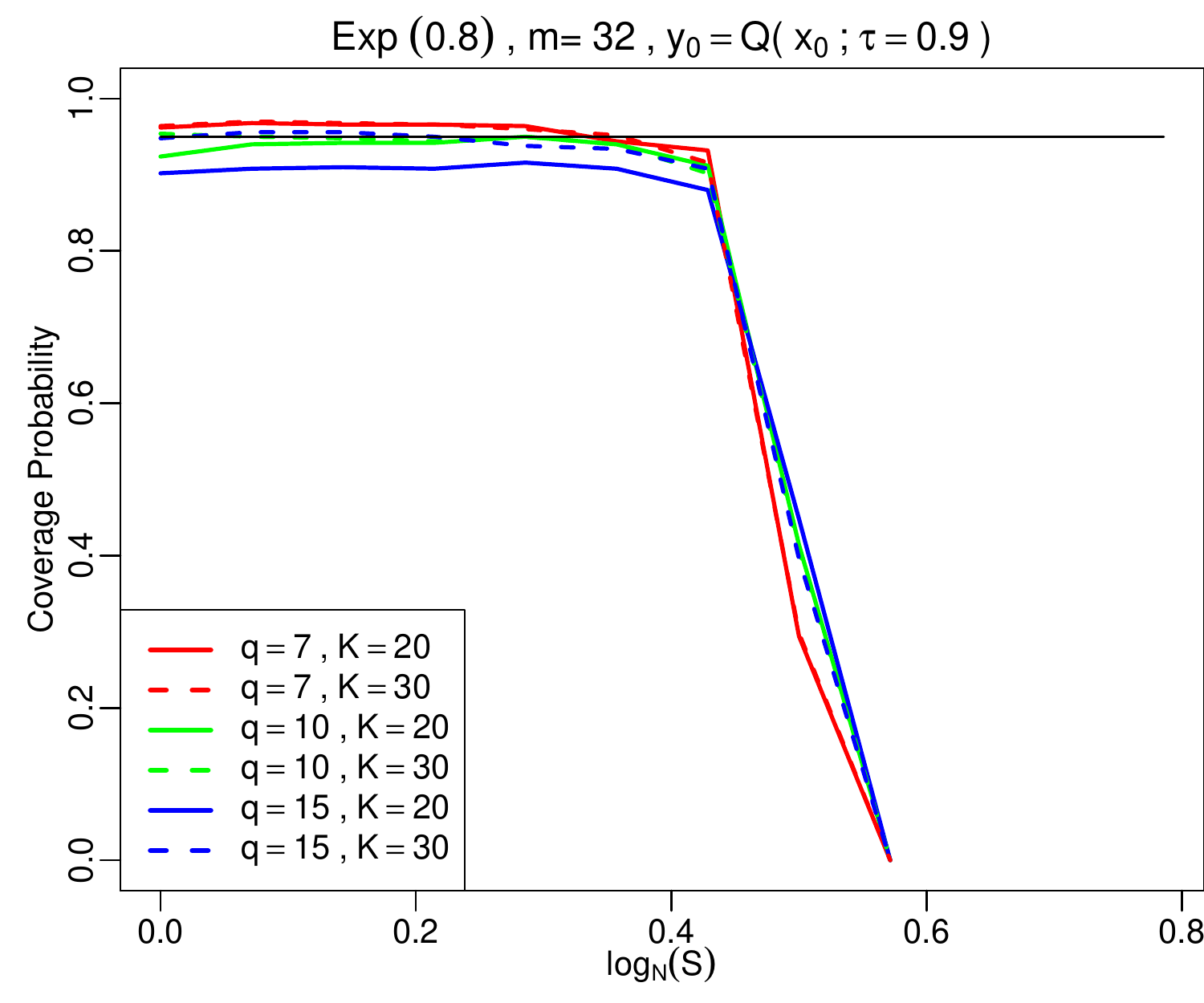}\\
\caption{Coverage probability under exactly the same setting as Figure \ref{fig:simlincovp_fun}, only except for $\varepsilon\sim\mbox{Exp}(\lambda)$.}\label{fig:simlincovp_funexp}
\end{figure}
	\begin{figure}[!ht]
		\centering
		\includegraphics[width=4cm, height = 3.2cm]{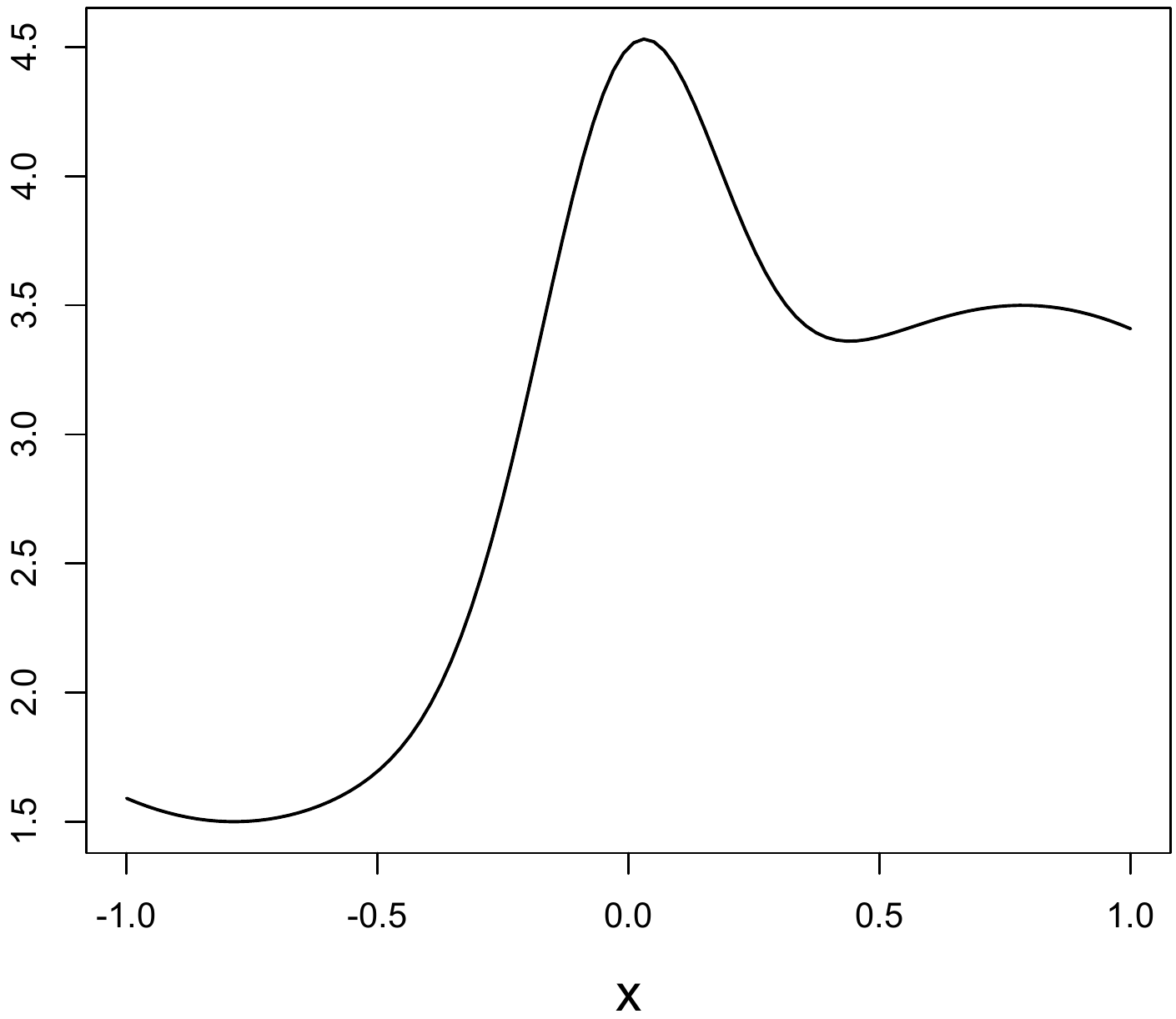}
		\caption{The plot for the function $x\mapsto 2.5 + \sin(2x)+2\exp(-16x^2)$, which is used in model \eqref{eq:npmodel}.}\label{fig:npfun}
	\end{figure}

\newpage
\vspace{-0.6cm}
	\begin{figure}[!ht]	\centering
		\includegraphics[width=4cm, height = 3.2cm]{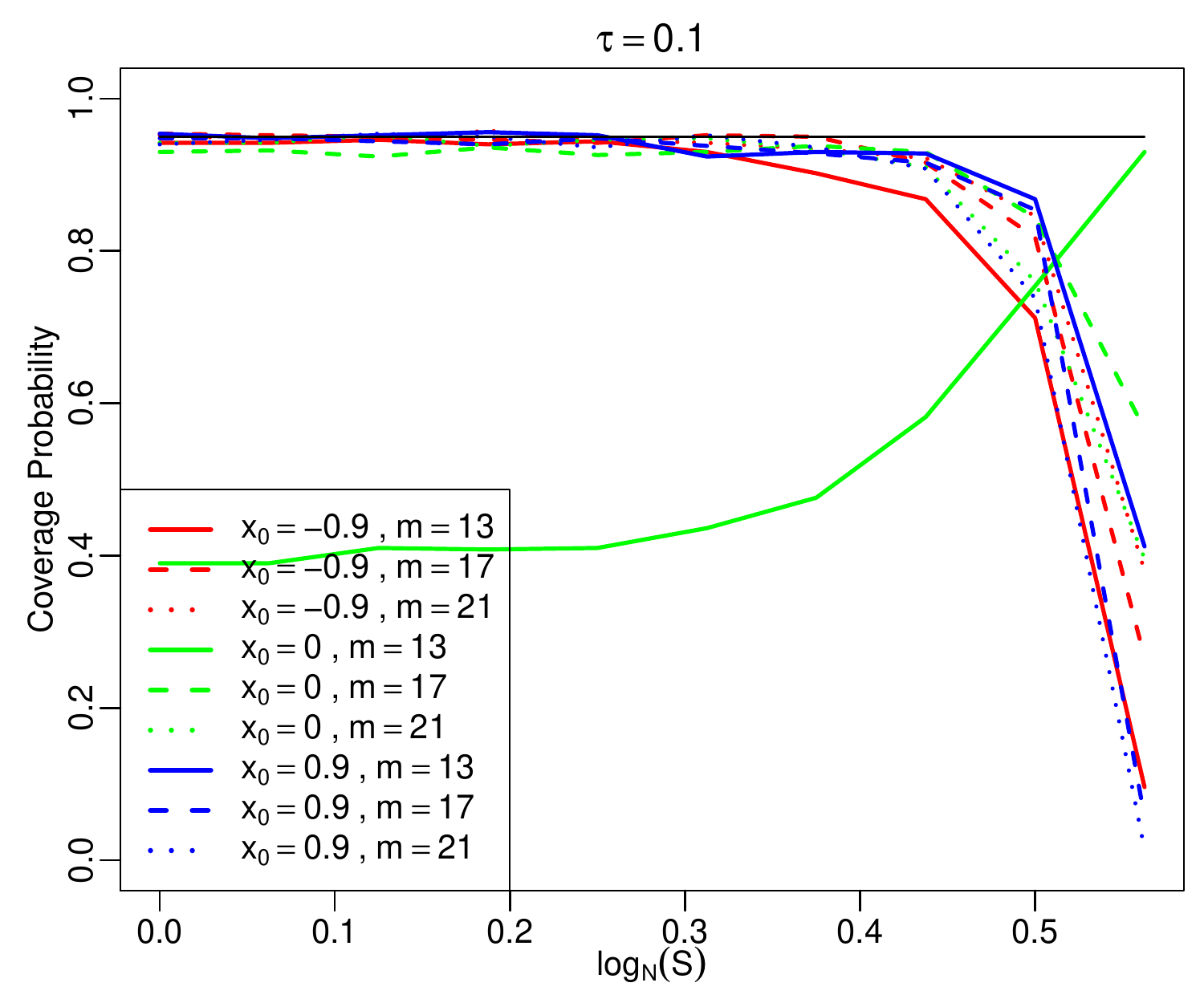}	
		\includegraphics[width=4cm, height = 3.2cm]{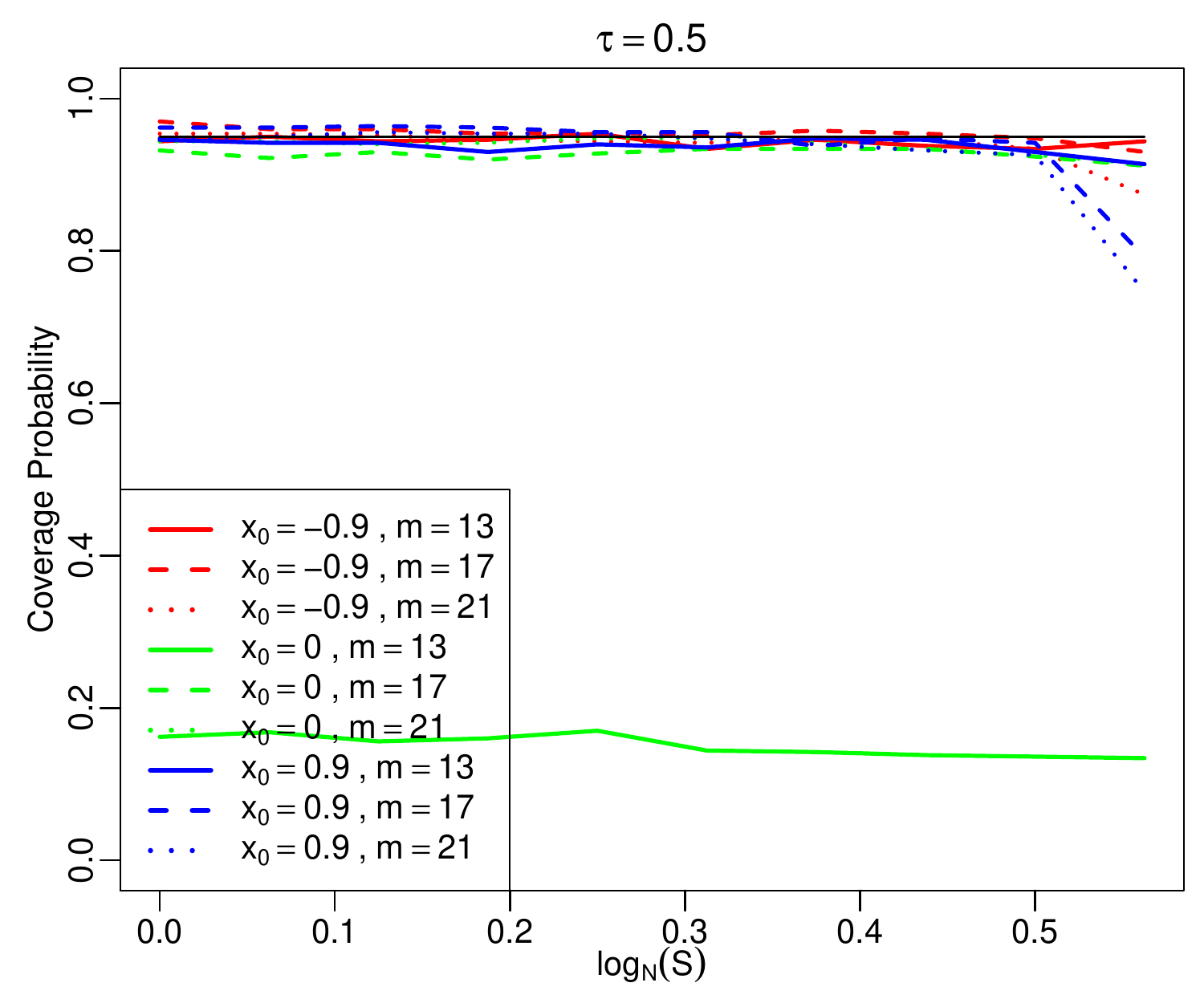}	
		\includegraphics[width=4cm, height = 3.2cm]{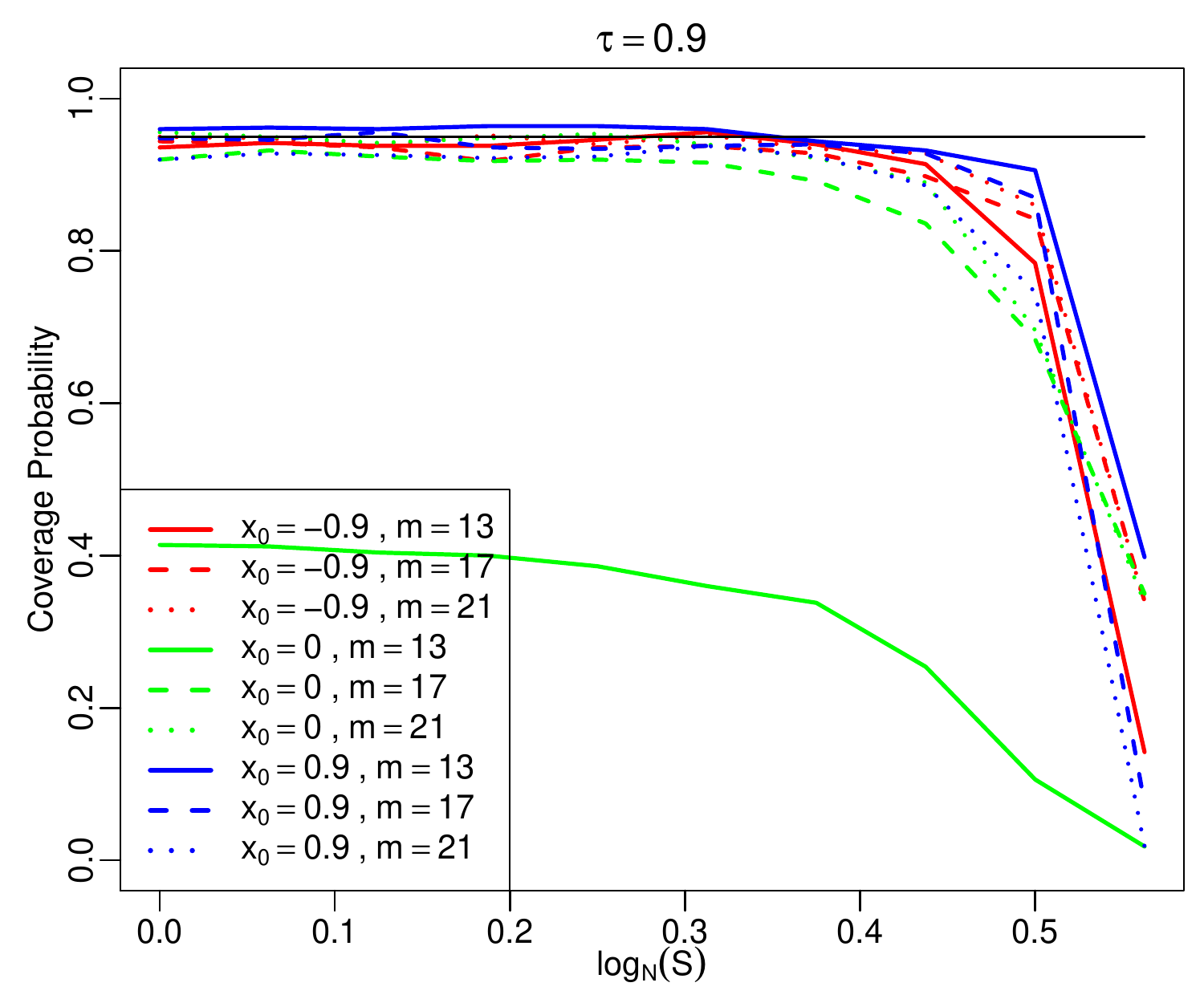}		
		\caption{Coverage probability of the $1-\alpha=95\%$ confidence interval for $Q(x_0;\tau)$ in \eqref{eq:npcovp}: 
		$\big[\ZZ(x_0)^\top \overline{\zb}(\tau) \pm N^{-1/2} \phi_{\sigma}^{-1}(\Phi_\sigma^{-1}(\tau))\sqrt{\tau(1-\tau)} \sigma_0(\ZZ) \Phi^{-1}(1-\alpha/2)\big]$ under the nonlinear model \eqref{eq:npmodel}, where $m=\dim(\ZZ)$ (corresponds to basis expansion in $x$) and $N=2^{16}$.}\label{fig:simnpcovp_fix}	
	\end{figure}
	\begin{figure}[!ht]
		\centering
		\centering {\scriptsize (a) $m=17$}\\	
		\includegraphics[width=4cm, height = 3.2cm]{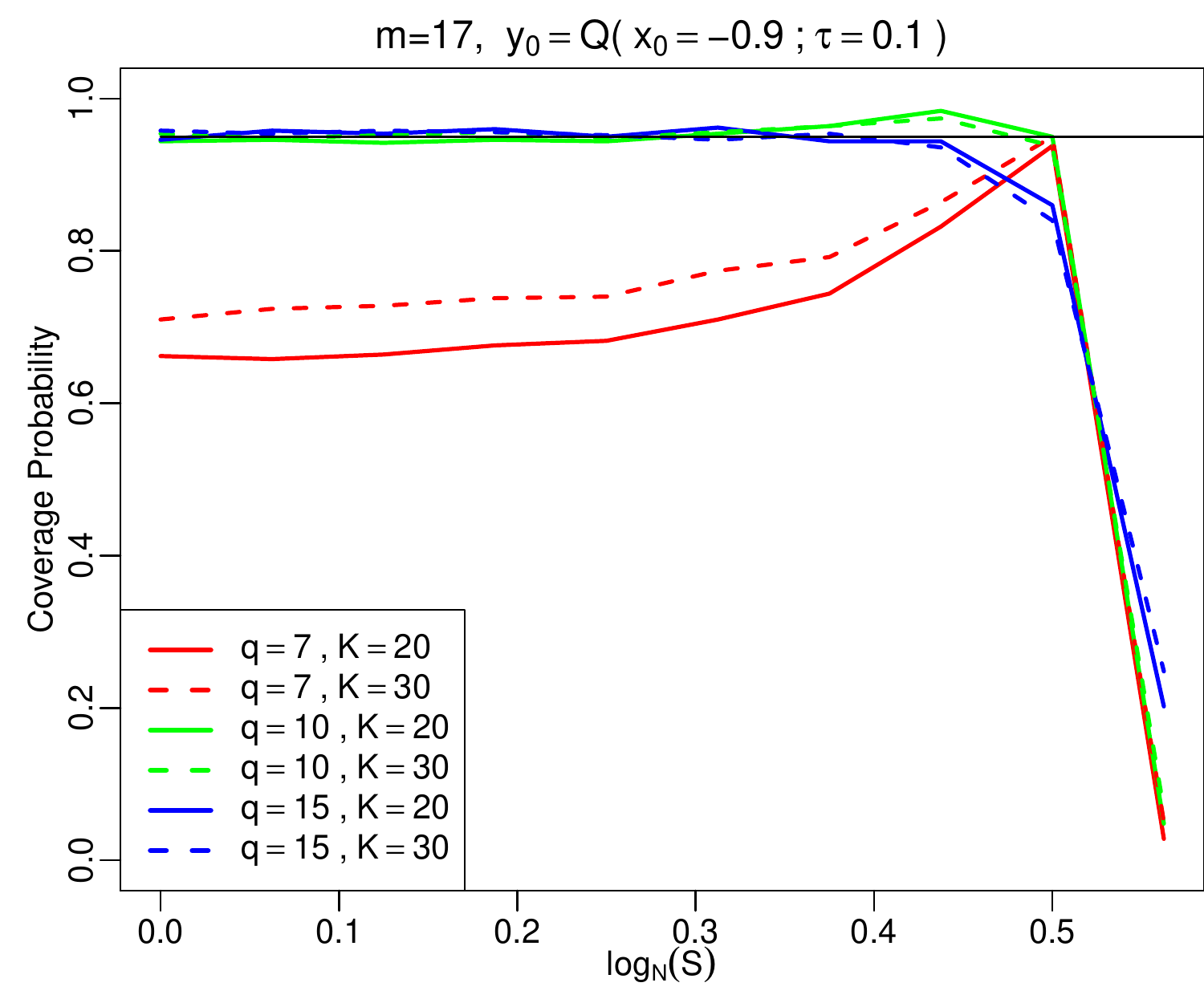}	
		\includegraphics[width=4cm, height = 3.2cm]{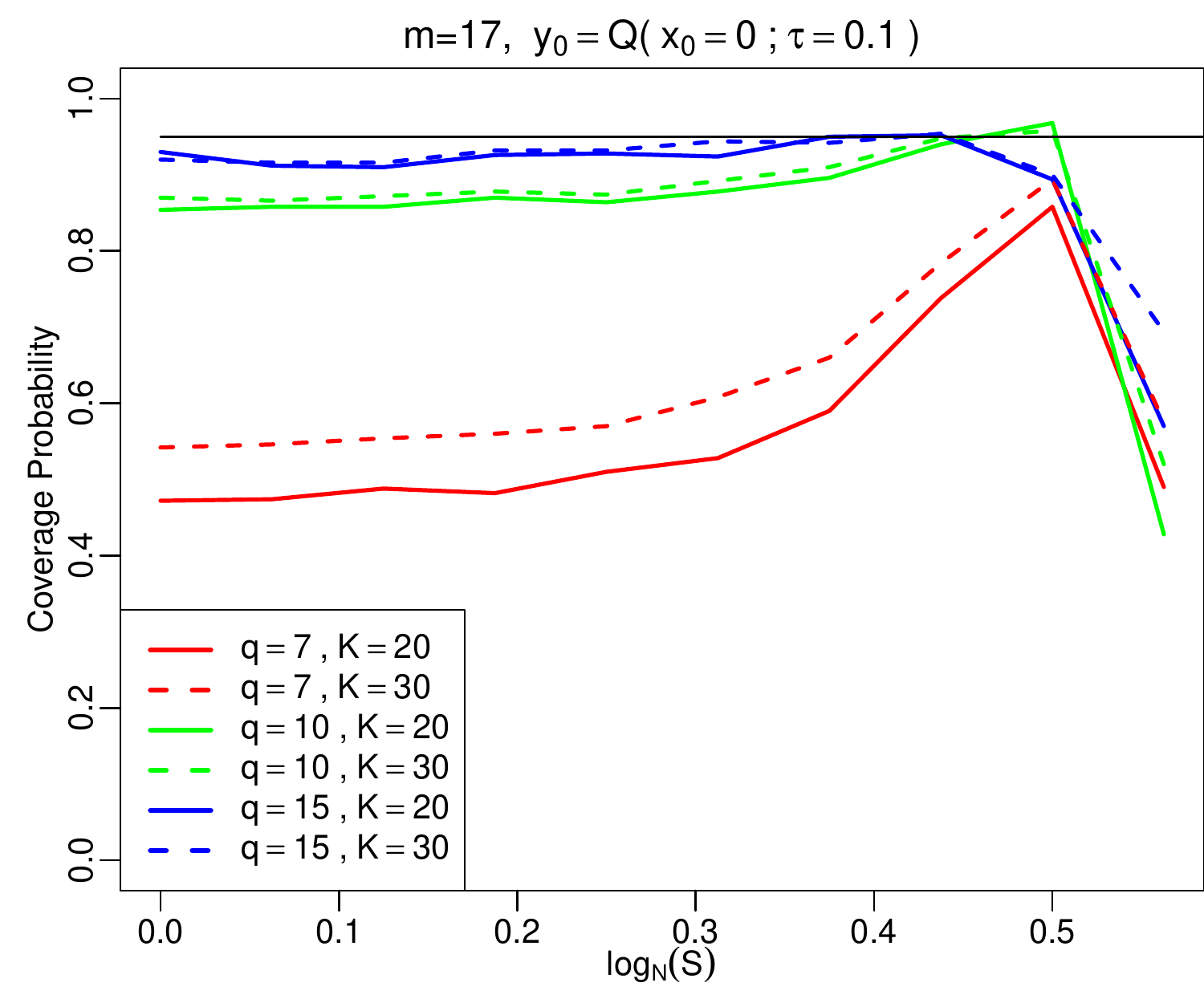}	
		\includegraphics[width=4cm, height = 3.2cm]{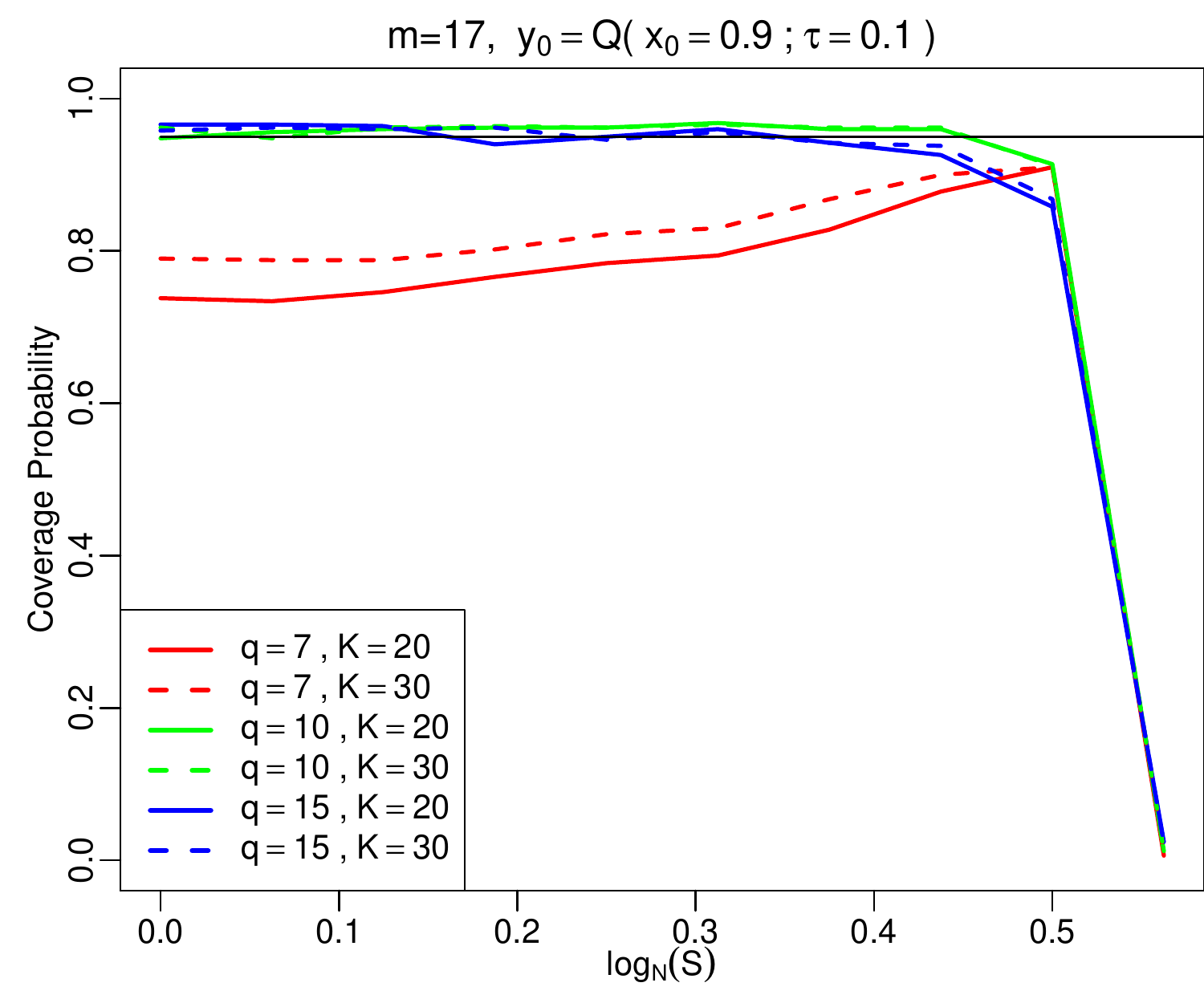}\\
		\centering {\scriptsize (b) $m=23$}\\	
		\includegraphics[width=4cm, height = 3.2cm]{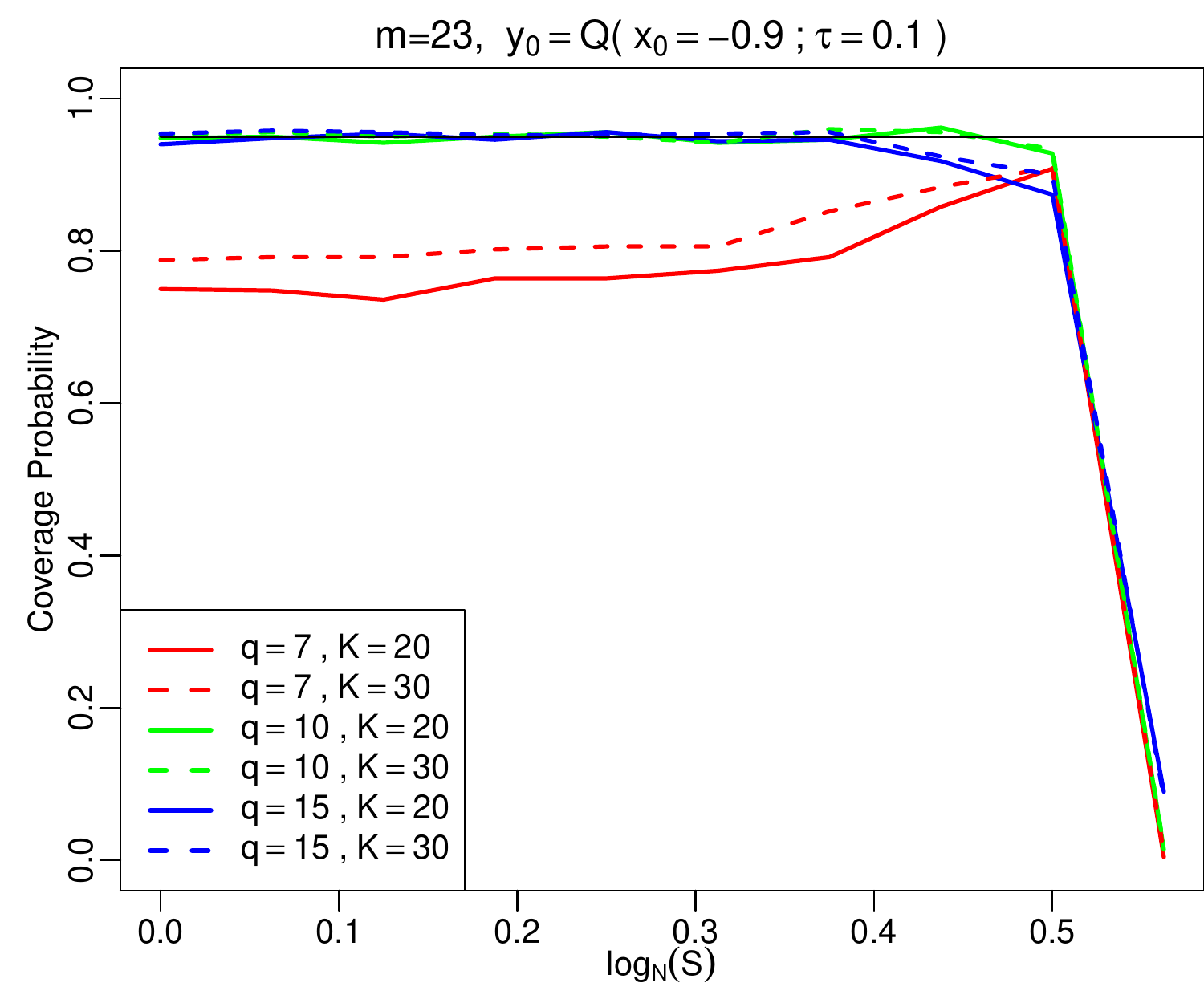}	
		\includegraphics[width=4cm, height = 3.2cm]{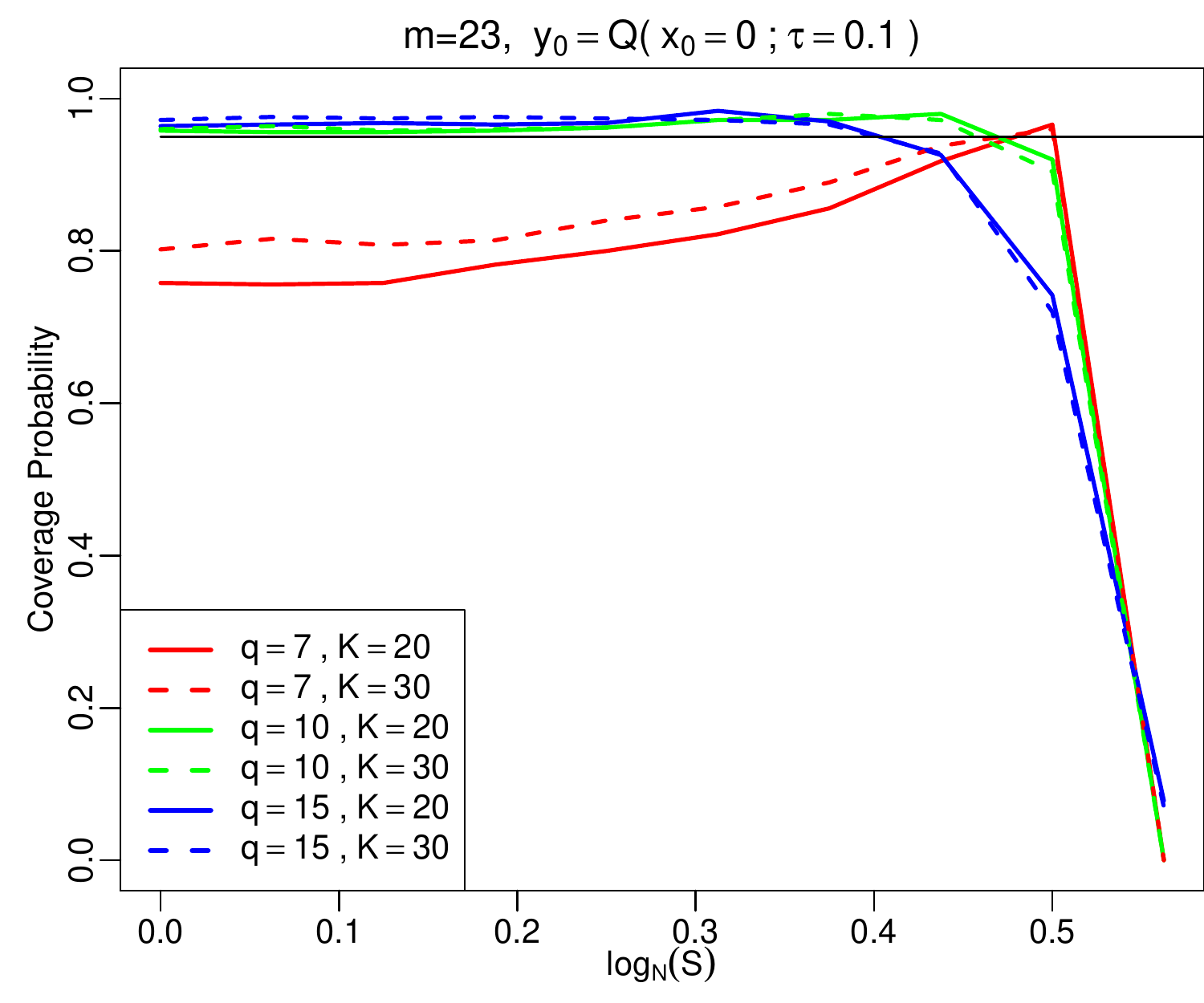}	
		\includegraphics[width=4cm, height = 3.2cm]{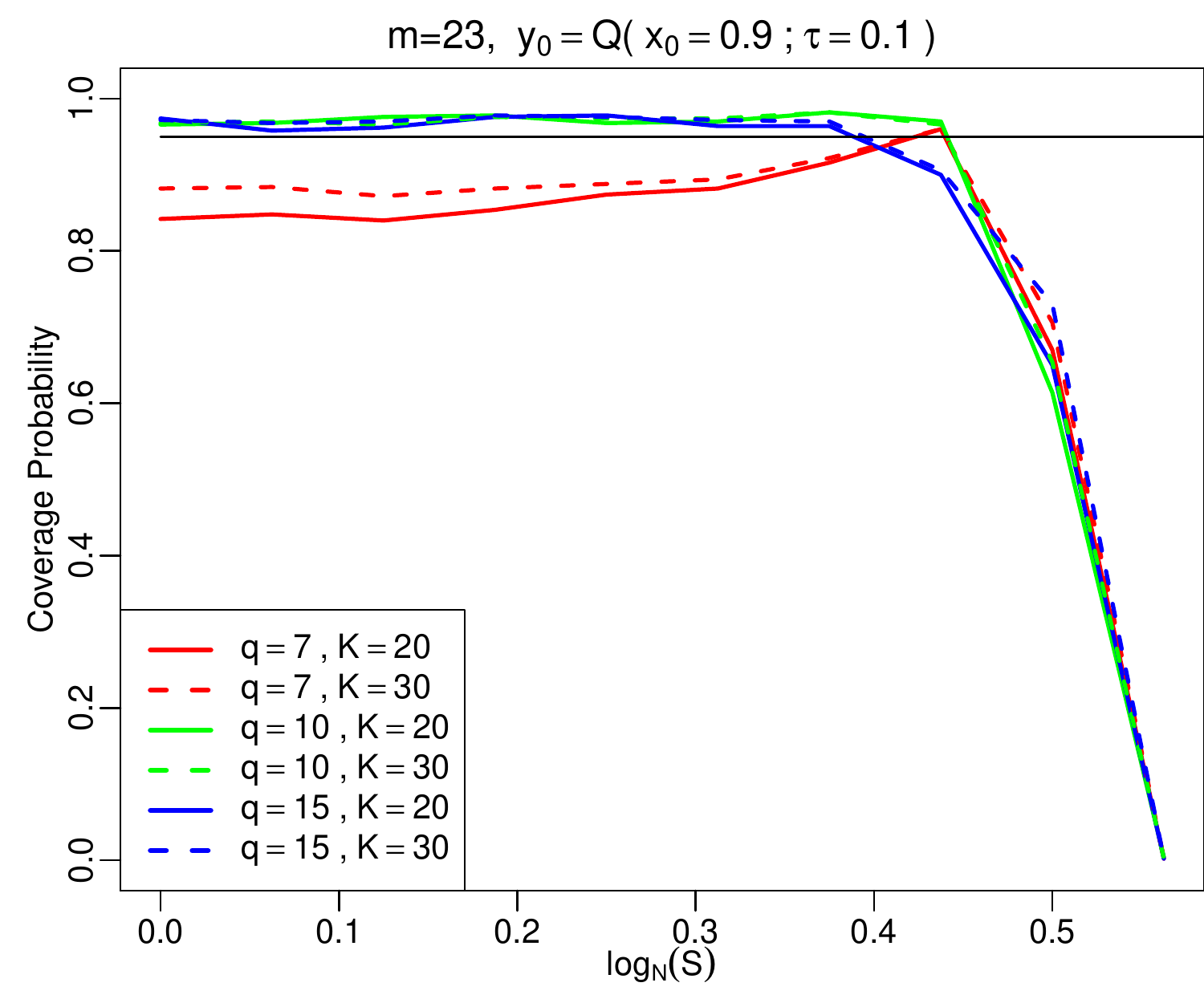}	
		\caption{Coverage probability of the $1-\alpha=95\%$ confidence interval for $F(y_0|x_0)$ in \eqref{eq:npfuncovp}: $\big[\hat F_{Y|X}(Q(x_0;\tau)|x_0) \pm N^{-1/2} \sqrt{\tau(1-\tau)} \sigma_0(\ZZ)\Phi^{-1}(1-\alpha/2)\big]$ under the nonlinear model \eqref{eq:npmodel}, where $m=\dim(\ZZ)$ (corresponds to basis expansion in $x$), $S$ is the number of subsamples, $N=2^{16}$, $q=\dim(\BB)$ (corresponds to projection in $\tau$ direction) and $K$ is the number of the quantile grid points.}\label{fig:simnpcovp_inter}
	\end{figure}

\end{document}